\documentclass[11pt]{amsbook}

\numberwithin{section}{chapter}
\numberwithin{subsection}{section}
\usepackage{amsmath}
\usepackage{amssymb}
\usepackage{amsthm}
\usepackage{eucal}
\usepackage{xcolor} 
\usepackage[matrix, arrow,curve]{xy}
\usepackage{verbatim}
\usepackage{MnSymbol}
\usepackage{hyperref}
\usepackage[utf8]{inputenc}
\usepackage{graphicx}

\newtheorem{lemma}{Lemma}[chapter]
\newtheorem{prop}[lemma]{Proposition}
\newtheorem{thm}[lemma]{Theorem}
\newtheorem{cor}[lemma]{Corollary}
\newtheorem{theorem}[lemma]{Theorem}
\newtheorem{conjecture}[lemma]{Conjecture}
\newtheorem{summary}[lemma]{Summary}
\newtheorem{thmdefn}[lemma]{Theorem/Definition}

\theoremstyle{definition}
\newtheorem{defn}[lemma]{Definition}

\newtheorem{ex}[lemma]{Example}

\newtheorem{rem}[lemma]{Remark}

\newtheorem{remark}[lemma]{Remark}
\newtheorem{notation}[lemma]{Notation}

\newcommand{\Vect}{\mathrm{-Vect}}
\newcommand{\Mod}{\mathrm{-Mod}}
\newcommand{\Zfree}{\Z\mathrm{-Proj}}

\newcommand{\End}{\mathrm{End}}
\newcommand{\Hom}{\mathrm{Hom}}
\newcommand{\Mor}{\mathrm{Mor}}
\newcommand{\Ext}{\mathrm{Ext}}
\newcommand{\Spec}{\mathrm{Spec}}

\newcommand{\Ah}{\mathcal{A}}
\newcommand{\Ch}{\mathcal{C}}
\newcommand{\Ph}{\mathcal{P}}
\newcommand{\isom}{\cong}
\newcommand{\id}{\mathrm{id}}
\newcommand{\ohne}{\smallsetminus}
\newcommand{\tensor}{\otimes}
\newcommand{\im}{\mathrm{im}}
\newcommand{\Ker}{\mathrm{ker}}
\newcommand{\Coker}{\mathrm{coker}}
\newcommand{\cone}{\mathrm{cone}}
\newcommand{\Lie}{\mathrm{Lie}}
\newcommand{\coLie}{\mathrm{coLie}}
\newcommand{\rk}{\mathrm{rk}}
\newcommand{\Pic}{\mathrm{Pic}}

\newcommand{\res}{\mathrm{res}}
\newcommand{\Gal}{\mathrm{Gal}}
\newcommand{\eff}{\mathrm{eff}}
\newcommand{\Sm}{\mathrm{Sm}}
\newcommand{\dR}{\mathrm{dR}}
\newcommand{\an}{\mathrm{an}}
\newcommand{\RGamma}{R\Gamma}
\newcommand{\RGammatilde}{R\tilde{\Gamma}}
\newcommand{\DM}{\mathsf{DM}}
\newcommand{\DA}{\mathsf{DA}}
\newcommand{\Cor}{\mathrm{Cor}}
\newcommand{\gm}{\mathrm{gm}}
\newcommand{\DMgm}{\DM_{\gm}}
\newcommand{\DMgmeff}{\DM_{\gm}^{\eff}}
\newcommand{\onemot}{1\mathrm{-Mot}}
\newcommand{\onemotgen}{1\mathrm{-MOT}}
\newcommand{\grp}{\mathcal{G}}
\newcommand{\MHS}{\mathrm{MHS}}
\newcommand{\VV}{(K,L)\Vect}
\newcommand{\VVarg}[2]{(#1,#2)\Vect}
\newcommand{\VVneu}{(\Q,\Qbar)\Vect}
\newcommand{\MMN}{\mathcal{MM}_{\Nori}}
\newcommand{\Nori}{\mathrm{Nori}}
\newcommand{\MMNeff}{\mathcal{MM}_\mathrm{Nori}^\eff}
\newcommand{\sing}{\mathrm{sing}}
\newcommand{\Per}{\mathcal{P}} 
\newcommand{\Perform}{\widetilde{\Per}} 
\newcommand{\Ann}{\mathrm{Ann}}
\newcommand{\VdR}[1]{V_\dR^\vee({#1})}
\newcommand{\Vsing}[1]{V_\sing(#1)}
\newcommand{\pairs}{\mathrm{Pairs}}
\newcommand{\pairseff}{\pairs^\eff}
\newcommand{\hodge}{\mathrm{Hdg}}

\newcommand{\Gd}{\mathrm{Gd}}
\newcommand{\ev}{\mathrm{ev}}
\newcommand{\gr}{\mathrm{Gr}}

\newcommand{\tate}{\mathrm{Ta}}
\newcommand{\baker}{\mathrm{Bk}}
\newcommand{\ab}{2}
\newcommand{\alg}{\mathrm{alg}}
\newcommand{\inc}{\mathrm{inc2}}
\newcommand{\red}{\mathrm{red}}
\newcommand{\mix}{\mathrm{inc3}}
\newcommand{\sat}{\mathrm{sat}}
\newcommand{\trans}{\mathrm{tr}}
\newcommand{\locus}{\mathcal{N}}

\newcommand{\tor}{\mathrm{tor}}

\newcommand{\Q}{\mathbb{Q}}
\newcommand{\Qbar}{{\overline{\Q}}}
\newcommand{\Z}{\mathbb{Z}}
\newcommand{\R}{\mathbb{R}}
\newcommand{\C}{\mathbb{C}}
\newcommand{\G}{\mathbb{G}}
\newcommand{\Pe}{\mathbb{P}} 
\newcommand{\A}{\mathbb{A}}
\newcommand{\B}{\mathsf{B}} 

\newcommand{\Na}{\mathbb{N}}
\newcommand{\Oh}{\mathcal{O}}
\renewcommand{\gg}{\mathfrak{g}}
\newcommand{\bb}{\mathfrak{b}}
\newcommand{\hh}{\mathfrak{h}}
\newcommand{\mm}{\mathfrak{m}}
\renewcommand{\aa}{\mathfrak{a}}

\newcommand{\Ga}{\G_a}
\newcommand{\Gm}{\G_m}
\newcommand{\uL}{\underline{L}}
\newcommand{\uG}{\underline{G}} 
\newcommand{\Uf}{\mathfrak{U}}

\makeindex
\begin{document}
\title{Transcendence and linear relations of $1$-periods}
\author{Annette Huber}
\address{Math. Institut\\ Ernst-Zermelo-Str. 1\\ 79104 Freiburg\\ Germany}
\email{annette.huber@math.uni-freiburg.de}
\author{Gisbert W\"ustholz}
\address{ETH Z\"urich\\ Department of Mathematics\\ R\"amistrasse 101\\ 8092 Zurich\\ Switzerland}
\email{gisbert.wuestholz@math.ethz.ch}
\date{April 2022}
\dedicatory{\vskip20ex {\Large In memoriam Alan Baker}}
\frontmatter
\maketitle
\begin{abstract} 
In this monograph we study  four fundamental questions about $1$-periods and give complete answers. These complex numbers are the values of integrals of rational algebraic $1$-forms  over not necessarily closed paths, or equivalently periods in cohomological degree $1$, or of Deligne $1$-motives over $\Qbar$.  
\begin{enumerate}
\item
We give a necessary and sufficient condition for a period integral to be \emph{transcendental}. We make this result explicit in the case of the Weierstraß $\sigma$-function on an elliptic curve.
\item
We give a qualitative description of all \emph{$\Qbar$-linear relations} between $1$-periods. This establishes Kontsevich's version of the Period Conjecture for such periods. 
\item
Periods may vanish and we determine all cases when this happens.1
\item For a fixed $1$-motive, we derive a general \emph{formula for the dimension} of the space of its periods  in the spirit of Baker's theorem, which appears as a very special case.
\end{enumerate}
These long-standing open questions lie in the heart of modern transcendence theory.
They look back to  a long history starting with the  transcendence of $\pi$.
\end{abstract} 
 
\tableofcontents


\chapter*{Prologue}
The study of transcendence properties of periods has a long history. It began in 1882 with the famous theorem of Lindemann on the transcendence of $\pi$ which showed that squaring the circle is not possible. 
This settled a problem more than 2000 years old from the time of the Greeks. At the same time,  he showed that $\alpha$ and $e^{\alpha}$ cannot both be algebraic unless $\alpha=0$.  In particular, $\log \alpha$ is transcendental for algebraic $\alpha\neq 0,1$. 
 Lindemann actually proved more: his method gives us that if $\alpha_1, \ldots \alpha_N$ are pairwise distinct algebraic numbers then $e^{\alpha_1}, \ldots , e^{\alpha_N}$ are linearly independent over $\Qbar$. This was carried out  in full detail, and with the  approval of Lindemann, in 1885 by 
Weierstraß in \cite{weierstrass}\footnote{see  p. 1067, footnote 2}.
 \vspace{.3cm}
 
   In his famous address at the ICM 1900,
Hilbert went further. In the seventh of his 23 problems, he asked when $\alpha, \beta$ and $\gamma=\alpha^\beta$ \,can all three be algebraic numbers. There are some obvious cases where this is true, namely when  $\alpha=0$, $\alpha=1$ or $\beta$ is rational. But Hilbert went further and asked whether these were the only cases.
He considered this problem as more difficult to prove than the Riemann hypothesis. 

To much surprise Gelfond \cite{gelfond} and Schneider \cite{S1},  applying different methods, independently succeeded in 1934  in answering Hilbert's problem. An almost equivalent formulation is that for $\alpha\neq 0$ the three numbers are algebraic  if  $\log (\alpha)$ and $\log(\alpha^\beta)$ are linearly dependent over $\Q$.  
\vspace{.3cm}

\subsection*{Linear Forms of Logarithms} 
 The work of Gelfond and Schneider initiated extensive work on so-called linear forms in logarithms.
 It was known that lower bounds for such linear forms would give solutions to several outstanding problems. One of them is the famous class number $1$ problem; another is finding effectively the integral  solutions of classes of diophantine equations. The main open problem in this context was to deal with linear forms in three logarithms of algebraic numbers with algebraic coefficients.
 
 The methods that had been developed so far could not handle more than two logarithms
and it was considered a very difficult problem to make progress on. This hurdle was overcome in the case of classical logarithms in 1966 by Baker.  In his  famous paper he was able to solve this problem in full generality not only for three  but even for any finite number logarithms. 
In particular he showed that if $\alpha_1,\ldots, \alpha_n$ are non-zero algebraic numbers, then a linear form with algebraic coefficients in  $\log\alpha_1,\ldots, \log \alpha_n$ vanishes if and only if  the logarithms are linearly dependent over $\Q$.  This result
is exactly in the spirit of Gelfond and Schneider's solutions of the Hilbert
Problem.

\subsection*{Elliptic and Abelian Integrals}
The number $2\pi i$ is a period of the integral $\int\frac{dx}{x}$ of the rational differential form $\frac{dx}{x}$ taken over a closed path in $\C$. 
Other period numbers appear in the theory of the Weierstraß elliptic function as elliptic integrals of the first kind and, consequently, it was no surprise that Siegel \cite{siegel} took up the topic. He  proved that not all periods of the Weierstraß elliptic functions with algebraic invariants $g_2$ and $g_3$ are algebraic, in particular it follows that in the complex multiplication case all non-zero periods are transcendental.

\vspace{.3cm}

Schneider further developed the theory also to deal with elliptic integrals of the first and  second kind and complete or incomplete periods. In a series of papers based on his solution of Hilbert's seventh problem \cite{S1, S2} he went as far as the technical tools of the time allowed. He proved in \cite{S3} for example that if $u$ is chosen such that the Weierstraß $\wp$-function takes an algebraic value, then 
$1$, $u$ and $\zeta(u)$ are linearly independent over $\Qbar$. In particular if $\omega$ is in the period lattice then $1$, $\omega$ and $\eta(\omega)$ are linearly independent over $\Qbar$. This was the first paper in which he proved a result about the Weiserstrass $\wp$-function and the exponential function. In particular he showed that $\frac{\pi}{\omega}$ is transcendental.

In the subsequent years Schneider extended his work, studying the transcendence properties of abelian functions and integrals and obtaining the first, albeit partial, results. The most striking example was the transcendence of the values of the $\B$-function at rational arguments, see also \cite{siegel-buch}. In his book \cite{S} he asked, as open problems, for proofs of similar results also for elliptic integrals of the third kind
and for abelian integrals. As he noted, it was clear that the methods were exhausted and new methods would be necessary for solving the problems. 
\vspace{.3cm}

In a series of papers following his work on logarithms, Baker  was also able to extend in a very limited way Schneider's results about transcendence of values of elliptic functions in the case of the Weierstraß $\wp-$ and $\zeta$-functions which contain also transcendence results on elliptic integrals of the second kind, see for example \cite{baker-goettingen}. His results were extended  to linear independence results on periods of integrals of the second kind by himself, Coates, Masser, Laurent and Bertrand. However, they were also strongly limited by the lack of a general tool for handling the case of an arbitrary number of elliptic or abelian logarithms.
The problem was that Baker's approach did not work in general in the case of abelian varieties or more generally in the case of commutative algebraic groups.
\vspace{.3cm}

This changed only  when the Analytic Subgroup Theorem \cite{wuestholz-subgroup} became available in 1982. It allowed one  to deal with $1$-periods in general and linear relations between them.
The case of  \emph{complete} periods in the general case where $\omega$ is an algebraic $1$-form on a curve of arbitrary genus and $\gamma$ a closed path on the corresponding Riemann surface, was settled in 1986 by the second author in \cite{wuestholz-icm}: If a period is non-zero, it is transcendental.
Both cases can arise. A simple example is a hyperelliptic curve whose Jacobian is isogenous to a product of two elliptic curves. 
Then $8$ of the $16$ standard periods are $0$. The others are transcendental. 

The Analytic Subgroup Theorem opened
a fruitful interplay  between transcendence theory and algebraic groups through the exponential map for Lie groups. More recently it turned out that the right frame is the theory of $1$-motives introduced by Deligne in 1974, see \cite{hodge3}. It added to the algebraic groups extra data  in the shape of a homomorphism of a free abelian group into the group, which lets one deal with so-called \emph{incomplete} periods in an elegant and natural way. We develop this point of view in full detail in the present monograph.

\subsection*{Grothendieck's Period Conjecture} 
The transcendence properties of periods make it natural to ask questions about linear and algebraic relations between them.
A conceptual interpretation of possible relations is provided by what came to be known as the \emph{Period Conjecture}. Periods are given a cohomological interpretation and all relations between them should be induced by relations between motives.

This string of ideas was started by Grothendieck in  \cite[p.~101]{grothendieck_66}. He discussed the comparison of the de Rham cohomology of a
 smooth variety $X$ over a number field $K$ with its  singular cohomology. 
The entries of the comparison matrix comparing $H^1_\dR(X)$ and $H^1_\sing(X^\an,\Q)$ of a complete non-singular curve $X$ are classical periods of the first and second kind. Grothendieck
 asks for instance if Schneider's theorem generalises in some way to these periods.\footnote{In footnote 10 Grothendieck recalls the belief that the periods
$\omega_1,\omega_2$ of a non-CM elliptic curve should be algebraically independent.     ``This conjecture extends in an obvious way to the set of periods $(\omega_1, \omega_2, \eta_1,\eta_2)$ and can be rephrased also for curves of any genus, or rather for abelian varieties of dimension $g$, involving $4g$ periods.''} 
Subsequently, he came to a conceptual conjecture and went further by  predicting the transcendence degree of the field of periods of $H^n(X)$ for
a  smooth projective variety $X$  (or more generally of a pure motive $M$) as the dimension of  the motivic Galois group or alternatively the dimension of the Mumford-Tate group of the Hodge structure on $M$.

However, he did not publish the conjecture himself. We refer to the first hand account of Andr\'e in \cite{andre_letter} on the history of the conjecture. 
A complete formulation  and discussion was finally given by Andr\'e \cite[Chapter~23]{andre2}. This includes a straightforward extension of the conjecture  to the case of mixed motives.
The formulation of the conjecture for $1$-motives is discussed by Bertolin in \cite{bertolin} and more recently \cite{bertolin2}. 
The only known result in this direction is a theorem of Chudnovsky, who showed in \cite{chudnovsky} that for any elliptic curve 
defined over $\Qbar$ at least two of the numbers $\omega_1$, $\omega_2$, $\eta(\omega_1)$ and $\eta(\omega_2)$ are algebraically independent provided that  $\omega_1$ and $\omega_2$ generate the period lattice over $\Q$. In the case of
complex multiplication this implies that $\omega$ and $\eta(\omega)$ are algebraically independent, which confirms the prediction in the CM case.
In general Grothendieck's conjecture is out of reach. We shall present an outlook on further developments below.

\vspace{.3cm}

\subsection*{The Period Conjecture as Formulated by Kontsevich and Zagier}
In a series of papers, Kontsevich and Zagier \cite{kontsevich_zagier}, \cite{kontsevich} promoted also the study of periods of non-smooth, non-projective varieties, or more generally of mixed motives. In \cite{kontsevich} Kontsevich formulated formulates an alternative version of the Period Conjecture. As Andr\'e pointed out, it has a very different flavour based on calculus rather than algebraic geometry. Relations between periods are induced by the transformation rule and Stokes's Theorem. This approach puts relative cohomology front and centre. Kontsevich
 views periods as the numbers in the image of the period pairing given by period integrals for relative cohomology
\[ H^*_\dR(X,Y)\times H_*^\sing(X,Y;\Q)\to \C\]
for algebraic varieties $X$ over $\Qbar$ and subvarieties $Y\subset X$.
By \emph{Kontsevich's Period Conjecture} all $\Qbar$-linear relations between such periods should be induced by bilinearity and functoriality of mixed motives. More explicitly, he
introduces an algebra of formal periods $\tilde{\Per}$ (also called motivic periods by some authors) with explicit generators and relations. His conjecture predicts that the evaluation map $\tilde{\Per}\to \C$ (sending a formal period to the actual value of the integral)  is injective.
In the present monograph, we give an answer in the case of periods in degree $1$, or equivalently, periods of curves.
\vspace{.3cm}

Grothendieck's Period Conjecture on algebraic relations between periods is essentially equivalent to Kontsevich's Period Conjecture on linear relations between  periods,
see e.g. \cite{ayoub}, \cite{andre3}, \cite[Section~13.2.1]{period-buch} or \cite[Section~5.3]{huber_galois} for the precise relation. 
It rests on a key insight of Nori, who realised that
 $\Spec(\Perform)$ of the algebra $\Perform$ is a torsor under a motivic Galois group, in fact the torsor of tensor isomorphisms between the de Rham realisation and the singular realisation of the category of mixed Nori motives.
 \vspace{.2 cm}

\subsection*{Dimensions of Period Spaces}  
For any given variety, the space of periods is finite dimensional. This makes it natural to ask about their dimension. A qualitative prediction already follows from the Period Conjecture.
The precedent for the kind of formula that we have in mind is Baker's theorem, \cite{baker} on logarithms.
Such a formula can be made  explicit: 
for $\beta_1,\dots,\beta_n\in\Qbar^*$ let $\langle \beta_1,\dots,\beta_n\rangle$ be the multiplicative subgroup of $\Qbar^*$ generated by these numbers. Then the dimension of the vector space generated by the principal determinations of $\log\beta_1,\dots,\log\beta_n$ over $\Qbar$ (but modulo multiples of $\pi$) is equal to the rank of the group generated by $\beta_1,\dots,\beta_n$.

In addition to  Baker's Theorem, a number of cases have been considered in the past; for example, the case of elliptic logarithms, see \cite[Chapter~6.2]{baker-wuestholz} or 
the extension of an elliptic curve by a torus of dimension $n$ in
\cite{wuestholz-trans}.
An interesting new case came up  recently in connection with curvature lines and geodesics for  billiards on a triaxial ellipsoid, see \cite{wuestholz-ell}. This leads to a period space generated by $1$, $2\pi i$ and the periods $\omega_1,\omega_2, \eta(\omega_1), \eta(\omega_2), \lambda(u,\omega_1), \lambda(u,\omega_2)$ of the first, second and third kind. Its dimension over $\Qbar$ is $8$, $6$ or $4$ depending on the endomorphisms of the elliptic curve involved and on the nature of the differential of the third kind. 
This case serves as a model for a completely general result. 

In the case of $1$-motives  this is the question about the dimension of the vector space generated over
$\Qbar$ by their periods.
It turns out that to state and prove such a general formula for the dimension of the period space of a 1-motive is difficult. The difficulties arise from periods of the third kind and the formulae we shall give are quite involved.

\subsection*{Outlook}
As we have already stated, the Period Conjecture itself seems currently far out of reach. Even the special case of values of the Riemann $\zeta$ functions  is widely open. The Period Conjecture implies that the $\zeta(2n+1)$ for $n\in\Na$ are algebraically independent. On one side we have the theory of motives and on the other hand transcendence theory. The interaction between both is a wonderful topic. It has been exploited in order to deduce upper bounds for the spaces of periods. 
However, on the transcendence side only comparitively weak lower bounds are available. Only for the case of $1$-motives over $\Qbar$ we have a complete description of the transcendental aspects. This is what our book explains. 

It is appropriate to give a short overlook on other types of motives which were studied with respect to transcendence. 
This concerns motives over $\Q$, motives over function fields, both over $\Q$ as well as over finite fields.
The more structure  the base field has the more complete gets the transcendence situation.
In the following we go through some cases of motives for which transcendence has been studied and we also mention some problems about effectivity.

\subsubsection*{Mixed Tate Motives}
The Riemann $\zeta$-function has been, since the work of Euler, one of the central objects in number theory. Euler showed that its value at positive integers $2n$ is, up to a constant, of the form $(2\pi)^{2n}$. This implies that they are transcendental over the rationals.
This is Lindemannn's theorem. The only other known fact about irrationality or transcendence of integral values of the $\zeta$-function is \'Apery's discovery of irrationality of $\zeta(3)$; see \cite{apery}.  
Only since about 2000 has there been more intensive study about these values, starting with Rivoal, Zudilin and others; see, for example, \cite{zudilin}, \cite{ball-rivoal}. They considered the space over $\Q$ generated by odd $\zeta$-values up to a
fixed integer $n$ for small $n$ and first showed that its dimension is at least $1$. More recently they have been able to prove  that the dimension tends to infinity with $n$ at a rate at least of order $\log n$.

From the other side, upper bounds on spaces generated by $\zeta$- and multi-zeta values (the periods of mixed Tate motives over $\Z$) were provided by Deligne--Goncharov in \cite{deligne-goncharov}. The reason why these results could be established is that  the motivic picture is completely understood, see also Brown's work \cite{brown-annals} on the structure of the motivic Galois group in this case.
Closely related to the  study of $\zeta$-values is the study of multi-zeta values, dilogarithms and multi-logarithms.

\subsubsection*{Motives over Function Fields over $\Q$}
Ayoub reformulated Kontsevich's Period Conjecture with fewer generators. Only polydisks are needed as domains of integration. Based on this description, he was able to formulate and prove a function field version of the conjecture in \cite{ayoub2}. 
 
For a closed polydisc $\overline{\mathbb D}^n$  
he considered the subspace 
\[\Oh_{\mathrm{alg}}^\dagger(\overline{\mathbb D}^n)\subset \Oh(\overline{\mathbb D}^n)[[T]][T^{-1}]\]
 of Laurent series $F=\sum_{i>-\infty} f_i(z_1,\ldots,z_n) T^i$  with coefficients in 
$\Oh(\overline{\mathbb D}^n)$ which are algebraic over $\C(T,z_1,\ldots,z_n)$. The dimension $n$ is allowed to vary and 
$\Oh_{alg}^\dagger(\overline{\mathbb D}^\infty)
=\bigcup_{n\in \mathbb N}\Oh_{alg}^\dagger(\overline{\mathbb D}^n).$
 In analogy to Kontsevich's space of formal periods $\Perform$, he defined $\mathcal P^\dagger$ as a quotient of  $\Oh_{alg}^\dagger(\overline{\mathbb D}^\infty)$
by certain relations and showed that there is an evaluation map $\mathcal P^\dagger \rightarrow \C((T))$. The main result (and geometric analogue of the Period Conjecture) is the injectivity of this evaluation map. There is also independent work of Nori (unpublished) in the same direction.

\subsubsection*{Motives over Function Fields over Finite Fields}
All that is known about transcendence over $\Q$ is also known for function fields over finite fields: indeed often more  is known.  Let $p$ be a prime, $q=p^n$ and $\mathbb F_q[x]$ the ring of polynomials in one variable 
over the finite field $\mathbb F_q$.
In 1935 Carlitz introduced  the so-called Carlitz $\psi$-functions in
\cite{carlitz},  defined  as 
\[
\psi(t)= \sum_0^\infty \frac{(-1)^k}{F_k} t^{q^k}
\]
where $F_k=[k][k-1]^{\,q} \cdots [1]^{\,q^{k-1}}$ and $[k] =x^{q^k} - x$.
On the basis of these functions Wade started studying transcendence theory 
over function fields over finite fields  \cite{wade}. Their minimal algebraic closure is complete with respect to the standard valuation so that the function $\psi(x)$ exists and is a replacement of the exponential function. Its inverse exists as a multi-valued function and is the analogue of the logarithm. Wade proved among other things the  analogue of the theorem of Gelfond and Schneider. 

In 1983 Jing Yu took up the topic and proved the analogue of Lindemann's theorem in the realm of Drinfeld modules and started a very interesting transcendence theory for Drinfeld modules and $t$-motives. As a highlight of a sequence of papers including periods and quasi-periods of Drinfeld modules as well as special zeta values in characteristic $p$,  he  obtained an analogue of the Analytic Subgroup Theorem for Drinfeld modules and, more generally, Anderson's $t$-motives.
 Once the theory and the techniques had been established, a whole spectrum of applications
followed, including linear independence of zeta-values, by Yu, Chieh-Yu Chang, Papanikolas and Thakur; see \cite{yu}, \cite{chang-yu}, \cite{carlitz-zeta}. They even were able to determine the transcendence degree of fields generated by logarithms or zeta-values in this setting. The survey paper of Chang \cite{chang-iccm} gives a very nice and substantial report about the newest achievements in this theory.

 \subsubsection*{Hypergeometric Period Relations, Periods of Higher Weight}
It is well-known that values of hypergeometric functions can be expressed as quotient of two abelian integrals, in general of the second kind. This leads to a period relation between the two periods of the second kind with the hypergeometric function as coefficient. Algebraic values of the hypergeometric function provide  linear relations between the two periods with algebraic coefficients. 
This cannot be true in general and leads to special  points on certain Shimura varieties as explained very carefully in Chapter~5 of Tretkoff's beautiful monograph \cite{tretkoff}.  In particular new transcendence results are given for  the Appel-Lauricella (hypergeometric) functions in $n \geq 2$ variables. They exceed the known results on the values of  the classical hypergeometric function in one variable. 

\subsubsection*{Hodge Level $1$}

An obvious problem is to extend the transcendence results to periods of higher weight. In general this seems to be a hopeless undertaking. However, there are cases when periods of higher weight can be related to $1$-periods.  Tretkoff gives some nice examples dealing with periods of Fermat hypersurface.
This is the case for certain algebraic $K3$-surfaces and  smooth complete intersections over $\C$ of Hodge level~$1$ as explained in  \cite{D2}; see also \cite{wuestholz-icm}.  

We are not going to expand our monograph in this direction but mention some interesting research 
dealing with this kind  of problems. The starting point was given in \cite{wuestholz-icm} and then taken up by 
Tretkoff.

In Chapter~6 of her monograph \cite{tretkoff}  P. Tretkoff discusses among other things algebraic K3 surfaces $X$
defined over a number field. She considers a holomorphic 2-form $\omega$ on $X$ and shows that if the vector space generated by the periods $\int_\gamma\omega$ for $\gamma \in H_2(X,\Z)$ has dimension~$1$, then $X$ has complex multiplication, i.e. its Mumford-Tate group is abelian. This implies that if $X$ has complex multiplication these periods are all transcendental unless $\gamma=0$. 

In Chapter~7 she deals with arbitrary smooth projective varieties $X$ defined over $\Qbar$.  One of the questions she raises is whether the Hodge filtration of $H^k(X,\C)$ for $0 \leq k \leq \dim X$ has complex multiplication if it is defined over $\Qbar$. She gives some nice examples dealing with periods on Fermat hypersurfaces.

All this is restricted to holomorphic differential forms and complete periods. In our monograph we deal with meromorphic differential forms and incomplete periods. It would be interesting to try to get more general cases involving incomplete periods.

\subsubsection*{Effectivity, Lower Bounds}
The main applications of Baker's work on linear forms in logarithms were the lower bounds he derived. He showed that if $\Lambda$ is a linear form in logarithms of algebraic numbers
with algebraic coefficients a lower bound for the absolute value of $\Lambda$ can be obained.
For a detailed account on this see \cite{baker-wuestholz}. A similar theory exists also for the
$p$-adic analogue, started by Coates and finally brought to the level of the archimedean case by
Kunrui Yu. Similar results were also obtained for elliptic and abelian logarithms. One might ask whether this can be extended to $1$-motives in a modified way. First steps in this direction were the work of Masser and Wüstholz \cite{masser-wuestholz-annals}  on isogeny estimates. It says that, given an isogeny between two abelian varieties, there exists an isogeny with degree bounded by the original data: height of
the source isogeny, degree of the number field, dimension of the abelian variety. The result is 
completely effective and a crucial ingredient for the proof of the famous Tate conjecture.
It has also been used 
for the proof of the Andr\'e--Oort conjecture for the coarse moduli space of prinicipally polarised abelian varieties \cite{ts} by Tsimerman. It would be interesting to formulate an isogeny estimate type statement for $1$-motives and to find applications in diophantine geometry.
\chapter*{Acknowledgments}

The two authors thankfully acknowledge the hospitality of the Freiburg Institute for Advanced Study, in particular the Research Focus Cohomology in Algebraic Geometry and Representation Theory, in the academic year 2017/18. Most of the book  was written during this stay. 

We thank Fritz H\"ormann, Florian Ivorra and Simon Pepin-Lehalleur for discussions on $1$-motives and the generalised Jacobian.
We also thank Cristiana Bertolin for comments on an earlier version and
J\"urgen Wolfart for help with hypergeometric series. Alin Bostan explained the corrected proof of the explicit formula for values of the hypergeometric series in Chapter~\ref{sec:hyper} to us. The argument for an alternative proof of
Theorem~\ref{thm:transc_curves} was pointed out to us by David Masser.
We explain a variant at the end of the chapter.

We are grateful to Yves Andr\'e for his historical comments, reading some part of the manuscript in detail and graciously sharing his insights on the Period Conjecture.

Finally, the first author would like thank all participants of her 2020 lecture on the Analytic Subgroup Theorem, in particular Fabian Glöckle, for their lively interest, comments and corrections.

\mainmatter
\chapter{Introduction}

This introduction aims to present the principal actors of the book and to explain the main results of our monograph. We begin with the question about transcendence of periods of integrals of 1-forms over closed or non-closed paths. Historically integrals over non-closed paths were not considered as periods.
The change came from looking at them in the relative cohomology. This leads us to distinguish between complete and incomplete periods.
 
The presentation does not follow the order in the main text, but is, we hope, designed to help those readers without a background in transcendence.

\section{Transcendence}

The vector space $\Per^1$ over $\Qbar$  of one-dimensional periods, complete or incomplete,  has a number of different descriptions. In the most elementary situation its elements are given by the period integrals 
\[ \alpha=\int_\sigma\omega\]
 where
\begin{itemize}
\item 
$X$ is a smooth projective curve over $\Qbar$;
\item $\omega$ is a rational differential form on $X$;
\item $\sigma=\sum_{i=1}^na_i\gamma_i$ is a chain in the Riemann surface $X^\an$
defined by $X$ which avoids the singularities of $\omega$ and has boundary divisor
$\partial \sigma$ in $X(\Qbar)$; in particular $\gamma_i:[0,1]\to X^\an$ is
a path and $a_i\in\Z$.
\end{itemize}
This set includes many interesting numbers like $2\pi i$, $\log\alpha$  for
algebraic $\alpha$ and the periods of elliptic curves over $\Qbar$. 
We study their transcendence properties.

The case of  \emph{complete} periods in the general case, i.e. $X$ and $\omega$ arbitrary,  $\gamma$ closed,   was settled in 1986 by the second author in \cite{wuestholz-icm}: if a period is non-zero, it is transcendental.
Both cases can arise. A simple example is a hyperelliptic curve whose Jacobian is isogenous to a product of two elliptic curves. 
Then $8$ of the $16$ standard periods are $0$. The others are transcendental. 

When $X$ is an elliptic curve we refer to \cite[Chapter~6.2]{baker-wuestholz} for the case of \emph{incomplete} periods. The general case has been described as an open problem in \cite{wu84bis}.   Often the values are transcendental, e.g.
$\int_1^2dz/z=\log 2$ but certainly not always, e.g.,  $\int_0^2dz=2$. 
Again, it is not difficult to write down a list of simple cases  in which the period is a non-zero algebraic number. However, it was not at all clear whether the list was complete and what the structure behind the examples was; see \cite{wuestholz}. The answer that we give now is surprisingly simple:

\begin{thm}[See Theorem~\ref{thm:transc_curves}]\label{thm:transc_curves_intro}
Let $\alpha=\int_\sigma\omega$ be a one-dimensional period on $X$. Then
$\alpha$ is algebraic if and only if
\[ \omega=df+\omega'\]
 with
$f\in\Qbar(X)^*$ and $\int_\sigma\omega'=0$ with $\omega'$ a form with no extra poles.
\end{thm}
The condition is clearly sufficient because the integral evaluates
to 
\[ \sum_ia_i(f(\gamma_i(1))-f(\gamma_i(0)))\in\Qbar\]
 in this case. 

Theorem \ref{thm:transc_curves} gives a complete answer to two of the seven problems listed in Schneider's book \cite[p.~138]{S},
\footnote{Problem 3. Es ist zu versuchen, Transzendenzresultate \"uber elliptische Integrale dritter Gattung zu beweisen.}
\footnote{Problem 4. Die Transzendenzs\"atze \"uber elliptische Integrale erster und zweiter Gattung sind in weitestm\"oglichem Umfang auf analoge S\"atze \"uber abelsche Integrale zu verallgemeinern.} open for more than 60 years.
Actually we even include periods of abelian integrals of the third kind.


\section{Relations between Periods}

Questions on transcendence can be viewed as a very special case of
the question on $\Qbar$-linear relations between $1$-periods:
a complex number is transcendental if it is $\Qbar$-linearly independent of
$1$. The most general problem of this kind is to determine the dimension of the period space generated over $\Qbar$ by the periods of all rational $1$-forms of an algebraic variety. It is easy to give an upper bound for this dimension in terms of cohomological data. The problem is then  to decide whether the upper bound is the correct number or whether there are linear relations between periods. 
\vspace{.3cm}

This fundamental question will be one of the central topics in this monograph. We establish a complete description of the linear relations  between (not necessarily complete) 
periods for all rational differential forms of degree $1$.
It is crucial to use here the  more conceptual descriptions of $\Per^1$ either as periods in cohomological degree $1$ or as cohomological periods of curves, or even better periods of $1$-motives. 

The following theorem gives a first answer. It  establishes Kontsevich's version of the Period Conjecture for $\Per^1$ and  furnishes a qualitative description of the period relations.

\begin{thm}[Kontsevich's Period Conjecture for $\Per^1$, Theorem~\ref{thm:main_kontsevich}]\label{thm:kont_pc}
All $\Qbar$-linear relations between elements of $\Per^1$ are induced by bilinearity and
functoriality of pairs $(C,D)$ where $C$ is a smooth affine curve over $\Qbar$ and
$D\subset C$ a finite set of points over $\Qbar$. 
\end{thm}

The conjecture  has an alternative formulation in terms of motives. Actually, we deduce
Theorem~\ref{thm:kont_pc} from the motivic version below together
with the result of Ayoub and Barbieri-Viale in \cite{ayoub-barbieri} which says that
the subcategory
of $\MMN^\eff$ generated by $H^*(C,D)$ with $C$ of dimension at most $1$
agrees with Deligne's much older category of $1$-motives; see \cite{hodge3}.

\vspace{.3cm}

Every $1$-motive $M$ has  a singular realisation $\Vsing{M}$ and a de Rham realisation $V_\dR(M)$. They are linked via a period isomorphism 
\[ \Vsing{M}\tensor_\Q\C\isom V_\dR(M)\tensor_\Qbar\C.\]
There is a  well-known relation between curves and $1$-motives provided by
the theory of generalised Jacobians.
From this fact we see that  the set
$\Per^1$ has another alternative description as the union of the images of the period pairings
\[ \Vsing{M}\times \VdR{M}\to \C\]
for all $1$-motives $M$ over $\Qbar$.

\begin{thm}[{Period conjecture for $1$-motives, Theorem~\ref{thm:main_one}}]\label{thm:main_one_intro}
All $\Qbar$-linear relations between elements of $\Per^1$ are induced
by bilinearity and functoriality for morphisms of iso-$1$-motives over
$\Qbar$.
\end{thm}

The theorem does not say anything about the actual dimension of the period space. We need a quantitative answer. In other words the space of relations has to be determined. It turns out that to find the answer  is rather difficult in some cases.

\section{Dimensions of Period Spaces}\label{background}

The above qualitative theorems can be refined into an  explicit computation of the dimension $\delta(M)$ of the $\Qbar$-vector space generated by the periods of a given $1$-motive $M$. 
The result depends on the subtle and very unexpected interplay between the constituents of $M$.
\vspace{.3cm}

Not only for the proofs, but also for the very formulation of the dimension formulas, we rely on the theory of $1$-motives introduced by Deligne; see \cite{hodge3}. They form an abelian category that captures all cohomological properties of algebraic varieties in degree $1$, including all one-dimensional periods. 

We review the basics: a $1$-motive over $\Qbar$ is a complex $M=[L\to G]$ where
$G$ is a semi-abelian variety over $\Qbar$ and $L$ a free abelian group of finite rank. The map is a group homomorphism. As mentioned before, every $1$-motive has
de Rham and singular realisations and a period isomorphism between them after extension of scalars to $\C$.  

If $C$ is a smooth curve over $k$, $D\subset C$ a finite set of $\Qbar$-points,
then there is a $1$-motive $M_1(C)$ such that  $H_1^\sing(C^\an,D;\Q)$ agrees
with the singular realisation of $M_1(C)$, and  $H^1_\dR(C,D)^\vee$ agrees with
the de Rham realisation of $M_1(C)$. Hence the periods of the pair $(C,D)$
agree with the periods of $M_1(C)$. Explicitly, $M_1(C)=[\Z[D]^0\to J(C)]$ where
$J(C)$ is the generalised Jacobian of $C$ and $\Z[D]^0$ means divisor of degree $0$ supported on $D$.

We denote by $\Per(M)$ the image of the period pairing for $M$ and by $\Per\langle M\rangle$ the abelian group (or, equivalently $\Qbar$-vector space) generated by $\Per(M)\subset \C$.
\vspace{.3cm}

We fix a $1$-motive $M=[L\to G]$ with $G$ an extension of an abelian variety $A$ by a torus $T$ and $L$ a free abelian group of finite rank. For the definition of its singular realisation $\Vsing{M}$ and its de Rham realisation
$\VdR{M}$, we refer to Chapter~\ref{sec:one-mot}.

The weight filtration on $M$, explicitly given by
\[ [0\to T]\subset [0\to G]\subset [L\to G]\]
induces 
\[ \Vsing{T}\hookrightarrow \Vsing{G}\hookrightarrow \Vsing{M}.\]
and dually
\[ \VdR{M}\hookleftarrow\VdR{[L\to A]}\hookleftarrow\VdR{[L\to 0]}.\]
Together, they introduce a bifiltration 
\[\begin{xy}\xymatrix{
 \Per\langle T\rangle \ar@{^{(}->}[r]&\Per\langle G\rangle\ar@{^{(}->}[r]&\Per\langle M\rangle\\
         &\Per\langle A\rangle\ar@{^{(}->}[r]\ar@{^{(}->}[u]&\Per\langle [L\to A]\rangle\ar@{^{(}->}[u]&\\
         &&\Per\langle [L\to 0]\rangle\ar@{^{(}->}[u]&
}\end{xy}\]
on $\Per\langle M\rangle $.
\vspace{.3cm}

We introduce the notation and terminology
\begin{align*}
\Per_\tate(M)&=\Per\langle T\rangle&&\text{Tate periods,}\\
\Per_\ab(M)&=\Per\langle A\rangle &&\text{2nd kind wrt closed paths,}\\
\Per_\alg(M)&=\Per\langle [L\to 0]\rangle&&\text{algebraic periods,}\\
\Per_3(M)\ \ &=\Per\langle G\rangle/(\Per_\tate(M)+\Per_\ab(M))&&\text{3rd kind wrt closed paths,}\\
\Per_\inc(M)&=\Per\langle[L\to A]\rangle/(\Per_\ab(M)+\Per_\alg(M))&&\text{2nd kind wrt non-cl. paths,}\\
\Per_\mix(M)&=\Per\langle M\rangle/(\Per_3(M)+\Per_\inc(M))&&\text{3rd kind wrt non-cl. paths,}
\end{align*}
where wrt and non-cl. are abbreviatios for 'with respect to' and'non-closed'.
After choosing bases, we can  organise the periods into a period matrix of the form
\[\left(\begin{matrix}
  \Per_\tate(M)&\Per_3(M)&\Per_\mix(M)\\
  0 &\Per_\ab(M)&\Per_\inc(M)\\
 0&0&\Per_\alg(M)\end{matrix}\right)\]
The contribution of $\Per_\tate(M)$ (multiples of $2\pi i$) and $\Per_\alg(M)$
(algebraic numbers) is readily understood.
Note that the off-diagonal entries are only well-defined up to periods on the diagonal. This can also be seen in the case of Baker periods, which are contained in $\Per_\mix(M)$ for special $M$. The value
of $\log\alpha$ depends on the chosen path and is only-well-defined up to
multiples of $2\pi i$.
The total dimension is obtained by adding up these dimensions.
In particular, we have e.g.
\[ \Per\langle [L\to A]\rangle\cap \Per\langle [0\to G]\rangle =\Per\langle [0\to A]\rangle.\]

The complete result takes a rather complicated form. In order to state it we 
write $\delta(M)=\dim\Per\langle M\rangle$ and $\delta_?(M)=\dim\Per_?(M)$ for the different entries of the period matrix. If $B$ is a simple abelian variety,  $g(B)$ will be its genus and $e(B)$ the $\Q$-dimension of
$\End(B)_\Q$. We also need the invariants $\rk_B(L,M)$, $\rk_B(T,M)$ as introduced in Notation~\ref{not:data_all}.

\begin{thm}[{Corollary~\ref{cor:sum}, Proposition~\ref{prop:list}}]
We always have
\[ \delta(M)=\delta_\tate(M)+\delta_\ab(M)+\delta_\alg(M)+\delta_3(M)+\delta_\inc(M)+\delta_\mix(M).\]
 \begin{enumerate}
\item 
All Tate periods are $\Qbar$-multiples of $2\pi i$. 
All algebraic
periods are in $\Qbar$. In particular $\delta_\tate(M)$ and $\delta_\alg(M)$ take the values $0$ or $1$, depending on the (non)-vanishing of $T$ and $L$.
\item We have
\[ \delta_\ab(M)=\sum_B \frac{4g(B)^2}{e(B)}\]
where the sum is taken over all simple factors of $A$, without multiplicities.
\item We have
\begin{align*}
 \delta_3(M)&=\sum_B 2g(B)\rk_{B}(L,M)\\[2ex]
  \delta_\inc(M)&=\sum_B2g(B)\rk_B(T,M)
\end{align*}
\end{enumerate}
\end{thm}

 In the special
case $A=0$, we get back Baker's theorem. 
The most interesting and hardest contribution is $\Per_\mix(M)$. The computation of this contribution was not possible without the methods that we develop here. Up to particular cases the formulae for the other contributions were not in the literature either. For an overview seee for example\cite[Chapter~6.2]{baker-wuestholz}, \cite{wu84bis} \cite{wuestholz} and \cite{wuestholz-ell}.

The formula for $\Per_\mix(M)$ simplifies in the case of motives that we call saturated,
see Definition~\ref{defn:red-sat}. 

\begin{thm}[{Theorem~\ref{thm:dimension_sat}}]
If $M=M_0\times M_1$ is the product of a Baker motive $M_0=[L_0\to T_0]$, i.e. with 
vanishing abelian part, and a saturated motive $M_1=[L_1\to G_1]$, then
\[ \delta_\mix(M)=\rk_\gm(L,M_1)+\sum_Be(B)\rk_B(G_1,M_1)\rk_B(L_1,M_1)\]
\end{thm}

Fortunately,  by Theorem~\ref{thm:dimension_sat}~(\ref{it:contained}) the periods of a general motive  are always included
in the period space of $M_0\times M_\sat$ with $M_0$ of Baker type ($A_0=0$)
and $M_\sat$ saturated.

There is a precise recipe for $\delta_\mix(M)$ for any $1$-motive $M$. It is spelled out in Chapter~\ref{ch:precise}, in particular Theorem~\ref{thm:precise1}.
We also refer to Chapter~\ref{ch:ex}  for the examples of elliptic curves without  and with CM.

\section{Method of Proof}\label{ssec:method}

As  in the case of closed paths, the main ingredient of our proof (and the only input from transcendence
theory) is the Analytic Subgroup Theorem of \cite{wuestholz-subgroup}. We
give a reformulation as Theorem~\ref{thm:annihilator}:
given a smooth connected commutative algebraic group over $\Qbar$ and $u\in\Lie(G^\an)$ such that $\exp_G(u)\in G(\Qbar)$, there is a canonical short
exact sequence 
\[ 0\to G_1\to G\xrightarrow{\pi} G_2\to 0\]
of algebraic groups over $\Qbar$
such that $\Ann(u)=\pi^*(\coLie(G_2))$ and $u\in \Lie(G_1^\an)$.
 Here $\Ann(u)\subset \coLie(G)$ is the largest subspace
such that $\langle \Ann(u),u\rangle=0$ under the canonical pairing.

Given a $1$-motive $M$, Deligne constructed a vector extension $M^\natural$ of $G$ such that $V_\dR(M)=\Lie(M^\natural)$. This is the group
we apply the Subgroup Theorem to. 

\begin{thm}[Subgroup Theorem for $1$-motives, Theorem~\ref{thm:annihilator_mot}]
Given
a $1$-motive $M$ over $\Qbar$ and $u\in\Vsing{M}$, there is  a short exact sequence
of $1$-motives up to isogeny
\[ 0\to M_1\xrightarrow{i} M\xrightarrow{p} M_2\to 0\]
such that $\Ann(u)=p^*\VdR{M_2}$ and $u\in i_*\Vsing{M_1}$.
Here $\Ann(u)\subset\VdR{M}$ is the left kernel under the period pairing.
The sequence is uniquely determined by these properties.
\end{thm}
Given a pair of non-zero $u\in \Vsing{M}$ and $\omega\in\VdR{M}$ with vanishing period, the theorem provides a proper submotive $M_1$ of $M$ such that
$u=i_*u_1$ for $u_1\in \Vsing{M_1}$ and $\omega=p^*\omega_2$ for $\omega_2\in\VdR{M_2}$.
Any $\Qbar$-linear relation between periods
can be translated into the vanishing of a period. Then the Subgroup Theorem
for $1$-motives is applied.

As a byproduct, we also get a couple of new results on $1$-motives over
$\Qbar$: they are a full subcategory of the category of $\Q$-Hodge structures
over $\Qbar$ (see  Proposition~\ref{prop:hodge_ff}) and
of the category of (non-effective) Nori motives (see Theorem~\ref{thm:ff})
and of the category $\VVneu$ of pairs of vector spaces together with a period matrix. The last statement was also obtained independently by Andreatta, Barbieri-Viale and Bertapelle, see \cite{ABB}. The case of Hodge structures has just recently been considered by Andr\'e in \cite{andre4}. He proves that
the functor from $1$-motives into $\Q$-Hodge structures is fully faithful for all algebraically closed fields $k\subset \C$.

\section{Why $1$-Motives?}

This seems the right moment to address the question if our emphasis on $1$-motives is necessary. We think that the answer is yes. 

Obviously, all proofs using $1$-motives could be rewritten in terms of commutative algebraic groups because this is how the Subgroup Theorem for $1$-motives itself is deduced. However, the dimension formulas depend on the constituents of the $1$-motive and do not admit a transparent formulation in terms of the constituents of the algebraic group. 

More generally, $1$-motives are the link between the classical objects of transcendence theory \`a la Lindemann, Schneider, or Baker and the structural
predictions linked with  Grothendieck, Andr\'e, or Kontsevich. 

\section{The Case of Elliptic Curves}

The above results are very general and depend on a subtle interplay between the data. It is a non-trivial task to make them explicit in particular examples. We have carried this out to some extent in the case of an elliptic curve $E$ defined over $\Qbar$.

Recall the Weierstraß $\wp$-, $\zeta$- and $\sigma$-function for $E$. We obtain:

\begin{theorem}[see Theorem~\ref{thm:sigma}]
Let $u\in\C$ be such that $\wp(u)\in \Qbar$ and $\exp_E(u)$ is non-torsion in $E(\Qbar)$. Then
\[ u\zeta(u)-2\log \sigma(u)\]
is transcendental.
\end{theorem}
This is an incomplete period integral of the third kind. The proof of the above result actually is not a direct consequence of Theorem~\ref{thm:transc_curves_intro} but rather uses the insights of our dimension computations.

We also carry out the dimension computation in this case:
let $M=[L\to G]$ with $L\isom \Z$, $G$ an extension of $E$ by 
$\Gm$ that is non-split up to isogeny, $L_\Q\to E_\Q$ injective.
Then by Proposition~\ref{prop:ell} and \ref{prop:with_CM}
\[ \dim\Per\langle M\rangle =\begin{cases} 11& E\text{\ without CM},\\
                                          9 &E\text{\ CM}.

                           \end{cases}\]
The incomplete periods of the third kind become more difficult already if we consider  $M=[L\to G]$ with $L\isom\Z^2$, $G$ an extension of $E$ by $\Gm^2$, again
$L_\Q\to E_\Q$ injective and $G$ completely non-split up to isogeny. 
If $E$ does not have CM, then
\[ \dim\Per\langle M\rangle = 18.\]
If $E$ is CM, then
\[ \dim\Per\langle M\rangle = 16,14,12,10,\]
depending on the interplay of the complex multiplication and $L$ and $G$.
The extreme case occurs when $\End(M)$ is the CM-field. Then the resulting
dimension is $10$.

\section{Values of Hypergeometric Functions}
Euler had already known that the hypergeometric function $F(a,b,c;z)$ can be written as a quotient of two integrals. If $a,b,c$ are rational numbers, these integrals can be regarded as periods on certain explicit algebraic curves. Knowledge about $\Qbar$-linear indepdendence of periods translates then into transcendence statements for the values $F(a,b,c;\lambda)$ for $\lambda\in\Qbar\ohne\{0,1\}$. This insight is exploited by
Wolfart's in
\cite{wolfart} or by Chudnovski--Chudnovsky in \cite{chudchud}. 
We explain the argument in detail for $a=b=1/2$ and $c=1$:

\begin{prop}[{Proposition~\ref{prop:1/2}}]
The value
$F(1/2,1/2,1;z)$ of the hypergeometric function is transcendental for $z\in\Qbar\ohne\{0,1\}$. 
\end{prop}
The proposition follows from the $\Qbar$-linear independence of $\pi$ and the complete periods of elliptic curves established first by Schneider in 1936, see \cite[Satz~IIIa]{S3}.

In the case of general $a,b,c\in\Q$ with least common denominator $N$, the Euler integrals can be seen as periods for the algebraic curve with affine equation
\[ y^N=x^r(1-x)^s(1-\lambda x)^t\]
for suitable $r,s,t$. For the formula in the case of $N=p$ a prime see Proposition~\ref{prop:abc}. These curves have been intensely studied. Using results of Gross--Rohrlich  \cite{gross-rohrlich}, Archinard \cite{arch2} and Asakura--Otsubo \cite{AO} we work out another example. 

\begin{thm}[Corollary~\ref{cor:jetzt_aber}]
Let $p$ be a prime such that $p\nequiv 1\mod 3$, $1\leq n\leq p-1$. Let $
0<r,s<p$ such that $p$ does not divide $r+s$, put $t=p-s$ and
\[
u=\left[\frac{nr}{p}\right], \;
v=\left[\frac{ns}{p}\right], \; w= \left[\frac{nt}{p}\right].
\]
We further assume
\[ \langle \frac{nr}{p}\rangle +\langle \frac{ns}{p}\rangle-\langle \frac{n(r+s)}{p}\rangle\neq 1.\]
Then for all $\lambda\in\Qbar\ohne\{0,1\}$, the corresponding value $F(a,b,c;\lambda)$ is zero or transcendental and transcendental if $\lambda\in (0,1)$.
\end{thm}
An explicit example where the assumptions are satisfied is
$p=11$, $r=s=2$, $n=1,2,6,7,8$. We deduce, for example, that the numbers
$F(6/11,6/11,12/11;\lambda)$ are zero or transcendental, provided $\lambda\in\Qbar\ohne\{0,1\}$.

We should stress that this application only relies on complete periods on abelian varieties and not on the more general theory developed in our monograph. It should be seen as a proof of concept: the same argument can be applied to other geometric families of curves, allowing families of differential forms of the third kind and non-closed paths. The hyergeometric function would be replaced by the solutions of differential equation defined by the Gauß--Manin connection.

\section{Structure of the Monograph}
We have tried to make the monograph accessible to readers not familiar with either motives or periods.

The first part provides foundational material that will be used throughout, fore example terminology from category theory, a review of the theory of the generalised Jacobian or the basics on singular homology and de Rham cohomology. We provide precise references for the facts that we need later. Along the way we also fix notation and normalisations. Depending on their background, readers are invited to skip on some or all of these chapters and only use them for reference. 

Chapter~\ref{ch:subgroup} and Chapter~\ref{ch:formalism} address less classical material. The first deduces a reformulation or our key tool, of the Analytic Subgroup Theorem. We apply it to the comparison between analytic and algebraic homomorphisms between connected commutative algebraic groups. 

Chapter~\ref{ch:formalism} presents an abstract formulation of the theory of periods and the Period Conjecture for abelian categories without a tensor structure.

Part~\ref{part2} is the heart of the monograph and presents our main result.
It addresses periods of $1$-motives.
After settling some notation, 
Chapter~\ref{sec:one-mot}  starts by reviewing Deligne's category of $1$-motives and its properties. We then establish auxiliary results that are needed in the next chapter.

Chapter~\ref{sec:periods_eins} discusses  periods of $1$-motives and proves the version
of the Period Conjecture purely in terms of $1$-motives. We then
consider examples: in Chapter~\ref{ch:first} we treat the classical cases
like the transcendence of $\pi$ and values of $\log$ in our language.
In Chapter~\ref{ch:ex}
we apply the general results in the
case of a $1$-motive whose constituents are as small as possible without being trivial and compute the dimensions of their period spaces.

In Part~\ref{part3} we turn to periods of algebraic varieties. 
Chapter~\ref{sec:nori} clarifies the notion of a
cohomological period. After defining $\Per^1$ in a down to earth way,
the interpretation of cohomological periods as 
the periods of $1$-motives is explained. Finally, we explain the interpretation as periods of  Nori or Voevodsky motives.

In Chapter~\ref{sec:results} we use the results on periods of $1$-motives to
deduce the qualitative results on $\Per^1$ and periods of curves: the criterion on transcendence and the Period Conjecture. The results are made explicit in the classical terms of differential forms of the first, second and third kind on an algebraic curve in Chapter~\ref{ch:van}.

Part~\ref{part4} aims at a dimension formula for the space of periods of a $1$-motive in terms of its data. 
Chapter~\ref{ch:dim_comp_part1} treats mainly the saturated case. This can be applied to deduce complete structural results in Chapter~\ref{ch:struct}.
Finally, Chapter~\ref{ch:precise} is devoted to an explicit dimension computation  for the space of incomplete periods of the third kind, which is very involved.
In this rather complicated case the results were unexpected.

In the next chapter~\ref{sec:elliptic} we deal with the case of elliptic curves and make our results  explicit in terms of the classical Weierstraß functions
$\wp,\zeta,\sigma$. 

We explain in Chapter~\ref{ch:hypergeom} how transcendence results on special values of the hypergeometric function can be deduced from $\Qbar$-linear independence of $1$-periods.

There are three appendices: 
The first two sketch the theories of Nori and Voevodsky motives
to the extent used in the proof of Theorem~\ref{thm:main_kontsevich}.

The last appendix is
of technical nature: we need to verify that the singular and de Rham realisations
of a $1$-motive agree with the realisation of the attached geometric motive.

\part{Foundations}
\chapter{Basics on Categories}
In this monograph some basic concepts of the theory of categories are used frequently.  For the convenience of the reader, we recall them here.  At some places
however the requirements are higher than what can be found in this chapter. In such a situation  we give a precise reference to the literature where the reader can find the full information needed. 

\section{Additive and Abelian Categories}

\begin{defn}
A \emph{category}\index{category} $\Ch$ consists of a class of \emph{objects} $\mathrm{Ob}(\Ch)$ and for every pair of objects $X,Y\in\mathrm{Ob}(\Ch)$ a set of \emph{morphisms} $\Mor_\Ch(X,Y)$ together with a composition law 
\begin{align*} \circ:\Mor_\Ch(X,Y)\times \Mor_\Ch(Y,Z)&\to \Mor_\Ch(X,Z)\\
                             (f,g)\hspace{5ex}&\mapsto g\circ f
\end{align*}
for every triple of objects $X,Y,Z$. These data are subject to the following conditions:
\begin{enumerate}
\item The composition law is associative.
\item For every object $X$ there is a morphism $\id_X\in\Mor_\Ch(X,X)$ which is a left and right identity for the composition law, i.e.
$f\circ \id_X=f$ for all $Y$ and all $f:X\to Y$ and $\id_X\circ g=g$ for all $Z$ and all $g:Z\to X$.
\end{enumerate}
\end{defn}

A category is $\Z$-linear (or $\Q$-linear) if the morphism sets are given the structure of abelian groups (or $\Q$-vector spaces) and the composition of morphisms is bilinear. We usually write $\Hom_\Ch$ instead of $\Mor_\Ch$ in this case. It is \emph{additive}\index{category!additive} if it is $\Z$-linear and, in addition, has finite direct sums. The direct sum of two objects $X$ and $Y$ is denoted $X\oplus Y$. As the particular case of the empty direct sum, this also requires the existence of a $0$-object characterised uniquely by the property that
\[ \Hom(0,X)=\Hom(X,0)=0\]
for all objects $X$.

An additive category $\Ah$ is called \emph{abelian},\index{category!abelian} if the following two properties are satisfied:
\begin{enumerate}
\item Every morphism $f:X\to Y$ in $\Ah$ has a kernel and cokernel.
\item For every morphism $f:X\to Y$ in $\Ah$ the natural map $X/\Ker(f)\to \im(f)$ is an isomorphism.
\end{enumerate}
In this abstract setting, the image of a morphism is defined as the kernel of
the cokernel.

\begin{ex}
For every ring $R$, the category of $R$-modules is abelian. The category of finitely generated modules is additive. It is abelian if $R$ is noetherian.
\end{ex}

Many examples of additive and abelian categories are going to be used throughout the book. In order of appearance:

\begin{ex}
\begin{enumerate}
\item 
Let $k$ be a field. Then the category of connected commutative algebraic groups schemes over $k$ is additive, but not abelian. We refer to Chapter~\ref{ch:alg_groups} for details.
\item Let $K,L\subset\C$ be subfields. Then the category $\VV$ introduced in
Section~\ref{sec:formalism_2} is $\Q$-linear and abelian.
\item Let $k$ be a field. Then the category of filtered $k$-vector spaces is additive, but not abelian. Every morphism has a kernel and a cokernel, but the
isomorphism between $X/\Ker(f)$ and $\im(f)$ fails in general.
\item Let $k$ be an algebraically closed field of characteristic $0$. The category $\onemot_k$ of iso-$1$-motives is abelian. A thorough review is given in 
Chapter~\ref{sec:one-mot}.
\end{enumerate}
\end{ex}

Given a $\Z$-linear category $\Ah$, we obtain a $\Q$-linear category
$\Ah\tensor\Q$ with the same objects as $\Ah$ and morphism
\[ \Hom_{\Ah\tensor\Q}(X,Y)=\Hom_{\Ah}(X,Y)\tensor_\Z\Q.\]
We refer to it as the \emph{isogeny category of $\Ah$.}\index{isogeny category}
If $\Ah$ is additive or abelian, then so is $\Ah\tensor\Q$.

\begin{ex}
\begin{enumerate}
\item Let $k$ be a field. If $\Ah$ is the category of abelian varieties over $k$, the $\Ah\tensor\Q$ is what is referred to as the \emph{category of abelian varieties up to isogeny} in the literature. It is abelian. The same remark also applies to the category of connected commutative group schemes.
\item By Definition~\ref{defn:einsmot}, $\onemot_k$ is defined as the isogeny category of the category of $1$-motives over $k$.
\end{enumerate}
\end{ex}

For any category $\Ch$, we define the \emph{additive hull} $\Z[\Ch ]$ with the
objects formal direct sums $\bigoplus_{i=1}^nX_i$ for $n\geq 0$ and 
$X_1,\dots,X_n\in\Ch$. We interpret the empty direct sum as an object $0$.
Morphisms are defined by the formula
\[ \Hom_{\Z[\Ch]}\left(\bigoplus_{i=1}^nX_i,\bigoplus_{j=1}^mY_j\right)=\bigoplus_{i,j}\Z[\Hom_\Ch(X_i,Y_j)].\]
Here for a set $S$, we denote by $\Z[S]$ the free abelian group with basis $S$.

\section{Subcategories}

Given a category $\Ch$, a \emph{subcategory} of $\Ch$ is a
category $\Ch'$ such that every object and every morphism of $\Ch'$ is an object and morphism in $\Ch$, respectively. The composition of morphisms in $\Ch'$ is defined as their composition in $\Ch$ and the identity morphisms in $\Ch'$ agree with the identity morphisms in $\Ch$.
A subcategory is called \emph{full}\index{full subcategory} if 
\[ \Mor_{\Ch'}(X,Y)=\Mor_{\Ch}(X,Y)\]
for all objects $X,Y$ of $\Ch'$.

\begin{rem}
If $\Ch$ is additive or abelian, then a subcategory $\Ch'$ is not necessarily
additive or abelian itself. If $\Ch$ and $\Ch'$ are both abelian, this does not imply that the kernels and cokernels are the same when computed in $\Ch$ or
$\Ch'$. We are not going to consider such pathological situations, which only appear if the subcategory is not full.
\end{rem}

\begin{ex}
The category of abelian varieties (up to isogeny) is a full subcategory of the category of connected commutative algebraic groups (up to isogeny).
\end{ex}

\begin{ex}
The category of $\Q$-vector spaces is a full subcategory of the category of abelian groups.
\end{ex}

Let $\Ah$ be an abelian category. A \emph{subquotient}\index{subquotient} of an object $X$ in $\Ah$ is
a quotient of a subobject of $X$, or equivalently, a subobject of a quotient of $X$.

\begin{defn}
Let $\Ah$ be an abelian category and  $X\in\Ah$ an object. We define
$\langle X\rangle$ as the smallest full subcategory closed under subquotients containing $X$.
\end{defn}
More explicitly, this means that $\langle X\rangle$ contains $X$ and all the quotients and subobjects of every object $Y\in\langle X\rangle$. 

\begin{lemma}
The category $\langle X\rangle$ is abelian. We have
\[ \Ah=\bigcup_{X\in\Ch}\langle X\rangle.\]
\end{lemma}
\begin{proof}
Let $f:Y\to Z$ be a morphism in $\langle X\rangle$. Then $\ker(f)\subset Y$ exists in $\Ah$. As a subobject of an object in $\langle X\rangle$ it is itself an object of $\langle X\rangle$. The universal property of a kernel holds because it holds in $\Ah$. The same argument gives the existence of cokernels. The natural map $Y/\ker(f)\to \im(f)$ has an inverse in $\Ah$ because the category is
abelian. This inverse is in $\langle X\rangle$ because the subcategory is full.

Obviously all objects of $\Ah$ are contained in the union of all
$\langle X\rangle$. We have to check that the same is true for morphisms.
Let $f:X\to Y$ be a morphism in $\Ah$. Both $X$ and $Y$ are subobjects
of $X\oplus Y$, hence they are both objects of $\langle X\oplus Y\rangle$.
As $\langle X\oplus Y\rangle\subset\Ah$ is a full subcategory, the morphism
$f$ is a morphism in $\langle X\oplus Y\rangle$.
\end{proof} 

\section{Functors}

\begin{defn}
Let $\Ch$ and $\Ch'$ be categories. A \emph{covariant functor}\index{functor} $F:\Ch\to\Ch'$ is an assignment $F:\mathrm{Ob}(\Ch)\to\mathrm{Ob}(\Ch')$ together with
a map
\[ \Mor_\Ch(X,Y)\to\Mor_{\Ch'}(F(X),F(Y))\]
for every pair of objects $X,Y$ of $\Ch$. It is subject to the following conditions.
\begin{enumerate}
\item Compatibility with composition: $F(g)\circ F(f)=F(g\circ f)$ for
all objects $X,Y,Z$ in $\Ch$ and morphisms $f:X\to Y$, $g:Y\to Z$;
\item Compatibility with identities: $F(\id_X)=\id_{F(X)}$ for all objects $X$ of $\Ch$.
\end{enumerate}
In the case of a \emph{contravariant functor}, we are given maps
\[ \Mor_\Ch(X,Y)\to\Mor_{\Ch'}(F(Y),F(X))\]
and the compatibility condition reads $F(f)\circ F(g)=F(g\circ f)$ instead.
\end{defn}

A functor is called $F:\Ch\to\Ch'$ \emph{faithful},\index{faithful functor} \emph{full}, or \emph{fully faithful},\index{fully faithful functor} if for all objects $X,Y\in\Ch$ the natural map
\[ \Hom_{\Ch}(X,Y)\to \Hom_{\Ch'}(F(X),F(Y))\]
is injective, surjective or bijective, respectively.

\begin{ex}
If $F$ is the inclusion of a subcategory of a category, then it is faithful. If the subcategory is full, the inclusion is fully faithful.
\end{ex}

A functor $F:\Ch\to\Ch'$ between $\Z$-linear or $\Q$-linear categories is called \emph{additive} or $\Q$-linear if for all $X,Y\in\Ch$ the map
\[ \Hom_\Ch(X,Y)\to\Hom_{\Ch'}(F(X),F(Y))\]
is $\Z$-linear or $\Q$-linear, respectively.
Such a functor automatically respects direct sums, provided that they exist (e.g., because the categories are additive).

An additive functor $F:\Ah\to\Ah'$ between abelian categories is called \emph{exact}\index{exact functor} if it sends short exact sequences to short exact sequences.

\begin{lemma}
Let $F:\Ah\to\Ah'$ be an exact functor between abelian categories. Then $F$ is faithful if and only if for all $X$ in $\Ah$ the assumption $F(X)=0$ implies $X\isom 0$.
\end{lemma}

\begin{proof}Assume that $F$ is faithful and that $X$ is an object of $\Ah$ such that
$F(X)=0$. Then this implies $F(0)=F(\id_X)$. By faithfulness this
gives $0=\id_X$ and hence $X\isom 0$.

Conversely assume the condition on objects. 
Let $f$ be in the kernel of the map $\Hom_{\Ah}(X,Y)\to \Hom_{\Ah'}(F(X),F(Y))$. The functor $F$ maps the exact sequence
\[ 0\to \Ker(f)\to X\xrightarrow{f}Y\to \Coker(f)\to 0\]
to the exact sequence
\[ 0\to F(\Ker(f))\to F(X)\xrightarrow{F(f)=0}F(Y)\to F(\Coker(f))\to 0.\]
This gives $F(\Ker(f))\isom F(X)$ and $F(Y)\isom F(\Coker(f))$. Now consider
the short exact sequence
\[ 0\to\Ker(f)\to X\to X/\Ker(f)\to 0\]
and its image
\[ 0\to F(\Ker(f))\to F(X)\to F(X/\Ker(f))\to 0.\]
We had established that the first map is an isomorphism, so $F(X/\Ker(f))\isom 0$. By assumption this implies that $X/\Ker(f)\isom 0$ or $\Ker(f)\isom X$. The same type argument also shows that $Y\isom\Coker(f)$. 
Taken together this means that $f=0$.
\end{proof}

Faithful functors allow us to test for inclusions.
\begin{lemma}\label{lem:is_sub}
Let $F:\Ah\to\Ah'$ be a faithful exact functor between abelian categories, 
$X\in\Ah$ an object and $X_1,X_2\subset X$ subobjects. If $F(X_2)\subset F(X_1)$, then $X_2\subset X_1$.
\end{lemma}
\begin{proof}Let $X_3=X_1\cap X_2$ (or more abstractly, let $X_3$ be the pull-back of $X_1\to X$ via $X_2\to X$). We need to show that the natural inclusion $X_3\to X_2$ is an isomorphism, whence $X_2\subset X_1$. By the exactness of $F$, we have $F(X_3)\isom F(X_1)\cap F(X_2)$. By assumption this is $F(X_2)$.
We apply $F$ to the exact sequence
\[ 0\to X_3\to X_2\to C\to 0.\]
As $F(X_3)=F(X_2)$, we get $F(C)=0$. By the faithfulness of $F$, this implies $C\isom 0$.
\end{proof}

As a consequence of our results on transcendence and the Period Conjecture, we are are also going to establish  results on fullness of certain functors,
see Proposition~\ref{prop:hodge_ff}, Theorem~\ref{thm:fulness_VV} and Theorem~\ref{thm:ff}. The following criterion will be useful.

\begin{lemma}\label{lem:crit_full}
Let $F:\Ah\to\Ah'$ be a faithful exact functor between abelian categories.
Assume that the image of $F$ is closed under subquotients, i.e, if
\[ 0\to Y'\to F(X)\to Y''\to 0\]
is an exact sequence in $\Ah'$, then there is a short exact sequence
\[ 0\to X'\to X\to X''\to 0\] 
in $\Ah$ mapping to  the given  exact sequence in $\Ah'$.
 Then $F$ is full.
\end{lemma}
\begin{proof}
Let $f:F(Y_1)\to F(Y_2)$ be a morphism in $\Ah'$ and $\Gamma\subset F(Y_1)\times F(Y_2)$ its graph. We find it as the image of
$F(Y_1)$ under the map
\[ F(Y_1)\xrightarrow{\Delta}F(Y_1)\times F(Y_1)\xrightarrow{(\id,f)}F(Y_1)\oplus F(Y_2).\]
It is a subobject. By assumption there is $G\subset Y_1\times Y_2$ in  $\Ah$ such that $F(G)=\Gamma$. The projection $p:G\to Y_1\times Y_2\to Y_1$ is an isomorphism
because this is true for the image $\Gamma\to F(Y_1)$ and $F$ is faithful. Let
$i$ be its inverse. The composition
\[ Y_1\xrightarrow{i} G\subset Y_1\times Y_2\to Y_2\]
is the preimage of $f$ we were looking for. 
\end{proof}

\chapter{Homology and Cohomology}\label{ch:homology}
A key step in our approach to periods is the reinterpretation of paths as homology classes and differential forms as classes in algebraic de Rham cohomology. We survey the key definitions. For an in-depth review with complete references
we refer to \cite[Part~I]{period-buch}.

\section{Singular Homology}\label{sec:homology_intro}

All topological spaces in this section are locally compact, Hausdorff and satisfy the second countable axiom. The analytification of an algebraic variety over $\C$ is an example of such a space. We refer to standard text books on algebraic topology like \cite{spanier} or \cite{hatcher} for full details and proofs.

\begin{defn}
The \emph{topological $n$-simplex} $\Delta_n$ is defined
by
\[ \Delta_n=\left\{ (x_0,\dots,x_n)\in\R^{n+1}\;|\; x_0\geq 0,\dots,x_n\geq 0,\; \sum_{i=0}^nx_i=1\right\}.\]
\end{defn}
By setting one coordinate $x_i$ to $0$, we get the inclusion of the codimension~$1$ \emph{faces} $F_i$. They are homeomorphic to the $(n-1)$-simplex in an obvious way.

\begin{ex}
The $0$ simplex is a single point. The $1$-simplex is the interval from $F_0=(0,1)$ to $F_1=(1,0)$. We often identify it with the unit interval $[0,1]\subset\R$. 
\end{ex}

\begin{defn}
Let $X$ be a topological space.   For an integer $n\geq 0$ a  \emph{singular $n$-chain} is a formal $\Q$-linear combinations of continuous maps $f:\Delta_n\to X$. The space of singular chains $S_n(X)$ is a $\Q$-vector space. Together with the natural differential 
\[ d_n:S_n(X)\to S_{n-1}(X)\]
which maps a basis element $f$ in $S_n(X)$ to
\[ d_n(f)=\sum_{i=0}^n(-1)^if|_{F_i}\]
we obtain a chain complex 
$(S_*(X),d_*)$, the \emph{singular chain complex of $X$}.\index{singular chain complex}
Its 
homology \[ H^\sing_n(X,\Q):=H_n( (S_*(X),d_*)).\]
 is called the \emph{singular homology of $X$ with rational coefficients}.\index{singular homology}
\end{defn}

\begin{ex}
If $X$ is a finite discrete set, then
\[ H^\sing_n(X,\Q)=\begin{cases}\Q^{|X|}&n=0,\\
                           0&\text{else.}
              \end{cases}
\]
\end{ex}

\begin{ex}
Take $\sigma=\sum_{i=1}^na_i\gamma_i\in S_1(X)$, where $a_i\in\Q$ and
$\gamma_i:\Delta_1\to X$ is continuous. We identify
$\Delta_1$ with the unit interval and view the $\gamma_i$ as paths. 
Then $\sigma$ is in the kernel of $d_1$ if the formal linear combination
\[ \sum_{i=1}^na_i\gamma_i(1)-\sum_{i=1}^na_i\gamma_i(0))\]
vanishes. We say that $\sigma$ is \emph{closed} or a \emph{cycle}. Cycles
in the image of $d_2$ are called \emph{boundaries}. They are the ones that can be filled in by discs.
\end{ex}

In particular, every closed path gives rise to a homology class and  homotopic paths are homologous. We get a well-defined map
\[ \pi_1(X,x_0)\to H^\sing_1(X,\Q),\]
the \emph{Hurewicz homomorphism}. Its image generates $H^\sing_1(X,\Q)$ if $X$ is
path connected. In fact, we have
\[ H^\sing_1(X,\Q)\isom \pi_1(X,x_0)^{\mathrm{ab}}\tensor_\Z\Q\]
in this case.

\begin{ex}
Let $C$ be a smooth complete curve of genus $g$ over $\C$ and $C^\an$ the compact Riemann surface defined by its analytification. Then
\[ \dim_\Q H^\sing_1(C^\an,\Q)=2g.\]
For $C^\circ=C\ohne S$ with $S$  a non-empty finite set we have
\[ \dim_\Q H^\sing_1({C^\circ}^\an,\Q)=2g-|S|+1.\]
\end{ex}

\begin{ex} 
The projective line satisfies
\[ H^\sing_1({\Pe^1}^\an,\Q)=0.\]
For ${\A^1}^\an=\C$ and $\Gm^\an=\C^*$ we obtain
\[ H^\sing_1(\C,\Q)=0, \quad H^\sing_1(\C^*,\Q)\isom \Q\]
with the last group generated by a loop around $0$, e.g., the boundary of the unit disc.
\end{ex}

We also want to handle non-closed paths. They define classes in relative homology.

\begin{defn}Let $X$ be a topological space, $A\subset X$ a closed subset. We call
\[ S_*(X,A;\Q)=S_*(X,\Q)/S_*(A,\Q)\]
the \emph{singular chain complex for $(X,A)$}. Its homology is \emph{singular homology of the pair $(X,A)$ with rational coefficients}\index{singular homology! of a pair} or \emph{singular homology of $X$ relative to $A$},\index{singular homology! relative}\index{relative singular homology} written as
\[ H^\sing_n(X,A;\Q)=H_n( S_*(X,A;\Q),d_*).\]
\end{defn}
By definition relative singular homology is functorial for pairs.

\begin{ex}Let $\gamma:[0,1]\to X$ be a path with end points in $A$. Then it
is a cycle relative to $A$ and gives rise to a class in $H^\sing_1(X,A;\Q)$.
\end{ex}

The short exact sequence
\[ 0\to S_*(A,\Q)\to S_*(X,\Q)\to S_*(X,A;\Q)\to 0\]
gives rise to a long exact sequence in homology
\[ \dots\to H^\sing_n(A,\Q)\to H^\sing_n(X,\Q)\to H^\sing_n(X,A;\Q)\to H^\sing_{n-1}(A,\Q)\to\cdots.\]
Of particular interest for us is $n=1$. 
\begin{ex}
If $A$ is a finite set of points, then the sequence simplifies to
\[ 0\to H^\sing_1(X,\Q)\to H^\sing_1(X,A;\Q)\to \Q^{|X|-1}\to 0.\]
\end{ex}
More generally, we get \emph{boundary maps} for triples $B\subset A\subset X$
\[ \partial: H^\sing_n(X,A;\Q)\to H^\sing_{n-1}(A,B;\Q)\]
and natural long exact sequences
\[ \dots\to H^\sing_n(A,B;\Q)\to H^\sing_n(X,B;\Q)\to H^\sing_n(X,A;\Q)\to H^\sing_{n-1}(A,B;\Q)\to\cdots.\]

\begin{rem}\label{rem:smooth_chains}
If $X$ is a manifold, it suffices to work with \emph{smooth singular chains}. 
Let $S_*^\infty(X)\subset S_*(X)$ be  the subcomplex of linear combinations of
$C^\infty$-maps $f:\Delta_n\to X$, see \cite[Definition~2.2.2]{period-buch}.
By \cite[Theorem~2.2.5]{period-buch} the complex $S_*^\infty(X)$  can be used to compute singular homology of $X$. 
\end{rem}

There is another tool that allows us to reduce questions on the homology of algebraic varieties to the smooth case. We formulate it in cohomology obtained by replacing all vector spaces by their duals.
\begin{prop}[Blow-up sequence]\label{prop:blow-up}
Let $X$ be an algebraic variety over $\C$, $\pi:\tilde{X}\to X$ a proper map,
$Z\subset X$ a closed subvariety with preimage $E$ in $\tilde{X}$ such that
$\pi$ induces an isomorphism $\tilde{X}\ohne E\to X\ohne Z$. Then there is a natural long exact sequence
\begin{multline*}
 \dots \to H^n_\sing(X^\an,\Q)\to H^n_\sing(\tilde{X}^\an,\Q)\oplus H^n_\sing(Z^\an,\Q)\to H^n_\sing(E^\an,\Q)\\
\to H^{n+1}_\sing(X^\an,\Q)\to\cdots.
\end{multline*}
\end{prop}
\begin{proof}
In the case of analytic spaces attached to algebraic varieties, we can identify singular cohomology with sheaf cohomology. The statement then follows from
proper base change (see \cite[Proposition~2.6.7]{KS}) for $\pi$.
\end{proof}
\begin{ex}\label{ex:singsing}
Let $C$ be a curve with normalisation $\tilde{C}$. Let $Y\subset C$ be the set of singular points with preimage $\tilde{Y}\subset\tilde{C}$. The assumptions of the proposition are satisfied and the long exact sequence degenerates to the short exact sequence
\[ 0\to\Q^N\to H^1_\sing(C^\an,\Q)\to H^1_\sing(\tilde{C}^\an,\Q)\to 0.\]
with $N=|\tilde{Y}|-|Y|$.
\end{ex}

\section{Algebraic de Rham Cohomology}\label{sec:deRham_intro}

Algebraic de Rham cohomology was introduced by Grothendieck in \cite{grothendieck_66} as a purely algebraic way to define the Betti numbers of algebraic varieties. 
In this section $k$ is a field of characteristic $0$.

\subsection{The Smooth Case}
We start by presenting the much easier smooth case.

\begin{defn}Let $X$ be a smooth variety over $k$. We define \emph{algebraic de Rham cohomology of $X$}\index{de Rham cohomology}\index{de Rham cohomology! of a smooth variety} as hypercohomology of the complex of sheaves of differential forms on $X$,
\[ H_\dR^n(X)=\mathbb{H}^n(X,\Omega^*_X).\]
\end{defn}
This is particularly easy if $X$ is, in addition, affine. In this case
\[ H_\dR^n(X)=H^n( \Omega^*(X),d ).\]
There are different ways to compute hypercohomology. We  make the approach via \v{C}ech-cohomology explicit for later use.

Let $X$ be a smooth variety and  $\Uf=(U_1,\dots,U_n)$ be an open cover of $X$ by affine subvarieties $U_i$. For every $I\subset \{1,\dots,n\}$ we put
$U_I=\bigcap_{i\in I}U_i$ and for every $p,q\in\Na_0$ we define
\[ C^p(\Uf,\Omega_X^q)=\prod_{|I|=p+1}\Omega^q(U_I).\]
The $C^p(\Uf,\Omega_X^q)$  form a double complex. The differential $d$ in $q$-direction is induced
by the differential of $\Omega^*_X$. The differential $\delta$ in $p$-direction is
the differential of the \v{C}ech-complex given as follows: for $\alpha=(\alpha_I)\in C^p(\mathfrak{U},\Omega^q_X)$ we have
\[ \delta^p(\alpha)_{i_0\leq i_1\leq \dot i_p}=\sum_{j=0}^p(-1)^p\alpha_{i_0\leq\dots\leq \hat{i}_j\leq \dots i_p}\]
where $\hat{i_j}$ means that the index is omitted.

\begin{defn} Let
$\RGammatilde_\dR(X,\Uf)$ be the total complex of the double
complex $C^p(\Uf,\Omega^q)$ consisting of
\[ \RGammatilde_\dR(X,\Uf)^n=\bigoplus_{p+q=n}C^p(\Uf,\Omega^q)\]
with differential $ \sum_{p+q=n} (d^p+(-1)^{q}\delta^q)$.
\end{defn}

\begin{rem}This complex is nice because it is explicit and bounded. However, the boundary depends on the choice of an ordering of $U_1,\dots,U_n$.
In consequence, these complexes have bad functorial properties, unless
$f:Y\to X$ is affine. 
\end{rem}

\begin{lemma}The complex $\RGammatilde_\dR(X,\Uf)$ computes algebraic de Rham cohomology of $X$.
\end{lemma}
\begin{proof}We take the stupid filtration on the complex $\Omega^*_X$, which induces 
a filtration on the double complex $C^p(\Uf,\Omega^q)$. By the spectral sequence for the filtration, it suffices to consider the individual $\Omega^q$.
They are coherent, so by \cite[III~Theorem~4.5]{hartshorne_71} \v{C}ech-cohomology agrees with sheaf cohomology.
\end{proof}
We also need relative cohomology. Again the smooth case is easier.

\begin{defn}Let $X$ be a smooth variety, $Y\subset X$ a smooth closed subvariety, and $\Uf$ a finite open affine cover of $X$. We put
\[ \RGammatilde_\dR(X,Y,\Uf):=\cone(\RGammatilde_\dR(Y,\Uf\cap Y)\to \RGammatilde_\dR(X,\Uf))[-1]\]
and define \emph{algebraic de Rham cohomology of the pair $(X,Y)$}\index{de Rham cohomology!of a pair}\index{relative de Rham cohomology} or \emph{algebraic de Rham cohomology of $X$ relative to $Y$} as
\[ H^n_\dR(X,Y)=H^n(\RGammatilde_\dR(X,Y,\Uf)).\]
\end{defn}
These groups satisfy the same formal properties as singular cohomology.
\begin{lemma}Relative algebraic de Rham cohomology is well-defined. There is a natural long exact sequence
\[ \dots \to H^n_\dR(X,Y)\to H^n_\dR(X)\to H^n_\dR(Y)\to H^{n+1}_\dR(X,Y)\to \cdots.\]
\end{lemma}
\begin{proof}The exact sequence is the long exact sequence attached to the distinguished triangle
\[ \RGammatilde_\dR(X,\Uf)\to\RGamma_\dR(Y,\Uf\cap Y)\to\RGammatilde_\dR(X,Y,\Uf)[1]\]
 or, in different language, the short exact sequence of complexes
\[ 0\to \RGammatilde_\dR(Y,\Uf\cap Y)[-1]\to \RGammatilde_\dR(X,Y,\Uf)\to \RGammatilde_\dR(X,\Uf)\to 0.\]
To verify that algebraic de Rham cohomology is well-defined we need to check independence of the choice of cover.
We sketch the argument. 

As a first step replace $\RGammatilde_\dR(X,\Uf)$ by the quasi-isomorphic
complex $\RGammatilde_\dR(X,\Uf)'$ which  has all tuples $I=(i_0,\dots,i_n)$ as indices rather than only the ordered ones.  
Given two covers $\Uf_1$ and $\Uf_2$, there is a common refinement $\Uf_3$. It suffices to compare $\Uf_1$ and $\Uf_2$ with $\Uf_3$. The choice of a refinement map
from $\Uf_3$ to $\Uf_1$ induces homomorphisms
\[ \RGammatilde_\dR(X,\Uf_1)'\to\RGammatilde_\dR(X,\Uf_3)'\]
and also  for $Y$ and the pair $(X,Y)$.
They are quasi-isomorphisms for $X$ and $Y$ because the complexes compute algebraic de Rham cohomology. By the above-mentioned long exact sequence the comparison map is a quasi-iso\-mor\-phism for $(X,Y)$ as well.
\end{proof}

If $\dim X=0$, then  $\Omega^*_X=\Oh_X[0]$ 
and $X$ is affine, hence
\[ \RGammatilde_\dR(X)=\Oh_X(X).\]

We now spell out the curve case in degree $1$. Let $C$ be a smooth affine curve over $k$,
$D\subset C$ a finite set of closed points viewed as a smooth subvariety of dimension $0$. We use the trivial cover $\Uf=(C)$ and get
\begin{equation}
\label{eq:de_rham_einfach} \RGammatilde_\dR(C,\Uf)=[\Oh(C)\xrightarrow{f\mapsto (df,-f|_D)} \Omega^1(C)\oplus \Oh(D)]
\end{equation}
with $\Oh(C)$ in degree $0$. 

More generally, if $C$ is not necessarily affine, let $U_1,\dots,U_n$ be an open affine cover. We write $D_i=U_i\cap D$ and more generally $D_I=D\cap U_I$. By definition $\RGammatilde(C,D,\Uf)$ is the shifted cone of the homomorphism 
$\RGammatilde(C,\Uf)\to \RGammatilde(D,\Uf\cap D)$ of complexes. If we write the complexes vertically this takes the form
\[\begin{xy}\xymatrix{ 
\vdots &\vdots\\
\displaystyle\prod_{i<j}\Omega^1(U_i\cap U_j)\oplus \prod_{|I|=3}\Oh(U_I)\ar[u]\ar[r]&
\displaystyle\prod_{|I|=3}\Oh(D_I)\ar[u]\\
\displaystyle\prod_{i=1}^n\Omega^1(U_i)\oplus \prod_{i<j}\Oh(U_i\cap U_j)\ar[u]\ar[r]&\displaystyle\prod_{i<j}\Oh(D_i\cap D_j)\ar[u]\\
\displaystyle\prod_{i=1}^n\Oh(U_i)\ar[u]\ar[r]&\displaystyle\prod_{i=1}^n\Oh(D_i)\ar[u]
}\end{xy}\]
When taking the cone, we obtain the following description of cohomology in degree $1$.

\begin{lemma}\label{lem:de_rham_explicit}\index{de Rham cohomology!explicit cocycle}
The group $H^1_\dR(C,D)$ is given by the cohomology in degree~$1$ of 
\begin{multline*}
 \prod_{i=1}^n\Oh(U_i)\longrightarrow \prod_{i=1}^n\Omega^1(U_i)\oplus \prod_{i<j}\Oh(U_i\cap U_j)\oplus \prod_{i=1}^n\Oh(D_i)\\[1ex]
\longrightarrow
\prod_{i<j}\Omega^1(U_i\cap U_j)\oplus \prod_{|I|=3}\Oh(U_I)\oplus \prod_{i<j}\Oh(D_i\cap D_j)\longrightarrow \dots
\end{multline*}
with differentials 
\begin{align*} 
d^0( (f_i)_i) &=((df_i)_i, (f_j-f_i)_{ij}, (-f_i|_{D_i})_i)\\
d^1( (\omega_i)_i, (f_{ij})_{ij}, (g_i)_i)&=\\
 ( (-\omega_j+\omega_i+df&_{ij})_{ij}, (-f_{i_1i_2}+f_{i_0i_2}-f_{i_0i_1})_{i_0i_1i_2}, (-f_{ij}|_{D_{ij}}-g_j+g_i)_{ij})
\end{align*}
\end{lemma}.

\begin{rem}
The signs in the differentials depend on the sign conventions used for total complexes,  shifts and cones. We are using the normalisations
of \cite[Section~1.3]{period-buch}. However, any other choice
of sign conventions will lead to isomorphic cohomology groups. We only
have to ensure that $d^1\circ d^0=0$.
\end{rem}

\subsection{The General Case}
\index{de Rham cohomology!the general case}
The definition of (relative) algebraic de Rham cohomology can be generalised to the singular case by different methods, all yielding the same cohomology groups. The first was Hartshorne's, who embedded a singular variety $X$ into a smooth variety $P$ and worked with the completion of the complex of sheaves of differential forms on $P$ with respect to the vanishing ideal of $X$. In the context of Hodge theory, Deligne used smooth proper hypercovers: given $X$, he constructs a simplicial variety $X_\bullet$ over $X$ with smooth components and such that 
the 
singular cohomology of $X$ agrees with the singular cohomology of $X_\bullet$. (Our Proposition~\ref{prop:blow-up} is an instance of this fact.) He then defines algebraic de Rham cohomology of $X$ as the de  Rham cohomology of $X_\bullet$. In \cite{period-buch}, we use a variant of this approach. Algebraic de Rham cohomology is defined as sheaf cohomology of the complex $\Omega^*_h$ of $h$-differentials in the $h$-topology introduced by Voevodsky with the theory of triangulated motives in mind.

Instead of explaining the construction, we summarise the result:

\begin{thm}[{\cite[Section~3.2]{period-buch}}]
There is a sequence of functors
\[ H^n_\dR:\; (X,Y)\mapsto H^n_\dR(X,Y)\]
which attach a finite dimensional $k$-vector space to every pair $(X,Y)$ consisting of an algebraic variety and
a closed subvariety $Y$ and which extends the functors for smooth $X$ and $Y$.

If $Z\subset Y\subset X$ are closed subvarieties, there are natural coboundary maps
\[ \partial: H^n_\dR(Y,Z)\to H^{n+1}_\dR(X,Y)\]
fitting into a long exact sequence  
\[ \dots \to H^n_\dR(X,Z)\to H^n_\dR(Y,Z)\xrightarrow{\partial}  H_\dR^{n+1}(X,Y)\to \cdots.\]
\end{thm} 

\begin{ex}
Let $C$ be a curve with normalisation $\tilde{C} \to C$,  and $Y\subset C$ the set of singular points with preimage $\tilde{Y}\subset\tilde{C}$. Then the shift by $[-1]$ of the cone of
\[ \RGammatilde_\dR(\tilde{C})\oplus \RGammatilde_\dR(Y)\to\RGammatilde_\dR(\tilde{Y})\]
computes algebraic de Rham cohomology of $C$. In particular, this leads to a short exact sequence  
\[ 0\to k^N\to H^1_\dR(C)\to H^1_\dR(\tilde{C})\to 0\]
with $N=|\tilde{Y}|-|Y$. This is the same result as in Example~\ref{ex:singsing}.
\end{ex}

\section{The Period Pairing}\label{sec:pairing_intro}

We now fix an embedding of $k$ into $\C$, which allows us to define an analytification functor from $k$-varieties to complex spaces.
Again we start under the simplifying assumption that $X$ is smooth and affine.
 In this case there is the natural pairing
\[ S_n^\infty(X^\an)\times \Omega^n(X)\to\C; \quad(\sigma,\omega)\mapsto \int_\sigma \omega.\]
By Stokes's theorem the pairing induces a well-defined map on homology
\[ H_n^\sing(X,\Q)\times H^n_\dR(X)\to\C,\]
the \emph{period pairing}. \index{period pairing!for varieties}

\begin{rem}
 For $n=1$, which is the most important case for us, the pairing is even defined on all of $S_1(X)$. We get its value by analytic continuation.
\end{rem}

For fixed $n$ and $(X,Y)$, singular cohomology is a $\Q$-vector space and de Rham cohomology is  a $k$-vector space. After extension of scalars to $\C$, the pairing becomes perfect.
In the smooth and proper case, this is a direct consequence of GAGA. In full generality, it was established by Deligne as part of the development of Hodge theory for non-proper and singular varieties. It even extends to pairs.

\begin{thm}[{\cite[Chapter~5]{period-buch}}]
There is a natural \emph{period pairing}
\[ H_n^\sing(X,Y;\Q)\times H^n_\dR(X,Y)\to \C\]
inducing \emph{period isomorphisms}\index{period isomorphism!for varieties}
\[ H^n_\sing(X,Y;\Q)\tensor_\Q\C\to H^n_\dR(X,Y)\tensor_k\C.\]
The period isomorphism is functorial for pairs of $k$-varieties and compatible with coboundary maps for triples $Z\subset Y\subset X$.
\end{thm}

\subsection{The Case of Smooth Affine Curves}\label{ssec:period_compute_1}

For later use, we make the period pairing explicit in the first interesting case.
Let $C$ be a smooth affine curve, $Y\subset C$ a finite set of $k$-points. 
By definition algebraic de Rham cohomology
of $(C,Y)$ is the cohomology of the complex 
\[ \left[\Omega^0(C)\to\Omega^1(C)\oplus\Omega^0(Y)\right].\]
Hence every class in $H^1_\dR(C,Y)$ is represented
by a pair $(\omega,\alpha)$ where $\omega$ is a $1$-form and $\alpha:Y(k)\to k$ is a set-theoretic map. Conversely, every such pair defines a class in relative de Rham cohomology.

Singular homology of the pair is defined as homology of the complex
\[ S^\infty_*(C)/S^\infty_*(Y).\]
Its $1$-cycles are represented by $\Z$-linear combinations $\sigma=\sum n_i\gamma_i$
of smooth maps $\gamma_i:[0,1]\to C^\an$ such that 
$\partial(\sum n_i\gamma_i)=\sum n_i\gamma_i(1)-\sum n_i\gamma_i(0)$ is 
in $S_0(Y)$. Up to homotopy such a cycle can be replaced by a formal $\Z$-linear
combination of closed paths and paths with end points in $Y(k)$. The definition of the period pairing requires to replace the singular complex of $(C,Y)$ by the quasi-isomorphic
\[ \cone\left( S^\infty_*(Y)\to S^\infty(C)_*\right)=
[ S^\infty_0(C)\leftarrow S^\infty_1(C)\oplus S^\infty_0(Y)\leftarrow\cdots]
\]
The homology class of $\sigma$ is represented by the pair $(\sigma,-\partial \sigma)$. 
In these explicit terms, the period pairing is given by 
\[ ((\omega,\alpha),(\sigma,\partial\sigma))=\int_{\sigma}\omega-\alpha(\partial\sigma).\]

\chapter{Commutative Algebraic Groups}\label{ch:alg_groups}
\section{The Building Blocks}

We fix an algebraically closed field $k$ of characteristic zero. The cases of relevance for us are $k=\Qbar$ and $k=\C$.
We denote by $\grp$ the category of commutative connected algebraic groups over $k$. They are automatically smooth. This category is not abelian because the kernel of a morphism is not necessarily connected. However, the category of all commutative algebraic groups over $k$ is abelian and so is the isogeny category of $\grp$, where morphisms are tensored by $\Q$. For a careful review and analysis, see \cite{brionI}.

\begin{ex}
The additive group $\Ga=\Spec(k[T])$\index{additive group $\Ga$} and the multiplicative group \index{multiplicative group $\Gm$}
$\Gm=\Spec(k[T,T^{-1}])$ are objects of $\grp$.
\end{ex}
A connected commutative algebraic group is called \emph{vector group} \index{vector group} if it is isomorphic to a power of $\Ga$ and \emph{torus}\index{torus} if it is isomorphic to a power of $\Gm$.

\begin{ex}
Every abelian variety over $k$ is an object of $\grp$.
\end{ex}

Note that there are no non-trivial morphisms between vector groups, tori and abelian varieties.

\begin{thm}[Structure theory]\label{thm:structure}
\index{structure theorem for commutative algebraic groups}
Let $G$ be a connected commutative algebraic group. Then there is a canonical short exact sequence
\[ 0\to L\to G\to A\to 0\]
with an abelian variety $A$ and a linear connected commutative algebraic group 
$L$. Moreover, there is a canonical split short exact sequence
\[ 0\to V\to L\to T\to 0\]
with a torus $T$ and a vector group $V$.
\end{thm}
\begin{proof}The first sequence is the commutative case of
\cite{barsotti}; see also \cite{chevalley}. By \cite[Ch. IV \S 3 Th\'eor\'eme~1.1]{demazure-gabriel} or \cite[Ch. III Proposition~12]{serre_class} we have $L\isom V\times T$ with $V$ unipotent and a torus $T$. By \cite[Ch. VII \S 2.7]{serre_class}, all unipotent groups are powers of $\Ga$ in characteristic $0$, hence $V$ is a vector group.
\end{proof}

\begin{cor}
An object of $\grp$ is simple if and only if it is isomorphic to
$\Ga$, $\Gm$ or a simple abelian variety.
\end{cor}

A connected commutative algebraic group is called a  \emph{semi-abelian variety} \index{semi-abelian variety} if it is an extension of an abelian variety by a torus, or, equivalently, if its vector part is trivial. The Structure Theorem then implies that for any connected commutative algebraic group $G$ there is also a canonical short exact sequence
\[ 0\to V\to G\to G^{sa}\to 0\]
with a vector group $V$ and a semi-abelian variety $G^{sa}$.

The aim of this chapter is on the one hand to give an alternative description of the category of semi-abelian varieties and, on the other hand, explain the construction of the universal vector extension.

\begin{defn}
\begin{enumerate}
\item Let $V$ be a vector group. We define $V^\vee=\Hom(V,k)$ and view it as a vector group.
\item Let $A$ be an abelian variety. We define $A^\vee=\Pic^0(A)$, the \emph{dual abelian variety}.\index{dual abelian variety}
\item Let $T$ be a torus. We define $X(T)=\Hom(T,\Gm)$, the \emph{character group of $T$}.\index{character group of  torus}
\item Let $\Xi$ be a free abelian group of finite rank. We define
$\Gm(\Xi)=\Hom(\Xi,\Gm)$, the \emph{dual torus of $\Xi$}.
\end{enumerate}
\end{defn}
Note that $A^\vee$ has a canonical structure of abelian variety by \cite[\S 13 Theorem]{mumford_av}; see also \cite[p.~40]{milne_av}. 
All these functors are contravariant. Given a morphisms of abelian
varieties $\alpha:A\to A'$, by pull-back of line bundles we get an induced morphism
$\alpha^\vee:A'^\vee\to A^\vee$. Given a morphism of tori
$\tau:T\to T'$, by composition we get an induced morphism of abelian groups
$\tau^*:X(T')\to X(T)$. Given a morphism of finitely generated abelian groups
$\xi:\Xi\to\Xi'$, by composition we get an induced morphism of tori
$\xi^*:\Gm(\Xi')\to \Gm(\Xi)$.

\begin{prop}
 Let $V$ be a vector group, $A$ an abelian variety and
$T$ a torus. Then there are canonical isomorphisms
\[ V^{\vee,\vee}\isom V,\quad A^{\vee,\vee}\isom A,\quad \Gm(X(T))\isom T.\]
\end{prop}
\begin{proof} The case of $V$ is linear algebra. In the case of
$T$, we get by adjunction a canonical map
\[ T\to \Gm(X(T)).\]
By naturality, it suffices to check that it is an isomorphism in the case $T=\Gm$. This case is again trivial.
For the double dual of an abelian variety, see \cite[\S Corollary, p.~132]{mumford_av}, see also \cite[Theorem~8.9]{milne_av}.
\end{proof}

\section{Group Extensions}\label{ext-groups}

In this section we give  a short introduction to the functor $\Ext^1$ in the abelian category  of commutative algebraic groups
defined over $k$. Most of the material can be found in Serre's book \cite{{serre_class}}. 

Let $A, B$ be objects of  $\grp$. To give an \emph{extension of $A$ by $B$} is the same as to give a triple $(C,\iota,\pi)$ with $C$ a commutative algebraic group $C$ and $(\iota,\pi)\in \Hom(B,C)\times \Hom(B,A)$ such that 
\begin{eqnarray}\label{extension*}
\xymatrix{
0\ar[r] & B\ar[r]^\iota & C \ar[r]^\pi&A\ar[r]&0
}
\end{eqnarray}\label{commdiag}
is exact. By abuse of notation we often call the group $C$
an \emph{extension}\index{extensions of commutative algebraic groups} of $A$ by $B$. It is automatically connected.

A morphism of extensions is a triple $\alpha:A\to A', \beta:B\to B',\gamma:C\to C'$ making the diagram
\begin{eqnarray}
\xymatrix{
0\ar[r] & B \ar[r]^{\iota}\ar[d]^{\beta}& C \ar[r]^{\pi}\ar[d]^{\gamma} & A \ar[r]\ar[d]^{\alpha}&0\\
0\ar[r] &B' \ar[r]^{\iota'}& C' \ar[r]^{\pi'} & A' \ar[r]& 0
}
\end{eqnarray}
commutative.  Note that $\gamma$ is an isomorphism if and only if $\alpha$ and $\beta$ are. In this case the extensions are isomorphic. We say that two extensions are \emph{equivalent} if there
is a homomorphism of extensions with $A=A'$, $B=B'$ and $\alpha=\id_A, \beta=\id_B$. 
The set of equivalence classes  of extensions  makes up a commutative group $\Ext^1(A,B)$,  the group of \emph{Yoneda-$1$-extensions}.  By abuse of notation we often write $C$ for its equivalence class $[C]$.

\begin{rem}
The morphisms $\alpha, \beta$ are uniquely determined by $\gamma$. If $\Hom(B,A')=0$ (for example because $B$ is a linear algebraic group and $A'$ an abelian variety), then the existence of $\alpha$ and $\beta$ is automatic.
\end{rem}

We now review the group structure on $\Ext^1$ via the Baer sum.\index{Baer sum} As for all abelian categories,
the bi-functor $\Ext^1$ which associates with a pair $(A,B)$ the set $\Ext^1(A,B)$ is contravariant in the
first and covariant in the second variable. The functoriality in the first variable is given by pull-back.  
Given a morphism $\alpha:A''\to A$, introduce
\[ C'':=C\times_{A}A''.\] 
By construction,
\[ 0\to B\to C''\to A''\to 0\]
is exact
and we define 
\[ \alpha^*[C]=[C'']\in\Ext^1(A'',B).\]
The functoriality in the second variable is given by push-out. Given a morphism
$\beta:B\to B'$, introduce 
\[ C'= C\times B'/B\] 
where $B$ acts both on $B'$ and on $C$. Let $B'\to C'$ be given by $b\mapsto (-b,\beta(b))$.  
By construction,
\[ 0\to B'\to C'\to A\to 0\]
is exact 
and we define 
\[ \beta_{*}[C]=[C']\in\Ext^1(A,B').\]  
The two transformations  $\alpha^*$ and $\beta_{*}$ commute in the sense that
\begin{eqnarray}\label{linear}
\xymatrix{
\Ext^1(A,B) \ar[r]^{\alpha^*}\ar[d]^{\beta_*} & \Ext^1(A'',B) \ar[d]^{\beta^*}\\
\Ext^1(A,B') \ar[r]^{\alpha^*} & \Ext(A'',B')
}
\end{eqnarray}
commutes. This is a general property of Ext-groups in abelian categories.

If $[C_1]$ and $[C_2]$ are in $\Ext^1(A,B)$, then their \emph{Baer sum} is
\[ [C_1]+[C_2]=\Delta^* s_* ([C_1\times C_2]),\] 
which makes $\Ext^1(A,B)$ into a commutative group with neutral element $0=[B\times A]$; here $\Delta$ is the diagonal map from $A$ into $A\times A$ and $B\times B \xrightarrow{s} B$ the addition on $B$. 
We deduce that multiplication by an integer $n$ can be defined inductively using addition. Equivalently take 
\[ n[B]=\Delta^* s_* ([B^n])\]
 where $\Delta=\Delta_n$ is the diagonal from $C$ to $C^n$ and $s=s_n$ is $n$-fold addition $A^{\,n}\rightarrow A$ on $A$.

The bi-functor $\Ext^1$ is additive in both variables, which  implies that 
\[ \Ext^1(A_1\times A_2,B) = \Ext^1(A_1,B)\times \Ext^1(A_2,B)\]
 and 
\[ \Ext^1(A,B_1\times B_2)=\Ext^1(A,B_1)\times\Ext^1(A,B_2).\]
Hence it is of particular importance to understand $\Ext^1$ for the simple building blocks.

\begin{prop}\label{prop:compute_ext}
\index{extension of commutative algebraic groups!computation}
Let $A$ be an abelian variety, $L, L'$ linear connected commutative algebraic groups. Then:
\begin{gather}
\Ext^1(A,\Ga)=H^1(A,\Oh),\tag{1}\label{eq:ext1}\\
\Ext^1(A,\Gm)=A^\vee(k)=\Pic^0(A)(k)\subset \Pic(A)(k)=H^1(A,\Oh^\times),\tag{2}\label{eq:ext2}\\
\Ext^1(L,L')=0.\tag{3}\label{eq:ext3}
\end{gather}
\end{prop}
\begin{proof}For the statements on abelian varieties 
see \cite[Ch VII \S 3 Theorem 7 and 6]{serre_class}. 
All linear connected commutative algebraic groups are split by Theorem~\ref{thm:structure}. In particular, there are no non-trivial extensions. 
\end{proof}

\begin{rem}\label{rem:poincare}
The identification of $\Ext^1(A,\Gm)$ with $A^\vee(k)$ is provided by the Poincar\'e bundle $\Ph$ on $A\times A^\vee$: see \cite[Chapter~II 8.~p. 78]{mumford_av}. Given a point $x\in A^\vee(k)$, the pull-back $\Ph_x$ to $A$ via $(\id,x)$ is a line bundle of degree $0$. After removing the zero-section, it is a $\Gm$-bundle and, in fact, a semi-abelian variety in $\Ext^1(A,\Gm)$. 
\end{rem}

\begin{cor}\label{cor:compare_ext}
Let $G$ be a  semi-abelian variety with abelian part $A$ and $V$ a vector group. Then the natural map
\[ \Ext^1(A,V)\isom \Ext^1(G,V)\]
is an isomorphism. In particular, $\Ext^1(G,V)$ is finite dimensional.
\end{cor}
\begin{proof}
We start with the short exact sequence
\[ 0\to T\to G\to A\to 0\]
with $T$ the torus part of $G$, and apply the long exact sequence
for $\Hom(-,V)$; see \cite[Ch. VII \S 1 2.]{serre_class}. This yields the exact sequence
\[  \dots\to   \Hom(T,V)\to \Ext^1(A,V)\to \Ext^1(G,V)\to\Ext^1(T,V)\to\dots\]
The first and the last term vanish.

Finite dimensionality holds because
\[ \Ext^1(A,V)\isom\Ext^1(A,\Ga)^s\isom H^1(A,\Oh)^s\]
where $s=\dim(V)$. 
\end{proof}

From this proposition, we get classifying maps: let $A$ be an abelian variety, $T$ a torus. Then we have the bilinear map 
\begin{align*} X(T)\times\Ext^1(A,T)&\to\Ext^1(A,\Gm)\\
                 (\chi,[G])&\mapsto \chi_*[G].
\end{align*}
Alternatively, let $[G]\in \Ext^1(A,T)$ and consider the exact sequence
of Theorem~\ref{thm:structure} with $L=T$. Applying the long exact $\Hom(-,\Gm)$-sequence of \cite[VII, \1 Proposition 2]{serre_class}
we get a long exact sequence
\begin{eqnarray*}
\xymatrix{
\cdots\;\ar[r]&\Hom(T, \Gm)\ar[r]^{d_G}& \Ext^1(A,\Gm)\ar[r]^{\pi^*}&\Ext^1(G,\Gm)\ar[r]&\;\cdots,
}
\end{eqnarray*}
where the connecting homomorphism is given by 
$d_G(\lambda)=\lambda_*([G])$.
The two descriptions are equivalent.

\begin{cor}\label{cor:sa_ext}
The induced map
\[ \Ext^1(A,T)\to \Hom(X(T),A^\vee)\]
is an isomorphism.
\end{cor}
\begin{proof}As $T\isom\Gm^r$ and both sides are natural in $T$, it suffices to treat the case $T=\Gm$. In this case we have $X(\Gm)=\Z$ and
$\Hom(X(\Gm),A^\vee)=A^\vee(k)$. The claim now follows from (\ref{eq:ext2}) in  Proposition~\ref{prop:compute_ext}.
\end{proof}

We refer to the image of $G$ in $\Hom(X(T),A^\vee)$ as the \emph{classifying map} of $G$.\index{classifying map of a semi-abelian variety}\index{semi-abelian variety!classifying map}

Now let $V$ be a vector group. Again we have a bilinear map
\begin{align*} V^\vee\times\Ext^1(A,V)&\to\Ext^1(A,\Ga)\\
                 (\lambda, [G])&\mapsto \lambda_*[G].
\end{align*}
As in the torus case, there is an alternative description as
\[ (\lambda, [G])\mapsto d_G(\lambda)\]
with respect to the long exact $\Hom(-,\Ga)$-sequence attached to
(\ref{extension*}) with $L=V$.

\begin{cor}\label{cor:vector_ext}
The induced map
\[ \Ext^1(A,V)\to\Hom(V^\vee,H^1(A,\Oh))\]
is an isomorphism.
\end{cor}
\begin{proof}
As $V\isom\Ga^s$ and both sides are natural in $V$, it suffices to treat the case $V=\Ga$. In this case $\Ga^\vee=\Ga$ and $\Hom(\Ga,H^1(A,\Oh))=H^1(A,\Oh)$.
 The claim follows from (\ref{eq:ext1}) in  Proposition~\ref{prop:compute_ext}.
\end{proof}

The same considerations also apply to extensions of semi-abelian varieties by
vector groups. We obtain: 

\begin{cor}\label{cor:vector_ext2}
Let $G$ be semi-abelian, $V$ a vector group. Then the natural map
\[ \Ext^1(G,V)\to \Hom(V^\vee,\Ext^1(G,\Ga))\]
is an isomorphism.
\end{cor}
\begin{proof}By Corollary~\ref{cor:compare_ext} we can replace $G$ by its abelian part $A$ on both sides. Then we are back in the situation of Corollary~\ref{cor:vector_ext}.
\end{proof}

\section{Semi-abelian Varieties}\label{sec:sa}

\index{semi-abelian variety}\index{semi-abelian variety!isogeny category}

As shown in Corollary~\ref{cor:sa_ext},
the datum of a semi-abelian variety $G$ over $k$  is equivalent to the datum of a homomorphism   $X(T)\to A^\vee(k)$.
This
construction is functorial. A morphism of semi-abelian varieties $\alpha:G_1\to G_2$
induces a commutative diagram
\[\begin{xy}\xymatrix{
 X(T_1)\ar[d]_{[G_1]}&X(T_2)\ar[l]_{\alpha^\vee}\ar[d]^{[G_2]}\\
 A_1^\vee(k)&A_2^\vee(k)\ar[l]^{\alpha^\vee}
}\end{xy}.\]
In conclusion:
\begin{prop}\label{prop:equiv_sa}
The assignment $G\mapsto [X(T)\to A^\vee(k)]$ yields an equivalence between the category
of semi-abelian varieties over $k$ and the category with objects given by 
triples $(X,A,\phi)$ where $X$ is a free abelian group of finite rank, $A$ an abelian variety and $\phi:X\to A^\vee(k)$ is a group homomorphism.
\end{prop}
\begin{proof}
We verify that the functor is faithful. Let $f:G\to G'$ be a morphism of semi-abelian variety mapping to $0$ under the functor. In particular, the induced morphisms on the torus part and the abelian part vanish. The composition $G\to G'\to A'$ vanishes because it factors over $0:A\to A'$. Hence
$f$ maps into $T'$. The restriction $f|_T$ vanishes, hence we get an induced
map $\bar{f}:A\to T'$. There are no such maps, hence $\bar{f}=0$.

The functor is also full:
given a commutative diagram as above, we get back the morphism $\alpha$
as the composition
\[ G_1\to G_{\alpha^\vee\circ [G_2]}=G_2\times_{A_2}A_1\to G_2.\]
We are often going to make use of this equivalence without mentioning it explicitly. To verify the equality $G_{\alpha^\vee\circ [G_2]}=G_2\times_{A_2}A_1$ we consider the cartesian diagram
\[\begin{xy}\xymatrix{
G_2\times_{A_2}A_1\ar[d]\ar[r]&G_2\ar[d]\\
A_1\ar[r]&A_2}
\end{xy}
\]
which corresponds to 
\[\begin{xy}\xymatrix{
X(T_2)\ar[d]_{[G_2]}\ar@{=}[r]&X(T_2)\ar[d]^{[G_2\times_{A_2}A_1]}\\
A_2^\vee\ar[r]_{\alpha^\vee}&{A_1}^\vee}.
\end{xy}
\]
This shows that $G_{\alpha^\vee\circ [G_2]}=G_2\times_{A_2}A_1=\alpha^*G_2$.

It remains to check that the functor is full on objects.
Given $(X,A,\phi)$, we construct $G$ as follows: let $e_1,\dots,e_n$ be
a basis of $X$. The elements $\phi(e_i)\in A^\vee(k)$ define
elements of $\Ext^1(A,\Gm)$ and hence extensions
\[ 0\to \Gm\to G_i\to A\to 0.\]
We put
\[ G=G_1\times_A\times \dots\times_A G_n.\]
By construction $G$ maps to $(X,A,\phi)$ under our functor.
\end{proof}    

\begin{rem}\label{rem:decomp_G}
The map $X(T)\to A^\vee(k)$ is zero if and only if
$G\isom A\times T$. Given two maps $s_1:X(T_1)\to A^\vee(k)$ and
$s_2:X(T_2)\to A^\vee(k)$ corresponding to $G_1$ and $G_2$, their 
sum defines $s:X(T_1)\oplus X(T_2)\to A^\vee(k)$. It corresponds
to the semiabelian variety $G$ obtained as the  pull-back of $G_1\times G_2\to A\times A$ via the diagonal $A\to A\times A$. Its torus part is
$T_1\times T_2$. If $s_1=0$, then the composition
$G\to G_1\isom A\times T_1\to T_1$ together with $G\to G_2$ induce
an isomorphism $G\isom T_1\times G_2$.
\end{rem}

\begin{defn}
The category of semi-abelian varieties \emph{up to isogeny} has the same objects as the category of semi-abelian varieties but with morphisms tensored by $\Q$.
\end{defn}

Proposition~\ref{prop:equiv_sa} implies that the category of semi-abelian varieties up to isogeny is equivalent to the category with objects given by triples $(X_\Q,A,\phi)$ where $X_\Q$ is a finite dimensional $\Q$-vector space, $A$ denotes an abelian variety up to isogeny and $\phi$ a 
$\Q$-linear map $X_\Q\to A^\vee(k)_\Q$.
We often write objects as $X\to A^\vee(k)_\Q$, where $X$ is a free abelian group of finite rank.

\begin{cor}
Let $G$ be a semi-abelian variety, $T$ a torus and $G\to T$ a surjective morphism of algebraic groups with kernel $G'$. Then $G\isom T\times G'$ up to isogeny.
\end{cor}
\begin{proof}
Since tori are semi-simple there is also an injective homomorphism $T \to G$ with image in the torus part of $G$. By the universal property of the direct product, this leads to an isomorphism $G \isom T \times G'$.
\end{proof}

\section{Universal Vector Extensions}\label{sec:universal}

\index{universal vector extension}
As shown in Corollary~\ref{cor:vector_ext2},
the datum of a vector extension of a semi-abelian variety \index{vector extension! of a semi-abelian variety} $G$ over $k$  is equivalent to the datum of a linear map   $V^\vee\to \Ext^1(G,\Ga)$ or dually
$\Ext^1(G,\Ga)^\vee\to V$.
As in the semi-abelian case, this
construction is functorial. 

\begin{prop}\label{prop:equiv_vector}
The assignment $G\mapsto [\Ext^1(G,\Ga)^\vee\to V]$ yields an equivalence between the category
of vector extensions of semi-abelian varieties over $k$ and the category with objects given by 
triples $(V,A,\phi)$ where $V$ is a finite dimensional $k$-vector space,
$A$ is an abelian variety and $\phi:\Ext^1(G,\Ga)^\vee\to V$ is a $k$-linear map.
\end{prop}
\begin{proof}
The argument is the same as in the semi-abelian case.
\end{proof}    

The vector space $\Ext^1(G,\Ga)$ is itself finite-dimensional by Corollary~\ref{cor:compare_ext}, hence there is a distinguished object in the category of vector extensions of $A$.

\begin{defn}
Let $G$ be a semi-abelian variety. We call the vector
extension 
\[ 0\to \Ext^1(G,\Ga)^\vee\to G^\natural\to G\to 0\]
corresponding to $\id:\Ext^1(G,\Ga)^\vee\to \Ext^1(G,\Ga)^\vee$ the \emph{universal vector extension of $G$.}
\end{defn}

\begin{prop}\label{prop:universal}
The universal vector extension of a semi-abelian variety $G$ has the following universal property: given a vector extension
\[ 0\to V\to G'\to G\to 0\]
there is a unique morphism $G^\natural\to G'$ compatible with the projection to $G$.
\end{prop}
\begin{proof}
Under the equivalence of Proposition~\ref{prop:equiv_vector} the vector extension $G'$ corresponds to the 
triple $(V,G,\phi:\Ext^1(G,\Ga)^\vee\to V)$ and $G^\natural$ corresponds to
$(\Ext^1(G,\Ga),A,\id)$.
A morphism $G'\to G^\natural$ corresponds to a linear
map $\Ext^1(G,\Ga)^\vee\to V$ compatible with the structure maps. The only such maps is $\phi$.
\end{proof}

\begin{rem}
Let $A$ be the abelian part of $G$. By the computation of
$\Ext^1(A,\Ga)$ in Proposition~\ref{prop:compute_ext}, we have
\[ 0\to H^1(A,\Oh)^\vee\to A^\natural\to A\to 0.\]
Moreover, the isomorphism $\Ext^1(A,\Ga)\isom\Ext^1(G,\Ga)$
of Corollary~\ref{cor:compare_ext} implies that $G^\natural$ is explicitly constructed as
\[ G^\natural=A^\natural\times_AG.\]
\end{rem}

If $V$ is a vector group, $G\in \grp$, both $\Hom(G,V)$
and $\Ext^1(G,V)$ are $k$-vector spaces and do not change when we replace
the category $\grp$ with its isogeny category $\grp_\Q$.  This remark implies:

\begin{cor}
The universal vector extension $G^\natural$ of a semi-abelian variety $G$ also satisfies the universal property of a vector extension in the
isogeny category $\grp_\Q$.
\end{cor}

\section{Generalised Jacobians}\label{app:B}
\index{generalised Jacobian}

Let $k$ be an algebraically closed field of characteristic zero. 
Let $Y$ be a smooth connected algebraic curve over $k$ with a chosen base point $y_0$.

The following theorem is  a special case of the theory of generalised Jacobians. 
They were introduced by Rosenlicht. We follow the presentation of Serre, see \cite[Chapter V]{serre_class}. We recall briefly the deduction.
\begin{thm}[{Rosenlicht: see Serre \cite[Chapter V]{serre_class}}]\label{thm:gen_jacob} 
There is a semi-abelian variety $J(Y)$ over $k$ and a morphism
\[ Y\to J(Y)\]
depending only on $y_0$ such that $H_1^\sing(Y,\Z)\to H_1^\sing(J(Y),\Z)$ is an isomorphism.
\end{thm}

\subsection{Construction of $J(Y)$}

Let $\bar{Y}$ be the smooth compactification   of $Y$ and $S=\bar{Y}\setminus Y$ the set of points in the complement of $Y$. 
We define the divisor $\mm$ as $\sum_{P\in S}P$. In the terminology of \cite{serre_class} this is a (special case of a) \emph{modulus}. The case $\mm=0$ (i.e. $S=\emptyset$) is allowed.

A rational function $\varphi$ on $Y$ is congruent to 1 mod $\mm$  if  $\nu_P(1-\varphi) \geq 1$ for all $P\in S$ where $\nu_P$ denotes the valuation at $P$. 
We write
\begin{itemize}
\item  $C_{\mm}$ for the group of classes of divisors on $\bar{Y}$ which are prime to $S$ modulo those which can be written as $(\varphi)$ for some rational function $\varphi \equiv 1\mod \mm$;
\item  $J_\mm=C_{\mm}^0$ for the subgroup of classes which have degree $0$;
\item  $J=C^0$ for the usual group of divisor classes of degree $0$.
\end{itemize}
There is a surjective homomorphism 
\[ \pi: J_\mm\rightarrow J\]
 with kernel $L_\mm$ consisting of those classes in $J_\mm$   which are invertible  at each $P\in S$. 
Moreover, let 
\[ \theta: Y\to J_\mm\]
  be the map assigning 
to a point
$y\in Y$  the class of the divisor $y-y_0$.

Alternatively, the group $J_\mm$ can be described as the group of isomorphism classes of pairs $(L,\iota)$ where $L$ is a line bundle of degree $0$ on $\bar{Y}$ and
$\iota$ is a trivialisation of $L$ on $S$. The image of a divisor $D$ of degree $0$  on $\bar{Y}$ is the line bundle $\Oh(D)$ together with the canonical trivialisation, which exists because $\Oh(D)|_U=\Oh|_U$ outside the support of the divisor, in particular on $S$. In the case $S=\emptyset$, this identification is the familiar isomorphism $J\isom\Pic^0(\bar{Y})$.

By \cite[Chapter V Proposition 2]{serre_class}, the group $J_\mm$ is an algebraic group and by Serre's Proposition 4 the map 
 $\theta$ is a morphism of algebraic varieties. 
By \cite[Chapter V Theorem 2]{serre_class}, the pair has a universal property for morphisms into commutative algebraic groups mapping $y_0$ to $0$: given a rational map $f : \bar{Y} \to G$ to  a commutative algebraic group $G$ admitting $\mm$ for a modulus; see \cite[Chapter I, Theorem 1]{serre_class}, there exists a unique algebraic homomorphism $F: J_\mm \to G$ such that 
\[ f = F  \circ \theta + f(y_0).\]

The structure of $J_\mm$ is explained in \cite[Chapter V Section  13]{serre_class}.
In the case $\mm=0$, we get back the usual Jacobian of $\bar{C}$. This is an abelian variety. In our special case, the kernel $L_\mm$ is isomorphic to $\Gm^r$ where
\[r=\begin{cases}0&\text{for }\mm=0,\\
              \deg\mm-1&\text{for }\mm\neq 0.
\end{cases}\]

We put $J(Y):=J_\mm$ and obtain up to isomorphism the short exact sequence
\[ 1\to \Gm^r\to J(Y)\to J(\bar{Y})\to 0\]
of commutative algebraic groups, making $J(Y)$ semi-abelian.

The classifying map of $J(Y)$  maps a lattice of rank $r$ to $J(\bar{Y})^\vee\isom\Pic^0(\bar{Y})\isom J(\bar{Y})$. 

\begin{lemma}[Serre {\cite[Section~1]{serre_troisieme}}]\label{lem:class_jac}
The classifying map of $J(Y)$ is given by the map
\[ \Z[S]^0\to J(\bar{Y})\isom J(\bar{Y})^\vee\]
induced by $\theta$, where $\Z[S]^0$ is the group of divisors of degree $0$ supported on the $S$.
\end{lemma}

\subsection {Generalised Jacobian over $\C$}

The structure of $J_\mm$ over $\C$ is explained in \cite[Chapter V, \S 19]{serre_class}.
Serre shows
\[ J_\mm(\C)\isom H^0(\bar{Y},\Omega^1(-\mm))^\vee/H_1^\sing(Y,\Z).\]                                 
This implies that the map induced by $Y\to J(Y)$ induces
an isomorphism
\[ H_1^\sing(Y,\Z)\to H_1^\sing(J(Y)(\C),\Z).\]

\begin{rem}
Actually, this isomorphism is shown in \cite{serre_class} on the way to establishing the
formula for $J_\mm(C)$.
\end{rem}

\chapter{Lie Groups}\label{lie}
We review the construction and properties of the exponential map,
fixing notations and normalisations for later.
\index{Lie group}
\section{The Lie Algebra}
\index{Lie algebra}
Let $G^\an$ be a connected commutative complex Lie group.
We denote $\gg_\C$ or $\Lie(G^\an)$ the Lie algebra of invariant vector fields on $G^\an$ and by $\gg_\C^\vee$ or $\coLie(G^\an)$ the dual space of invariant differential forms. 
Note that $\gg_\C$ is abelian, i.e. the Lie bracket is trivial and does not play a role in what follows.

If $V$ is a $\C$-vector space, we can view it as a complex commutative Lie group $V^\an$. In this case $\Lie(V^\an)=V$.

\begin{ex}
For $\Ga=\Spec(\Z[t])$ we have $\Ga^\an=\C$. It has a canonical coordinate with the property $t(1)=1$. Then its Lie algebra $\gg_a^\an$ is generated by $\frac{d}{dt}$ and its dual by $dt$. The canonical identification
of $\gg_a^\an$ with $\Ga^\an$ maps $\frac{d}{dt}$ to $1$.
\end{ex}

Morphisms in the category of connected commutative complex Lie groups are called \emph{analytic homomorphisms}.\index{analytic homomorphism} The assignments $\Lie$ and $\coLie$ are functors. Given an analytic homomorphism $\varphi:G^\an\to H^\an$ of connected commutative complex Lie groups, we get by push-forward of vector fields and pull-back of differential forms  $\C$-linear maps
\[ \varphi_*=d\varphi:\gg_\C\to \hh_\C,\quad \varphi^*=\delta\phi:\hh_\C^\vee\to\gg_\C^\vee.\]
The linear maps are adjoint with respect to the natural pairings between a space and its dual space and this  means that $(\varphi_*)^\vee = \varphi^*$.
If $(\ ,\  )$ is the pairing which defines duality, then
\[
\bigl(\varphi^*(\omega),X)=\bigl(\omega,\varphi_*X\bigr)
\]
for every invariant differential form $\omega$ on $H$ and invariant vector field $X$ on $G$.

Of particular interest is the case of analytic homomorphisms $\varphi:\C\to G^\an$ and $X={d\over{dt}}$.

\section{The Exponential Map}\label{ssec:exp}
\index{exponential map!of a Lie group}
For any  given vector field $X\in \gg_\C$ there exists a unique analytic homomorphism $\varphi_X:\Ga^{\an}\rightarrow G^{\an}$ such that its tangent map $d\varphi_X$ satisfies
$$(d\varphi_X)\left({d\over dt}\right)=X$$ 
as can be found in \cite{warner}.
It amounts to solving an ordinary linear differential equation.
This is used to construct the exponential map of the Lie group $G^\an$ in the following way. We choose  $X\in \gg_\C$ and put 
$$ \exp_{G}(X):=\varphi_X(1).$$
This defines  an analytic homomorphism $\exp_{G}:\gg_\C\rightarrow G^\an$.
In other words, $\exp_G$ is uniquely characterised by functoriality with respect to analytic group homomorphisms and $d\exp_G=\id$. 
This leads to an exact sequence 
\begin{eqnarray}\label{expseq}
\xymatrix{
0\ar[r]&\Lambda\ar[r] &\gg_\C  \ar[r]^{\exp_G\;\;\;} & G^\an \ar[r] &0.
}
\end{eqnarray}
\begin{prop}
Let $G^\an$ be a connected commutative complex Lie group. Then
$\exp_G:\gg_\C\to G^\an$ is the universal cover.
\end{prop}
\begin{proof}
The map $\exp_G$ is unramified because $d\exp=\id$ is an isomorphism.
It is a cover by the sequence (\ref{expseq}). 
This makes it the universal cover.
\end{proof}

\begin{rem}
This also means that $\Lambda$ is the fundamental group of
$G^\an$. We will deduce an explicit identification from the point of view
of paths below.
\end{rem}

An analytic homomorphism $\varphi: G^\an \rightarrow H^\an$
induces a commutative diagram
\begin{eqnarray}
\xymatrix{
0\ar[r] & \Lambda_G \ar[r]^{\iota_G}\ar[d]^{\nu}& \gg_\C \ar[r]^{\exp_G}\ar[d]^{d\varphi} & G^\an \ar[r]\ar[d]^{\varphi}&0\\
0\ar[r] &\Lambda_H \ar[r]^{\iota_H}& \hh_\C \ar[r]^{\exp} & H^\an \ar[r]& 0
}
\end{eqnarray}
The converse  does not hold necessarily: A  Lie algebra homomorphism $\theta:\gg_\C\to\hh_\C$ does not necessarily descend to an analytic homomorphism from $G^\an$ to $H^\an$. However, this does hold if $G^\an$ is simply connected:

\begin{lemma}\label{lem:sconn}
Let $G^\an$ be simply connected,  $\theta: \gg_\C \rightarrow \hh_\C$ a linear map. Then there exists an analytic homomorphism $\Theta: G^\an\rightarrow H^\an$ such that $d\Theta=\theta$. 
\end{lemma}
\begin{proof}
As $G^\an$ is simply connected, it agrees with its universal cover. Hence
$\exp_G$ is an isomorphism of connected commutative complex Lie groups. We define 
\[ \Theta= \exp_H\circ\theta\circ\exp_G^{-1}.\]
\end{proof}

The group $G^\an$ is simply connected  if and only if $G^\an$ is a vector group $V\isom \Ga^n$.

\section{Integration over Paths}\label{ssec:paths}
Another pairing is obtained by integration. Let $\gamma:[0,1]\rightarrow G^{\an}$ be any path in $G^{\an}$. This path defines an element $I(\gamma)$ in $\gg_\C$ by putting 
$$
I(\gamma)(\omega):=\int_\gamma\omega \quad\text{for all $\omega\in\gg^\vee_\C$}.
$$
By Stokes's Theorem we see that $I(\gamma)$ depends only on the homotopy class of $\gamma$. The pairing is non-degenerate so that  $I(\gamma)=0$ implies that $\gamma$ is closed and homotopically equivalent relative $\{0,1\}$ to a constant path. 

\begin{ex}
In the case that $G=\Ga$ and $\epsilon:[0,1]\rightarrow{\G}_a^{\an}$ is the path going from $\epsilon(0)=0$ straight to $\epsilon(1)=1$ we have 
\[ I(\epsilon)=\left(- ,{d\over{dt}}\right).\]
 In fact every invariant differential form on $\Ga$ is a constant multiple of $dt$  and everything reduces to the calculation 
\[
\int_0^1{dt}=1=\left(dt,{d\over{dt}}\right).
\]
\end{ex}

Let $\varphi:G^\an\to H^\an$ be an analytic homomorphism. Then we have for all
invariant differential forms $\omega\in\hh_\C$ and paths $\gamma$ on $G^\an$
\[ I(\varphi_*\gamma)(\omega)=I(\gamma)(\varphi^*\omega)\]
by the transformation rule.

Fix now an element $X$ in $\gg_\C$ and let $\varphi_X$ be the analytic homomorphism from $\Ga^{\an}$ to $G^{\an}$ determined by $X$. Also let 
\[ \gamma_X:[0,1]\rightarrow G^{\an}\]
 be the path obtained by restricting the analytic homomorphism $\varphi_X$ to the interval $[0,1]$. 
Note that by definition $\gamma_X=\varphi_X\circ\epsilon=\varphi_{X,*}\epsilon$.

We thus have defined maps $I$ and $X\mapsto \gamma_X$ assigning tangent vectors to 
paths and conversely.

\begin{lemma}
We have $I(\gamma_X)=X$.
\end{lemma}
\begin{proof}
Let $\omega$ be in $\gg_\C^\vee$. 
Since $\gamma_X$ is a restriction of the analytic homomorphism $\varphi_X$ we see that $\gamma_X^*\omega $ is 
an invariant differential form in 
$\gg_{a,\C}$.
Then
\begin{align*}\label{dual}
I(\gamma_X)(\omega)&= I(\varphi_{X,*}\epsilon)(\omega)=I(\epsilon)(\varphi_X^*\omega)=\left(\varphi_X^*\omega,\frac{d}{dt}\right)=\left(\omega,\varphi_{X}^*\frac{d}{dt}\right)=(\omega,X)
\end{align*}
by the transformation formula for integrals together with  $I(\epsilon)=( -,d/dt)$. 
This means that $I(\gamma_X)=X$. 
\end{proof}
In particular we may start with the element $X=I(\gamma)$. Then 
\[
I(\gamma)=X=I(\gamma_X))=I(\gamma_{I(\gamma)})
\] 
whence $\gamma$ is homotopic to $\gamma_{I(\gamma)}$. This gives $\gamma(1)=\gamma_{I(\gamma)}(1)=\exp_G(I(\gamma))$ and  leads to  the following 

\begin{cor}
Let $P$ be a point in $G^\an$ and $\gamma$ a path from $0$ to $P$. Then we have 
$$
\exp_{G}\bigl(I(\gamma)\bigr)=P.
$$
\end{cor}

The lemma shows that integration is inverse to exponentiation as it should be.  But this is precisely the definition of a logarithm and we may write 
$$
\log_G(P):=I(\gamma).
$$ 
Note that $\log_G$ is multivalued.
The map $\gamma\mapsto I(\gamma)$ from the path space $\mathcal L_G(0)$ of $G$, with the unit element of the group as base point, taken modulo homotopy into the Lie algebra $\gg_\C$ identifies $\gg_\C$ with the universal covering space  of $G^\an$. 

We now restrict to closed paths. The maps 
\begin{equation}\label{eq:fund} \begin{xy}\xymatrix{  
\Lambda\ar@/^/[rr]^{X\mapsto \gamma_X}&&\pi_1(G^\an,0)\ar@/^/[ll]^{I}
}\end{xy}\end{equation}
are inverse to each other; in particular $\Lambda\isom\pi_1(G,0)$ and the fundamental group is abelian.

Let $\sigma=\sum_{i=1}^na_i\gamma_i$ be a chain with $a_i\in\Z$, $\gamma_i:[0,1]\to G^\an$ continuous. We extend the definition of $I$ and put
\[ I(\sigma)=\sum_{i=1}^n a_i I(\gamma_i)\in\gg_\C.\]
If $\gamma$ is closed, but $\gamma(0)\neq 0$, then $p(I(\gamma))$  is still homologous to $\gamma$. Hence
\begin{equation}\label{eq:homol} \begin{xy}\xymatrix{  
\Lambda\ar@/^/[rr]^{X\mapsto\gamma_X}&&H^\sing_1(G^\an,\Z)\ar@/^/[ll]^{I}
}\end{xy}\end{equation}
are inverse to each other. The two identifications are compatible with
the Hurewitz map $\pi_1(G^\an,0)\to H_1^\sing(G^\an,\Z)$, which is an isomorphism in this case.


\chapter{The Analytic Subgroup Theorem}\label{ch:subgroup}

In this chapter, we give a  new formulation of the Analytic Subgroup Theorem. We then explore its consequences for the comparison of analytic and algebraic homomorphisms.

\section{The Statement}

Let $G$ be a commutative and connected  algebraic group defined over $\Qbar$ and $\gg$ its Lie algebra. The associated complex manifold $G^\an$ is a complex Lie group and its Lie algebra is $\gg_\C=\gg\otimes_{\Qbar} \C$. The exponential map 
\[ \exp_{G}: \gg_\C \rightarrow G^\an\]
  from the Lie algebra 
into $G^\an$ defines  an analytic homomorphism.
If $\bb \subseteq \gg$ is a subalgebra and $\bb_{\C}=\bb_\Qbar\otimes\C$ we denote by $B$ the analytic subgroup 
$\exp_G(\bb_\C)$. An obvious question one can ask is whether $B(\Qbar):=B\cap G(\Qbar)$ can contain an 
algebraic point different from $0$, the neutral element. The answer is given by the Analytic Subgroup Theorem.

\begin{theorem}[Wüstholz {\cite{wuestholz-icm}, \cite{wuestholz-subgroup}}]
\label{ast1} 
\index{Analytic Subgroup Theorem}
The group of algebraic points $B(\Qbar)$ is non-trivial if and only if there is a connected algebraic subgroup $H\subseteq G$ with Lie algebra $\hh$ such that $\{0\}\neq \hh \subseteq \bb$.
\end{theorem}

We conclude that the only source for algebraic points is the obvious one. Note that $B(\Qbar)\ne \{0\}$ implies that $\bb \neq  \{0\}$. 

There is a refined version of the theorem. 
To state it let $G$ be a connected commutative algebraic group over $\Qbar$ with Lie subalgebra $\gg$ and let $\langle \;, \; \rangle$ be the duality pairing between $\gg^\vee$ and $\gg$. For
 $u$ in $\gg_\C $ with $\exp_G(u) \in G(\Qbar)$ we denote by $\Ann(u)$ the largest subspace of $\gg^\vee$ (sic) such that $\langle \Ann(u), u\rangle=0$.

\begin{theorem}\label{thm:annihilator}
Assume that $\exp_G(u)\in G(\Qbar)$. Then
there exists an exact sequence 
\begin{eqnarray*} 
0\rightarrow H \rightarrow G \xrightarrow{\pi} G/H\rightarrow 0
\end{eqnarray*}
of connected commutative algebraic groups defined over $\Qbar$ such that $\Ann(u) = \pi^*({\gg/\hh})^\vee$ and $u\in\hh_\C$, where $\hh$ is the Lie algebra of $H$. The sequence is uniquely determined by these properties.
\end{theorem}
\begin{proof}
We write $P=\exp_G(u)$. 
If $u=0$, the theorem holds with $H=0$. If $u\neq 0$, but
$P=\exp_G(u)=0$, we may replace $u$ by $\frac{1}{n}u$ for a big enough $n\in\Na$. We then have $\exp_G(\frac{1}{n}u)\neq 0$ because the kernel of $\exp_G$ is discrete. Moreover, the image point is a torsion point of $G$, hence in $G(\Qbar)$. Without loss of generality, we may assume that $P\neq 0$.

Let $\langle\;,\;\rangle: \gg^\vee\times \gg\rightarrow \Qbar$ be the natural duality pairing and for any  subalgebra $\aa\subset\gg$ denote the left kernel by 
\[ \aa^\perp=\{\lambda\in \gg^\vee; \langle\lambda,\aa\rangle = 0\}.\]
 The right kernel is defined similarly. We put  $\bb:= \Ann(u)^\perp \subseteq \gg$. Then $\bb_\C$ contains $u$  and  Theorem~\ref{ast1} 
gives an algebraic subgroup $H\subset G$ with Lie algebra $\hh$. We may assume that $u\in\bb_\C$, otherwise we apply our arguments to $G/H$. It has smaller dimension, so the process must stop after finitely many steps.
Taking the left kernels  gives  $\bb^\perp\subset \hh^\perp\subset\Ann(u)$ and then $\bb^\perp=\Ann(u)=\hh^\perp$ by the maximality property of $\Ann(u)$. We get an exact sequence of Lie algebras 
\begin{eqnarray*}
0\rightarrow \hh \rightarrow \gg \xrightarrow{\pi_*} \gg/\hh \rightarrow 0
\end{eqnarray*}
which corresponds to an exact sequence 
\begin{eqnarray*}
0\rightarrow H \rightarrow G \xrightarrow{\pi} G/H \rightarrow 0
\end{eqnarray*}
of algebraic groups and by duality to the exact sequence
\begin{eqnarray*}
0\rightarrow \left(\gg/\hh\right)^\vee \xrightarrow{\pi^*} \gg^\vee \rightarrow \hh^\vee \rightarrow 0.
\end{eqnarray*}
We have $\hh^\perp=\pi^*(\gg/\hh)^\vee$ and $(\gg/\hh)^\perp=\hh^\vee$, which we prove as follows:
we have $\lambda\in \hh^\perp$ if and only if the restriction of $\lambda$  to  $\hh$ is zero. This implies that $\lambda$ descends to $\gg/\hh$ and that there is an element $\mu\in (\gg/\hh)^\vee$ with $\lambda=\pi^*\mu$. This leads to  $\hh^\perp\subseteq \pi^*(\gg/\hh)^\vee$. Conversely 
\[\langle \pi^*(\gg/\hh)^\vee,\hh\rangle=\langle (\gg/\hh)^\vee,\pi_*\hh\rangle=0\]
since $\pi_*\hh=0$ and we conclude  that $\Ann(u)=\hh^\perp=\pi^*(\gg/\hh)^\vee$ as stated.

Suppose that there is another short exact sequence
\[ 0\to H'\to G\xrightarrow{\pi'} G/H'\to 0\]
with the same properties. In particular $\pi'^*(\gg/\hh)^\vee= \pi^*(\gg/\hh')^\vee$ as subobjects of $\gg^\vee$. This implies that $\hh=\hh'$ as subspaces
of $\gg$. As $H$ and $H'$ are connected, this also implies $H=H'$ as subgroups of $G$.
\end{proof}

\section{Analytic vs. Algebraic Homomorphisms}

The Subgroup Theorem also has a consequence for the category of groups itself.
A connected commutative algebraic group over $\Qbar$ gives rise to a complex Lie group. We recall that morphisms in the category of complex Lie groups are called \emph{analytic homomorphisms}. 

\begin{thm}\label{thm:hom_algebraic}
\index{analytic homomorphism}
Let $G,G'$ be connected commutative algebraic groups defined over $\Qbar$ with Lie algebras $\gg$ and $\gg'$ and let $\phi:G^\an\to {G'}^\an$ be an analytic group homomorphism such that $\gg_\C\to\gg'_\C$ maps $\gg$ to $\gg'$.

Then there exists a vector group $V_\trans$, a connected commutative algebraic group $G_\alg$ and a decomposition 
\[ G\isom V_\trans\times G_\alg\]
such that $\phi|_{G_\alg^\an}$ is algebraic over $\Qbar$ and $\phi|_{V_\trans^\an}$ is
purely transcendental, i.e.  $\phi(V_\trans(\Qbar))$ does not contain any non-zero algebraic values.
\end{thm}

\begin{rem}
\begin{enumerate}
\item
An earlier version claimed the same corollary but without the $V_\trans$-factor. We thank the referee for pointing out the mistake in
the argument. Indeed, the statement would be false as the example
$\exp:\C\to\C^*$ shows. The theorem says that all counterexamples are of a similar nature, see Corollary~\ref{cor:class_hom_algebraic} for a complete classification.
\item
The assumption on the induced map on Lie algebras is necessary as the following example shows: let $G_1=\G_m$, and $G_2=E$ an elliptic curve over $\Qbar$. 
Let
\[ \C\xrightarrow{z\mapsto \exp(2\pi iz)}\C^*=\G_m^\an\]
be the standard uniformisation. For $E^\an$ we use the explicit uniformisation
\[ \exp_E: \C\to E^\an\]
with kernel $\Lambda=\Z\omega_1+\Z\omega_2$ of Section~\ref{ssec:class}.
In these coordinates the $\Qbar$-coLie algebras of $\G_m$ and $E$ are
generated by $dz/z$ and $dz$ respectively.

We get a well-defined analytic homomorphism 
\[ \phi:\G_m^\an=\C/ 2\pi i \Z\to E^\an=\C/\Lambda\]
by mapping $z\mapsto (\omega_1/2\pi i)z$. 
It is not algebraic. Note that this does not
contradict Theorem~\ref{thm:hom_algebraic} because  it
does not map $\gg_1$ to $\gg_2$ since $\omega_1/2\pi i$ is not in $\Qbar$ as we shall see later.
\end{enumerate}
\end{rem}
The proof of this theorem will take the rest of this chapter.

\begin{lemma}\label{lem:A.5}
Suppose that the set of torsion points $G_\tor$ is dense in $G$. Under the 
assumptions of Theorem~\ref{thm:hom_algebraic}, the morphism $\phi$
is algebraic.
\end{lemma}
\begin{proof}
Let $B\subset G^\an\times {G'}^\an$ be the graph of $\phi$. It is connected because it is isomorphic to $G^\an$ via the first projection.
By assumption its Lie algebra is defined over $\Qbar$. 
Let $g\in G(\Qbar)$ be
an $N$-torsion point. Then $\phi(g)\in G'(\C)$ is also an $N$-torsion point, hence in $G'(\Qbar)$. 
This implies that
\[ T:=\{(g,\phi(g))|g\in G_\tor\}\subset B(\Qbar).\]
By the Analytic Subgroup Theorem there is an algebraic subgroup $H\subset G\times G'$ defined over $\Qbar$ such that
$H^\an\subset B$ and containing all of $T$. The projection $B\hookrightarrow G^\an\times {G'}^\an\to G^\an$
is an isomorphism, hence its restriction  to $H$ is a closed immersion. The image contains the set $G_\tor$. It is Zariski dense, hence the inclusion is surjective. In other words, $H^\an= B$. The group $H\subset G\times G'$ is the graph of the morphism we wanted to find.
\end{proof}

\begin{lemma}\label{lem:A.6}
The theorem holds if $G=V$ is a vector group.
\end{lemma}
\begin{proof}Let $\Sigma=\phi^{-1}(G'(\Qbar))\cap V(\Qbar)$. We denote by $V_\Sigma\subset V$ the smallest algebraic subgroup containing $\Sigma$. We choose
 a direct complement $V_\trans$ of $V_\Sigma$ in $V$. 
By construction,
$\phi|_{V_\trans^\an}$ is purely transcendental. Indeed, any $\sigma\in V_\trans(\Qbar)$ with $\phi(\sigma)\in G'(\Qbar)$ is already in $\Sigma$ and hence
in $V_\Sigma(\Qbar)$.

It remains to show
that $\phi|_{V_\Sigma^\an}$ is algebraic.
As in the last lemma we consider its graph $B$ in $V_\Sigma^\an\times {G'}^\an$.
 Its Lie algebra is defined over $\Qbar$ and it contains the set
\[ T:=\{(g,\phi(g))|g\in \Sigma\}\subset B(\Qbar).\]
By the Analytic Subgroup Theorem there is an algebraic subgroup $H\subset V_\Sigma\times G'$ defined over $\Qbar$ such that
$H^\an\subset B$ and containing all of $T$. The projection $B\hookrightarrow V_\Sigma^\an\times {G'}^\an\to V_\Sigma^\an$
is an isomorphism, hence its restriction  to $H$ is a closed immersion. The image is an algebraic subgroup containing the set $\Sigma$, hence equal to
$V_\Sigma$.
In other words, again $H^\an= B$. The group $H\subset G\times G'$ is the graph of the morphism we wanted to find.
\end{proof}

\begin{lemma}
Let $G_1\to G_2$ be a vector extension. Then
$(G_1)_\tor=(G_2)_\tor$.
\end{lemma}
\begin{proof}It suffices to check the statement over the complex numbers and in the analytification. We have
\[ G_1^\an\isom\C^{n_1}/H_1^\sing(G_1^\an,\Z)\to G_2^\an\isom \C^{n_2}/H_1^\sing(G_2^\an,\Z).\]
By homotopy invariance, $H_1^\sing(G_1^\an,\Z)\isom H_1^\sing(G_2^\an,\Z)$.
The torsion
is computed as $G_{i,\tor}\isom H_1^\sing(G_i^\an,\Z)\tensor \Q/\Z$, hence the torsions of of $G_1$ and $G_2$ are isomorphic. 
\end{proof}

Let $V$ be the vector part of $G$, i.e. we have a short exact sequence
\[ 0\to V\to G\to G^{sa}\to 0\]
with $V$ a vector group and  $G^{sa}$ semi-abelian.
We say that $G$ is \emph{completely non-trivial}\index{completely non-trivial vector extension} (as vector extension) if it does not have a direct factor
$\G_a$. In other words: the classifying 
map
\[ V^\vee\to \Ext^1(G^{sa},\G_a)\]
is injective, see  Corollary~\ref{cor:vector_ext2}.

\begin{lemma}\label{lem:A.7}
Let $G_1$ be the Zariski closure of $G_\tor$ in $G$.
Then $G_1$ is a completely non-trivial vector extension of $G^{sa}$. Moreover, 
there is  a decomposition 
\[ G\isom V_1\times G_1\] 
with a vector group $V_1$, i.e. $G_1$ is the maximal completely non-trivial subextension of $G^{sa}$ contained in $G$.
\end{lemma}
\begin{proof}
We have $G_\tor\isom G^{sa}_\tor$, hence the image of
$G_1\to G^{sa}$ contains all torsion points. They are dense in $G^{sa}$, hence
$G_1\to G^{sa}$ is surjective. This makes $G_1$ a vector extension of $G^{sa}$.
By construction, $(G_1)_\tor$ is dense in $G_1$.
If it was not completely non-trivial, we would have a decomposition
$G_1=G_2\times \G_a$ and the torsion points would not be dense.

Finally, let $V$ be the vector part of $G$, $W=G_1\cap V$ and choose a direct complement $V_1$ of $W$ in
$V$. The natural map
\[ G_1\times V_1\to G\]
is an isomorphism.
\end{proof}

\begin{proof}[Proof of Theorem~\ref{thm:hom_algebraic}.]
By Lemma~\ref{lem:A.7} we have
\[ G\isom G_1\times V_1\]
such $G_\tor$ is dense in $G_1$ and $V_1$ is a vector group.
By Lemma~\ref{lem:A.5}, the theorem holds for $G_1$. 

By Lemma~\ref{lem:A.6}, there is a decomposition
$V_1\isom (V_1)_\Sigma\times V_\trans$ such that $\phi$ is algebraic on
$(V_1)_\Sigma$ and purely transcendental on $V_\trans$.
This completes the proof.
\end{proof}

The interplay between algebraic and transcendental morphisms is subtle.
In the situation of Theorem~\ref{thm:hom_algebraic} let
\[ G\isom V_1\times G_1, \hspace{2ex}G'\isom V_2\times G_2\]
be decompositions of the algebraic groups $G$ and $G'$ into a vector group and a completely non-trivial
vector extension of its semi-abelian part, as in Lemma~\ref{lem:A.7}. The analytic homomorphism
$\phi$ decomposes as a
$(2\times 2)$-matrix
\begin{eqnarray*} 
\phi=\begin{pmatrix}
\phi_{11} & \phi_{12} \\
\phi_{21} & \phi_{22}
\end{pmatrix}
\end{eqnarray*}
with $\phi_{11} \in \Hom(V_1^\an,V_2^\an)$, $\phi_{12}\in \Hom(V_1^\an,G_2^\an)$ and with  $\phi_{21}\in \Hom(G_1^\an,V_2^\an)$ and $\phi_{22}\in 
\Hom(G_1^\an,G_2^\an)$.

\begin{cor}\label{cor:class_hom_algebraic}
In this situation, we have:
\begin{enumerate}
\item $\phi_{11}$ and $\phi_{22}$ are algebraic and defined over $\Qbar$;
\item $\phi_{21}=0$;
\item There is a decomposition $V_1\isom V_{1,\trans}\times  V_{1,\alg}$  such
that  the maps 
\[ V_{1,\trans}^\an\to G_2^{sa,\an},\quad V_{1,\trans}^\an\to G_2^{sa,\an}\]
induced by $\phi_{12}$ 
are purely transcendental   and algebraic over $\Qbar$, respectively. 
\end{enumerate}
\end{cor}
\begin{proof}
By the proof of Theorem~\ref{thm:hom_algebraic},
we have $G_1\subset G_\alg$ and $V_\trans\subset V_1$. In particular,
$\phi_{21}$ and $\phi_{22}$ are algebraic and defined over $\Qbar$.
If $\phi_{21}:G_1^\an\to V_2^\an$ was non-zero, we would be able to split off a factor $\G_a$ from $G_1$. This is impossible because $G_1$ is completely non-trivial. 

All analytic homomorphisms $\phi_{11}:V_1^\an\to V_2^\an$ are algebraic over $\C$. It agrees with the analytification of the $\C$-linear map
$\mathfrak{v}_{1,\C}\to\mathfrak{v}_{2,\C}$.  By assumption it is induced by a $\Qbar$-linear
map $\mathfrak{v}_1\to\mathfrak{v}_2$, hence it is even algebraic over $\Qbar$.

We decompose $V_1$ as in Theorem~\ref{thm:hom_algebraic} in this special case.
Then $\phi_{12}$ is algebraic and defined over $\Qbar$ on $V_{1,\alg}$ and
purely transcendental on $V_{1,\trans}$. 
It remains to show that the composition
$V_{1,\trans}\to G_2\to G_2^{sa}$ is purely transcendental.
In order to simplify notation, we write $W$ for $V_{1,\trans}$ and $G'$ for $G_2$. We apply Theorem~\ref{thm:hom_algebraic} to $W\to {G'}^{sa}$. Accordingly there is a decomposition $W\isom W_\trans\times W_\alg$ such that the map is algebraic on $W_\alg$ and purely transcendental on $W_\trans$. The algebraic map
$W_\alg\to {G'}^{sa}$ vanishes because $W_\alg$ is a vector group and
${G'}^{sa}$ is semi-abelian. This implies that we get an induced algebraic map
$W_\alg\to V'$ where $V'$ is the vector part of $G'$. This contradicts
that $W_\alg^\an\to {G'}^\an$ is purely transcendental. We conclude that  $W_\alg$ is
in fact $0$ and $W^\an\to {G'}^{sa,\an}$ is purely transcendental.
\end{proof}

\chapter{The Formalism of the Period Conjecture}\label{ch:formalism}
The Period Conjecture predicts relations between the periods of algebraic varieties or, more generally, periods of motives. We explain the abstract set-up behind the explicit formulation. Our machinery will be applied
mostly to periods of $1$-motives, but also in a couple of other cases.

\section{Periods}

We first introduce periods and formal periods, 
following \cite[Definition~5.1.1]{period-buch} and \cite[Definition~3.6]{huber_galois}. 

Throughout we fix subfields $K,L\subset\C$. Their compositum $KL$ is the subfield generated by $K$ and $L$. To simplify notation, we work under the hypothesis $K\cap L=\Q$.

\begin{defn}\label{defn:VV}
\begin{enumerate}
\item
Let $\VV$ be the category of tuples $ V=(V_K,V_L,\phi)$
 where
$V_K$ and $V_L$ are finite dimensional vector spaces over $K$ and $L$, respectively, and $\phi:V_K\tensor_K\C\to V_L\tensor_L\C$ a $\C$-linear isomorphism.
    Morphisms are pairs of linear maps such that the diagram 
\[\begin{xy}\xymatrix{
   V_K\tensor_K\C\ar[r]^{f_K\tensor_K\C}\ar[d]_{\phi_V}&W_K\tensor_K\C\ar[d]^{\phi_W}\\
 V_L\tensor_L\C\ar[r]_{f_L\tensor_L\C}&W_L\tensor_L\C
}\end{xy}\]
commutes.
\item
Given $V\in \VV$, we define the \emph{set of periods} \index{periods!of $\in\VV$}
of $V$ as
\[ \Per(V)=\im (V_K\times V_L^\vee\to \C),\hspace{2ex} (\sigma,\omega)\mapsto \omega_\C(\phi(\sigma_\C))\]
and the \emph{space of periods} $\Per\langle V\rangle$ as the additive group generated by it. Here we write $\sigma_\C$ and $\omega_\C$ for the images
of $\sigma$ and $\omega$ in $V_K\tensor \C$ and $V_L^\vee\tensor \C$, respectively.
\item
If $\Ch$ is a category, $V:\Ch\to\VV$ a functor, we put
\[ \Per(\Ch)=\bigcup_{X\in\Ch}\Per(V(X)).\]
\end{enumerate}
\end{defn}

\begin{rem}
\begin{enumerate}
\item
The category $\VV$ is $\Q$-linear and abelian. We  could even turn into a rigid tensor category and then $V\mapsto V_K$ becomes a so-called fibre functor if $K\subset L$. We are not going to use this fact. 
\item The abelian group $\Per\langle V\rangle$ is even a $KL$-vector space because of the bilinearity of the map $V_K\times V_L^\vee\to\C$.
It has an alternative interpretation as the $KL$-vector space generated by the entries of the \emph{period matrix},\index{period matrix} the matrix of $\phi$ in a $K$-basis of $V_K$ and an $L$-basis of
$V_L$. 
\item The set $\Per(\Ch)$ only depends on the objects in the image $V(\Ch)$. 
\item With $L=\Q$ this is the definition given in in \cite{period-buch} and \cite{huber_galois}. In the present monograph, the case $K=\Q$ , $L=\Qbar$ will be of most interest because $1$-motives are a homological theory, whereas
the other references take the cohomological point of view.
In both cases we want to compare de Rham \emph{cohomology} (the $\Qbar$-component) with singular \emph{homology} (the $\Q$-component).
\item We may replace the category $\Ch$ and the functor
$V$ by a diagram $D$ (i.e. an oriented graph) and a representation $V$. Its periods are simply defined as the periods of the path category of $D$ and the induced functor with values in $\VV$. This is the point of view taken originally by Nori and also in \cite{period-buch}. It will only play a very minor role in our monograph,  in the proof of Theorem~\ref{thm:main_kontsevich} on the Period Conjecture for curves.
\end{enumerate}
\end{rem}

\begin{ex}
The main case of interest for us is the category of iso-$1$-motives
over $\Qbar$, see Chapter~\ref{sec:one-mot} below. The functor
$V$ is given by the singular realisation, the de Rham realisation  and by the period isomorphism.
\end{ex}

\begin{ex}\label{ex:triangle}
Given a short exact sequence
\[ 0\to V_1\to V\to V_2\to 0\]
in $\VV$, the period matrix for $V$ (in adapted bases) is upper block triangular, i.e. of the form
\[\left(\begin{matrix} A&B\\ 0&C\end{matrix}\right)\]
such that $A$ is the period matrix of $V_1$ and $C$ the period matrix of $V_2$.
In particular,
\[ \Per\langle V_1\rangle+\Per\langle V_2\rangle\subset \Per\langle V\rangle.\]
This is not an equality in general.
\end{ex}

\begin{lemma}\label{lem:vsp_abstract}
Let $\Ch$ be an additive category and $V:\Ch\to \VV$ an additive functor. Then $\Per(\Ch)$ is a $KL$-vector space.
For $M\in\VV$ we have 
\[ \Per\langle M\rangle=\Per(\langle M\rangle)\]
where $\langle M\rangle\subset \VV$ is the full abelian subcategory generated by $M$ and closed under subquotients, i.e. the morphisms in $\langle M\rangle$ are the same as in $\VV$, and for $X\in \langle M\rangle$ and $Y$ a subquotient of $X$ in $\VV$, the object $Y$ is also in $\langle M\rangle$.
\end{lemma}
\begin{proof}It suffices to show that $\Per(\Ch)$ is closed under addition. If $\alpha_1$ is a period of $X_1$ and $\alpha_2$ is a period of $X_2$, then
$\alpha_1+\alpha_2$ is a period of $X_1\oplus X_2$.

As a consequence, the periods of the category $\langle V\rangle$ form an abelian group. They contain the periods of $V$, hence
\[ \Per\langle V\rangle \subset\Per(\langle V\rangle).\]
For the converse inclusion, note that if $V_1$ is a subquotient of
$V_2$, then $\Per(V_1)\subset\Per(V_2)$ by Example~\ref{ex:triangle}. 

Moreover, $\Per(M^n)\subset\Per\langle M\rangle$. As all objects of $\langle V\rangle$ are subquotients of
$M^n$ for some $n$, this shows that $\Per(W)\subset\Per\langle V\rangle$ for
all objects of $\langle V\rangle$.
\end{proof}

There are obvious relations between the periods of a category $\Ch$. They are encoded in a space of formal periods.

\begin{defn}\label{defn:formal_abstract}
Let $\Ch$ be an additive category, $V:\Ch\to\VV$ be an additive functor. The 
\emph{space of formal periods}\index{formal periods} $\Perform(\Ch)$ is the $KL$-vector space generated by symbols
$(\sigma,\omega)$ for $\sigma\in V_K(X)$, $\omega\in V_L(X)^\vee$ for all
objects $X$ of $\Ch$ with relations given by
\begin{itemize}
\item (Bilinearity) for all objects $X$ and   
$\sigma_1,\sigma_2\in V_K(X)$, 
 $\omega_1,\omega_2\in V_L(X)^\vee$, $a_1,a_2\in K$, $b_1,b_2\in L$,
\[
(a_1\sigma_1+a_2\sigma_2,b_1\omega_1+b_2\omega_2)= a_1b_1(\sigma_1,\omega_2)+\cdots+ a_2b_2(\sigma_2,\omega_2).
\] 
\item (Functoriality) for all morphisms $f:X\to Y$ and $\omega\in V_\Qbar(Y)^\vee$, $\sigma\in V_\Q(X)$,
\[ (f^*\omega,\sigma)=(\omega,f_*\sigma).\]
\end{itemize}
\end{defn}
Equivalently, the vector space $\Perform(\Ch)$ can be characterised as the quotient space
\[ \Perform(\Ch)=\left(\bigoplus_{X\in\Ch}V_K(X)\tensor_{\Q} V_L(X)^\vee\right)/\text{\ functoriality}.\]
The bilinearity relation is incorporated into the tensor product. 
\begin{rem}
We could also apply the same definition to formal periods of a diagram $D$ and a representation $V:D\to\VV$. This is the point of view taken in
\cite{period-buch}. The resulting space of formal periods agrees with the space of formal periods of the additive hull of the path category of $D$. 
\end{rem}

It is often useful to break $\Ch$ into smaller pieces.

\begin{defn}\label{defn:single}
Let $\Ch$ be an abelian category, $X$ an object of $\Ch$. By $\langle X\rangle$ \index{subcategory!generated by $X$} we denote the smallest full subcategory of $\Ch$ that contains $X$ and  is closed under subquotients.
\end{defn}
We have $\Per\langle X\rangle=\Per(\langle X\rangle)$ and this shows that this is the right category if we try to understand linear relations between periods of $X$.

\begin{lemma}\label{lem:dim_X}
For an additive functor $\Ch=\langle X\rangle\to \VV$, the elements of
$V_K(X)\tensor_\Q V_L(X)^\vee$ generate $\Perform(\Ch)$ as $KL$-vector space.
In particular,
\[ \dim_{KL}\Perform(\Ch)\leq (\dim_KV_K(X))^2.\]
\end{lemma}

\begin{proof} We need to verify that all generators of $\Perform(\Ch)$ can be expressed in terms of elements of $V_K(X)\tensor_\Q V_L(X)^\vee$.

If $f:Y\to Y'$ is a surjective morphism in $\Ch$, then all elementary tensors in the tensor product $V_K(Y')\tensor_\Q V_L(Y')^\vee$ can be identified with elementary tensors of $V_K(Y)\tensor_\Q V_L(Y)^\vee$ because
$V_K(Y)\to V_K(Y')$ is surjective, i.e. every element in $V_K(Y')$ has the 
form $f_*\sigma$ for some $\sigma\in V_K(Y)$, and in consequence 
\[ f_*\sigma\tensor \omega=\sigma\tensor f^*\omega\in\Perform(\Ch).\]
In the same way, if
$f:Y\to Y'$ is injective, then the elementary tensors on $Y$ can be identified with some elementary tensors on $Y'$ because $V_L(Y')\to V_L(Y)$ is surjective.

 By assumption, every object of $\langle X\rangle$ is
a subquotient of some $X^n$ for $n\geq 1$. Hence it suffices to consider
elementary tensors on $X^n$. 
Note that $V_K(X^n)\isom V_K(X)^n$ and that  we can write an
elementary tensor $\sigma\tensor\omega\in V_K(X^n)\tensor_\Q V_L(X^n)$ as
    \[ \sigma\tensor \omega=\sum_{k=1}^n (i_k)_*\sigma_k\tensor\omega, \]
where $\sigma_1,\dots,\sigma_n$ are the components of $\sigma$. By the functoriality relation this yields the identification 
\[ \sigma\tensor\omega=\sum_{k=1}^n\sigma_k\tensor i_k^*\omega\]
with an element of $V_K(X)\tensor_\Q V_L(X)^\vee$.
\end{proof}

Following Hörmann in \cite{hoermann-notiz}, there is an interesting alternative description of the space of relations in the abelian case. It is closer to the shape in which they will appear in the context of $1$-motives and was motivated by it. 

Given a short exact sequence
\[ 0\to X_1\xrightarrow{i} X^n\xrightarrow{p} X_2\to 0\]
in a $\Q$-linear abelian category $\Ch$ and elements 
$(\sigma_1,\dots,\sigma_n)\in i_*(V_K(X_1))$, 
$(\omega_1,\dots,\omega_n)\in p^*(V_L(X_2)^\vee)$, the functoriality relation implies that
$\sum_{i=1}^n \sigma_i\tensor  \omega_i$ vanishes in $\Perform\langle X\rangle$.
Actually, even the converse is true.

\begin{prop}[{Hörmann \cite{hoermann-notiz}}]\label{prop:hoermann}For an additive functor 
$\Ch=\langle X\rangle\to\VV$,
an element of the form $\sum_{i=1}^n \sigma_i\tensor  \omega_i$  is in
 the kernel of the map $V_K(X)\tensor_\Q V_L(X)^\vee\to\Perform\langle X\rangle$ if and only if there is a is short exact sequence
\[  0\to X_1\xrightarrow{i} X^n\xrightarrow{p} X_2\to 0\]
with $(\sigma_1,\dots,\sigma_n)\in i_*(V_K(X_1))$, $(\omega_1,\dots,\omega_n)\in p^*(V_L(X_2)^\vee)$.
\end{prop}
We omit the proof as we are not going to need this fact.

By construction, the space of formal periods comes with an evaluation map to $\C$.

\begin{defn}
Let $\Ch$ be an additive category, $V:\Ch\to \VV$ an additive functor. We define the \emph{evaluation map}\index{evaluation map!for $V\in\VV$}
\[ \ev:\Perform(\Ch)\to \C\]
by sending a symbol $(\sigma,\omega)\in V_K(X)\times V_L(X)^\vee$ to
\[ (\sigma,\omega)\mapsto \omega_\C( \phi(\sigma_\C) ).\]
\end{defn}

The map is obviously well-defined, $KL$-linear and has image $\Per(\Ch)$.

\begin{defn}
 We define the \emph{external duality functor}\label{defn:external}\index{external duality}
\[ \cdot^\vee: \VV\to \VVarg{L}{K}\]
by assigning the triple $(V_K,V_L,\phi)$ to $(V_L^\vee, V_K^\vee,\phi^\vee)$.
\end{defn}
This functor should not be confused with the internal duality functor
on $\VV$ which maps $(V_K,V_L,\phi)$ to $(V_K^\vee,V_L^\vee, (\phi^\vee)^{-1})$.

\begin{lemma}\label{lem:periods_dual}
 Let $V:\Ch\to \VV$ be an additive functor. Then
the period spaces $\Per(\Ch)$ and $\Perform(\Ch)$ do not change when applying the external duality functor.
\end{lemma}
\begin{proof}
Let $X\in\Ch$. The definition of $\Perform(\Ch)$ via $V$ uses
    $V_K(X)\tensor_{\Q}V_L(X)^\vee$, whereas the definition via $\,\cdot^\vee\circ V$ uses $V_L^\vee\tensor_{\Q} (V_K^\vee)^\vee$. These spaces are identified by exchanging the factors. The compatibility with the evaluation map is the very definition of the dual $\phi^\vee$ of $\phi$. 
\end{proof}

\begin{rem}
In the case of the internal duality, we get the same statement for
$\Perform(\Ch)$, but not longer for actual periods. External duality maps a period matrix \index{period matrix} to its transpose, so the period space remains the same. In contrast, internal duality maps the period matrix to the inverse of the transpose, hence the period space is divided by the determinant, and its periods are  divided by the determinant.  
\end{rem}

\section{The Period Conjecture}\label{sec:formalism_2}

The Period Conjecture asserts that in certain cases the obvious relations are the only ones. We follow \cite{huber_galois}. As in the previous section we fix
subfields $K,L\subset\C$ with $K\cap L=\Q$. The cases of interest for the Period Conjecture are $K=\Q$, $L\subset\Qbar$ or conversely.

\begin{defn}[Huber {\cite[Definition~3.7]{huber_galois}}]\label{defn:conj_abstract}
Let $\Ch$ be an additive category, $V:\Ch\to\VV$ an additive functor. We say that \emph{the Period Conjecture holds} \index{Period Conjecture!for a category $\Ch$} for $\Ch$ if the evaluation map $\Perform(\Ch)\to\Per(\Ch)$ is injective.
\end{defn}

\begin{rem}
If $\Ch$ is the category of all Nori motives over $\Qbar$ (see Appendix~\ref{sec:app_nori}), then this is the Period Conjecture as formulated by Kontsevich in \cite{kontsevich}. We refer to \cite[Part~III]{period-buch} for a detailed discussion. In Theorem~\ref{thm:main_one}, the conjecture is proved for the category of iso-$1$-motives. Note that the above statement does not mention the tensor structure, which exists on the category of all motives (but not on $\onemot_\Qbar$). 
We refer to \cite{huber_galois} for the discussion of tensor products and the comparison of the above conjecture with Grothendieck's version predicting the transcendence degree of the algebra generated by the periods of a motive. The latter does not a play a role in our monograph.
\end{rem}

Note that the space of formal periods $\Perform(\Ch)$ and hence the Period Conjecture only depends on the image of $\Ch$ under $V$. Hence we may assume without loss of generality that
$V$ is faithful.

\begin{prop}[Huber {\cite[Proposition~5.2]{huber_galois}}]\label{prop:is_full}
Let $F:\Ch'\to\Ch$ and $V:\Ch\to \VV$ be a faithful exact functor between $\Q$-linear abelian categories. Then
\[ \Perform(\Ch')\to\Perform(\Ch)\]
is injective if and only if $F$ is full with image closed under taking subquotients.
\end{prop}

\begin{rem}
The proof of this general fact relies on Nori's description of such categories as
categories of comodules for an explicit coalgebra. In the cases of interest for us, we will give a direct proof in later chapters.
\end{rem}

\begin{cor}[Fullness: Huber {\cite[Corollary~5.3]{huber_galois}}]\label{cor:fullness_abstract}\index{Period Conjecture! and fullness}\index{fullness! and the Period Conjecture}
Let $\Ch$ be a $\Q$-linear abelian category and $V:\Ch\to\VV$ a faithful exact functor.
If the Period Conjecture holds for $\Ch$, then $V$ is full with image closed under taking subquotients. 
\end{cor}
\begin{proof} Let $\bar{\Ch}$ be the full subcategory of $\VV$ closed under taking subquotients generated by $\Ch$. 
Then the evaluation map factors as 
\[ \Perform(\Ch)\to\Perform(\bar{\Ch})\to\C.\]
If the composition map is injective, so is the first map.
By applying Proposition~\ref{prop:is_full} to $F=V$ we deduce that $V$ is full, and as a consequence $\Ch$ is
equivalent to $\bar{\Ch}$. 
\end{proof}

The Period Conjecture for an abelian category $\Ch$ can be broken into parts.
Recall from Definition~\ref{defn:single} the subcategory
$\langle X\rangle$ generated by a single object.
We have $\Per\langle X\rangle=\Per(\langle X\rangle)$, so this is the right category if we want to understand linear relations between periods of $X$.

\begin{lemma}[Huber {\cite[Proposition~5.6]{huber_galois}}]\label{lem:conj_single}Let $\Ch$ be an abelian category, $V:\Ch\to\VV$ a faithful exact functor. Then the following statements are equivalent:
\begin{enumerate}
\item\label{it:5.20.1} The Period Conjecture holds for $\Ch$.
\item \label{it:5.20.2} The Period Conjecture holds for $\langle X\rangle$ for all objects
$X$ of $\Ch$.
\end{enumerate}
\end{lemma}
\begin{proof}By Proposition~\ref{prop:is_full} applied to $\langle X\rangle \to \Ch$, the natural map
\[ \Perform(\langle X\rangle )\to \Perform(\Ch)\]
is injective. If $\Perform(\Ch)\to\C$ is injective, so is the composition 
\[ \Perform(\langle X\rangle)\to \Perform(\Ch)\to \C\]
 for every object $X$. This shows that (\ref{it:5.20.1}) implies (\ref{it:5.20.2}). Conversely, we have
\[ \Ch=\bigcup_{X\in\Ch}\langle X\rangle\]
because a morphism $f:X\to Y$ in our abelian category $\Ch$ is already a morphism in the subcategory
$\langle X\oplus Y\rangle$.
As a consequence we have 
\[ \Perform(\Ch)=\varinjlim_{X\in \Ch}\Perform(\langle X\rangle).\]
If the evaluation map is injective for every $X$, it is injective on the inductive limit.
\end{proof}

Using Hörmann's alternative description of the space of formal periods, we can reformulate the conjecture.

\begin{cor}The Period Conjecture holds for the abelian category $\Ch=\langle X\rangle$ if and only if for element $\sum_{i=1}^n\sigma_i\tensor \omega_i\in V_K(X)\tensor_\Q V_L(X)^\vee$ in the kernel of the evaluation map there is a short exact sequence
\[  0\to X_1\xrightarrow{i} X^n\xrightarrow{p} X_2\to 0\]
with $(\sigma_1,\dots,\sigma_n)\in i_*(V_K(X_1))$, $(\omega_1,\dots,\omega_n)\in p^*(V_L(X_2)^\vee)$.
\end{cor}
\begin{proof}Apply Proposition~\ref{prop:hoermann}.
\end{proof}

The advantage of taking the  subcategory $\langle X\rangle$ for an individual $X\in \Ch$  is that its period space is finite dimensional over $KL$; indeed its dimension is bounded by
$\dim_K V_K(X)^2$. Hence it makes sense to ask what the dimension actually is.

\begin{cor}
Let $\Ch$ be a $\Q$-linear additive category, $X$ an object of $\Ch$.
Then the Period Conjecture holds for $\langle X\rangle$ if and only if
\[ \dim_{KL}\Perform(\langle X\rangle)=\dim_{KL}\Per\langle X\rangle.\]
\end{cor}
\begin{proof}The evaluation map $\Perform(\langle X\rangle) \to \Per\langle X\rangle$ for $\langle X\rangle$ is surjective. Hence it is injective if and only if it is an isomorphism and if and only the dimensions of the two finite dimensional vector spaces agree.
\end{proof}
It remains to understand
the dimension of $\Perform(\langle X\rangle)$. This question is answered via Nori's version of Tannaka theory without a tensor product. We recall the main player.

We keep concentrating on the case of an abelian category generated by a single object $X$, in the sense of $\Ch=\langle X\rangle$.

\begin{defn}\label{defn:End}
Let 
$T:\Ch=\langle X\rangle \to \Q\Vect$ be a faithful exact functor. We introduce the spaces
\[
\End(T)=\left\{ (f_Y)\in \prod_{Y\in\Ch}\End_\Q(T(Y))\mid \forall g:Y\to Y':f_{Y'}\circ T(g)=T(g)\circ f_Y\right\}
\]  
and 
\[
\Ah(\Ch,T)=\End(T)^\vee.
\]
\end{defn}

\begin{lemma}\label{lem:is_finite}
If $\Ch=\langle X\rangle$, then elements 
$(f_Y)$ of $\End(T)$ are uniquely determined by $f_X$. In other words,
$\End(T)\subset \End_\Q(T(X))$ and in consequence
\[ \dim_\Q \End(T)\leq (\dim_\Q T(X))^2\]
In particular, the dimension is finite. 
\end{lemma}

\begin{proof}The argument is the same as in the proof of Lemma~\ref{lem:dim_X}, phrased in a slightly different language.

Let $f=(f_Y)$ be in $\End(T)$. Assume that $g:Y\to Y'$ is surjective in $\langle X\rangle$, then so is $T(g):T(Y)\to T(Y')$. From the commutative diagram
\[\xymatrix{
 T(Y)\ar@{->>}[r]^{T(g)}\ar[d]_{f_Y}&T(Y')\ar[d]^{f_{Y'}}\\
 T(Y)\ar@{->>}[r]^{T(g)} &T(Y')
}\]
we deduce that $f_{Y'}$ is uniquely determined by $f_Y$. In the same way, if $g:Y\to Y'$ is 
injective, then $f_Y$ is uniquely determined by $f_{Y'}$.
By assumption, every object of $\langle X\rangle$ is
    a subquotient of $M^n$ for some $n\geq 1$. 
Hence $f$ is determined by the components $f_{X^n}$. We write
$f_{X^n}\in \End(T(X)^n)$ as a matrix $\phi_{ij}$ with entries in
$\End(T(X))$. The entry $\phi_{ij}$ is the composition
\[ T(X)\xrightarrow{T(\iota_i)}T(X)^n\xrightarrow{f_{X^n}}T(X)^n\xrightarrow{T(p_j)}T(X)\]
with the injection $\iota_i:X\to X^n$ and
the projections $p_j:X^n\to X$. We have commutative diagrams
\[\xymatrix{
T(X)\ar[d]_{f_X}\ar[r]^{T(\iota_i)} &T(X)^n\ar[d]_{f_{X^n}}\ar[r]^{T(p_j)} &T(X)\ar[d]^{f_X}\\
T(X)\ar[r]^{T(\iota_i)} &T(X)^n\ar[r]^{T(p_j)} &T(X).
}\]
The map from the top left to the bottom right is $\phi_{ij}$. For $i\neq j$, the
composition $p_j\circ \iota_i$ vanishes, hence so does $\phi_{ij}$. For $i=j$ the composition is the identity, hence $\phi_{ii}=f_{X}$. In particular, the map
$f_{X^n}$ is uniquely determined by $f_X$. This finishes the proof of the first claim. The others follows directly from this fact.
\end{proof}

The set $\End(T)$ is stable under composition, making it into
a unital $\Q$-algebra. Dually, $\Ah(\Ch,T)$ is a \emph{counital coalgebra}: it is equipped with a \emph{comultiplication}, \index{coalgebra $\Ah(\Ch,T)$} i.e. a $\Q$-linear map
\[ \Ah(\Ch,T)\to \Ah(\Ch,T)\tensor_\Q\Ah(\Ch,T)\]
    satisfying axioms dual  to the axioms of a unital algebra.

For every $Y\in\langle X\rangle$, the vector space $T(Y)$ has a natural action  of $\End(T)$ where an element $f$ operates via its $f_Y$-component. This defines a functor
\[\tilde{T}: \langle X\rangle\to\End(T)\Mod\]
where $\End(T)\Mod$ denotes the category of finitely generated $\End(T)$-modules, or equivalently, $\End(T)$-modules whose underlying $\Q$-vector space is finite dimensional.
By adjunction, the structure map
\[ \End(T)\tensor V\to V\]
of an $\End(T)$-module $V$ induces a $\Q$-linear map
\[ V\to \Ah(\langle X\rangle,T)\tensor_\Q V,\]
turning $V$ into a \emph{comodule} whose underlying $\Q$-vector space is finite
dimensional. For the axioms of a comodule, see for example \cite[Section~7.5.2]{period-buch}. They are dual to the axioms of a module under an algebra. 

The significance of $\End(T)$ and $\Ah(\Ch,T)$ is the following strong property.

\begin{prop}\label{prop:comodule_category_finite}
Let  $T:\Ch=\langle X\rangle\to\Q\Vect$ be a faithful and exact functor. Then $\Ch$ is equivalent to the category of finite dimensional $\End(T)$-modules, or, equivalently to the category of finite dimensional $\Q$-vector spaces equipped with the structure of an $\Ah(\Ch,T)$-comodule.
\end{prop}
\begin{rem}We are not going to use this structural result in our applications to the Period Conjecture. We defer the deduction from the existing literature toward the end of the chapter.
\end{rem}

For later use, we give a more explicit description of $\Ah(\Ch,T)$.

\begin{lemma}\label{lem:coalgebra_finite}
In the situation of the Definition~\ref{defn:End}, we have
\[ \Ah(\Ch,T)=\left( \bigoplus_{Y\in\Ch}\End_\Q(T(Y))^\vee\right)/\text{\ functoriality}.\]
The functoriality relations are generated by elements of the form
$\sigma\tensor f^*(\omega)-f_*(\sigma)\tensor \omega$ for all
$f:Y\to Y'$ in $\Ch$ and $\sigma\in T(Y)$, $\omega\in T(Y')^\vee$ under the identification $\End(T(Y))^\vee\isom T(Y)\tensor T(Y)^\vee$.
\end{lemma}
\begin{proof}Let $\Ah'(\Ch,T)$ be the object on the right hand side. By definition, $\Ah'(\Ch,T)^\vee\isom \End(T)$ (the direct sum turns into a product and the quotient into a subobject). Taking duals again, we get
\[ (\Ah'(\Ch,T)^\vee)^\vee\isom \End(T)^\vee=\Ah(\Ch,T).\]
As the vector spaces are finite dimensional, they are isomorphic to their double duals. Hence we have shown 
\[ \Ah'(\Ch,T)\isom \Ah(\Ch,T).\]
\end{proof}

The coalgebra point of view has the advantage of generalising to all $\Ch$.

\begin{defn}
Let $\Ch$ be a $\Q$-linear abelian category, $T:\Ch\to\Q\Vect$ a faithful exact functor into the category of finite dimensional $\Q$-vector spaces. We put
\[  \Ah(\Ch,T)=\left( \bigoplus_{Y\in\Ch}\End_\Q(T(Y))^\vee\right)/\text{\ functoriality}\]
with functoriality interpreted as in Lemma~\ref{lem:coalgebra_finite}.
\end{defn}

As in the case where $\Ch$ is generated by a single object, this vector space has a natural $\Q$-coalgebra structure and every $T(Y)$ inherits the structure of
a comodule.

\begin{thm}[Nori: see {\cite[Chapter~7]{period-buch}}]\label{thm:nori}
Let $\Ch$ be a $\Q$-linear abelian category, $T:\Ch\to\Q\Vect$ a faithful exact functor into the category of finite dimensional $\Q$-vector spaces. Then $\Ch$ is equivalent to the category of $\Ah(\Ch,T)$-comodules whose underlying $\Q$-vector spaces are finite dimensional.
\end{thm}
The proof of this theorem is non-trivial. Proposition~\ref{prop:comodule_category_finite} is a key step in its proof. In our monograph we use the opposite approach and deduce Proposition~\ref{prop:comodule_category_finite} from Theorem~\ref{thm:nori} instead. 

\begin{proof}[Proof of Proposition~\ref{prop:comodule_category_finite}.]
We apply the theorem to $\Ch=\langle X\rangle$. By Lemma~\ref{lem:coalgebra_finite}, the coalgebra $\Ah(\Ch,T)$ is dual to the algebra
$\End(T)$. It is finite dimensional by Lemma~\ref{lem:is_finite}, hence the category of finitely generated $\End(T)$-modules is equivalent to the category of
$\Ah(\Ch,T)$-comodules which are finite dimensional over $\Q$.
\end{proof}

\begin{rem}
We have decided to deduce Proposition~\ref{prop:comodule_category_finite} from Theorem~\ref{thm:nori}, but actually the converse is also true and implicitly this is the way Theorem~\ref{thm:nori} is shown in \cite{period-buch}. The argument in \cite{period-buch} is made more complicated by allowing other  base rings than $\Q$. At heart, the result is even older. It appears as a step in the proof of Tannaka duality in \cite{LNM900}. Let us explain how Theorem~\ref{thm:nori} relates to Tannaka duality, even if these issues are not relevant for the rest of the monograph.

If $\Ch$ is not only abelian, but a so-called Tannaka category with fibre functor $T:\Ch\to\Q\Vect$ (i.e. equipped with with a unitary commutative associative
tensor product and $T$ being a faithful tensor functor), then the coalgebra $\Ah(´\Ch,T)$ is endowed with a commutative multiplication, making it into a Hopf algebra. It follows that $G=\Spec(\Ah(\Ch,T))$ is a group scheme. There is an equivalence of categories between $\Ah(\Ch,T)$-comodules and representations of $G$, identifying
$\Ch$ (as a Tannakian category) with finite dimensional representations of $G$.
\end{rem}

Having introduced $\End(T)$ and $\Ah(\Ch,T)$ and explained their significance, we can now establish the dimension formula for the space of abstract periods that we are after. As before, let $L$ be a subfield of $\C$.

\begin{prop}\label{prop:dim_formula_abstract}
\index{formal periods!dimension formula}\index{dimension formula!for formal periods}
 Let $\Ch=\langle X\rangle$ be a $\Q$-linear additive category and
 $V:\Ch\to \VVarg{\Q}{L}$ a faithful and exact additive functor. Then
\[ \dim_{L}\Perform(\Ch)=\dim_\Q\Ah(\Ch,V_\Q)=\dim_\Q\End(V_\Q).\]
\end{prop} 
\begin{proof}
The second equality holds because the vector spaces are $\Q$-dual  to each other.

The definitions of $\Perform(\Ch)$ and $\Ah(\Ch,V_\Q)$ are very similar. They become isomorphic after base change to $\C$, in particular they have the same dimensions. In detail:
\begin{align*}
 \Ah(\Ch,V_\Q)\tensor_\Q\C&=\left( \bigoplus_{Y\in\Ch}\End_\C(V_\C(Y))^\vee\right)/\text{\ functoriality}.\\
& =\left( \bigoplus_{Y\in\Ch}V_\C(Y)\tensor_\C V_\C(Y)^\vee\right)/\text{\ functoriality}.
\end{align*}
where we write $V_\C:=V_\Q\tensor_\Q\C$. On the other hand
\begin{align*}
\Perform(\Ch)\tensor_L\C&=\left( \bigoplus_{Y\in\Ch}V_\Q(Y)\tensor_\Q V_L(Y)^\vee\tensor_\Qbar\C\right)/\text{\ functoriality}\\
&=\left( \bigoplus_{Y\in\Ch}V_\C(Y)\tensor_\C (V_L(Y)^\vee\tensor_L\C)\right)/\text{\ functoriality}\\
&=\left( \bigoplus_{Y\in\Ch}V_\C(X)\tensor_\C (V_\C(Y)^\vee\C)\right)/\text{\ functoriality}
\end{align*}
because $V_L\tensor_L\C\isom V_\C$. 
\end{proof}

In the language of \cite{huber_galois}: $\Perform(\Ch)$ is a semi-torsor under
$\Ah(\Ch,V_\Q)$. 

\begin{cor}
Let $\Ch$ be a $\Q$-linear abelian category, $X$ an object of
$\Ch$, and $V:\Ch\to (\Q,L)\Vect$ a faithful and exact additive functor. Then the Period Conjecture holds for
$\langle X\rangle$ if and only if
\[ \dim_L\Per\langle X\rangle=\dim_\Q \End(V_\Q|_{\langle X\rangle}).\]
\end{cor}


\part{Periods of Deligne $1$-motives}\label{part2}


\chapter{Deligne's $1$-Motives}\label{sec:one-mot}
In this chapter we give a review on Deligne's category of  of $1$-motives over an algebraically closed field $k$ embedded into $\C$. The cases of interest for us are $k=\C$ and $k=\Qbar$. In the latter case, Section~\ref{sec:hodge} and in particular Section~\ref{sec:key} contain new results.

\section{The Category and the Realisation Functors}

\begin{defn}[Deligne {\cite[Ch. 10]{hodge3}}]\label{defn:einsmot}\index{$1$-motive}\index{motive!$1$-motive}
A \emph{ $1$-motive} $M=[L\to G]$ over $k$ is the datum given by a semi-abelian group $G$ over $k$, a free abelian group $L$ of finite rank and a group homomorphism $L\to G(k)$. Morphisms of $1$-motives are morphisms of complexes $L\to G$. The category $\onemot_k$ of \emph{iso-$1$-motives} \index{iso-$1$-motive} has the same objects, but morphism tensored by $\Q$.
\end{defn}

\begin{rem}
The category of iso-$1$-motives is abelian. 
In this monograph, we are working in the category of  iso-$1$-motives. The arguments
often involve replacing a $1$-motive $[L\to G]$ by an isogenous $1$-motive $[L'\to G']$. This will
sometimes happen tacitly. 
\end{rem}

We need to spell out the singular and de Rham realisation defined in \cite[Ch. 10]{hodge3} in detail.

\subsection{The Singular Realisation}

Let $M=[L\xrightarrow{u} G]$ be a $1$-motive. We write $G^\an$ for the commutative Lie group over $\C$ attached to $G$. The associated exponential sequence is
\[ 0\to H_1^\sing(G^\an,\Z)\to \Lie (G^\an)\xrightarrow{\exp} G^\an\to 0, \]
where we have used the identification made in Section~\ref{ssec:paths}, Equation (\ref{eq:homol}). 

\begin{defn}
Let $T_\sing(M)$ be fibre product of $L$ and $\Lie(G^\an)$ over $G^\an$ under  the structure map $u:L\to G^\an$ and the exponential map $\exp$
\[\begin{xy}\xymatrix{
T_\sing(M)\ar[r]\ar[d]&\Lie(G^\an)\ar[d]^{\exp}\\
L\ar[r]^u&G.
}\end{xy}\]
 
The vector space $V_\sing(M)=T_\sing(M)\tensor\Q$ is called the \emph{singular realisation}\index{singular realisation!of a $1$-motive} of $M$.
\end{defn}

By construction, there is a short exact sequence
\[ 0\to H_1(G^\an,\Q)\to V_\sing(M)\to L\tensor\Q\to 0.\]
In particular this gives $V_\sing(M)\isom H_1(G^\an,\Q)$ if $L=0$ and $V_\sing(M)=L_\Q$ if
$G$ is trivial.
The vector space $V_\sing(M)$ carries a weight filtration \index{weight filtration} with
\[W_nV_\sing(M)=\begin{cases} \quad \quad  0 \quad&n\leq -3,\\
                 H_1(T,\Q) & n=-2,\\ 
                 H_1(G,\Q)&n=-1,\\
                 V_\sing(M)&n\geq 0.\end{cases}\]
Here $0\to T\to G\to A\to 0$
is the decomposition of $G$ into a torus and an abelian variety.

\begin{lemma}
The functor $V_\sing:\onemot_k\to \Q\Vect$ is faithful and exact.
\end{lemma}
\begin{proof}Exactness holds because taking Lie algebras and pull-backs are exact functors. In order to verify faithfulness it suffices
to show that $\Vsing{M}=0$ implies $M=0$ in $\onemot_k$. Hence we consider some $M$ with $\Vsing{M}=0$. This implies
$H_1(G,\Q)=0$, and in consequence $G$ is trivial and $\Vsing{M}=L_\Q$. In this case $L$ also has to be trivial.
\end{proof}

\subsection{The Universal Vector Extension}                                                                           

As in Chapter~\ref{ch:alg_groups} let $\grp$ be the category of connected commutative algebraic groups. We denote by $\grp_\Q$ its isogeny category.
We  enlarge the category  $\onemot_k$ to a bigger abelian category which we call $\onemotgen_k$. Its objects are of the form $[L\to G]$ with $G$ in the category  $\grp$ and its morphisms are given by morphisms of complexes tensored by $\Q$. The category $\grp_\Q$ can be identified with a full subcategory of $\onemotgen_k$ by $G \mapsto [0 \to G]$. It is closed under extensions. This means that the bifunctor  $\Ext^1$ in $\grp_\Q$
is the same as  the bifunctor $\Ext^1$ in $\onemotgen_k$ restricted to $\grp_\Q$.

We briefly recall the universal vector extension of \cite[Constr. 10.1.7]{hodge3}. \index{universal vector extension!of a $1$-motive}
In Section~\ref{sec:universal} we introduced the universal vector extension \[ 0\to \Ext^1(G,\Ga)^\vee\to G^\natural\to G\to 0.\]  
It depended on the computation of $\Ext^1_{\grp}(G,\Ga)=\Ext^1_{\grp_\Q}(G,\Ga)$ for the additive group $\Ga$. 

\begin{lemma}
For  $M=[L\to G]$ in $\onemot_k$ there is a natural short exact sequence of $k$-vector spaces
\[ 0\to \Hom_{\mathrm{ab}}(L,\G_a)\to\Ext^1_{\onemotgen}(M,\G_a)\to\Ext^1_{\grp}(G,\G_a)\to 0.\]
In particular, they  are finite dimensional vector spaces. 
\end{lemma}

\begin{proof} 
An application of the functor $\Hom(-, [0\to\Ga])$ to the short exact sequence
\[ 0\to [0\to G]\to M\to [L\to 0]\to 0\]in $\onemotgen_k$
and identifying the 1-motives $[0\to G]$ and $[0\to \Ga]$ with $G$  and $\Ga$, respectively, by  our convention above,  induces a long exact sequence
\begin{multline}\label{eq:gl9} 
\Hom(G,\Ga)\to \Ext^1([L\to 0],\Ga)\to  \Ext^1(M,\Ga)
\to\Ext^1(G,\Ga). 
\end{multline}
The first term vanishes because $G$ is semi-abelian. Given an extension
\[ 0\to \Ga\to E\to G\to 0\]
in $\Ext^1(G,\Ga)$
we can lift the structure map $L\to G$ to $L\to E$ because $L$ is free. 
We obtain  a short exact sequence
\[ 0\to [0\to \Ga]\to [L\to E]\to [L\to G]\to 0\]
which is in $\Ext^1_{\onemotgen_k}(M,\Ga)$ because $[L\to G]=M$.
This makes the last map of (\ref{eq:gl9}) surjective.

Elements of $\Ext^1([L\to 0],\Ga)$ are of the form 
\[ 0\to [0\to\Ga]\to [L\to\Ga]\to [L\to 0]\to 0,\]
hence they can be identified with homomorphisms $L\to\Ga$. 

The group $\Ext^1(G,\Ga)$ is finite dimensional by Corollary~\ref{cor:compare_ext}. On the other hand $\dim \Hom(L,\Ga)=\rk(L)$  and the last statement follows.
\end{proof}

Our construction can be easily extended to the case when the group $\Ga$ is replaced by a finite dimensional vector space $V$. 

\begin{defn}Let $M=[L\to G]$ be in $\onemot_k$. A \emph{vector extension}\index{vector extension!of a $1$-motive}
of $M$ is an extension of the form
\[ 0\to [0\to V]\to [L\to G']\to [L\to G]\to 0\]
for a vector group $V$. 
\end{defn}

By the same arguments as in the case of semi-abelian varieties in
Section~\ref{sec:universal}, the datum of a vector extension of $M$ is equivalent to the datum of a classifying map
\[ \Ext^1(M,\Ga)^\vee\to V.\] 
The identity is a distinguished choice for $V=\Ext^1(M,\Ga)^\vee$.  

\begin{defn}\label{defn:Mnatural}  
For a 1-motive $M=[L\to G]$ in $\onemot_k$  
we call the extension  
\[ 0\to  \Ext^1(M,\Ga)\to [L\to M^\natural]\to [L\to G]\to 0\]
in $\onemotgen_k$ corresponding to 
the classifying map 
\[\id:\Ext^1(M,\Ga)^\vee\to\Ext^1(M,\Ga)^\vee\]
the \emph{universal vector extension}\index{universal vector extension!of a $1$-motive} of $M$.
\end{defn}

\begin{rem}
Our notation deviates from Deligne's.  
Our $M^\natural$ as defined in Definition~\ref{defn:Mnatural} corresponds to Deligne's $G^\natural$ and our $[L\to M^\natural]$ to his $M^\natural$.
The reason is that we want to be able to distinguish between the universal vector extension $G^\natural$ of $G$ (in $\grp_\Q$) and $M^\natural$
when discussing $M=[L\to G]$. 
\end{rem}

\begin{lemma}\label{lem:lift_inj}
The universal vector extension 
$[L\to M^\natural]$ of $[L\to G]$
is universal in the following sense: given a vector extension $[L\to G']$ there is a unique morphism
\[ [L\to M^\natural]\to [L\to G'].\]
Moreover, $L\to M^\natural$ is injective.
\end{lemma}
\begin{proof}The proof of the universal property here is the same as in the case of
abelian varieties, see Proposition~\ref{prop:universal}.

Let $l\in L$ be an element with image $0$ in $M^\natural$. A fortiori, its image in $G$ vanishes. Hence it suffices to show injectivity in the case $G=0$. Then
$\Ext^1(M,\Ga)=\Hom(L,\Ga)$,  hence
\[ M^\natural=\Hom(L,\Ga)^\vee.\]
The natural map $L\to M^\natural$ is the evaluation map $l\mapsto (\chi\mapsto \chi(l))$.
It is injective.
\end{proof}
The universal property of $G^\natural$ induces
a canonical map $G^\natural\to M^\natural$. Comparing their  vector group components, we see that we have a short exact sequence
\[0\to G^\natural\to M^\natural\to \Hom(L,\Ga)^\vee \to 0.\]
By the structure theory of commutative algebraic groups, this gives even a (non-canonical) decomposition
\[ M^\natural\isom G^\natural\times \Hom(L,\Ga)^\vee.\]
\begin{lemma}\label{lem:natural_exact} 
The universal vector extension defines faithful exact functors
\[ \onemot_k\to\grp,\quad M\mapsto M^\natural\]
and
\[ \onemot_k\to\onemotgen_k,\quad M=[L\to G]\mapsto [L\to M^\natural].\]
\end{lemma}

\begin{proof}
The dimension of $M^\natural$ is given by the formula $\rk(L)+\dim(A^\natural)+\dim(T)$
for $M=[L\to G]$ with $0\to T\to G\to A\to 0$ the decomposition into the torus part and the abelian part. If $M^\natural=0$, then $M=0$. This is faithfulness.

Let $0\to M_1\to M_2\to M_3\to 0$ be an exact sequence in $\onemot_k$.  We have the commutative diagram
{\small
\[\begin{xy}\xymatrix{
0\ar[r]&\Ext^1(M_1,\Ga)^\vee\ar[r]\ar[d]&\Ext^1(M_2,\Ga)^\vee\ar[r]\ar[d]&\Ext^1(M_3,\Ga)^\vee\ar[d]\ar[r]&0\\
0\ar[r]&M_1^\natural\ar[r]\ar[d]&M_2^\natural\ar[r]\ar[d]& M_3^\natural\ar[r]\ar[d]&0\\
0\ar[r]&G_1\ar[r]&G_2\ar[r]& G_3\ar[r]&0\
}\end{xy}\]
}
in which the sequence of semi-abelian varieties is the sequence of the semi-abelian parts of the motives. In the Ext-sequence the composition of the two maps vanishes by functoriality. In order to deduce exactness, it suffices to show that the dimensions add up in the right way. This is the case, because the dimension of $\Ext^1(M,\Ga)$ is linear in the dimension of the constituents. Together this implies exactness in the middle row.
\end{proof}

\subsection{The de Rham Realisation}

\begin{defn}
Let $M$ be a $1$-motive over $k$.
 We define the \emph{de Rham realisation of $M$}\index{de Rham realisation!of a $1$-motive} as
\[ V_\dR(M):=\Lie(M^\natural)\]
with $M^\natural$ as in Definition~\ref{defn:Mnatural}.
\end{defn}

The de Rham realisation carries a weight filtration \index{weight filtration}in the same way as the singular realistion, and in addition a \emph{Hodge filtration}\index{Hodge filtration}
\[ F^pV_\dR(M)=\begin{cases}0&p>0\\
                           \ker(\Lie(M^\natural)\to \Lie(G))& p=0\\
                            V_\dR(M)&p\leq -1,
               \end{cases}\]
see \cite[Constr. 10.1.8]{hodge3}. Note that $F^0V_\dR(M)=\Ext^1_{\onemotgen}(M,\G_a)^\vee$. 

\begin{lemma}
The functor $V_\dR:\onemot_k\to k\Vect$ is faithful and exact.
\end{lemma}
\begin{proof}Exactness follows from the exactness of the functors
$M\mapsto M^\natural$ and $G\mapsto \Lie(G)$. Hence it suffices to check that
$V_\dR(M)=0$ implies $M=0$ in $\onemot_k$. The assumption implies
$M^\natural=0$ and in consequence that $G$ is trivial. In this case
$M^\natural=\Hom(L,\Ga)^\vee$. Its vanishing implies also $L=0$.
\end{proof}

\subsection{The Period Isomorphism}
In addition to the two realisations, there is a filtered comparison isomorphism $V_\sing(M)_\C\isom V_\dR(M)_\C$, the \emph{period isomorphism},\index{period isomorphism!of a $1$-motive} which is constructed as follows: the structural map $L\to G$ has a canonical lift to the universal vector extension $M^\natural$, see Definition~\ref{defn:Mnatural}. We obtain the commutative diagram
\[\begin{xy}\xymatrix{
&H_1^\sing(M^{\natural,\an},\Z)\ar[d]\ar[r]^{\isom}&H_1^\sing(G^\an,\Z)\ar[d]\\
&\Lie(M^\natural)_\C\ar[d]^{\exp}\ar[r]&\Lie(G)_\C\ar[d]^{\exp}\\
L\ar[r]&M^{\natural,\an}\ar[r]&G^\an
}\end{xy}\]
The map at the top is an isomorphism by homotopy invariance because $M^\natural$ is a  vector bundle over $G$. 
        Hence the pull-back $T_\sing(M)=L\times_{G^\an}\Lie(G)_\C $ of $L\to G^\an$ to $\Lie(G)_\C$  agrees with the pull-back $L\times_{M^{\natural,\an}}\Lie(M^\natural)_\C$
of $L\to M^{\natural,\an}$ to $\Lie(M^\natural)_\C$.
 Let
\[ \phi_M:V_\sing(M)_\C\to \Lie(M^\natural)_\C\]
be the map obtained by this identification of pull-backs.

\begin{lemma}[Deligne {\cite[Constr. 10.1.8]{hodge3}}]\label{lem:filtered_iso}
The morphism $\phi_M$ is a filtered isomorphism.
\end{lemma}
\begin{proof}It is sufficient to consider the three cases $[L\to 0]$, $[0\to T]$ for a torus $T$ and $[0\to A]$ for an abelian variety $A$ separately. If they are isomorphisms, then $\phi_M$ is a filtered isomorphism in general.

For $M=[L\to 0]$, we have $\Vsing{M}=L\tensor\Q$, $V_\dR(M)=\Hom(L,\Ga)^\vee$
and the period map is the natural map, hence an isomorphism after extension of scalars.

For $M=[0\to T]$, we have $T_\sing(M)=H_1(T^\an,\Z)$, $M^\natural=T$ and the period map is induced from the inclusion $H_1(T^\an,\Z)\to\Lie(T)^\an$. It is an isomorphism after extension of scalars to $\C$.

We now turn to the case $M=[0\to A]$. We have $V_\dR(M)\tensor_k\C=\Lie(A^\natural)^\an$.  The complex Lie group $(A^\natural)^\an$ is the universal vector extension of $A^\an$ because $H^1(A_\C,\Oh)=H^1(A^\an,\Oh)$. To simplify notation we now write $A$ instead of $A^\an$ and also use the abbreviations   $T=T_\sing(M)\isom H_1(A,\Z)$
and $T_\C=T\tensor_\Z\C$. We want to show that
\[ \phi:T_\C\to \Lie(A^\natural)\]
is an isomorphism. Both sides have the same dimension $2\dim A$. The group $T\subset T_\C$ is in the kernel of the composition $T\to \Lie(A^\natural)\xrightarrow{\exp_{A^\natural}} A$. We get an induced vector extension
\[ T_\C/T\to A.\]
It suffices to show that it is the universal one. Let
\[ 0\to V\to G\to A\to 0\]
be a vector extension. By homotopy invariance, the map $T\to\Lie(A)$ lifts
to $T\to \Lie(G)$. (This is the same argument that we used to construct
 $\phi: T\to\Lie(A^\natural)$ in the first place.) It induces a $\C$-linear map
\[ T_\C\to\Lie(G)\]
and a holomorphic group homomorphism
\[ T_\C/T\to G.\]
In conclusion we have verified that $T_\C/T$ satisfies the universal property.
\end{proof}

\section{The Functor  to Mixed Hodge Structures}\label{sec:hodge}
Deligne introduced the category of $1$-motives because of its close relation to Hodge theory. We use a modification that also takes the $k$-structure into account. 

\subsection{The Category of Mixed Hodge Structures}\label{ssec:mhs}

All filtrations on vector spaces are \emph{separated} and \emph{exhaustive}, i.e. they start with $0$ and end with the full space.
If $W_\bullet V$ is an ascending filtration, we write
$\gr_n^WV=W_nV/W_{n-1}V$. If $F^\bullet V$ is a descending filtration, we write
$\gr^p_FV=F^pV/F^{p+1}V$.

\begin{defn}[{Deligne~\cite{hodge2}}]\label{defn:mhs}
\index{mixed Hodge structure}
Let $k\subset\C$ be a subfield.
A \emph{mixed $\Q$-Hodge structure defined over $k$} consists of
\begin{itemize}
\item a finite dimensional $\Q$-vector space $V_\Q$ equipped with an ascending filtration $W_\bullet V_\Q$, the \emph{weight filtration},
\item a finite dimensional $k$-vector space $V_\dR$ equipped with an ascending filtration $W_\bullet V_\dR$ and a descending filtration $F^\bullet V_\dR$, the \emph{Hodge filtration},
\item a filtered isomorphism $\phi:(V_\Q,W_\bullet)\tensor_\Q\C\to (V_\dR,W_\bullet)\tensor_k\C$
\end{itemize}
such that for every $n\in\Z$ the data 
$(\gr_n^WV_\Q,\gr_n^W V_\C,\gr_n^W\phi)$ with $V_\C=V_\dR\tensor\C$ and  the induced Hodge filtration is a pure Hodge structure of weight $n$. This means that
\[ \bigoplus_{p+q=n} (F^p\cap \bar{F}^q)(\gr_n^WV_\C)=\gr_n^WV_\C.\]
Here $\bar{F}^q$ is the complex conjugate of $F^q$ with respect to the $\R$-structure induced by $\phi(V_\Q)$.

A \emph{morphism} $f:V\to V'$ of mixed $\Q$-Hodge structures over $k$ consists of a filtered $\Q$-linear map $f_\Q:V_\Q\to V'_\Q$ and a bifiltered $k$-linear map $f_\dR:V_\dR\to V'_\dR$ compatible with the period isomorphisms of $V$ and $V'$.

We denote the category of mixed $\Q$-Hodge structures over $k$ by $\MHS_k$.
\end{defn}
The category is obviously additive and $\Q$-linear. Less obviously it is even abelian, see \cite[Th\'eor\`eme~1.2.10]{hodge2}.

\begin{ex} Let $X$ be a smooth projective variety over $\C$. By the Hodge Decomposition Theorem, every cohomology class has a unique harmonic representative. The decomposition of harmonic forms into $pq$-forms gives
\[ H^n_\sing(X^\an,\C)\isom\bigoplus_{p+q=n} H^{pq}\]
satisfying $\overline{H^{pq}}=H^{qp}$. With the Hodge filtration
$F^pH^n_\sing(X^\an,\C)=\bigoplus_{p'\geq p}H^{p'q'}$ this turns
$H^n_\sing(X^\an,\Q)$ into a pure $\Q$-Hodge structure of weight $n$. Alternatively, the Hodge filtration is induced by the stupid filtration
\[ F^p\Omega^*_X=[0\to\dots \to 0\to \Omega^p_X\to\dots] \]
on the complex of holomorphic or algebraic differential forms on $X$. If $X$ is defined over a subfield of $k$, this point of view allows us to also define  the Hodge filtration over the subfield. 
\end{ex}

The main result of \cite{hodge2} and \cite{hodge3} is the construction of a natural mixed Hodge structure on the cohomology of any algebraic variety over $k$.

\subsection{The Relation with $1$-Motives}
By Lemma~\ref{lem:filtered_iso},  the assignment $M\mapsto (V_\sing(M),V_\dR(M),\phi_M)$
defines a functor 
\[ V:\onemot_k\to\MHS_k\]
from the category of iso-$1$-motives to
the category of mixed Hodge structures over $k$.

\begin{thm}[Deligne {\cite[Construction~10.1.3, pp.~53--56]{hodge3}}]\label{thm:del_equiv}\index{mixed Hodge structure! of a $1$-motive}
For $k=\C$, the functor $V:\onemot_\C\to \MHS_\C$ is fully faithful. Its image consists of the full subcategory of  the polarisable Hodge structures 
whose only non-zero Hodge numbers are
$(-1,-1), (-1,0), (0,-1), (0,0)$.
\end{thm}
Note that a mixed Hodge structure with Hodge numbers as above is polarisable if and only if the graded piece in weight $-1$ is polarisable.
Using the Analytic Subgroup Theorem Deligne's result can be sharpened:

\begin{prop}\label{prop:hodge_ff} 
In the case $k=\Qbar$, the functor $V:\onemot_\Qbar\to \MHS_\Qbar$ is fully faithful.
\end{prop}
\begin{proof}Let $M=[L\to G],M'=[L'\to G']$  be in $\onemot_\Qbar $ and 
\[ \gamma:V(M)\to V(M')\]
 a morphism
of Hodge structures over $\Qbar$. By extension of scalars we get
a morphism of Hodge structures 
\[ (V_\dR(M)_\C,V_\sing(M),\phi)\xrightarrow{\gamma_\C} (V_\dR(M')_\C,V_\sing(M'),\phi)\]  over $\C$. By Deligne's theorem $\gamma_\C$ is induced by
a  morphism $\gamma_{\mathrm{Mot}}$ in $\onemot_\C$ and after replacing $L$  by a rational multiple it is represented by
\[ \gamma_{\mathrm{Mot}}: [L\to G_\C]\to [L'\to G'_\C].\]
It remains to show that the induced morphism of algebraic groups $G_\C\to G'_\C$ is even defined over $\Qbar$. Using that
\[\Lie(G)= V_\dR(M)/F^0V_\dR(M)\]
the morphism of Hodge structures over $\Qbar$ also induces a compatible homomorphism
\[ \Lie(G)\to\Lie(G')\]
defined over $\Qbar$. By the Analytic Subgroup Theorem (see Corollary~\ref{cor:class_hom_algebraic}) this is enough to imply that the group homomorphism is defined over $\Qbar$.
\end{proof}
By composition with the forgetful functor
$\MHS_\Qbar\to\VVarg{\Q}{k}$ (see Definition~\ref{defn:VV}) we also get a functor
\[ \onemot_\Qbar\to\VVarg{\Q}{k}.\]
We shall show in Theorem~\ref{thm:fulness_VV} that it is still fully faithful.

\begin{rem}
In the meantime  Andr\'e has   shown in \cite{andre4} that 
$\onemot_k\to\MHS_\C$ is fully faithful even for  all algebraically closed fields $k\subset\C$.
\end{rem}

\subsection{The Weight Filtration}
\index{weight filtration}
The weight filtration on the Hodge structure is \emph{motivic}, i.e. induced by a filtration on the motive.

\begin{defn}Let $M=[L\to G]$ be a $1$-motive, $G$ an extension of the abelian variety $A$ by a torus $T$, and define the \emph{weight filtration of $M$} as
\[ W_nM=\begin{cases} 0&n\leq -3,\\
                     [0\to T]& n=-2,\\
                       [0\to G]&n=-1,\\
                         M&n\geq 0.
          \end{cases}
\]
\end{defn}
The associated gradeds of this filtration are
$[0\to T]$, $[0\to A]$ and $[L\to 0]$. The simple buidlings blocks are $\Gm$,
$[0\to B]$ for simple abelian varieties $B$ and $[\Z\to 0]$. The functors
$W_n:\onemot_k\to\onemot_k$ are exact. This is often used to deduce the existence of splittings. For example:
\begin{lemma}Let $[L\to G]$ be a $1$-motive, $[L'\to 0]\hookrightarrow M$ be an injective morphism. Then 
\[ M\isom [L'\to 0]\times [L''\to G]\]
where $L''=L/L'$ (modulo torsion).
\end{lemma}
\begin{proof}
We choose $L\isom L'\times L''$ (after replacing $L$ by an isogenuous lattice). The natural map $[L''\to G]\to M$ together with the given $[L'\to 0]\to M$
define $[L'\to 0]\times [L''\to G]\to M$. This map is an isomorphism because
it is an isomorphism on all gradeds with respect to the weight filtration.
\end{proof}

\section{The Key Comparison}\label{sec:key}

We come back to the functor $(\cdot)^\natural:M\mapsto M^\natural$ from $1$-motives to commutative algebraic groups. By Lemma~\ref{lem:natural_exact} it is faithful and exact. 

\begin{prop}\label{prop:subquot_is_mot}
Let $H$ be an object of $\grp$ of the form $H=M^\natural$ for a $1$-motive $M$.
 Given a short exact sequence
\[ 0\to H_1\to H\to H_2\to 0\]
in $\grp$, there is a short exact sequence
\[ 0\to M_1\to M\to M_2\to 0\]
in $\onemot_k$ and a commutative diagram
\[\begin{xy}\xymatrix{
0\ar[r]&H_1\ar[r]&H\ar[r]\ar@{=}[d]&H_2\ar[r]&0\\
0\ar[r]&M^\natural_1\ar[r]\ar@{^(->}[u]&M^\natural\ar[r]&M^\natural_2\ar[r]\ar@{->>}[u]&0
}\end{xy}\]
such that
\[ V_\sing(M)\cap \Lie(H_1)_\C=V_\sing(M_1).\] 
The sequence is uniquely determined by these properties.
\end{prop}
\begin{proof}
Now let $M=[L\to G]$ be a $1$-motive and  
\[ 0\to H_1\to M^\natural\to H_2\to 0\]
a short exact sequence of connected commutative algebraic groups. 
By the structure theory of
commutative algebraic groups, there are  canonical decompositions
\[ 0\to V_i\to H_i\to G_i\to 0\]
with $G_i$ semi-abelian and $V_i$ a vector group. Moreover, the sequence 
\[ 0\to G_1\to G\to G_2\to 0\]
is exact. The data organise as 
\[\begin{xy}\xymatrix{
0\ar[r]&V_1\ar[r]\ar@{_(->}[d]&V\ar[r]\ar@{_(->}[d]&V_2\ar[r]\ar@{_(->}[d]&0\\
0\ar[r]&H_1\ar[r]\ar@{->>}[d]&M^\natural\ar[r]\ar@{->>}[d]&H_2\ar[r]\ar@{->>}[d]&0\\
0\ar[r]&G_1\ar[r]&G\ar[r]&G_2\ar[r]&0
}.\end{xy}\]

Let $L\to M^\natural$ be the canonical lift from $G$ to $M^\natural$, see
Definition~\ref{defn:Mnatural}.
 Note that it is injective, by Lemma~\ref{lem:lift_inj}.
 We define $L_1$ as the intersection of $L$ with $H_1$ and $L_2$ as $L/L_1$ (modulo torsion).
 By construction,
$L_1\to L\to G$ factors via $G_1$ and $L\to G\to G_2$ via $L_2\to G_2$.
We put $M_i=[L_i\to G_i]$. By construction the sequence
\[ 0\to M_1\to M\to M_2\to 0\]
is exact and there are  maps $L_i\to H_i$. By the universal property of $M_i^\natural$
this induces morphisms $M_i^\natural\to H_i$, compatible with the morphisms
to/from $M^\natural$. For $i=1$, the composition
$M_1^\natural\to H_1\to M^\natural$ is injective, hence so is $M_1^\natural \to H_1$. The dual argument gives surjectivity of $M_2^\natural\to H_2$.

We abbreviate $T_\sing(H_1)=T_\sing(M)\cap \Lie(H_1)_\C$.
Since $L_1\to H^\an_1$ is the pull-back of $L\to M^{\natural,\an}$
and $\Lie(H_1)_\C\to \Lie(M^\natural)_\C$ is the pull-back of
$H_1\to M^\natural$ via the exponential map, we deduce that
\[ T_\sing(H_1)=T_\sing(M)\cap \Lie(H_1)_\C=\exp_{H_1}^{-1}L_1.\]
It follows that the sequence
\[ 0 \to \ker(\exp_{H_1})\to T_\sing(H_1)\to L_1\to 0\]
is exact.
We compare it with the same sequence for $M_1^\natural$. In both cases, the kernel computes $H_1^\sing(G_1,\Z)$ because they are vector groups over $G_1$. Hence they are the same, which implies the claim $T_\sing(H_1)\isom T_\sing(M_1)$ and hence $\Vsing{M_1}\isom \Vsing{M}\cap \Lie(H_1)_\C$.

Given a second short exact sequence of motives
\[ 0\to M'_1\to M\to M'_2\to 0\]
with the same properties, the assumptions imply $\Vsing{M_1}=\Vsing{M'_1}$ inside $\Lie(H_1)_\C$. Without loss of generality even $T_\sing(M_1)=T_\sing(M'_1)$. Let $M'_1=[L'_1\to G'_1]$. We have $G'_1\subset G_1$ because this is the semi-abelian part of $H_1$. We also have $L'_1\subset L_1=L\cap H_1$. This means $M'_1\subset M_1$. We get equality of iso-$1$-motives from equality of the singular realisations.
\end{proof}
It is shown later (see Theorem~\ref{thm:ff}) that the image
of the functor $\onemot_\Qbar\to \MHS_\Qbar$ is closed under subquotients. 

\begin{rem}
\begin{enumerate}
\item
The above result is the central input into our proof of the main result of our monograph: a version of the Subgroup Theorem for $1$-motives in Theorem~\ref{thm:annihilator_mot}.
\item The proof gives even an integral construction of $M_1$ in a suitable abelian enlargement of the category of $1$-motives such that 
\[ T_\sing(M)\cap \Lie(H_1)_\C=T_\sing(M_1).\]
We do not need this for our applications to periods. 
\end{enumerate}
\end{rem}


\chapter{Periods of $1$-Motives}\label{sec:periods_eins}

In this chapter we introduce the set of periods of a $1$-motive as obtained from the comparison isomorphism between the singular and the de Rham realisation of the motive.  Then in Section \ref{sec:periods1} an alternative  description of periods as explicit integrals will be derived. This turns out to be  more useful in transcendence theory. As a next step  we  determine in Section \ref{sec:relations} the relations between periods. This is  a classical open question which has been answered in the past in very special cases only. As a first main result of the Chapter  we show that all relations  are induced by  trivial ones. This gives an answer to Kontsevich's Period Conjecture in the case of $1$-motives. For this we need an extension of the Subgroup Theorem to $1$-motives. As a first application, we given in Section \ref{sec:transcendence1} a necessary and sufficient condition for periods of 1-motives to be algebraic.

\section{Definition and First Properties}\label{sec:periods1}

Let $M=[L\to G]$ be a $1$-motive over $\Qbar$. 

\begin{defn}\label{defn:period_eins}\index{periods!of a $1$-motive}
The set of \emph{periods of $M$} is the union of the sets of entries of the period matrices of the comparison isomorphism $V_\sing(M)_\C\to V_\dR(M)_\C$ with respect to all $\Q$-bases of $V_\sing(M)$ and $\Qbar$-bases of $V_\dR(M)$. We denote it $\Per(M)$ and by
$\Per\langle M\rangle$ the $\Qbar$-subvector space of $\C$ generated by
$\Per(M)$.

For any subcategory $\Ch\subseteq\onemot_\Qbar$ we write $\Per(\Ch)$ for the union of the
$\Per(M)$ for all objects $M\in C$ and  $\Per\langle \Ch\rangle$ for the vector space over $\Qbar$ generated by $\Per(\Ch)$. In particular we write  $\Per(\onemot_\Qbar)$ for the union of all $\Per(M)$.
\end{defn}
Equivalently, we can define $\Per(M)$ as the image of the period pairing \index{period pairing!for a $1$-motive}
\[ \Vsing{M}\times \VdR{M} \to\C.\]
This description makes clear that it is not a vector space.
We write $\omega(\sigma)$ or more suggestively $\int_\sigma\omega$  for the value of the period
pairing for $\sigma\in \Vsing{M}$ and $\omega\in \VdR{M}$. The notation will be justified  later.

Note that this definition is an instance of the abstract definition of periods in Chapter~\ref{ch:formalism}.

\begin{lemma}\label{lem:vsp}
The vector space $\Per\langle M\rangle$ agrees with the set of periods
of the additive subcategory generated by $M$ and  with the set of
periods of the full abelian subcategory closed under subquotients generated by $M$.
The set $\Per(\onemot_\Qbar)$ is a $\Qbar$-subvector space of $\C$.
\end{lemma}
\begin{proof}Apply Lemma~\ref{lem:vsp_abstract} to $\Ch=\onemot_\Qbar$
and $V:\onemot_\Qbar\to\VVarg{\Q}{k}$.
\end{proof}

We want to make the definition of $\int_\sigma\omega$ explicit. Let $M=[L\to G]$ be  a $1$-motive over $\Qbar$. Recall that
$V_\dR(M)$ was defined as $\Lie(M^\natural)$ for a certain vector extension $M^\natural$ of $G$. Hence
$\VdR{M}=\coLie(M^\natural)$. Every cotangent vector of $M^\natural$ defines a unique $M^\natural$-equivariant differential form on $M^\natural$ and any
equivariant global differential arises in this way. Therefore we may make the identification
\[ \VdR{M}\isom \Omega^1(M^\natural)^{M^\natural}.\]
with the space of invariant differentials and we may from now on  view $\omega$ as a differential form on $M^\natural$.
Recall also that $T_\sing(M)$ is a lattice in $\Lie(M^{\natural,\an})$.     The latter maps to $M^{\natural,\an}$ via the exponential map for the Lie group $M^{\natural,\an}$, see Section~\ref{ssec:exp}. Hence a vector field $u\in T_\sing(M)$  defines a point
\[ \exp(u)\in M^{\natural,\an}\]
and by composition of a straight path from $0$ to $u$ with the exponential map a path 
\[ \gamma_u:[0,1]\to M^{\natural,\an}\]
from $0$ to $\exp(u)$.
The period pairing is then computed as
\[ \int_u\omega=\int_{\gamma_u}\omega\]
where the right hand side is an honest integral on a manifold; see Section~\ref{ssec:paths} for more details.
In particular, for $u\in\ker(\exp:\Lie(M^{\natural,\an})\to M^{\natural,\an})$, the path $\gamma_u$ is closed and hence it defines an element of $H_1(M^{\natural,\an},\Z)\isom H_1(G^\an,\Z)$. Conversely, every element $\sigma\in H_1(M^{\natural,\an},\Z)$ is represented by a formal linear combination 
$\sigma=\sum a_i\gamma_i$ of closed loops $\gamma_i$. The cycle defines a linear map
\begin{align*} I(\sigma):\coLie(M^{\natural,\an})&\to \C\\
                      \omega&\mapsto \sum a_i\int_{\gamma_i}\omega.
\end{align*}
In other words, $I(\sigma)\in\Lie(M^{\natural,\an})^{\vee\vee}\isom \Lie(M^{\natural,\an})$. As spelled out in Section~\ref{ssec:paths}, the two operations are inverse to each other: 
$I(\gamma_u)=u$ for $u\in\ker(\exp)$ and $\gamma(I(\sigma))$ is homologous to $\sigma\in H_1(M^{\natural,\an},\Z).$

\begin{lemma}\label{alg}
We have $\exp(u)\in M^\natural(\Qbar)$.
\end{lemma}
\begin{proof}
By definition, we have the commutative diagram 
\[\xymatrix{
T_\sing(M)\ar[r]\ar[d]&\Lie(M^{\natural,\an})\ar[d]^\exp\\
L\ar[r]&M^{\natural,\an}
}\]
which we evaluated  at $u\in T_\sing(M)$. The value is in $M^\natural(\Qbar)$ because $L$ takes values there.
\end{proof}

For the record, the above argument proves:

\begin{prop}\label{prop:explicit}
Every period of a $1$-motive is of the form
\[ \int_\gamma\omega\]
where $\omega$ is an algebraic $1$-form on a commutative algebraic group $G$ over $\Qbar$ and $\gamma$ is a path from $0$ to a point $P\in G(\Qbar)$.
\end{prop}

\begin{ex}\label{ex:algebraic}
Let $M$ be $[\Z\to 0]$. Then $M^\natural=\Ga$, $\Vsing{M}=\Q$, $\VdR{M}=\coLie(\Ga)=\Qbar dt$ and the period map sends $1$ to $\frac{\partial}{\partial t}$. The image of the period pairing $\Vsing{M}\times\VdR{M}\to \C$ is simply $\Qbar$. In terms of integration: an element $u\in T_\sing(M)\subset\Lie(\Ga)_\C$ gives rise to a path $\gamma_u$ from $0$ to $u$ in $\C$. An element $\omega\in \VdR{M}$ is identified with a differential form $\alpha dt$ on $\Ga$ and the period is
\[ \int_0^u\alpha dt=\alpha u\in\Qbar.\]
More generally, all periods of motives of the form $[L\to 0]$ are algebraic.
\end{ex}

We are going to see many more explicit examples in subsequents chapters.

\section{Relations between Periods}\label{sec:relations}

There are two types of obvious sources of relations between periods of $1$-motives:
\begin{enumerate}
\item (Bilinearity)
Let $M$ be a $1$-motive over $\Qbar$, $\sigma_1,\sigma_2\in \Vsing{M}$ and 
take $\omega_1,\omega_2\in \VdR{M}$, $\mu_1,\mu_2\in\Q$, $\lambda_1,\lambda_2\in\Qbar$. Then
\[ \int_{\mu_1\sigma_1+\mu_2\sigma_2}(\lambda_1\omega_1+\lambda_2\omega_2)=\sum_{i,j=1,2}\mu_i\lambda_j\int_{\sigma_i}\omega_j.\]
\item (Functoriality) Let $f:M\to M'$ be a morphism in  $\onemot_\Qbar$. Let
$\sigma\in \Vsing{M}$ and $\omega'\in \VdR{M}$. Then
\[ \int_\sigma f^*\omega'=\int_{f_*\sigma}\omega'.\]
\end{enumerate}

As a special case we get the relations coming  from short exact sequences: consider
\[ 0\to M'\xrightarrow{i} M\xrightarrow{p} M''\to 0\]
in $\onemot_k$. Then the period matrix for $M$ will be block triangular, hence will contain  plenty of zeroes. Explicitly: for $\sigma'\in \Vsing{M'}$ and
$\omega''\in \VdR{M''}$, we have
\[ \int_{i_*\sigma'}p^*\omega''=\int_{\sigma'}i^*p^*\omega''=\int_{p_*i_*\sigma'}\omega''=0\]
because $i^*p^*\omega''=0$ and $p_*i_*\sigma'=0$.

This is actually the \emph{only} source of relations between  periods of $1$-motives as we show.

\begin{defn}
For $u\in V_\sing(M)$ we define $\Ann(u)\subseteq \VdR{M}$ as the left kernel of the period pairing.
\end{defn}

\begin{thm}[Subgroup Theorem for $1$-motives]\label{thm:annihilator_mot}
\index{Subgroup Theorem for $1$-motives}\index{vanishing of periods}\index{periods!vanishing}
Given a $1$-motive $M$ over $\Qbar$  and $u\in \Vsing{M}$,
there exists an exact sequence in $\onemot_\Qbar$
\[0\rightarrow M_1\xrightarrow{i} M \xrightarrow{p} M_2\rightarrow 0\] 
such that $\Ann(u)=p^*\VdR{M_2}$ and $u\in i_*\Vsing{M_1}$. It is uniquely determined by these properties.
\end{thm}
\begin{proof}We consider the connected commutative algebraic group
$M^\natural$. Without loss of generality, $u\in T_\sing(M)$. By Lemma~\ref{alg}, $u\in \Lie(M^\natural)_\C$ with $\exp(u)\in M^\natural(\Qbar)$. We apply the Analytic Subgroup Theorem in the version Theorem~\ref{thm:annihilator}. Hence there is a short exact sequence in $\grp$
\[0\to H_1\to M^\natural\xrightarrow{\pi} H_2\to 0\]
such that $u\in \Lie(H_1)_\C$ and $\Ann(u)=\pi^*\coLie(H_2)$. By Proposition~\ref{prop:subquot_is_mot}, we find an associated short sequence in $\onemot_\Qbar$
\[ 0\to M_1\to M^\natural\to M_2\to 0\]
such that $T_\sing(M_1)=T_\sing(M)\cap \Lie(H_1)_\C$. In particular,
$u\in T_\sing(M_1)$. Because of the short exact sequence of motives, we have 
$\coLie(M_2^\natural)=\VdR{M_2}\subset\Ann(u)$. On the other hand, Proposition~\ref{prop:subquot_is_mot}  also gives a surjection $M_2^\natural\to H_2$, hence
\[ \Ann(u)=\coLie(H_2)\subset \coLie(M_2^\natural)\subset\Ann(u).\]
This implies equality of co-Lie algebras and hence even $M_2^\natural\isom H_2$.

Suppose that there is a second short exact sequence 
\[ 0\to M'_1\xrightarrow{i'} M\xrightarrow{p'}M'_2\to 0\]
with the same properties. In particular $p'^*\VdR{M'_2}=p^*\VdR{M_2}$ inside
$\VdR{M}$. This means that $M_2^\natural={M'}_2^\natural$ as quotients of $M^\natural$. As the functor $\cdot^\natural$ is faithful, this gives $M_2=M'_2$ as quotients of $M$.
\end{proof}

\begin{rem}
The proof shows that the decompositions in terms of algebraic groups (Theorem~\ref{thm:annihilator}) and in terms of $1$-motives agree.
\end{rem}

Most of the time we apply the theorem through the following consequence:

\begin{cor}
Given a $1$-motive $M$ over $\Qbar$, $u\in\Vsing{M}$, $\omega\in\VdR{M}$ such that $\int_u\omega=0$. Then there is a short exact sequence
\[ 0\to M_1\xrightarrow{i} M\xrightarrow{p} M_2\to 0\]
of $1$-motives and $u_1\in\Vsing{M_1}$, $\omega_2\in\VdR{M_2}$ such that
$u=i_*u_1$, $\omega=p^*\omega_2$.
\end{cor}
\begin{proof}We apply the theorem. This already gives the existence of
$u_1$.
By assumption, we have $\omega\in\Ann(u)$, hence $\omega$ is in the image of
$p^*$.
\end{proof}

Our main aim is to prove the following, which was formulated as a conjecture in
\cite{wuestholz}. We remind that $\langle M\rangle$ is the full subcategory of
$\onemot_\Qbar$ generated by $M$ and closed under subquotients.

\begin{thm}[Kontsevich's Period Conjecture for $1$-motives]\label{thm:main_one}
\index{Period Conjecture!for $1$-motives}
All $\Qbar$-linear relations between elements of $\Per(\onemot_\Qbar)$ are induced by bilinearity and functoriality.

More precisely, for every $1$-motive $M$ the relations between elements of $\Per(M)$ are generated by bilinearity and functoriality for morphisms in  $\langle M\rangle$, or equivalently, morphisms of the induced mixed Hodge structures over $\Qbar$.
In other words, the Period Conjecture in the sense of Definition~\ref{defn:conj_abstract} holds for
$\langle M\rangle$.
\end{thm}

\begin{proof}
We consider a linear relation between periods.
 For $i=1,\dots,n$ let
 $\alpha_i=\int_{\sigma_i}\omega_i$ be periods for $1$-motives $M_i$ and
let $\lambda_i\in \Qbar$  be such that
\begin{equation}\label{eq:sum} \lambda_1\alpha_1+\lambda_2\alpha_2+\dots+\lambda_n\alpha_n=0.
\end{equation}
We have already argued in Lemma~\ref{lem:vsp_abstract}  that
a linear combination of periods can be represented as single  period. We now have to go through the construction carefully in order to check that no relations other than bilinearity and functoriality are used. 

We put $M=M_1\oplus\dots\oplus M_n$. By pull-back via the projection, we can
view each $\omega_i$ as an element of $\VdR{M}$. By push-forward via the inclusion
we may view each $\sigma_i$ in $\Vsing{M}$. We then put $\sigma=\sum \sigma_i$ 
and $\omega=\sum\lambda_j\omega_j$. From the additivity relation we deduce
\[ \int_\sigma\omega=\sum_{i,j}\lambda_j\int_{\sigma_i}\omega_j.\]
The functoriality relation leads to $\int_{\sigma_i}\omega_j=0$ for $i\neq j$ 
and gives $\alpha_i$ for $i=j$. Hence the left hand side of (\ref{eq:sum})
equals $\int_\sigma\omega$. 

We are now in the situation 
\[ \int_\sigma\omega=0\]
 on the $1$-motive $M$. In other words, $\omega\in\Ann(\sigma)$. By Theorem~\ref{thm:annihilator_mot} there is
a short exact sequence 
\[ 0\to M'\xrightarrow{i} M\xrightarrow{p} M''\to 0\]
in $\onemot_\Qbar$ such that $\sigma=i_*\sigma'$ for $\sigma'\in \Vsing{M'}$ and
$\omega=p^*\omega''$ for $\omega''\in \VdR{M''}=\Ann(\sigma)$. 
Hence 
the vanishing of
\[ \int_\sigma\omega=\int_{i_*\sigma'}p^*\omega''=\int_{\sigma'}0\]
 is now implied by functoriality
of $1$-motives.
\end{proof}

\section{Transcendence of Periods of $1$-Motives}\label{sec:transcendence1}

All this implies a result on transcendence of periods.

\begin{thm}[Transcendence]\label{thm:transc_1}\index{transcendence!for periods of $1$-motives}
Let $M=[L\to G]$ be a $1$-motive, $\sigma\in \Vsing{M}$, and $\omega\in \VdR{M}$. Then the integral
\[ \int_\sigma\omega\]
is in $\Qbar$ if and only if 
there are $\phi,\psi\in \VdR{M}$ with  
\[ \omega=\phi+\psi\] 
such that
$\int_\sigma\psi=0$ and the image of $\phi$ in
$\VdR{G}$ vanishes.
\end{thm}
\begin{proof}
We write $\alpha=\int_\sigma\omega$ and
begin with the easy direction.  The short exact sequence 
\[ 0\to [0\to G]\to M\to [L\to 0]\to 0 \]
induces a short exact sequence
\[ 0\to \VdR{[L\to 0]}\to \VdR{M}\to \VdR{[0\to G]}\to 0.\]
Suppose that the image of $\phi$ in $\VdR{[0\to G]}$ vanishes or, equivalently, that $\phi\in \VdR{[L\to 0]}$.

Let $\bar{\sigma}$ be the image of
$\sigma$ in $\Vsing{[L\to 0]}$. Then
\[ \alpha=\int_\sigma\phi=\int_{\bar{\sigma}}\phi\]
is a period for $[L\to 0]$. It is a general fact; see Example~\ref{ex:algebraic}, and that $\VdR{[\Z^r\to 0]}=\coLie(\A^r)$
and all its periods are algebraic.

Conversely, assume that $\alpha$ is algebraic. 
If $\alpha=0$, the theorem holds with $\omega=\psi$. Assume that it is non-zero from now on. The algebraicity of $\alpha$ means that we can
write $\alpha$ as $\int_{\sigma'}\omega'$ with
$\omega'=\alpha dt$ and $\sigma'\in \Vsing{[\Z\to 0]}$ the standard basis vector. 
Let $\Sigma=(\sigma,-\sigma')\in \Vsing{ M\oplus [\Z\to 0]}$
and $\Omega=(\omega,\omega')\in \VdR{M\oplus [\Z \to 0]}$.
By assumption
\[\int_\Sigma\Omega=\int_\sigma\omega-\int_{\sigma'}\omega'=0.\]
By the Subgroup Theorem for $1$-motives, Theorem~\ref{thm:annihilator_mot}, there is a short
exact sequence of iso-$1$-motives
\begin{equation}\label{eq:alg} 0\to M_1\xrightarrow{(\iota,\iota_\Z)} M\oplus [\Z\to 0]\xrightarrow{p+q} M_2\to 0\end{equation}
and $\sigma_1\in \Vsing{M_1}$, $\omega_2\in \VdR{M_2}$ such that
\[ (\iota,\iota_{\Z})_*\sigma_1=(\sigma,-\sigma'),\hspace{2ex} \ p^*\omega_2=\omega, \hspace{2ex}q^*\omega_2=\omega'.\]
The map $q:[\Z\to 0]\to M_2$ does not vanish because the
pullback  of $\omega_2$ is $\omega'$. The latter is non-zero because  $\alpha$ is assumed to be non-zero.
The non-vanishing of the map already implies that $[\Z\to 0]$ is 
a direct summand of $M_2$. We explain the  argument as follows. We write $M_2$ in the form $[L_2\to G_2]$. Then the composition
\[ [\Z\to 0]\to [L_2\to G_2]\to [L_2\to 0]\]
is non-trivial. The composition $\Z\to L_2$ does not vanish, hence we can
decompose $L_2$ up to isogeny into the image of $\Z$ and a direct complement.
This defines a section of our map. 

We decompose $M_2$ as $[\Z\to 0]\oplus M'_2$.
Hence $\omega_2=\phi_2+\psi_2$ with
$\phi_2$ coming from $[\Z\to 0]$ and $\psi_2$ from the complement $M'_2$. 
Let $\phi=p^*\phi_2$ and $\psi=p^*\psi_2$ be their images in $\VdR{M}$. Then $\omega=p^*\omega_2=\phi+\psi$ and hence
\[ \alpha=\int_\sigma \omega=\int_\sigma \phi+\int_\sigma\psi.\]
By splitting off the direct summand $[\Z\to 0]$ from the sequence (\ref{eq:alg}), we obtain the short exact sequence 
\[ 0\to M_1\to M\to M'_2\to 0.\]
The element $\sigma=\iota_*\sigma_1\in \Vsing{M}$ is induced from $M_1$ and $\psi=p^*\psi_2$ from $M'_2$. Hence $\int_\sigma\psi=0$.

The element $\phi=p^*\phi_2$ is induced from $[\Z\to 0]$, and we conclude that it is in the image of
\[ \VdR{[\Z\to 0]}\to \VdR{[L\to 0]}\to \VdR{M}.\]
 It follows that its image in $\VdR{G}$ vanishes. 
\end{proof}

The general period formalism explained Chapter~\ref{ch:formalism} also implies a dimension formula. 
We put (see Definition~\ref{defn:End} with $T=V_\sing|_{\langle M\rangle}$)
\[ E(M)=\left\{ (\phi_N)\in \prod_{N\in\langle M\rangle}\End_\Q(\Vsing{N})|
  \phi_ {N'}\circ f=f\circ \phi_{N} \forall f:N\to N'\right\}.\]
This is a subalgebra of $\End_\Q(\Vsing{M})$, hence finite dimensional over $\Q$. In fact it agrees with $\End(V_\Q|_{\langle M\rangle})$ in the notation of
Section~\ref{sec:formalism_2}.

By Theorem~\ref{thm:nori}, 
 the category $\langle M\rangle$ is equivalent to the category of finitely generated $E(M)$-modules.

\begin{cor}\label{cor:dim_easy}
\index{dimension formula!for periods of $1$-motives}
We have
\[ \dim_\Qbar \Per\langle M\rangle=\dim_\Q E(M).\]
\end{cor}
\begin{proof}
By Theorem~\ref{thm:main_one} the Period Conjecture holds for $\langle M\rangle$. We now apply the abstract dimension formula of
Proposition~\ref{prop:dim_formula_abstract}. Note that $E(M)=\End(V_\Q|_{\langle M\rangle})$.
\end{proof}

\begin{rem}
This is a clear qualitative characterisation. However, it is by no means obvious to compute the  explicit value for a given $M$. We carry out this computation in Part~\ref{part3}.
\end{rem}

\section{Fullness}

As pointed out in Corollary~\ref{cor:fullness_abstract}, the validity of the Period Conjecture for a category $\Ch$ implies fullness of the functor
$F:\Ch\to\VV$ under consideration. In our case, we give a direct proof from the Analytic Subgroup Theorem for $1$-motives.

\begin{thm}[Fullness]\label{thm:fulness_VV}
\index{fullness!for $1$-motives}
The functor $V:\onemot_\Qbar\to \VVneu$ is fully faithful with image closed under subquotients.
\end{thm}

\begin{proof}Note that the functor $V$
is faithful and exact because the functors $V_\sing$ and $V_\dR$ are. 

 We choose $M\in\onemot_\Qbar$ and consider a short exact sequence in $\VVneu$
\[ 0\to V'\xrightarrow{\iota} V(M)\xrightarrow{\pi} V''\to 0.\]
It gives a 2-step filtration of $V(M)$, which furnishes a period matrix for $V(M)$ in block-triangular form, i.e. the period matrix contains a square of zeroes. We show that 
 $V'=V(M')$ and $V''=V(M'')$ for
objects $M',M''$ of $\onemot_k$. 
To see this,  consider $\Ann=\Ann(\iota_* V'_\Q)\subset \VdR{M}$. We record for later use that 
\begin{equation}\label{eq:ann} \pi^*({V''}_\Qbar^\vee)\subset \Ann.\end{equation}
 By Theorem~\ref{thm:annihilator_mot} applied to the elements of $V'_\Q$, there is a short exact
sequence
 \[ 0\to M'\xrightarrow{i} M\xrightarrow{p} M''\to 0\]
in $\onemot_\Qbar$
such that $p^*\VdR{M''}=\Ann$ and $V'_\Q\subset \Vsing{M'}$. We have 
\[ \iota_*V'_\Q\subset \ker(\Vsing{M}\to \Vsing{M''})=i_*\Vsing{M'}.\]
 Note that both $V(M')$ and $V'$ are subobjects of $V(M)$. We apply Lemma~\ref{lem:is_sub} to the faithful exact functor $V\mapsto V_\Q$. This gives even
$V'\subset V(M')$. This implies also $V(M'')\twoheadrightarrow V''$ and hence
\[ \Ann=p^*\VdR{M''}\subset \pi^*({V''}^\vee_\Qbar).\]
 Together with Equation~(\ref{eq:ann}) this gives
$p^*\VdR{M''}=\pi^*({V''}^\vee_\Qbar)$ inside $\VdR{M}$. Applying again
Lemma~\ref{lem:is_sub}, this time to the faithful exact functor
$V\mapsto V_\Qbar^\vee$, this leads to even $V(M'')=V''$ and in turn
$V(M')=V'$.  We have now shown that the image of
$f_3$ is closed under subobjects and quotients, implying the same for  subquotients.

Closedness of the image under subquotients implies fullness by Lemma~\ref{lem:crit_full} 
\end{proof}

\begin{rem}
Theorem~\ref{thm:fulness_VV} is not equivalent to Theorem~\ref{thm:main_one}. Consider for example the full abelian subcategory $\Ah$ closed under subquotients of $\VVneu$ generated 
by the single object $V=(\Qbar^2,\Q^2,\phi)$ with $\phi:\C^2\to\C^2$ given
by multiplication by the matrix
\[ \Phi=\left(\begin{matrix}1&\pi\\\log 2&1\end{matrix}\right).\]
It is easy to check that $V$ is simple and that its only endomorphisms are
multiplication by rational numbers. 
Indeed, a subobject $V'\subsetneq V$ is determined  by a vector  $v=\left(\begin{matrix}s\\t\end{matrix}\right)\in\Qbar^2$ such that 
\[ \phi(v)=\left(\begin{matrix}1&\pi\\
\log 2&1\end{matrix}\right)\left(\begin{matrix}s\\t\end{matrix}\right)
=\left(\begin{matrix}s+t\pi \\ s \log 2 +t\end{matrix}\right)\in\Q^2.\]
 As $\pi $ and $\log 2$ are transcendental, this implies $t=s=0$.  An endomorphism
$f:V\to V$ is represented by a pair of  matrices $A_\Qbar=\left(\begin{matrix}\alpha&\beta\\\gamma&\delta\end{matrix}\right)$ on $V_\Qbar$ and $A_\Q=\left(\begin{matrix}a&b\\ c&d\end{matrix}\right)$ on $V_\Q$ such that  $\Phi A_\Qbar=A_\Q\Phi$.

Spelling out the matrix equation we get   
\[ \left(\begin{matrix} \alpha+\pi\gamma&\beta+\pi \delta\\ \log 2\alpha +\gamma&\log 2\beta+\delta\end{matrix}\right)=\left(\begin{matrix}a+b\log 2 &a\pi + b\\c+ d\log 2&c\pi +d\end{matrix}\right)
\]
By the $\Qbar$-linear independence of $1,\pi$ and $\log 2$ it implies that
\[ \gamma=b=0, \ \alpha=a, \ \delta=a, \ \beta=b, \ \alpha=d,\ \gamma=c,\ \beta=c=0, \ d=\delta\]
and shows that that $A_\Qbar=A_\Q$ is a diagonal matrix for $a\in\Q$.
In other words, the full abelian category $\langle V\rangle\subset\VVneu$ generated by $V$ is semi-simple with a single simple object $V$ for which $\End(V)=\Q$.

The space of periods $\Per\langle V\rangle$ of $V$ has
$\Qbar$-dimension $3$ with basis $1,\pi,\log 2$. We compare it with the space $\Perform(\langle V\rangle)$, which was introduced in  Definition~\ref{defn:formal_abstract}.  The relations listed above imply that $\Perform(\langle V\rangle)$ is generated by
$4$ elements corresponding to the $4$ entries of the period matrix of $V$. 
There are no additional relations coming from subobjects or endomorphisms. This proves that $\dim_\Qbar\Perform(\langle V\rangle)=4$, which is not the same as $ 3= \dim_\Qbar \Per\langle V\rangle$, hence the Period Conjecture does not hold for $\langle V\rangle$.
As it does hold for $\onemot_\Qbar$, this implies that $\Phi$ does not occur as the period matrix of a $1$-motive, even though all entries are indeed in $\Per^1$. 
\end{rem}

\begin{rem}
The fact that the Period Conjecture implies Theorem~\ref{thm:fulness_VV}     is a special case of a general pattern, see Corollary~\ref{cor:fullness_abstract}.
For a careful analysis of this property we refer to the the discussion in \cite[Proposition~13.2.8]{period-buch} and \cite{huber_galois}.  The relation between the Period Conjecture and the Hodge conjecture is also explained there. The fullness question is also taken up by Andreatta, Barbieri-Viale and Bertapelle in their recent work, \cite{ABB}. They give an independent proof of Theorem~\ref{thm:ff}. Their second proof copies ours, but without the detour to periods.
\end{rem}

\chapter{First Examples}\label{ch:first}
Before turning to the case of period numbers of curves more generally, we give some examples, all of them very classical. They do not rely on the full strength of Theorem~\ref{thm:annihilator_mot}, but could be deduced directly from the Analytic Subgroup Theorem as in Theorem~\ref{thm:annihilator}.
We still prefer to go via Theorem~\ref{thm:annihilator_mot} in order to demonstrate the method.

\section{Squaring the Circle}

To prove the transcendence of $\pi$, or rather more naturally in our setting of $2\pi i$, we take the 1-motive $M_1=[0\to\Gm]$.
Then $M_1^\natural=\Gm$ because $\Ext^1_{\grp}(\Gm,\Ga)=0$ and this means that by definition we have  
\[ \VdR{M_1}=\coLie(\Gm)=\Omega^1(\Gm)^{\Gm}\] for the de Rham realisation of the motive.
On the singular side we get 
\[\Vsing{M_1}=\ker( \Lie(\Gm)^\an\to \Gm^\an)=\ker(\exp:\C\to\C^*).\]
We take   $\omega_1=\frac{dz}{z}\in \VdR{M_1}$ and the positively oriented loop $\gamma_1:[0,1]\to \C^*$ given by $\gamma_1(s)=e^{2\pi is}$ around $0$. In the notation of Section~\ref{ssec:paths} we define $\sigma_1=I(\gamma_1)\in\Lie(\Gm)^\an$
and obtain \[ \int_{\sigma_1}\omega_1=\int_{\gamma_1}\frac{dz}{z}=2\pi i.\]

\begin{cor}[Lindemann 1882]\label{cor:pi} 
The period $2\pi i$ is transcendental.\index{transcendence!of $\pi$}
\end{cor}
\begin{proof}[First proof.]
Assume that
$2\pi i$ is algebraic. Then, by Theorem~\ref{thm:transc_1}, we can express $\omega_1$ as 
\[ \omega_1=\phi+\psi \in\VdR{M_1}\]
such that the image of $\phi$ vanishes in $\VdR{\Gm}$ and such that $\int_{\sigma_1}\psi=0$. As $\Gm=M_1^\natural$, the map
$\VdR{M_1}\to \VdR{\Gm}$ is an isomorphism. Hence $\phi=0$. This means that
$\omega_1=\psi$ and hence $2\pi i=0$, which is false. 
\end{proof}

We offer a second proof where the result is deduced from the Analytic Subgroup Theorem for $1$-motives. 
Let $M_2=[\Z\to 0]$. Then $M_2^\natural=\Ga$ because
$\Ext^1_{\onemotgen}(M_2,\Ga)=\Hom(\Z,\Ga)=\Ga$. This gives 
\[ \VdR{M_2}=\coLie(\Ga)=\Omega^1(\Ga)^{\Ga}.
\]
In this case $\exp:\Lie(M_2^\natural)^\an\to M_2^{\natural,\an}$ is
the identity on $\Ga^\an=\C$. The lattice $\Z$ embeds naturally
into $M_2^{\natural,\an}=\C$. In conclusion 
\[ \Vsing{M_2}=\exp^{-1}(\Z)=\Z.\]
Let $\gamma_2$ be the straight path from $0$ to $1$ in $\Ga^\an=\C$ with image
$\sigma_2=I(\gamma_2)\in\gg_a^\an$. For $\omega_2$ we take $dt\in \VdR{M_2}$ to get 
\[ \int_{\sigma_2}\omega_2=\int_{\gamma_2}dt=1.\]

\begin{proof}[Second proof of Corollary~\ref{cor:pi}.]
Assume that $2\pi i$ is algebraic. This means that
\[ 2\pi i+\alpha=0\]
for  some $\alpha\in\Qbar$. In the notation from above we consider
\[ M=M_1\times M_2=[\Z\xrightarrow{0}\Gm]\]
and deduce
\[ \VdR{M}=\VdR{M_1}\times \VdR{M_2},\quad \Vsing{M}=\Vsing{M_1}\times \Vsing{M_2}.\]
We choose 
\[ \sigma=(\sigma_1,\sigma_2)\in\Vsing{M}\]
and
\[ \omega=(\omega_1,\alpha\omega_2)\in\VdR{M}.\] 
Here it is  used that  $\alpha$ is algebraic.  
We find that \[ \int_\sigma\omega=\int_{\sigma_1}\omega_1+\int_{\sigma_2}\alpha\omega_2=
2\pi i+ \alpha\cdot 1=0.\]
Now Theorem~\ref{thm:annihilator_mot} is applied to the motive $M$ and
the classes $\omega$ and $\sigma$. Their period vanishes, hence
there is a short exact sequence 
\[ 0\to M'\xrightarrow{i}M\xrightarrow{p}M''\to 0\]
such that $\sigma$ is in the image of $i_*$ and
$\omega$ is in the image of $p^*$. 

Note that $M$ is the product of two simple non-isomorphic $1$-motives, hence this leaves only
four possible choices for $M'$:
\[ 0,M_1\times 0, 0\times M_2, M.\]
We go through the cases. If $M'=0$, then the image of $i_*$ is zero and hence
$\sigma=0$, in contradiction to our hypothesis.

If $M'=M_1\times 0$, then the second component $\sigma_2$ of $\sigma$ is zero. This is false. The same argument also eliminates $M'=0\times M_2$.

We conclude that we have $M'=M$ and  $M''=0$, and this shows that the image of $p^*$ is zero and hence
$\omega=0$. This is false.

We have deduced  a contradiction and $2\pi i$ cannot be algebraic.
\end{proof}

\section{Transcendence of Logarithms}

We now turn to logarithms of algebraic numbers. For $\alpha\in\Qbar^*$, we have
\[ \int_1^\alpha\frac{dz}{z}=\log \alpha\]
(with the branch depending on the choice of the path from $1$ to $\alpha$). This is obviously an (incomplete) period of an algebraic variety. In order to apply our theorems directly, we identify it with the period of a $1$-motive.

We put 
\[ M(\alpha)=[\Z\xrightarrow{1\mapsto \alpha}\Gm].\]
This is often called \emph{Kummer motive} in the literature.\index{Kummer motive}
If $\alpha$ is a root of unity, this is the motive $M=M_1\times M_2$  from above. 
Otherwise the extension
\[ 0\to [0\to\Gm]\to M(\alpha)\to [\Z\to 0]\to 0\]
is non-trivial.
By definition, $M(\alpha)^\natural$
is an extension of $\Gm$ by $\Ga=\Hom(\Z,\Ga)$. By Theorem~\ref{thm:structure}
 $M(\alpha)^\natural$ is canonically isomorphic to $\Ga\times\Gm$.
Alternatively, we can also get the splitting by applying the functor $(\cdot)^\natural$ to the natural sequence of motives above.
We put 
\[
\omega=(0,dz/z)\in \VdR{M(\alpha)}=\coLie(\Ga\times\Gm).\]
The singular realisation of $M(\alpha)$ is 
\[ \exp_{\Gm}^{-1}(\alpha^\Z)\subset \Lie(\Gm)^\an\] 
  which gives
\[ \Vsing{M(\alpha)}=\langle I(\gamma_1), I(\gamma(\alpha))\rangle_\Q\subset\Lie(\Gm)^\an\] 
with $\gamma_1$   as before the positively oriented loop around $0$ and
$\gamma(\alpha)$ a path from $1$ to $\alpha$ in $\Gm^\an=\C^*$.
Note that the basis depends on the choice of path $\gamma(\alpha)$, but the lattice does not.  We introduce $\sigma(\alpha)=I(\gamma(\alpha))\in \Vsing{M(\alpha)}$.
Then 
\[ \int_{\sigma(\alpha)}\frac{dz}{z}=\int_{\gamma(\alpha)}\frac{dz}{z}=\log \alpha \]
with again the choice of logarithm determined by the choice of path.
The calculation of  the periods for $M(\alpha)$ uses the canonical embedding
\[ \Vsing{M(\alpha)}\subset \Lie(M(\alpha)^\natural)=\Lie(\Ga)^\an\times\Lie(\Gm)^\an=\C\times\C\]
given by
\[ \sigma_1\to (0,\sigma_1), \quad \sigma(\alpha)\mapsto \sigma(\alpha)^\natural=(1,\sigma(\alpha))=I(\gamma(\alpha)^\natural)\]
with $\gamma(\alpha)^\natural(s)=(s,\gamma(\alpha)(s))\in M(\alpha)^{\natural,\an}=\C\times\C^*$.
The period $\omega(\sigma(\alpha))$ is defined by applying the cotangent vector
$\omega$ to the tangent vector $\sigma(\alpha)^\natural$. Hence the period pairing gives
\[ \omega(\sigma(\alpha))=\int_{\gamma(\alpha)^\natural}\frac{dz}{z}=\int_{\gamma(\alpha)}\frac{dz}{z}=\log\alpha.\]

\begin{cor}[Transcendence of logarithms, Lindemann~1882]
For $\alpha \neq 1$  algebraic, $\log \alpha$\index{transcendence! of $\log\alpha$}
is transcendental, independent of the choice of the branch of the logarithm.
\end{cor}
\begin{proof} If $\alpha$ is a root of unity, then $\log\alpha$
is a rational multiple of $2\pi i$, whose transcendence we have already established. From now on let $\alpha$ not be a root of unity and as a consequence $M(\alpha)$ is non-split. Let $\omega$ and $\sigma(\alpha)$ be as above and assume that 
$\int_{\sigma(\alpha)}\omega$ is algebraic. By Theorem~\ref{thm:transc_1} there is a decomposition
\[ \omega=\psi+\phi\]
with  $\int_{\sigma(\alpha)}\psi=0$ and such that the image of $\phi$ vanishes in
$\VdR{\Gm}$. We first concentrate on $\psi$.  An application of Theorem~\ref{thm:annihilator_mot} to $\psi$ gives a short exact sequence
\[ 0\to M'\xrightarrow{i} M(\alpha)\xrightarrow{p} M''\to 0\]
such that $\sigma(\alpha)$ is in the image of $i_*$ and
$\psi$ is in the image of $p^*$. There are only three possibilities for
$M'$:
\[ 0,\quad [0\to \Gm],\quad M(\alpha).\]
We exclude $M'=0$ because $\sigma(\alpha)\neq 0$ and $M'=[0\to\Gm]$ because
$\sigma(\alpha)\notin\Vsing{[0\to \Gm]}$. This shows that $M'=M(\alpha)$  and implies that $M''=0$ and we conclude that $\psi=0$. As a consequence
$\omega=\phi$ vanishes when mapped to $\VdR{[0\to\Gm]}$. But actually
this image is $dz/z$, so we have a contradiction.
\end{proof}

\section{Hilbert's 7th Problem}

In his 7th problem Hilbert asked whether 
\[ \alpha,\beta,\alpha^\beta\]
can all  be algebraic unless $\alpha=0,1$ or $\beta$ rational. In other words:
\begin{equation}\tag{GS}\label{eq:GS} \alpha,\beta,\alpha^\beta\in\Qbar \Rightarrow\alpha=0,1\text{\ or\ }\beta\in\Q.\end{equation}
There is a logarithmic version of \ref{eq:GS}: let $\alpha,\gamma\in\Qbar^*$,
$\log\alpha$ and $\log\gamma$ choices of branches of logarithm.
\begin{multline*}
\tag{B}\label{eq:B} 
\text{$\log\alpha$ and  $\log\gamma$ are $\Qbar$-linearly dependent}\\
\text{$\Rightarrow$ $\log\alpha$ and $\log\gamma$ are $\Q$-linearly dependent}
\end{multline*}
We think of (GS) as the Gelfond-Schneider version of the implication
and of (B) as the Baker version. 

Note that the converse implication of (B) is obvious. However, the converse implication of (GS) fails in the case $\alpha=1$ if we do not use the principal branch of logarithm in the definition of $\alpha^\beta$. The problem disappears when restricting to real numbers.

\begin{lemma}
The implications (GS) for all $\alpha$, $\beta$ and $\gamma$ and (B) for all $\alpha$ and $\gamma$ are equivalent.
\end{lemma}
\begin{proof} We assume (GS). Let $\alpha,\gamma$ be non-zero algebraic numbers such that
$\log\alpha$ and $\log\gamma$ are $\Qbar$-linearly dependent.
By assumption
there are $u,v\in\Qbar^*$  such that
\[ u\log\alpha+v\log\gamma=0.\]
Without loss of generality we may assume that $\alpha\neq 1$. Indeed, if $\alpha=1$
and $\log\alpha=2\pi n$ for $n\in\Z$, we consider instead $\alpha=\gamma$
and different branches of $\log\gamma$.

We introduce $\beta=u/v$. As a consequence of the linear relation from above  we find that $\alpha^\beta=\exp(\beta\log \alpha)=\gamma^{-1}\in\Qbar^*$.
 By
(GS) this implies $\beta\in\Q$.

Conversely, we assume that (B). We take  $\alpha,\beta,\gamma\in\Qbar$, $\alpha\neq 0,1$ such that and $\gamma=\alpha^\beta$ and then
\[ \log\gamma=\beta\log\alpha+2\pi i n\]
for some choice of logarithms and an appropriate $n\in\Z$. We replace
the choice of branch of of the logarithm for $\gamma$ such that $n=0$.
Then  $\log\gamma$ and $\log\alpha$ are $\Qbar$-linearly dependent.
By (B) this implies $\beta\in\Q$.
\end{proof}

\begin{thm}[Gelfond--Schneider~1934]\label{thm:gelfond-schneider}
\index{Gelfond-Schneider, theorem of}
Implication (B) holds true.
\end{thm}
\begin{proof}We switch notation from (B) and write $\beta$ instead of $\gamma$.
Having fixed  $\alpha,\beta\in\Qbar^*$ and branches of logarithm
$\log\alpha$, $\log\beta$ such that the numbers are $\Qbar$-linearly dependent, we find $a,b\in\Qbar^*$ such that
\[ a\log\alpha+b\log\beta=0.\]
In order to show that they are $\Q$-linearly dependent,
we consider the $1$-motive $M=[\Z\to\Gm^2]$ with
structure morphism $1\mapsto (\alpha,\beta)$.  
The motive is split, i.e. isogenous to $[\Z\to 0]\oplus [0\to\Gm^2]$  
 if both $\alpha$ and $\beta$ are roots of unity. In this
case $\log\alpha$ and $\log\beta$ are rational multiples of
$2\pi i$, hence linearly dependent. This case is excluded  from now on.

Similar to the case of transcendence of logarithms we have
$M^\natural=\Ga^1\times \Gm^2$ and hence
$\VdR{M}=\coLie(M^\natural)$ has the basis $(dt,0,0)$, $(0,dz_1/z_1,0)$, $(0,0,dz_2/z_2)$.
We put $\omega=(0,adz_1/z_1,bdz_2/z_2)$.

Our next step is to compute $T_\sing(M)$. As $M$ is non-split, $T_\sing(M)$ is a subset
of $\Lie(\Gm^2)^\an$.
Recall that $\log\alpha\frac{d}{dz_1}=I(\gamma(\alpha))$ where $\gamma(\alpha)$ is
a suitable path in $\Gm^\an$ from $1$ to $\alpha$. The same relation holds for $\beta$. Let $\gamma:[0,1]\to\Gm^{\an}\times\Gm^\an$ be given by
$\gamma(t)=(\gamma(\alpha)(t),\gamma(\beta)(t))$ and
put $\sigma=I(\gamma)\in T_\sing(M)\subset \Lie(\Gm^2)^\an$. 
By construction 
\[ \omega(\sigma)=\int_{\gamma(\alpha)}a\frac{dz_1}{z_1}+\int_{\gamma(\beta)}b\frac{dz_2}{z_2}=a\log\alpha+b\log\beta=0\]
and Theorem~\ref{thm:annihilator_mot} furnishes a short exact sequence
\[ 0\to M_1\to M\to M_2\to 0\]
such that $\sigma$ is induced from $M_1$ and $\omega$ from $M_2=[\Z^s\to \Gm^t]$ with $s\leq 1$ and $t\leq 2$.

If $t=2$, then the surjection $M\to M_2$ is the identity on the torus part. The push-forward of 
$\sigma$ is given by $(\log\alpha\frac{d}{dz_1},\log\beta\frac{d}{dz_2})$. Hence both vanish. This implies that $\alpha=\beta=1$, a case we had excluded.

The case $t=0$ does not occur because $\omega$ is not a pull-back from
a motive of the form $[\Z^s\to 0]$. 

We are left with the case $t=1$. The torus part of the map
$M\to M_2$ is given by $(x,y)\to x^ny^m$ for $n,m\in\Z$. The induced
map on Lie algebras maps
$(\frac{d}{dt},\log\alpha\frac{d}{dz_1},\log\beta\frac{d}{dz_2})$ to $(n\log\alpha+m\log\beta)\frac{d}{dz}$.  This image
vanishes and gives the linear dependence
we were looking for.
\end{proof}

What we have just  seen is  a motivic reformulation of Gelfond's proof based on 
$\Gm\times \Gm$. In contrast, Schneider's argument uses
$\Ga\times\Gm$ but does not have a translation to our language. One would need a modification of the Analytic Subgroup Theorem, which would be desirable.

The same arguments also apply to more then two numbers, which leads to: 

\begin{thm}[Baker 1967] 
\index{Baker, theorem of}
Take $\alpha_1,\dots,\alpha_n\in\Qbar^*$. If $\log\alpha_1$,\dots,
$\log\alpha_n$ are $\Qbar$-linearly dependent, then they
are $\Q$-linearly dependent.

We even have
\[ \rk\langle \alpha_1,\dots,\alpha_n\rangle_\Z=\dim_\Qbar\langle \log\alpha_1,\dots,\log\alpha_n,2\pi i\rangle_\Qbar/ 2\pi i\Qbar
\]
for any choice of branches of logarithms.
\end{thm}

\begin{rem} 
In the literature we often find formulations with
$\alpha_1,\dots,\alpha_n$ multiplicatively independent.
  The above is the correct version that also allows roots of unity or even repetitions with different choices of branch of logarithm. We will discuss later (see Chapter~\ref{ch:struct}) in more detail
that the space of periods of the third kind with respect to non-closed paths is only well-defined up to other types of periods.
\end{rem}

\section{Abelian Periods for Closed Paths}

Another important case involves periods of abelian varieties in the classical sense.

\begin{cor}[W\"ustholz~\cite{wuestholz-icm}]
\index{transcendence!of periods of abelian varieties}
Let $A$ be an abelian variety,
$\omega\in\Omega^1(A)$ and $\gamma$ a closed path on $A^\an$. Then
\[ \int_\gamma\omega\]
is either $0$ or transcendental.
\end{cor}
\begin{proof}Consider $M=[0\to A]$. Its de Rham realisation is
$\coLie(A^\natural)^{A^\natural}\supset \coLie(A)^{A}$. All global differential forms on $A$ are $A$-invariant, hence $\omega$ defines an element
of $\VdR{M}$. It singular realisation is by definition the kernel
of $\exp_A:\Lie(A)^\an\to A^\an$. Let $\sigma\in \Vsing{A}$ be the element such that the image of a path from $0$ to $\sigma$ under $\exp_A$ is equal to
$\gamma$. Then
\[\int_\sigma\omega=\int_\gamma\omega.\]

Assume that the period is algebraic.
An application of Theorem~\ref{thm:transc_1} gives $\omega=\phi+\psi$ with
$\int_\sigma\psi=0$ and $\phi$ in the kernel of the restriction
to the group part of $M$. But $M$ is equal to its group part, implying that 
$\phi=0$.
\end{proof}

\chapter{On Non-closed Elliptic Periods}\label{ch:ex}
\index{elliptic periods}

The computation of the dimension of the period space is a classical problem
and has been studied in various cases by many authors.
In this chapter, we concentrate on the case of a 
non-classical elliptic $1$-motive. For instance we deduce the first examples of periods which were not known to be transcendental. At this point everything will be formulated in terms of $1$-motives. For the translation to periods of the first, second and third kind
on curves, see
Chapter~\ref{ch:van} and Chapter~\ref{sec:elliptic}.

The dimension formula is a special case of the generalised Baker Theory in Part~\ref{part4}. We give a direct proof that should be understood as a warm-up for the considerably more complicated general case. The special result
will not be needed later on.

\section{The Setting}\label{sec:setting_ell}

Let $A=E$ be an elliptic curve, $0\to \Gm\to G\to E\to 0$ a non-trivial
extension (which is even non-split up to isogeny) and $P\in G(\Qbar)$ a point whose image in $E(\Qbar)$ is not torsion.
We consider the 1-motive 
\[ M=[\Z \to G]\]
with $1$ mapping to $P$.
We denote $\delta(M)$ the dimension of the $\Qbar$-vector space $\Per\langle M\rangle$ generated by the periods of $M$ in $\C$.

We start by choosing bases in the singular and de Rham cohomology respecting the weight filtration. 
The inclusions
\[ [0\to\Gm]\hookrightarrow [0\to G]\hookrightarrow M\]
with cokernels $[0\to E]$ and $[\Z \to 0]$, respectively,  lead to a filtration 
\[\Vsing{\Gm}\subset \Vsing{G}\subset\Vsing{M}.\]
Extend a basis $\sigma$ for $\Vsing{\Gm}$ by $\gamma_1,\gamma_2$ to a
basis of $\Vsing{G}$ and further by $\lambda$ to a basis of the whole space.
Then their images $\bar{\gamma}_1,\bar{\gamma}_2$ in $E$ form a basis of $\Vsing{E}$
and $\bar{\lambda}$ forms a basis of $\Vsing{[\Z\to 0]}$.

For the de Rham realisation consider the cofiltration
\[ M\twoheadrightarrow [\Z \to E]\twoheadrightarrow [\Z \to 0]\]
with  kernels $[0\to \Gm]$ and $[0\to E]$, respectively.  They lead by pull-back of forms to a filtration 
 \[\VdR{M}\supset\VdR{[\Z\to E]}\supset\VdR{[\Z\to 0]}.\]
Extend a basis $u$ of $\VdR{[\Z\to 0]}$   
by $\omega,\eta$ to a basis of $\VdR{[\Z\to E]}$ and
by $\xi$ to a basis of $\VdR{M}$. The images $\bar{\omega},\bar{\eta}$ of $\omega$ and $\eta$
are a basis of $\VdR{E}$ and the image $\bar{\xi}$ of $\xi$ is a basis of
$\VdR{\Gm}$. The period space $\Per\langle M\rangle$ is spanned by the numbers obtained by pairing our basis vectors.

The pairing of an element of $\Vsing{M}$ coming from a subobject with an
element of $\VdR{M}$ coming from the corresponding quotient is zero, as we know.
Applying this observation to the two filtrations from above give
$u(\sigma)=\omega(\sigma)=u(\gamma_1)=u(\gamma_2)=\omega(\sigma)=0$.
Further one sees that $u(\lambda)=1$ and $\xi(\sigma)=2\pi i$ (at least after scaling). Taking this together gives
a period matrix \index{period matrix!elliptic example}of the shape
\[ \left(\begin{matrix}2\pi i&\xi(\gamma_1)&\xi(\gamma_2)&\xi(\lambda)\\
 0&\omega(\gamma_1)&\omega(\gamma_2)&\omega(\lambda)\\
 0&\eta(\gamma_1)&\eta(\gamma_2)&\eta(\lambda)\\
0&0&0&1
\end{matrix}\right).
\]
The calculation of the dimension of the associated space of periods needs to distinguish between two cases, the CM-case and the non-CM-case.  We deal with each of the two cases  separately.

\section{Without CM}

 In the case when there is no complex multiplication, the non-CM-case,  the endomorphism algebra $\End(E)=\Z$ is trivial.

\begin{prop}\label{prop:ell}
 Let $M$ be as just described. Then
\[ \delta(M)=11.\]
\end{prop}
This will be also  a corollary of the general theory in Part~\ref{part4}. The deduction of the corollary is 
explained in Example~\ref{ex:zahlen} and its continuation in Example~\ref{ex:delta_1}.
\begin{proof}[Direct proof.]
We use the notation fixed above. 
It has to bw shown that all entries of the period matrix are $\Qbar$-linearly independent.
If not, there is a relation 
\[ a\, 2\pi i+\sum_{i=1}^{2}(b_i\, \xi(\gamma_i)+c_i\,\omega(\gamma_i)+d_i\,\eta(\gamma_i))
+e\,\xi(\lambda)+f\,\omega(\lambda)+g \,\eta(\lambda)+h=0\]
with $a,b_1,b_2,c_1,c_2,d_1,d_2,e,f,g,h\in\Qbar$.
We consider the motive 
\[ \tilde{M}=[0\to\Gm]\times[0\to G]^2\times M=[0^3\times\Z\to \Gm\times G^3]\]
together with
\begin{align*}
\tilde{\gamma}&=(\sigma,\gamma_1,\gamma_2,\lambda)\in \Vsing{\tilde{M}},\\\tilde{\omega}&=(a\xi,b_1\xi+c_1\omega+d_1\eta,b_2\xi+c_2\omega+d_2\eta,e\xi+f\omega+g\eta+hu)\in\VdR{\tilde{M}}.
\end{align*}
Then  $\tilde{\omega}(\tilde{\gamma})=0$.
By the Analytic Subgroup Theorem for
$1$-motives, Theorem~\ref{thm:annihilator_mot}, there is a short exact sequence 
\[ 0\to M_1\xrightarrow{i} \tilde{M}\xrightarrow{p} M_2\to 0\]
of $1$-motives $M_1=[L_1\to G_1]$ and $M_2=[L_2\to G_2]$ with $\tilde{\gamma}=i_*\gamma_1$ for some $\tilde{\gamma}_1\in \Vsing{M_1}$  and $\tilde{\omega}=p^*\omega_2$ for some $\omega_2\in \VdR{M_2}$.

Let  $A_2$  be the abelian part of $M_2$. We want to show that  $A_2=0$. Assuming $A_2\neq 0$ we choose  
a non-zero map surjective map $\tilde{\kappa}: \tilde{M}\to [L \to E]$ which factors as $\kappa_2\circ p$ through $p$. As  $L$ is a quotient of $\Z$, there are (up to isogeny) only two possibilities, namely
$L=0$ or $L=\Z $. The map $\tilde{\kappa}$ factors via  
\[ \kappa: [0^3\times \Z\to 0\times E^3]\to [L \to E].\]
On the abelian part it is  
given by a vector $(0,n,m,k)$ with $n,m,k\in \End(E)=\Z$. Note that there is no complex multiplication. We deduce that  $L=k\Z$ which is non-zero if and only if $k\neq 0$.
Since $\tilde{\kappa}_* \tilde{\gamma} = {\kappa_2}_*\circ p_*\circ i_*\tilde{\gamma}_1=0$, we deduce that 
\[ 0=\kappa_*\tilde{\gamma}=n\gamma_1+m\gamma_2+k\lambda\in \Vsing{[L_\kappa\to E]}.\]
The elements $\gamma_1,\gamma_2$ (and if $k\neq 0$ also  $\lambda$) are linearly independent in the vector space $\Vsing{[L\to E]}$, which implies that $n=m=k=0$.
This contradicts the non-triviality of $\tilde{\kappa}$ and proves that $A_2=0$.

In conclusion we have $M_2=[L_2\to \Gm^r]$ for some $0\leq r\leq 4$.
The group part of the morphism of motives $p:\tilde{M}\to M_2$ has the form
\[ \theta:\Gm\times G^3\twoheadrightarrow \Gm^r.\]
Its components $G\to\Gm^r$ have to vanish since $G$ is non-split.
The surjectivity of $\theta$ implies that  $r\leq 1$ with $\theta=(?,0,0,0)$ either $0$ or the projection
to the factor $\Gm$. 

This gives us a lot of information on
$\tilde{\omega}$. Recall that $\tilde{\omega}=p^*\omega_2$ for
some $\omega_2\in \VdR{M_2}$. We have the commutative diagram
\[\xymatrix{
   [0\to \Gm\times G^3]\ar[r]^\theta\ar[d]&[0\to \Gm^r]\ar[d]\\
    \tilde{M}\ar[r]^p&M_2.
}\]
Hence the pull-back of  $\tilde{\omega}$
to $\VdR{\Gm\times G^3}$ is concentrated in the
first component, which gives 
\begin{eqnarray*}
b_1\xi+c_1\omega+d_1\eta&=&0\\
b_2\xi+c_2\omega+d_2\eta&=&0\\
e\xi+f\omega+g\eta&=&0 \;.
\end{eqnarray*}

As the three classes $\xi,\omega,\eta$ are linearly independent, we
get vanishing 
coefficients 
\[ b_1=b_2=c_1=c_2=d_1=d_2=e=f=g=0.\]
We are left with the case $\tilde{\omega}=(a\xi,0,0,hu)$. If $a\neq 0$, then
from the period relation also $h\neq 0$, and conversely. Assume we are in this case. As $\tilde{\omega}=p^*\omega_2$
it follows that the group part of $\tilde{M}\to M_2$ is the projection to
the first factor and $M_2=\Gm\times[\Z\to 0]$. The kernel $M_1$ is equal
to $0\times G^3$ in this situation.
However, $\tilde{\gamma}$ is induced from $M_1$ and therefore of the form
$(0,\dots)$. This contradicts $\sigma\neq 0$.
Hence we must have $a=0$ and $h=0$.
\end{proof}

\section{The CM-Case}\label{sseq:with_CM}

 We turn to the CM-case when $\End(E)_\Q=\Q(\tau)$ is an imaginary quadratic extension of $\Q$.

The action of $\Q(\tau)$ induces new relations between the entries of the period matrix. This is well-known in the language of periods of curves. We take the point of view of $1$-motives instead.

The singular realisation $\Vsing{E}$ is of dimension $2$ as a $\Q$-vector space because $E$ has genus $1$. As a $\Q(\tau)$-vector space it has dimension $1$. For any non-zero $\gamma$ in $\Vsing{E}$, the pair $(\gamma,\tau_*\gamma)$ is a $\Q$-basis of $\Vsing{E}$. After extension of scalars to $\C$, the operation is still semi-simple, meaning that $\Vsing{E}_\C$ decomposes as
a sum of two $\tau_*$-eigenspaces with complex conjugate eigenvalues. 
The dual operation $\tau^*$ induced on $\Vsing{E}^\vee$ has the same set of eigenvalues.    

This can also be described on the de Rham realisation.
We calculate $\VdR{E}=\coLie(E^\natural)=\Omega^1(E^\natural)^{E^\natural}$ by looking at 
\[ 0\to H^1(E,\Oh)^\vee\to E^\natural\to E\to 0. \]
By Theorem~\ref{thm:del_equiv}, $V^\vee(E)$ carries a Hodge structure. (This is the $1$-motivic incarnation of the Hodge decomposition of $H^1_\dR(E)$). It is explicitly given by
\[ F^1\VdR{E}=\Omega^1(E).\]
 This $\Qbar$-sub vector space is invariant under $\tau^*$, hence it is one of the eigenspaces. After extension of scalars to $\C$, we get even the decomposition by Hodge theory
\[ \VdR{E}_\C=F^1\VdR{E}_\C\oplus \overline{F^1\VdR{E}}_\C.\]
Everything is stable under the $\tau^*$-operation, so we have identified the eigenspace description. By definition the complex number $\tau$ corresponds to the unique endomorphism of $E$ which operates as multiplication by $\tau$ on the
Lie algebra and on the dual $\Omega^1(E)$, so the latter is the $\tau$-eigenspace. The eigenvalues are simply $\tau$ and $\bar{\tau}$.  Since $\tau^*$ acts on the $\Qbar$-vector space $\VdR{E}$ its eigenvectors are in $\VdR{E}$.
 Let $\omega',\omega''$ be a basis of $\tau^*$-eigenvectors of $\VdR{E}$, which has the property that
\[ \tau^*\omega'=\tau\cdot\omega',\quad \tau^*\omega''=\bar{\tau}\cdot\omega''.\]

\begin{cor} 
In the basis $(\gamma,\tau_*\gamma)$ of $\Vsing{E}$ and
$\omega',\omega''$ of $\VdR{E}$, the period relations for $E$ can be expressed as
\begin{gather*}
\omega'(\tau_*\gamma)=(\tau^*\omega')(\gamma)=\tau \omega'(\gamma),\\
\omega''(\tau_*\gamma)=(\tau^*\eta)(\gamma)=\bar{\tau}\omega''(\gamma).
\end{gather*}
\end{cor}


\begin{prop}\label{prop:with_CM}
Let $M$ be the motive introduced in Section~\ref{sec:setting_ell} with $\End_\Q(E)=\Q(\tau)$. Then
\[ \delta(M)=9.\]
\end{prop}
\begin{proof}[Direct proof.] 
\[
\tilde{M} = [0\to\Gm]\times [0\to G] \times M.
\]

We go through the proof in the non-CM case and make the necessary changes. Let $\sigma,\gamma_1,\gamma_2,\lambda$ be as in Section~\ref{sec:setting_ell}. We choose a little
more carefully $\gamma_2=\tau_*\gamma_1$ and then 
\begin{gather*}
\omega(\tau_*\gamma_1)=(\tau^*\omega)(\gamma_1)=a\omega(\gamma_1)+b\eta(\gamma_1)+c u(\gamma_1),\\
\eta(\tau_*\gamma)=(\tau^*\eta)(\gamma_1)=d\omega(\gamma_1)+e\eta(\gamma_1)+fu(\gamma_1)
\end{gather*}
with $a,b,c,d,e,f\in\Qbar$.
Indeed, the choice of basis used in the corollary leads to
$a=\tau$, $b=d=0$, $e=\bar{\tau}$; but we do not need the special shape.
Note that the argument does not apply to $\xi(\gamma_2)$ because $\tau^*\xi$ is not defined, or, in other words, would relate not to $M$ but a different $1-$motive.  

It remains to show that
$2\pi i,\xi(\gamma_1),\omega(\gamma_1),\eta(\gamma_1),\xi(\gamma_2),\xi(\lambda),\omega(\lambda),\eta(\lambda),1$ are linearly independent.
If not, there is a linear relation as in the first case, but omitting the
summands for $\omega(\gamma_2)$ and $\eta(\gamma_2)$, so $c_2=d_2=0$.
We consider the motive 
\[ \tilde{M}=[0\to \Gm]\times[0\to  G]\times M\]
and $\tilde{\gamma}$, $\tilde{\omega}$ analogously to before.
Again this gives $M_1,M_2$. Assume that $A_2\neq 0$ and choose $\kappa_2,L_\kappa$, $\kappa$ as in the first case. The composition
\[ \tilde{A}=0\times E^2\to E\]
is now given by a vector $(0,n,k)$ of elements of $\End(E)\subset\Q(\tau)$. The rest of the argument is the same as in the non-CM case.
\end{proof}

\section{Transcendence}

As a simple corollary of the explicit dimension computation, we also deduce the transcendence of periods of our $M$. We concentrate on the case where transcendence is not a simple consequence of the Analytic Subgroup Theorem. In the language of  Chapter~\ref{ch:struct} this refers to a period of the third kind with respect to a non-closed path.

\begin{cor}
\index{transcendence!of non-closed elliptic periods}
Let $M=[\Z\to G]$,
and $\sigma,\gamma_1,\lambda\in\Vsing{M}$, $\omega,\eta,\xi\in\VdR{M}$ be as
in Section~\ref{sec:setting_ell}.
Then the periods 
\[ \omega(\gamma_1),\omega(\lambda), \eta(\gamma_1), \eta(\lambda), \xi(\sigma)=2\pi i, \xi(\gamma_1), \xi(\lambda)\] 
are transcendental.
\end{cor}
\begin{proof} These elements agree with the elements with the ones from the proofs of Proposition~\ref{prop:ell} (non-CM case) and Proposition~\ref{prop:with_CM}. In both cases, $1$ appears as a period and we proved explicitly linear independence of the two periods.
\end{proof}

\begin{rem} 
The same transcendence result already appears in \cite{wuestholz-ell}.
For $\omega(\gamma_1)$, $\eta(\gamma_1)$  this means the transcendence of periods and quasi-periods of elliptic curves. These are old results of Siegel \cite{siegel} and Schneider \cite{S1,S2}. The transcendence of $\omega(\lambda)$, $\eta(\lambda)$ and
$\xi(\gamma_1)$ can be deduced from the Analytic Subgroup Theorem without the detour through $1$-motives.
\end{rem}

We may also deduce the same result more directly from the Analytic Subgroup Theorem for motives. We do the most interesting case $\xi(\lambda)$ as an example. The other cases can be treated in the same way. 

\begin{proof}[Second proof.]
Assume that $\xi(\lambda)$ is algebraic. We apply  the transcendence criterion given in Theorem~\ref{thm:transc_1} and write accordingly $\xi=\phi+\psi$ such that
$\psi(\lambda)=0$ and the image   of $\phi$ in $\VdR{G}$ vanishes. A fortiori the
image  of $\phi$ in $\VdR{\Gm}$ vanishes. This implies that $\psi$ is simply an alternative choice for $\xi$. It suffices to consider the case
\[ \xi(\lambda)=0.\]
By the Analytic Subgroup Theorem for $1$-motives, Theorem~\ref{thm:annihilator_mot} there is  short exact sequence
\[ 0\to M'\xrightarrow{i} M\xrightarrow{p} M''\to 0\]
such that $\lambda=i_*\lambda'$, $\xi=p^*\xi''$. 

The simple constituents of our $M$  are $\Gm, E,[\Z\to 0]$.
As the image of $\lambda$ in $\Vsing{[\Z\rightarrow 0]}$ and the image of $\xi$ in $\VdR{\Gm}$  are bases, we know that
$[\Z\to 0]$ must be a constituent of $M'$ and $\Gm$ a constituent of $M''$.
Hence there are only two possible shapes for $M'$:
\[ [\Z\to 0], [\Z\to E].\]
In the first case $M'=[\Z\to 0]$ and the inclusion $i$ is a section of the natural surjection
$M\to [\Z\to 0]$ and even of
$[\Z\to E]\to [\Z\to E]$. This contradicts the choice of $P$ in the definition
of $M$.

In the second case, $M''=\Gm$. The projection $p$ is section of the natural
inclusion $\Gm\to M$ and even of $\Gm\to G$. This contradicts the choice of
$G$ in the definition of $M$.

The two contradictions lead to $\xi(\lambda)\neq 0$.
\end{proof}

\part{Periods of algebraic varieties}\label{part3}


\chapter{Periods of Algebraic Varieties}\label{sec:nori}
We give an alternative  description of the set of periods of $1$-motives as 
periods of the first cohomology of algebraic varieties defined over the algebraic closure $\Qbar$ of $\Q$.
This needs to fix an embedding $\Qbar\to\C$.

\section{Spaces of Cohomological $1$-Periods}

The basic objects are triples $(X,Y,i)$ where $X$ is a $\Qbar$-variety, $Y$ a closed subvariety and $i\in\Na_0$. Let $H^i_\dR(X,Y)$ be its relative de Rham cohomology (see Section~\ref{sec:deRham_intro}) 
and $H^i_\sing(X,Y;\Q)$ its relative singular cohomology (see Section~\ref{sec:homology_intro}). The first is a $\Qbar$-vector space, the second a $\Q$-vector space. After base change to the complex numbers they become naturally isomorphic via the period isomorphism $\phi$ (see Section~\ref{sec:pairing_intro}).
In good cases, the period isomorphism can be explicitly described as integration of
closed differential forms over singular cycles.

In Chapter~\ref{ch:formalism} we had introduced the category $\VVarg{\Qbar}{\Q}$ with objects of the form $V=(V_\Qbar,V_\Q,\phi)$ where $V_\Qbar$ is a finite dimensional $\Qbar$-vector space, $V_\Q$ a finite dimensional $\Q$-vector space and $\phi:V_\Q\tensor_\Q\C\to V_\Qbar\tensor_\Qbar\C$ an isomorphism.

\begin{defn}
For algebraic varieties $X\supset Y$ over $\Qbar$ and $i\in\Na_0$, we denote by $H^i(X,Y)\in \VVarg{\Qbar}{\Q}$ the triple $(H^i_\dR(X,Y),H^i_\sing(X,Y;\Q),  \phi)$. 
\end{defn}
The assignment  $(X,Y,i) \to H^i(X,Y)$   is natural for morphisms of pairs $(X,Y)\to (X',Y')$. For every
triple $X\supset Y\supset Z$ there are connecting morphisms
\[ \partial: H^i(Y,Z)\to H^{i+1}(X,Y)\]
which are morphisms in $\VVarg{\Qbar}{\Q}$.

\begin{defn}\label{defn:coh_periods}
For $H^i(X,Y)$ as defined above the \emph{set of period numbers $\Per(X,Y,i)$} is the image of the period pairing
\[ H^i_\dR(X,Y)\times H^\sing_i(X,Y;\Q)\to \C.\] 
\index{periods!cohomological}\index{cohomological periods}
\index{$1$-periods}
The set of $i$-periods $\Per^i$ is the union of the $\Per(X,Y,i)$ for all
$X$ and $Y$.
\end{defn}

In our book we are primarily interested in the case $i=1$.

\begin{ex}
We have $\Per^0=\Qbar$ because $H^0(X,Y)$ only depends on the
connected components of $X$ and $Y$.
\end{ex}

\begin{lemma}\label{lem:per_contained}
For all $i$, the set $\Per^i$ is a $\Qbar$-subvector space of $\C$.
We have $\Per^i\subset\Per^{i+1}$.
\end{lemma}
\begin{proof}A sum of two periods of $H^i(X,Y)$ and $H^i(X',Y')$ can be realised as a period of  $H^i(X\amalg X',Y\amalg  Y')$. The set is stable under multiplication by numbers in $\Qbar$ because $H^i_\dR(X,Y)$ is a $\Qbar$-vector space.
Integration of the differential
form $d t$ over the path $s\mapsto e^{2\pi i s}$ on $[0,1]$ gives $1$ as a period of $H^1(\A^1,\{0,1\})$ whence the
periods of $H^i(X,Y)$ are contained (actually equal to) in the set of periods of  
$H^{i+1}(X\times\A^1,Y\times\A^1\cup X\times\{0,1\})\isom H^i(X,Y)\tensor H^1(\A^1,\{0,1\})$.
\end{proof}

\section{Periods of Curve Type}\label{ssec:curve-type}

Nori showed that every affine algebraic variety admits a filtration by subvarieties defined over $\Qbar$ such that their relative homology is
concentrated in a single degree. This ``good filtration'' should be seen as an analogue of the skeletal filtration of a simplicial complex or a CW-complex. Indeed, affine algebraic varieties have the homotopy type of a simplicial complex. The surprising insight is the existence of such a filtration by algebraic subvarieties, even over the ground field. This filtration goes into the construction of the category of Nori motives, but it also has immediate consequences for periods. For the general result see \cite[Section~11]{period-buch}.
We repeat the argument in our case.

\begin{prop}\label{prop:reduce_C}
In the definition of $\Per^1$ it suffices to consider
$H^1(C,D)$ where $C$ is a smooth affine curve and $D$ a finite collection of
points on $C$.
\end{prop}
\begin{proof}
Consider the periods of $H^1(X,Y)$ for arbitrary $X$ and $Y$. We first show that it suffices to deal with affine varieties $X$. 
By Jouanolou's trick  there is an $\A^n$-torsor $\tilde{X}\to X$ with $\tilde{X}$ affine.
Let $\tilde{Y}$ be the preimage of $Y$ in $\tilde{X}$. Since the map
$\tilde{X}\to X$ is a homotopy equivalence, we have
\[ H^1(X,Y)\isom H^1(\tilde{X},\tilde{Y}).\]
 Therefore we may without loss of generality assume that $X$ (and consequently also $Y$) is affine.

Nori's Basic Lemma, see  \cite[Proposition~9.2.3, Corollary~9.2.5]{period-buch}
provides our affine varieties with \emph{very good filtrations} by closed subvarieties
\[ X_0\subset X_1\subset\dots \subset X_n=X,\hspace{2ex} Y_0\subset Y_1\subset \dots\subset Y_n=Y\]
with $Y_i\subset X_i$.
By definition this means that   
\begin{enumerate}
\item $X_i\setminus X_{i-1}$ is smooth,
\item either $\dim X_i=i$ and $\dim X_{i-1}=i-1$ or $X_i=X_{i-1}$ and the  dimension of $\dim X_i$  is less than $i$,
\item $H^j(X_i,X_{i-1})$ vanishes for $j\neq i$.
\end{enumerate}
and the same for $Y_i$. The boundary maps in the long exact sequence for the
triple $(X_{i+1},X_i,X_{i-1})$ introduced in Section~\ref{sec:homology_intro} define a complex
\[ C(X_*)=[H^0(X_0)\to H^1(X_1,X_0)\to H^2(X_2,X_1)\to\dots].\]
By \cite[Theorem~2.35]{hatcher} its cohomology in degree $j$ agrees with $H^j(X)$.  We introduce 
\[ C(X_*,Y_*)=\cone\left( C(X_*)\to C(Y_*)\right)[-1],\]
explicitly given by
\[ [H^0(X_0)\to H^1(X_1,X_0)\oplus H^0(Y_0)\to H^1(X_2,X_1)\oplus H^1(Y_1,Y_0)\to \cdots].
\]
Its cohomology in degree $j$ agrees naturally with $H^j(X,Y)$. As a result,
$H^1(X,Y)$ can be identified with a subquotient of 
\[ H^1(X_1,X_0)\oplus H^0(Y_0).\]
This implies that the periods of $H^1(X,Y)$ are contained  in the space of periods
of $H^1(X_1,X_0)\oplus H^0(Y_0)$.  
As discussed in the proof of Lemma~\ref{lem:per_contained}, the periods of $H^0(Y_0)$ can also be seen as
periods of $H^1(Y_0\times\A^1)$, so they are periods of smooth affine curves.

In conclusion it remains to consider the case where $X$ is an affine curve, $
Y$ a finite set of points and, in addition, $X\setminus Y$ is smooth.
By normalisation, we resolve the singularities of $X$. We denote by $\tilde{X}$ the
normalisation and by $\tilde{Y}$ the preimage of $Y$ in $\tilde{X}$. By excision we have
\[ H^1(X,Y)\isom H^1(\tilde{X},\tilde{Y}).\] 
The curve $\tilde{X}$ is smooth and affine.
\end{proof}

\begin{defn}
We say that a  period is of \emph{of curve type}\index{periods!of curve type}
if it is the period of some $H^1(C,D)$ where $C$ is a smooth affine curve and $D$ a finite collection of points on $C$.
\end{defn}

Proposition~\ref{prop:reduce_C} asserts that all elements of $\Per^1$ are of curve type.

\begin{cor}\label{cor:from_curves}
All elements of $\Per^1$ are $\Z$-linear combinations of integrals of the form
\[ \int_\gamma \omega\]
where $\omega$  is  a regular algebraic $1$-form
on  a smooth affine curve $C$  over $\Qbar$   and $\gamma$ a differentiable path on $C(\C)$ which is either closed or has end points defined over $\Qbar$.
\end{cor}
This is a special case of the identification of normal crossings periods and periods of algebraic varieties, see \cite[Theorem~11.4.2]{period-buch}. The case $i=1$ is easier and we give the proof explicitly.

\begin{proof}
Given Proposition~\ref{prop:reduce_C}, this is a statement about the explicit description of relative singular  and de Rham cohomology. Let $C$ be a smooth affine curve, $Y\subset C$ a finite set of $\Qbar$-points. 
We have made the period computation explicit in Section~\ref{ssec:period_compute_1}
All period numbers are of the form
\[ ((\omega,\alpha),\sigma)=\int_{\sigma}\omega-\alpha(\partial\sigma).\] 
for an algebraic differential form $\omega$ on $C$, a set-theoretic map
$\alpha:Y(\Qbar)\to\Qbar$, and a
formal linear combination $\sigma=\sum n_i\gamma_i$
of smooth maps $\gamma_i:[0,1]\to C^\an$ such that 
$\partial(\sum n_i\gamma_i)=\sum n_i\gamma_i(1)-\sum n_i\gamma_i(0)$ is 
a divisor on $Y(\Qbar)$.

The second term only appears for non-closed paths. It is an algebraic number and can be expressed as a period integral of $\A^1$, namely $\int_{[0,1]}\alpha(\partial\sigma)d t$.   
We conclude that every element of $\Per^1$ is a $\Z$-linear combination of explicit
integrals as stated.
\end{proof}

Conversely, the following proposition  shows that all periods of curves are in $\Per^1$.
\begin{prop}
Let $C$ be a curve over $\Qbar$, $Y\subset C$ a finite set
of $\Qbar$-valued points. Then 
\[\Per\langle H^*(C,Y)\rangle\subset\Per^1.\]
\end{prop}
\begin{proof}Consider $H^i(C,Y)$ for $0\leq i \leq 2$. The assertion holds for $i=1$ by definition
and was shown in Lemma~\ref{lem:per_contained} for $i=0$. In the case $i=2$ dimension reasons show that $H^2(C,Y)\isom H^2(C)$. 
We replace $C$ by its normalisation. 
By the blow-up sequence formulated in Proposition~\ref{prop:blow-up} 
this does not change $H^2(C)$.
Without loss of generality, $C$ is connected.
If $C$ is affine, then $H^2(C)=0$ and $H^2(C)\isom H^2(\Pe^1)$ if $C$ is projective. From the long exact Mayer-Vietoris sequence
 \[ H^1(\A^1)\oplus H^1(\Pe^{1}\ohne\{0\})\to H^1(\Gm)\to H^2(\Pe^1)\to H^2(\A^1)\oplus H^2(\Pe^{1}\ohne\{0\})\]
for the cover of $\Pe^1$ by $\A^1$ and
$\Pe^{1}\ohne\{\infty\}$
and the vanishing of cohomology of affine spaces,
 we deduce an isomorphism
$H^2(\Pe^1)\isom H^1(\Gm)$.
 This shows that  its periods are also in $\Per^1$.
\end{proof}

\section{Comparison with Periods of $1$-Motives}\label{ssec:comp_periods}

From a conceptual point of view, it is also important to describe $1$-periods in terms of Jacobians of curves.

Let $C$ be a smooth curve over $\Qbar$, and  $D\subset C$ a finite set of $\Qbar$-points. Let $J(C)$ be its generalised Jacobian (see Section~\ref{app:B}) and
\[ \Z[D]^0=\left\{f:D\to\Z\;| \sum_{P\in D} f(P)=0\right\}\]
 the set of divisors of degree $0$ supported in $D$.  We consider the $1$-motive
\[ M=[\Z[D]^0\to J(C)].\]
\begin{lemma}\label{lem:per_J}
In this situation, we have
\[ \Per (H^1(C,D))=\Per(H^1(J(C),D))=\Per(M).\]
\index{periods!comparison}
\end{lemma}
\begin{proof}
We write $D$ as   $D=\{P_0,\dots,P_r\}$. The easy case when $D=\emptyset$ is left to the reader.
The point $P_0$ is used for the definition of the  inclusion $C\to J(C)$ which
induces by functoriality a morphism in $\VVarg{\Qbar}{\Q}$
\[ H^1(J(C),D)\to H^1(C,D).\]
We apply the long exact cohomology sequence for relative cohomology \[\begin{xy}\xymatrix{
H^0(J(C))\ar[r]\ar[d]_{\isom}&\ar[r]  H^0(D)\ar[r]\ar@{=}[d]&   H^1(J(C),D)\ar[d]\ar[r]&H^1(J(C))\ar[d]\ar[r]&0\\
H^0(C)\ar[r] & H^0(D)\ar[r]&   H^1(C,D)\ar[r]&H^1(C)\ar[r]&0
}\end{xy}\]
to both terms. 
 By Theorem~\ref{thm:gen_jacob},
the induced natural map $H^1(J(C))\to H^1(C)$ is an isomorphism. According to the Five Lemma the same is true
for $H^1(J(C),D)\to H^1(C,D)$ and so their periods agree.

In our second step we apply Proposition~\ref{prop:compare} to the 
$1$-motive $M=[\Z[D]^0\to J(C)]$.
Note that $e_i=P_i-P_0$ for $i=1,\dots,r$ is a basis of $\Z[D]^0$  and
the natural map $D\to J(C)$ maps $P_0$ to $0$ and all other $P_i$ to the corresponding $e_i$. By Proposition~\ref{prop:compare} this induces an isomorphism
\begin{equation}\label{eq:V} V(M)^\vee=H^1(J(C),D)\end{equation}
and Lemma~\ref{lem:periods_dual} implies
\[ \Per(M)=\Per(H^1(J(C),D)).\]
\end{proof}
Having now identified periods of curves with periods of semi-abelian varieties, we can make the step to $1$-motives.

\begin{prop}\label{prop:period_1}
A complex number is a period of some $1$-motive if and only if it
is  in $\Per^1$. In other words
\[\Per^1=\Per(\onemot_\Qbar).\]
\end{prop}
\begin{proof}
Let $\alpha$ be in $\Per^1$. Proposition~\ref{prop:reduce_C} tells us that it is of curve type and
by Lemma~\ref{lem:per_J} it is the period of a $1$-motive.

For the converse, let $M=[L\to G]$ be a $1$-motive. Up to isogeny  we can split
$L$ as $L=L_1\oplus L_2$ with $L_1\xrightarrow{0}G $ and $L_2\hookrightarrow G$. This gives 
\[ M\isom[L_1\to 0]\oplus [L_2\hookrightarrow G],\]
and it therefore suffices to consider the two special cases  $G=0$ or when the structure map $L\to G$ is injective. For $M=[\Z\to 0]$, Proposition~\ref{prop:compare} gives 
$V(M)^\vee\isom H^0(\Spec(\Qbar))$.In fact, the space of periods is simply $\Qbar$ in this case.
For $M=[\Z^r\hookrightarrow G]$ we have by Proposition~\ref{prop:compare}
that $V(M)^\vee\isom H^1(G,Z)$ with  $Z=\{0,P_1,\dots,P_r\}$ where $P_i$ is the image
of the $i$-th standard basis vector of $\Z^r$. Their periods are in $\Per^1$.
\end{proof}

\begin{summary}\label{sum:all_equal}
There are three different definitions of what a $1$-period might be:
\begin{enumerate}
\item the period of some $H^1(X,Y)$ (cohomological degree $1$),
\item the period of a curve relative to some points (dimension $1$),
\item the period of a Deligne $1$-motive.
\end{enumerate}
Our discussion shows that these three notions agree.
\end{summary}

\section{The Motivic Point of View}\label{ssec:mot}
 
We have only discussed periods of Deligne's category of $1$-motives so far. There are two other categories of mixed motives, due to Voevodsky and Nori, respectively. In both cases, periods can be defined. The purpose of this section is to
compare the sets of numbers that we obtain. 

The theories are technically very demanding. It would go too far to present  this here in detail. In Appendices~\ref{sec:app_nori} and \ref{sec:voe} 
the interested reader can find a more complete survey.

The motivic picture gives a lot more structure to the situation. It also shows that
most results comparing cohomological $1$-periods to other notions of $1$-periods can easily be deduced from the literature.
The results will only be used Chapter~\ref{sec:results}.

Let $k$ be a field with a fixed embedding into $\C$. 
We denote by $\DMgmeff(k,\Q)$ Voevodsky's triangulated category of effective
geometric motives, see Ap\-pen\-dix~\ref{sec:voe}. The category comes with a functor which attaches to every smooth $k$-variety $X$ its \emph{motive} $M(X)$.\index{geometric motive}\index{motive!geometric}
Let $d_1\DMgmeff(k,\Q)$ be the full thick subcategory 
generated by the motives of the form  $M(X)$ for $X$ a smooth variety of dimension at most $1$. There is a natural equivalence
triangulated categories
\[ D^b(\onemot_k)\to d_1\DMgm(k,\Q)\]
from the bounded derived category of the abelian category $\onemot_k$ to the traingulated category 
$d_1\DMgm(k,\Q)$. We refer to Theorem~\ref{thm:orgo_app} for more details.

Next we turn to Nori's abelian category $\MMNeff(k,\Q)$ of
effective motives. It is universal for all cohomological functors compatible with rational singular cohomology. In Appendix~\ref{sec:app_nori} a very brief introduction to Nori's theory is given.\index{motive!Nori}\index{Nori motives}

 Let
$d_1\MMNeff(k,\Q)$ be the smallest full subcategory containing $H^i_\Nori(X,Y)$ for
$Y\subset X$ with $i\leq 1$ and which is closed under subquotients.
Again, by the work of Ayoub and Barbieri-Viale,
see Theorem~\ref{thm:abv}, there is an anti-equivalence
\[ \onemot_k\to d_1\MMNeff(k,\Q).\]
Moreover, the abelian category $d_1\MMNeff(k,\Q)$ has an explicit description as the diagram category  in the sense of Nori  of the category of pairs
$(C,Y)$ where $C$ is a smooth curve and $Y\subset C(k)$ a finite subset. 

By Theorem~\ref{thm:harrer} both categories are linked by a triangulated realisation functor
\[ \DMgmeff(k,\Q)\to D^b(\MMNeff(k,\Q))\]
compatible with their singular realisations
 into the derived category of
$\Q$-vector spaces. 
By Proposition~\ref{prop:harrer_dim} it maps the subcategory $d_1\DMgm(k,\Q)$ to $D^b(d_1\MMN(k,\Q))$.

The universal property of Nori motives implies the existence of
 functors 
\[ \MMNeff(k,\Q)\hookrightarrow \MHS_k\hookrightarrow \VVarg{\Q}{k};\]
see Theorem~\ref{thmdefn:Nori}. The first functor associates to a Nori motive a mixed Hodge structure. By forgetting the filtrations we obtain an object of $\VVarg{\Q}{k}$. 
This allows us to define periods for
the various categories of motives.
All this is summed up  in one diagram: 

\[\begin{xy}\xymatrix{
D^g(\onemot_k)\ar[r]^{H^0}\ar[d]_{\simeq}&\onemot_k\ar[d]^{\simeq}\\
d_1\DMgmeff(k,\Q)\ar[r]^{H^0}\ar@{_{(}->}[d]&d_1\MMNeff(k,\Q)\ar@{_{(}->}[d]\\
\DMgmeff(k,\Q)\ar@/_1pc/[u]\ar[r]^{H^0}\ar[rd]\ar[rdd]&\MMNeff(k,\Q)\ar@/_1pc/[u]\ar[d]\\
&\MHS_k\ar[d]\\
&\VVarg{k}{\Q}
}\end{xy}\]

\noindent By Proposition~\ref{cor:one_vs_geom} the composition in the right column is the functor
$V:\onemot_k\to\VVarg{\Q}{k}$ of Section~\ref{sec:hodge} composed with the external duality functor
of Definition~\ref{defn:external}.

\begin{cor}
The sets of periods of the categories\index{periods!of motives}\index{periods!comparison}
$\onemot_k$, $d_1\DMgmeff(k,\Q)$ and $d_1\MMNeff(k,\Q)$ agree and are equal to $\Per^1$.
In particular, both 
Pro\-po\-si\-tion~\ref{prop:reduce_C} and Proposition~\ref{prop:period_1} hold true.
\end{cor}
\begin{proof}
Theorems~\ref{thm:orgo_app} and \ref{thm:harrer} combined with Theorem~\ref{thm:abv} provide the equivalences of categories
\[ D^b(\onemot_k)\to d_1\DMgmeff(k,\Q)\to D^b(d_1\MMNeff(k,\Q)).\]
 Proposition~\ref{cor:one_vs_geom} asserts that
 Deligne's construction of the realisation of a $1$-motive agrees (up to duality)
with the one obtained via the identification with the category $d_1\MMNeff(k,\Q)$. As a consequence the three categories have identical sets of periods. Finally,  the periods of $D^b(\onemot_k)$ coincide with the periods of $\onemot_k$ by definition.

 Also by definition, $\Per^1\subset \Per(d_1\MMNeff(k,\Q))$. The explicit description of
$d_1\MMNeff(k,\Q)$ in Theorem~\ref{thm:abv} yields the converse inclusion and even the more restrictive description of
Proposition~\ref{prop:reduce_C}.

As a byproduct, we get the equality $\Per^1=\Per(\onemot_k)$, reproving Proposition~\ref{prop:period_1}.
\end{proof}


\chapter{Relations between Periods}\label{sec:results}

In the last chapter, we established different descriptions for the
space of $1$-periods. We now turn to their relations. 

\section{Kontsevich's Period Conjecture}

There is a short list of obvious relations.
\begin{enumerate}
\item[(A)] 
\emph{Bilinearity:}
Let $X$ be a variety over $\Qbar$, $Y\subset X$ a closed subvariety and $i\in\Na_0$.  For all $\sigma_1,\sigma_2\in H_i^\sing(X,Y;\Q)$, $\omega_1,\omega_2\in H^i_\dR(X,Y)^\vee$, $\mu_1,\mu_2\in\Q$, $\lambda_1,\lambda_2\in\Qbar$, we have
\[ \int_{\mu_1\sigma_1+\mu_2\sigma_2}(\lambda_1\omega_1+\lambda_2\omega_2)=\sum_{i,j=1,2}\mu_i\lambda_j\int_{\sigma_i}\omega_j.\]
\item[(B)] 
\emph{Functoriality:}
Let $f:(X,Y)\to (X',Y')$ be a morphism of pairs of $\Qbar$-varieties and $i\in\Na_0$. For all
$\sigma\in H_i^\sing(X,Y;\Q)$ and $\omega'\in H^i_\dR(X',Y')$, we have
\[ \int_\sigma f^*\omega'=\int_{f_*\sigma}\omega'.\]
\item[(C)] 
\emph{Boundary maps:}
Let $X\supset Y\supset Z$ be subvarieties and $i\in\Na_0$.
For all $\sigma\in H_{i+1}^\sing(X,Y;\Q)$ and $\omega\in H^i_\dR(Y,Z)$, we have
\[ \int_{\partial\sigma}\omega=\int_\sigma\partial\omega\]
where $\partial$ denotes the boundary maps
$H^i_\dR(Y,Z)\to H^{i+1}_\dR(X,Y)$ on de Rham cohomology and $H^\sing_{i+1}(X,Y;\Q)\to H^i_\sing(Y,Z;\Q)$ on singular homology, respectively.
\end{enumerate}

To state the Period Conjecture 
we recall from Definition~\ref{defn:coh_periods} the set of
$i$-periods $\Per^i\subset \C$. Let $ \Per^\eff=\bigcup_{i=0}^\infty\Per^i$ be the
set of \emph{effective cohomological periods} and
$\Per=\Per^\eff[\pi^{-1}]$ the \emph{period algebra}. 

\begin{conjecture}[Period conjecture, Kontsevich \cite{kontsevich}]\label{conj:kontsevich_full}\index{Period Conjecture!Kontsevich's}
All $\Qbar$-linear relations between elements of $\Per$ are induced by the above relations.
\end{conjecture}

\begin{rem}
\begin{enumerate}
\item 
In the abstract formalism of Chapter~\ref{ch:formalism}
this is the Period Conjecture for the diagram $\pairseff$ for $k=\Qbar$, see Definition~\ref{defn:pairseff}.
\item
A close look shows that Conjecture~\ref{conj:kontsevich_full} is not identical to the conjecture originally formulated
by Kontsevich in \cite{kontsevich}. He was only considering smooth varieties $
X$ and divisors with normal crossings $Y$. We refer to the discussion in \cite[Remark~13.1.8]{period-buch}
for the  precise connection.  The version above implies that $\Spec(\Per)$
is a torsor under the motivic Galois group of $\Qbar$, a result due to
Nori. It was first formulated in  \cite[Theorem~6]{kontsevich}. A complete proof can be found in \cite[Theorem~13.1.4]{period-buch}.
\item 
By the Künneth formula, products of periods are in fact periods of products of varieties. Hence
the above conjecture also says something about \emph{algebraic} relations between periods and, indeed, it is equivalent to a Grothendieck style version of
the Period Conjecture. For a complete discussion, see \cite[Section~13.2]{period-buch}. We do not deal with  the latter because we are interested in the set $\Per^1$, which is \emph{not}
closed under multiplication.
\end{enumerate}
\end{rem}

\begin{thm}[Period conjecture for $\Per^1$]\label{thm:main_kontsevich}
\index{Period Conjecture!for $\Per^1$}
The Period Conjecture is true for the subset $\Per^1$. More explicitly, the following equivalent statements hold true:
\begin{enumerate}
\item \label{it:period_1}All relations between periods of $1$-motives are induced by bilinearity and functoriality of $1$-motives.
\item \label{it:period_2}All relations between periods of curve type are induced by bilinearity 
and functoriality of pairs $(C,D)\to (C',D')$ with $C, C'$ smooth affine curves and
$D, D'$ finite sets of points of $C$ and $C'$, respectively.
\item \label{it:period_3}All relations between periods in cohomological degree at most $1$ are induced
by the relations (A), (B), (C).
\end{enumerate}
\end{thm}

\begin{rem}
This theorem does not mention Nori motives. In contrast to Section~\ref{sec:nori}, we have not been able to eliminate them from the proofs, at least not without disproportionate effort.
\end{rem}

\begin{proof}[Motivic Proof of Theorem~\ref{thm:main_kontsevich}.]
Assertion~(\ref{it:period_1}) is precisely the 
statement of Theorem~\ref{thm:main_one}.
Hence it remains to show the equivalence with the others.  Here we rely significantly on the results of Appendix~\ref{sec:app_nori}.

By Theorem~\ref{thm:abv} the category $\onemot_\Qbar$ is equivalent to
$d_1\MMNeff(\Qbar,\Q)$. This category has a description as the diagram category of
the diagram of pairs $(C,D)$ with $C$ a smooth affine curve and $D$ a finite set of points on $D$. By the general results of \cite[Theorem~8.4.22]{period-buch} this implies
that all relations are induced by bilinearity and functoriality for the edges of the diagram,  i.e. functoriality for pairs.
This is the implication from Assertion~(\ref{it:period_1}) to Assertion~(\ref{it:period_2}).

The proof of Proposition~\ref{prop:reduce_C} shows that all elements in
$\Per^1$ can be related to periods of curve type using only the operations
(A), (B), (C). 
 Hence Assertion~(\ref{it:period_2}) implies Assertion~(\ref{it:period_3}).

In order to show the implication from Assertion~(\ref{it:period_3}) to
Assertion~(\ref{it:period_1}), we apply Theorem~\ref{thm:abv} and  replace
$\onemot_\Qbar$ by the equivalent $d_1\MMNeff(\Qbar,\Q)$. Every object $M$ in the latter category is a subquotient of
an object of the form $H^i(X,Y)$ for $i\leq 1$. Theorem~\ref{thm:abv} even shows that it suffices to take $i=1$ and $X$ a curve. The functoriality relation for periods of Nori motives identifies the periods of $M$ with the periods
of $H^i(X,Y)$ for $i\leq 1$. The relations 
(B) and (C)
are special cases of the functoriality relation for Nori-$1$-motives. 
This finishes the proof.
\end{proof}

We come back to the category of Nori-$1$-motives and its realisations as discussed in
Section~\ref{ssec:mot}.

\begin{thm}[Fullness]\label{thm:ff}
\index{fullness!for Nori motives and Hodge structures}
The three natural functors $f_1,f_2,f_3$ on $\onemot_{\Qbar}$
\[\begin{xy}\xymatrix{
\onemot_\Qbar\ar[r]^{\simeq\hspace{2em}}& d_1\MMNeff(\Qbar,\Q)\ar@{^(->}[r]^{f_1}\ar@{_(->}[rd]^{f_2}\ar@{_(->}[rdd]^{f_3}&\MMN(\Qbar,\Q)\ar[d]\\
 &&\MHS_\Qbar\ar[d]\\
&&\VVarg{\Qbar}{\Q}
}\end{xy}\]
are fully faithful with image closed under subquotients.
\end{thm}
\begin{proof}It suffices to consider the total functor $f_3$.  Composition of $f_3$ with the anti-equivalence with
$\onemot_\Qbar$ is simply $\cdot\circ V$. Therefore it suffices to consider
$V:\onemot_\Qbar\to \VVneu$. This functor is fully faithful by Theorem~\ref{thm:fulness_VV}
\end{proof}

\begin{rem}
We gave a direct proof for $\MHS_\Qbar$ earlier, see Proposition~\ref{prop:hodge_ff}. Both arguments rely on the Analytic Subgroup Theorem applied to the graph of a morphism, but applied in a different way.
\end{rem}

\section{The Case of Curves}

We turn now to the case of curves, which is of particular interest. We begin with some motivating discussion with historical background. Then we turn to the Period Conjecture for curves. One of the highlights is a precise criterium for a sum of periods of a single differential to be algebraic. 

\subsection{Motivating Examples}

In his book on transcendental numbers \cite{siegel-buch} Siegel mentioned several problems which were not accessible at the time. He wrote (see loc. cit. p. 97)

{\it All our transcendence proofs made essential use of the fact that the problem can be reduced to the proof of a property of entire functions.
This is the reason why the known methods do not work for elliptic integrals of the third kind and not even for integrals of the third kind in the still simpler fact of curves of genus 0. For instance, it is not known whether the number 
\[
\int_0^1\frac{dx}{1+x^3}=\frac{1}{3}(\log 2+\frac{\pi}{\sqrt{3}})
\]
is irrational.}

Integrals of this form along not necessarily closed paths are what we call  incomplete periods of the third kind on $\Pe^1$. As it turned out such integrals are not only irrational but even transcendental as follows from Baker's work on linear forms in logarithms. Indeed
one deduces from the inhomogeneous  case  of Baker's theorem about linear forms in logarithms of algebraic numbers that the numbers $1$, $\log 2$ and $\pi=-i \log(-1)$ are linearly independent over $\Qbar$.
Strictly speaking this is not a transcendence result but a result on linear independence of incomplete periods of the third kind in the case of a curve of  genus 0. However the transcendence of $\log 2$ and $\pi=-i \log(-1)$ is an immediate consequence. 

A.  van der Poorten, see \cite{poorten} considered a more general  complete and also incomplete period of the third kind on a curve of genus 0. In this case  a differential $\xi$  of the third kind takes the form 
\[
\xi = \frac{P(x)}{Q(x)} dx
\] 
where $P(x)$ and $Q(x)$ are polynomials. He considers a path $\gamma: [0,1] \rightarrow \Pe^1$ along which the differential form is defined and which satisfies $\gamma(0), \gamma(1) \in \Pe^1(\Qbar)$. We write  $\alpha_1,\ldots,\alpha_n$ for the zeroes of $Q$ and for $r_1,\ldots,r_n$ the residues at the poles of the differential form $\xi$. Then he deduces again from the inhomogeneous version of Baker's theorem on linear forms in logarithms that $\int_\gamma\xi$ is algebraic if and only
\[
\int_\gamma\big(\sum_{k=1}^n \frac{r_k}{x-\alpha_k}\big)\;dx=0.
\]
This follows from taking the partial fraction decomposition 
\[
\frac{P(x)}{Q(x)}dx= dF(x)+\sum_{k=1}^n\frac{r_k}{x-\alpha_k}dx.
\]
and integrating along $\gamma$.

In Theorem~\ref{thm:transc_curves} below we will give  a generalisation of van der Poorten's result to curves of any genus. In particular van der Poorten's Theorem is a special case of our result. Furthermore in the case of positive genus this includes abelian integrals of the third kind and proves transcendence of complete and incomplete periods.

An even older example was pointed out by Arnol'd in \cite{arnold}.
He gives a reference to a letter of Leibniz to Huygens, dated ${10\over20}$ April 1691. In this letter Leibniz formulated the problem of transcendence of the areas of segments cut off from an  algebraic curve, defined by an equation with rational coefficients, by straight lines with algebraic coefficients (see \cite[p. 93, footnote]{arnold}). In \cite[p. 105]{arnold} Arnol'd reformulated this problem turning it into modern language: an abelian integral along an algebraic curve with rational (algebraic) coefficients taken between limits which are rational (algebraic) numbers is generally a transcendental number. Again Theorem~\ref{thm:transc_curves} below gives the solution to Leibniz' problem. We refer to \cite{wuestholz}
for a more detailed discussion of the example.

As Arnol'd pointed out the problem is also very interesting from a historical point of view, in so far as it was previously believed that transcendence theory developed in the nineteenth century with Liouville, Hermite and others. The document of Leibniz, however, shows that already in the seventeenth century the concept of transcendence of numbers was present.

\subsection{The Period Conjecture for Curves}

\begin{thm}
\index{Period Conjecture!for curves}
Let $C$ be a smooth curve over $\Qbar$ with generalised Jacobian $J(C)$, $D\subset C$ a finite set of $\Qbar$-points with group of divisors of degree $0$ supported in $D$ denoted $L=\Z[D]^0$.
 Then all relations between periods of $H^1(C,D)$ are induced by bilinearity and morphisms between subquotients of sums of the $1$-motive $[L\to J(C)]$.
\end{thm}
\begin{proof}
By Lemma~\ref{lem:per_J}, the periods of $H^1(C,D)$ agree with the
periods of the $1$-motive as described in the theorem.
We then apply
Theorem~\ref{thm:main_one}.
\end{proof}

\begin{rem}
Note that this version of the Period Conjecture does not rely on the higher theory of geometric or Nori motives --- only $1$-motives are used. This version  is actually a lot more useful in computations. 
\end{rem}

We now turn from relations between periods to the question of transcendence of periods of differential forms.
It suffices to consider the case of a smooth projective curve $C$. 
With a rational function  $f\in\Qbar(C)^*$ we associate a  meromorphic differential form $\omega=df$. We also choose a path $\gamma$ in $C^\an$ which avoids the singularities of $f$ and has endpoints with beginning and end point in $C(\Qbar)$. Then
\[ \int_\gamma\omega= f(\gamma(1))-f(\gamma(0))\in\Qbar.\]
This is essentially the only way to produce algebraic periods from meromorphic differential forms as the following theorem shows.

\begin{thm}[Transcendence of periods]\label{thm:transc_curves}
\index{transcendence!of periods of curves}
Let $C$ be a smooth projective curve over $\Qbar$ and $\omega$ a meromorphic differential form defined over $\Qbar$. Let $\sigma=\sum_{i=1}^na_i\gamma_i$ where
 $\gamma_i:[0, 1]\to C$ for $i=1,\dots,n$ are  differentiable paths   avoiding the poles of
$\omega$ and $a_i\in\Z$. We assume that
$\partial\sigma$ has support in $C(\Qbar)$. 
In this situation the period
\[ \alpha=\int_\sigma\omega.\] is algebraic if and only if $\omega$ is the sum of an exact form with no extra poles  and a form with vanishing period.
\end{thm}

\begin{rem}
The theorem includes famous cases like the transcendence of $\pi$, $\log\alpha$ (for
$\alpha$ algebraic) and periods and quasi-periods of elliptic curves.
Forms of the first, second and third kind are allowed.
\end{rem}

\begin{proof} 
Let $C^\circ\subset C$ be an affine curve such that
$\omega$ is holomorphic on $C^\circ$. Let $D\subset C^\circ$ be the set of starting and end points
of the paths $\gamma_1,\dots,\gamma_n$. Then $\alpha$ can be considered as a period of 
$H^1(C^\circ,D)$. 
We introduce the generalised Jacobian $G=J(C^\circ)$ and fix an embedding $C^\circ\to G$. This translates  $\alpha$ into a period of
the $1$-motive $M=[\Z[D]^0\to G]$ by viewing $[\omega]\in H^1_\dR(C^\circ,D)$ as an element of $\VdR{M}$ and $[\sigma]\in H_1^\sing(C^\circ,D;\Q)$ as an element of $\Vsing{M}$.
By Theorem~\ref{thm:transc_1} the algebraicity of the period implies
that $\omega$ can be written as $\phi+\psi$ such that $\int_\sigma\psi=0$ and that the image of 
$\phi$ in $V_\dR(G)^\vee=H^1_\dR(G)\isom H^1_\dR(C^\circ)$ is zero. As $C^\circ$ is affine, this means that the differential form $\phi\in\Omega^1(C^\circ)$ is exact with poles only in $C\ohne C^\circ$. 

This finishes the proof except  when
$\omega\in\Omega^1(C)$. It remains to show that $\phi$ is not only exact but has no poles. This requires an extra argument that we give
in Proposition~\ref{prop:exc_case} below.
\end{proof}

\begin{prop}\label{prop:exc_case}
 Let $C$ be a smooth projective curve over $\Qbar$, $\omega\in\Omega^1(C)$, $\sigma$ a linear combination of paths with endpoints in $C(\Qbar)$. If
$\int_\sigma\omega$ is algebraic, then it is zero.
\end{prop}
\begin{proof}
Let $M=[\Z[D]^0\to J(C)]$ be as in the proof of Theorem~\ref{thm:transc_curves}. We consider $\sigma\in\Vsing{M}$ as an element of $\Lie(M^\natural)_\C$ via the inclusion $\Vsing{M}\subset \Lie(M^\natural)_\C$.
By construction $M^\natural$ is a vector extension of
$J(C)$ and by assumption $\omega$ is in the image of $\coLie(J(C))\to \coLie(M^\natural)$  . Hence the period $\alpha$ only depends on the image $\bar{\sigma}$ of $\sigma$ in $\Lie(J(C))_\C$. 
We now consider the connected algebraic group $G=J(C)\times\Ga$ and apply the
Analytic Subgroup Theorem~\ref{thm:annihilator} to the point $u=(\bar{\sigma},1)\in\Lie(G)^\an$ and to $(\omega,\alpha dt)\in\coLie(G)$. The theorem gives a short exact sequence
\[ 0\to H\to G\xrightarrow{\pi} G/H\to 0\]
such that $u$ is in the image of $\Lie(H)^\an$ and $(\omega,\alpha dt)$ in the image of $\coLie(G/H)$ under $\pi^*$. As $G$ is the product of an abelian variety and  $\Ga$, there are only two possibilities for $H$: it is either an abelian subvariety of $J(C)$ or a product of an abelian subvariety $B$  with  $\Ga$. 
In the first case, the inclusion $H\to G$ factors via $J(C)$. This contradict
the shape of $u$. In the second case $G/H=J(C)/B$ is an abelian variety. The subgroup $\Ga$ is contained in the kernel of $\pi$ and so $\pi^*\coLie(G/H)\subset \coLie(J(C))\times 0$, i.e. $\alpha dt=0$. This gives $\alpha=0$ as we wanted to prove.
\end{proof}

\begin{rem}
This proof goes back to the original formulation of the Analytic Subgroup Theorem. It is not enough to apply the transcendence criterion in Theorem~\ref{thm:transc_1}. The additional input here is the Hodge filtration on $V_\dR(M)$.  
\end{rem}

\begin{cor}\label{cor:closed}
Let $C$, $\omega$ and $\sigma$ be as in Theorem~\ref{thm:transc_curves}. If in addition $\omega\in\Omega^1(C)$ or if  the divisor
$\partial\sigma=\sum_{i=1}^na_i\gamma_i(1)-\sum_{i=1}^na_i\gamma_i(0)$ vanishes, then the period $\alpha$ is either transcendental or zero.
\end{cor}
\begin{proof} Suppose that $\alpha=\int_\sigma\omega$ is algebraic. 
By the theorem there is a decomposition $\omega=df+\phi$ such that $\int_\sigma\phi=0$.
If the boundary divisor vanishes, then
\[ \int_\sigma df=\int_{\partial \sigma} f=0\]
for all $f$.
If $\omega\in\Omega^1(C)$, then $f$ is in $\Oh(C)$, which gives $df=0$.
In both cases $\alpha=\int_\sigma\phi=0$.
\end{proof}

Masser pointed out to us that most of Theorem~\ref{thm:transc_curves} has an elementary reduction to the case of closed cycles announced in \cite{wuestholz-icm}. We explain a variant of his argument. 

\begin{proof}
Let $\sigma$ be as in the theorem such that $\partial\sigma\neq 0$. 
We write $\partial \sigma=\sum_{i=1}^mb_iP_i$ with $P_i\in C(\Qbar)$ and non-vanishing $b_i$. By assumption $m\geq 1$. Let $Q$ be in the polar locus of $\omega$. (If $\omega\in\Omega^1(C)$ pick any $Q\in C(\Qbar)$ not on any of the $\gamma_i([0,1])$.)

We consider the divisors $D_N=NQ-P_2-\dots-P_m$ and $D'_N=D_N-P_1$. For big enough $N$, Riemann-Roch gives $l(D_N)=l(D'_N)-1$. Let $f\in L(D'_N)\ohne L(D_N)$. This is a rational function with a pole only in $Q$ and a zero in
$P_2,\dots,P_m$, but not in $P_1$. As a consequence we obtain 
\[ \int_\sigma df=\sum_{i=1}^m b_i f(P_i)=b_1f(P_1)\neq 0.\]
The function $f$ can be normalised such that the value of the integral is $1$.
If $\alpha=\int_\sigma\omega$ is algebraic, then it is equal to
$\int_\sigma d(\alpha f)$. Introduce $\phi=\omega- d(\alpha f)$. By construction,
$\int_\sigma\phi=0$, as we wanted to show. 

As in the original proof of Theorem~\ref{thm:transc_curves} the argument does not allow to control the poles of $df$ in the special case where
$\omega$ is holomorphic. We deduced this case directly from the Analytic Subgroup Theorem in Proposition~\ref{prop:exc_case}.
\end{proof}

\begin{rem}
Qualitatively, we are in a situation very similar to the case of closed cycles: understanding algebraicity of periods requires understanding vanishing of periods.
\end{rem}

We come back to the case where $\sigma=\gamma$ is a single non-closed path later
in Corollary~\ref{cor:non-closed} under the simplifying assumption that
$J(C)$ is simple.

\chapter{Vanishing of Periods of Curves}\label{ch:van}

In this chapter, we translate our results obtained so far to the classical language and turn to the subtle question when periods integrals on curves vanish.
In this chapter let $C$ be a smooth projective curve over $\Qbar$ and
$\omega\in\Omega^1_{\Qbar(C)}$ a rational algebraic differential form on $C$.

\section{Classical Periods}
\index{classical periods}\index{periods!of curves}\index{periods!vanishing}\index{vanishing of periods}
We come back to the classical terminology.

\begin{defn}
 We say that $\omega$ is 
\begin{enumerate}
\item \emph{exact}\index{differential forms!exact} if it is of the form $\omega=df$ for some $f\in\Qbar(C)^*$,
\item \emph{of the first kind} \index{differential forms!first, second, third kind} if it does not  have poles, i.e. $f\in\Omega^1(C)$;
\item \emph{of the second kind} if the residues of $\omega$  are zero;
\item \emph{of the third kind} if it has at most simple poles.
\end{enumerate}
\end{defn}
This terminology goes back to the early 19th century, when Legendre 
studied elliptic integrals. We refer to \cite[Chapter~7]{Wells-history} for historical comments.

Following this definition, exact forms are of the second kind and differential forms of the first kind are also of the second and third kind. 

\begin{lemma}\label{lem:decomp_23}
Every differential form can be written in the form
\[ \omega=\omega_2+\omega_3\]
with $\omega_2$ of the second kind and $\omega_3$ of the third kind.
\end{lemma}
\begin{proof}
Let $\omega$ be an arbitrary meromorphic differential form. The sum of its residues is $0$. For example by \cite[Lemma~p.~233]{griffiths_harris} there is
a form $\omega_3$ of the third kind with the same poles and residues as $\omega$. Then $\omega_2=\omega-\omega_3$ has the desired property.
\end{proof}

\begin{defn}
A complex number is called \emph{classical period} if it is of the form 
\[ \int_\sigma\omega\]
where $\sigma=\sum_{i=1}^na_i\gamma_i$ is a formal $\Z$-linear combination of $C^\infty$-paths avoiding the singularities of $\omega$ and with endpoints in $C(\Qbar)$. We say that it is 
\begin{enumerate}
\item  \emph{simple} \index{periods!simple}\index{simple periods}if
$n=1$ (integral over a single path),
\item \emph{complete} \index{periods!complete}\index{complete periods} if all $\gamma_i$ are closed;
\item of \emph{the first, second or third kind} \index{periods!first, second, third kind}in the case that $\omega$ is of the first, second or third kind, respectively.
\end{enumerate}
\end{defn}
Note that the classical periods are algebraic for exact forms.
A cycle $\sigma=\sum_{i=1}^na_i\gamma_i$ with closed $\gamma_i$ can be seen  in the abelianisation of the fundamental group. In consequence, we may replace $\sigma$ by a homologous closed path defining the same period number. Accordingly all complete periods are simple. However, not all incomplete periods are simple. 

We get back the same periods that we considered earlier.

\begin{lemma}
The set of all classical periods agrees with the set of cohomological periods \index{periods!cohomological}\index{cohomological periods} $\Per^1$.
\end{lemma}
\begin{proof}
By Corollary~\ref{cor:from_curves}, all elements of $\Per^1$ are classical periods. Conversely, let $C^\circ$ be the complement of the set of poles of $\omega$,
$D\subset C^\circ(\Qbar)$ the finite set of end points of the $\gamma_i$. Then $\omega$ defines a class $[\omega]\in H^1_\dR(C^\circ,D)$ and
$\sigma$ defines a class $[\sigma]\in H_1^\sing(C^{\circ,\an},D;\Z)$. The integral computes the period pairing
\[ \int_\sigma\omega=\langle[\omega],[\sigma]\rangle\]
and makes our classical period a cohomological period. 
\end{proof}

Our next step is to find necessary and sufficient conditions for a differential form $\omega$ and a cycle $\sigma$ to satisfy 
\[ \int_\sigma\omega=0.\]
Here are some examples:
\begin{ex}\label{ex:vanishing}
\begin{enumerate}
\item Let $\gamma$ be contractible in the complement of the set of poles of $\omega$. Then the period vanishes by the Monodromy Theorem.
\item Let $\omega$ be of the third kind with poles in $Q_1,\dots,Q_m$. For every $i=1,\dots,m$ let $\gamma_i$ be the positively oriented boundary of a small disc in $C^\an$ centered at $Q_i$. Then
\[ \int_{\sum \gamma_i}\omega=\sum_{i=1}^m\res_{Q_i}\omega=0\]
by the Residue Theorem.
\item \label{it:van3}
To give an example for vanishing of periods we take the elliptic curve $E$ given by $y^2=x^3+1$ and define a curve $C$ by $y^2=x^6+1$. Sending $(x,y)$ to $(x^2,y)$ defines a morphism $\pi: C\to E$
of degree $2$. The genus of $E$ is 1 and the genus of $C$ is 2. The morphism $\pi$ induces a homomorphism
\[
\pi_*: H_1(C,\Z) \to H_1(E,\Z).
\]
We take $\omega\in H^0(E,\Omega^1(E))$ and $\gamma \in \ker \pi_*$ and obtain
\[
\int_\gamma \pi^*\omega = \int_{\pi_*\gamma} \omega =\int_0\omega =0.
\]
This shows that there are non-trivial examples for vanishing. 
\end{enumerate}
\end{ex}
We may ask how complete this list is. This question will be addressed  by first passing from paths and differential forms to homology and de Rham cohomology and then to $1$-motives. 

\section{The Setting}

The following data is fixed for the rest of the chapter:

\begin{notation}\label{not:curves_etc}
\begin{itemize}
\item $C$ denotes  a smooth projective curve over $\Qbar$ with base point $P_0\in C(\Qbar)$.
\item $\omega\in\Omega^1_{\Qbar(C)}$ is a meromorphic differential form.
\item $S=\{Q_1,\dots,Q_m\}$ stands for the set of poles of $\omega$ with non-trivial residue and  $S'=\{R_1,\dots,R_k\}$ the set of poles with vanishing residue; without loss of generality $P_0\notin S\cup S'$.
\item $C^\circ\subset C$ signifies the complement of the set of poles $S=\{Q_1,\dots,Q_m\}$ of $\omega$.
\item $J(C)$ and $J(C^\circ)$ are   the (generalised) Jacobians \index{generalised Jacobian} of $C$ and $C^\circ$ (in the sense of Chapter~\ref{app:B})  with embeddings $\nu^\circ:C^\circ\to J(C^\circ)$, $\nu:C\to J(C)$ via $P\mapsto P-P_0$.
\item $\sigma=\sum_{i=1}^n a_i\gamma_i$ is a formal $\Z$-linear combination of
$C^\infty$-paths $\gamma_i:[0,1]\to C^{\circ,\an}\ohne S'$ with endpoints defined over
$\Qbar$. 
\item 
Let $D\subset C^\circ\ohne S'$ be a  set of points such that
the divisor $\partial \sigma$ has support on $D$. We define $r=|D|-1$ if $D\neq \emptyset$ and $r=0$ if $D=\emptyset$. 
\item $\alpha=\int_\sigma\omega$ is the period of $\sigma$ and $\omega$.
\end{itemize}
\end{notation}

In Section~\ref{ssec:paths} we introduced  the map $I$ which assigns to a path or more generally a chain in a complex Lie group, an element of the complex Lie algebra.
Given a path $\gamma:[0,1]\to C^{\circ,\an}$, we put  
\[ l^\circ(\gamma)=I(\nu^\circ\circ\gamma)\in\Lie(J(C^\circ))^\an.\]
The operator $l^\circ$ has the property that
\[ \exp(l^\circ(\gamma))=\gamma(1)-\gamma(0)=:P(\gamma)\in J(C^\circ)^\an,\]
hence should be seen as a choice of logarithm.
We extend $l^\circ$ linearly to
\[ l^\circ(\sigma):=\sum_{i=1}^na_il^\circ(\gamma_i).\]
Then 
\[ \exp(l^\circ(\sigma))=\sum_{i=1}^n a_iP(\gamma_i)=:P(\sigma).\]
Let $l(\sigma)$ be the image of $l^\circ(\sigma)$ in $\Lie(J(C)^\an)$.

\subsection{Homological Interpretation}\label{ssec:2deRham} 
Obviously the chain $\sigma$ defines an element
$[\sigma]\in H_1^\sing(C^\circ,D;\Z)$.\index{singular homology} 
Less obviously,
the rational differential form $\omega$ defines an element $[\omega]$ of $H^1_\dR(C^\circ,D)$, see Section~\ref{ssec:de_rham}~(\ref{eq:de_rham_einfach}).\index{de Rham cohomology}
The argument is well-known for $D=\emptyset$. We explain the
construction in general.

If $S\cup S'=\emptyset$, then $\omega$ is in $\Omega^1(C)\subset H^1_\dR(C,D)$ and there is nothing to show. If $S\cup S'\neq \emptyset$,  the curve $C^\circ\ohne S'$ is affine and $\omega|_{C\circ\ohne S'}\in\Omega^1(C^\circ\ohne S')$ defines an element of $H^1_\dR(C^\circ\ohne S',D)$. By definition of $S'$, the cohomology class is in the kernel of
the residue map $H^1_\dR(C^\circ\ohne S',D)\to H^0(S')(-1)$, making it even an element of $H^1_\dR(C^\circ,D)$. We carry out the cocycle computation for $[\omega]$ in the representation of Lemma~\ref{lem:de_rham_explicit}.

Recall that $S'=\{R_1,\dots,R_k\}$ and  
choose $f_i\in\Qbar(C)^*$ such that the principal part of $df_i$ at $R_i$ coincides with the principal part of $\omega$ in $R_i$. Such a function exists
because the residue of $\omega$ vanishes. We write $\omega_i=\omega-df_i$ and $U_i\subset C$ for the complement of the set of poles of $\omega_i$. By construction $R_i\in U_i$. We introduce also $U_0=C^\circ\ohne S'$, $\omega_0=\omega$, $f_0=0$. 
Then $\Uf=\{U_0,\dots,U_k\}$ is an open cover of $C^\circ$. The differential form $\omega$
defines a cocycle in $H^1_\dR(C^\circ\ohne S',D;\Uf)$ given by the tuple  $\underline{\omega}=(\omega|_{U_i\ohne S'},0,0)$ cohomologous to the cocycle  
\[ \underline{\omega}-\partial f =((\omega-df_i)_i,(-f_j+f_i)_{i,j},(f_i|_{D})_i)\]
where $i,j$ run through $0$ to $k$.
It defines a cocycle in $H^1_\dR(C^\circ,D;\Uf)$ as required.

\begin{lemma}\label{lem:difff_deRham}
\index{de Rham cohomology!explicit cocycle}
The class $[\omega]\in H^1_\dR(C^\circ,D)$ is zero if and only if
$\omega$ can be represented as $\omega=df$ for a function $f\in\Qbar(C)^*$ which is regular in $D$ and vanishes there. Moreover, every element of $H^1_\dR(C^\circ,D)$ is of the form $[\omega]$ such that the set of poles of $\omega$ with non-trivial residue is contained in $S$.
\end{lemma}
\begin{proof}If $[\omega]=0$, then by Lemma~\ref{lem:de_rham_explicit} the cocycle is of the form
\[ \underline{\omega}-\partial f=((dg_i)_i, (g_j-g_i)_{ij},(-g_i|_{D_i})_i)\]
for $(g_i)_i\in\prod_i\Oh(U_i)$, in particular,
\[ -f_j+f_i=g_j-g_i\in \Oh(U_i\cap U_j)\]
for all $i,j$.
This implies that the collection of functions $f_i+g_i\in\Oh(U_i)$ glues to a global function  $F\in\Oh(C^\circ)$. Moreover,
\[ \omega-df_i=dg_i\in\Omega^1(U_i)\]
for all $i$ and hence $dF=\omega$.  We then have
\[ \underline{\omega}-\partial f\sim (dF|_{U_i},0,0)\sim (0,0,F|_{D_i}).\]
Note that the $g_i$ and hence also $F$ are unique up to an additive constant. The triviality of $[\omega]$ implies that $F|_D$ is constant. By adjusting the constant, we get $\omega=dF$ with $F|_D=0$ as claimed.

It remains to prove that every class in $H^1_\dR(C^\circ,D)$ can be represented by a differential form. It is well-known that every element of $H^1_\dR(C)$ is represented by a differential form of the second kind, see \cite[p.~459]{griffiths_harris}. The sequence  
\[ H^1_\dR(C)\to H^1_\dR(C^\circ)\xrightarrow{\res} H^0_\dR(S)(-1)\to H^0_\dR(C)(-1)\]
for relative cohomology is exact. For every element of $H^0_\dR(S)\isom \Qbar^m$ summing up to $0$ in $H^0_\dR(C)\isom \Qbar$, there is by \cite[Lemma~p.~233]{griffiths_harris} a form $\omega_3$ of the third kind with these residues. 

This shows that given a class $c\in H^1_\dR(C^\circ)$, there exists $\omega_3$ of the third kind with the same residues. The difference $c-[\omega_3]$ is in the image
of $H^1_\dR(C)$ and represented by a differential form of the second kind.

The sequence
\[H^0_\dR(D)\to H^1_\dR(C^\circ,D)\to H^1_\dR(C^\circ)\]
for the relative cohomology is also exact. Elements of $H^0_\dR(D)$ are functions $f:D\to\Qbar$. 
Let $U_0\subset C^\circ$ be an open affine subset containing $D$ and $U_0,\dots,U_n$ an affine cover of $C^\circ$.
The image of $f$ in $H^1_\dR(C^\circ,D,\Uf)$ is the cocycle 
\[ (0,0,f|_{D_i}).\]
As $U_0$ is affine, the closed immersion $D\subset U_0$ induces a surjection
$\Oh(U_0)\twoheadrightarrow \Oh(D)$.
We choose a lift $\tilde{f}\in\Oh(U_0)$ of $f$. The class of $d\tilde{f}$ agrees
with the class of $(0,0,f|_{D_i})$ from above. 

Given a class $c$ in $H^1_\dR(C^\circ,D)$ we have shown that  there is a form $\omega$ with correct residues such that the images of $c$ and $[\omega]\in H^1_\dR(C^\circ,D)$ in $H^1_\dR(C^\circ)$ agree. Their difference is in the image of $H^0_\dR(D)$ and can be represented by an exact differential form.
\end{proof}

Our integral computes the period pairing
\begin{align*}
\langle\;,\;\rangle: H_1^\sing(C^\circ,D) \times H^1_\dR(C^\circ,D) &\rightarrow \C\\
(\sigma, \omega) & \mapsto \alpha=\langle[\sigma],[\omega]\rangle = \int_\sigma\omega
\end{align*}
It vanishes if $[\sigma]=0$ or $[\omega]=0$. We are interested in the cases where this condition is not satisfied.
 
\begin{defn}
We say that the period pairing has \emph{non-trivial vanishing for $(\omega,\sigma)$}\index{non-trivial vanishing} if $\alpha=\langle [\omega],[\sigma]\rangle=0$ but
 $[\omega]\neq 0, [\sigma]\neq 0$.

 We denote by $\locus\subset H_1^\sing(C^\circ,D)\times H^1_\dR(C^\circ,D)$ the set of pairs $([\sigma],[\omega])$ such that
$\langle [\sigma],[\omega]\rangle=0$. We further denote by $\locus^i\subset\locus$ for $i=1,2,3$ the subset where the
differential form is of first, second and third kind, respectively.
\end{defn}
Our aim is to determine  under which conditions for $\sigma$ and $\omega$  the pairing has non-trivial vanishing. 

 Note that $D\hookrightarrow D'$ implies an inclusion $H_1^\sing(C^\circ,D)\hookrightarrow H_1^\sing(C^\circ,D')$
as well as a surjection $H^1_\dR(C^\circ,D')\twoheadrightarrow H^1_\dR(C^\circ,D)$.
Hence vanishing of $[\sigma]$ does not depend on $D'$, but vanishing of $[\omega]$ does.

\subsection{Translation to Motives}\label{ssec:trans_motives}

Properties of the periods of $H^1(C^\circ,D)$ depend on the $1$-motive
\[ M=[L\to J(C^\circ)]\]
 where $L=\Z[D]^0$ is the group of divisors of degree $0$ supported on $D$. It has rank $r$.
By Lemma~\ref{lem:per_J}, more precisely equation (\ref{eq:V}) of its proof, we have 
\[ \Vsing{M}\isom H_1^\sing(C^\circ,D;\Q), \hspace{1ex}\VdR{M}\isom H^1_\dR(C^\circ,D)\]
and 
\[ \alpha=\int_\sigma\omega\]
can be viewed as a period of $M$. 

\begin{lemma}
The class of $\sigma$ is given by 
\[ [\sigma]=(\partial\sigma,l^\circ(\sigma))\in T_\sing(M)\subset L\times \Lie(J(C^\circ))^\an.\]
Both components have image $P(\sigma)$ in $J(C^\circ)$.
\end{lemma}
\begin{proof}
This is precisely the identification of classes in $H_1^\sing(C^\circ,D;\Z)$ with $T_\sing(M)$ in Lemma~\ref{lem:lem2}.
 Classes in relative homology are represented by pairs in $S_0^\infty(D)\oplus S_1^\infty(C^\circ)$, in our case by $(\partial\sigma,\sigma)$.
 The first summand maps to $L$, the second maps a path $\gamma$ to $I(\gamma)$ and hence $\sigma$ to $l^\circ(\sigma)$.
\end{proof}

The key for determining the spaces $\mathcal N$ and $\mathcal N^i$ is Theorem~\ref{thm:annihilator_mot}. It shows that non-trivial vanishing of $\alpha$ is caused
by a non-trivial exact sequence 
\begin{equation}\label{eq:fund_seq} 
0\to M_1\xrightarrow{\iota} M\xrightarrow{p} M_2\to 0
\end{equation}
with $M_i=[L_i\to G_i]$, with $[\sigma]=\iota_*\sigma_1$ induced from $M_1$, $\sigma_1\in V_\sing(M_1)$,  and $[\omega]=p^*\omega_2$ induced from $M_2$, $\omega_2 \in \VdR{M}$. 

To analyse this explicitly we go through the various types of differential forms.
If $J(C)$ is not simple, there are many such sequences and there is  a lot of non-trivial vanishing, see Example~\ref{ex:vanishing}(\ref{it:van3}). The general case seems to be very complicated, so that
we restrict our discussion to the case where $J(C)$ is simple for the rest of this chapter.

\begin{rem} 
By construction, $[\sigma]$ is induced from the submotive 
\[ 
M'=[\Z_\sigma\to J(C^\circ)]\subset M
\]
where $\Z_\sigma$
 is the sublattice of $L$ generated by the element $\partial \sigma$.
\end{rem}

We go now  through the different cases, that is forms of the first kind, of the second and of the third kind, one after the other and analyse the conditions for non-trivial vanishing. This analysis is carried out by going through the different possible shapes of $M_1$.
We finally specialise to the easier case where, in addition, $\sigma$ is simple.

\section{Forms of the First Kind}\label{sssec:first}
 
A form of the first kind is a non-zero global  differential form $\omega\in H^0(C,\Omega^1_C)$ without pole. This implies that we use $C^\circ=C$. As already mentioned before we assume for simplicity that $J(C)$ is simple. The class of $\omega$ in $H^1_\dR(C,D)$  and even in $H^1_\dR(C)$ is non-zero because differential forms of the first kind cannot be exact.

In this situation the relevant  motive is $M=[\Z[D]^0\to J(C)]$ and for applying Theorem~\ref{thm:annihilator_mot} we have to determine the possible submotives  $M_1$.
Since $J(C)$ is assumed to be simple 
they have either the shape  $[L_1\to J(C)]$ or $[L_1\to 0]$ for some $L_1\subset \Z[D]^0$ with quotient $M_2$.  We discuss the two cases separately.

\subsection{The Case $M_1=[L_1\to J(C)]$}\label{ssec:first_1}
According to Theorem~\ref{thm:annihilator_mot}, the form $[\omega]$ is
a pull-back from $M_2$, hence its pull-back to $[L_1\rightarrow J(C)]$ is zero, as is its restriction to $[0\to J(C)]$. This is equivalent to $[\omega]=0$ in $H^1_\dR(C)$.
Since it is a differential form of the first kind this implies $\omega=0$. This case does
not occur.

\subsection{The Case $M_1=[L_1\to 0]\subset M$}

Since we are in the category of iso-$1$-motives the structure  map
$L_1\to J(C)$  is isogenous to the zero map and $\sigma=\iota_*(\sigma_1)$. 
This means that 
\[ P(\sigma)=\exp(l(\sigma))=\sum_i a_i P(\gamma_i)\in J(C)\]
 is a torsion point, a necessary, but not sufficient condition for non-trivial vanishing. This property has to be translated from a statement about $J(C)$ into a vanishing condition in the Lie algebra. The point is that the exponential map is not injective. This requires a more careful analysis of the situation.

Let $n\geq 1$ be chosen such that the image of $L$ in $J(C)$ under the composition
$L\to J(C)\xrightarrow{[n]_*}J(C)$ is torsion free. The  kernel of the composition has a complement $L_0$ with the property that the structure map of the modified motive $M_0=[L_0\to J(C)]$ (with structure map the restriction from the structure map of $M$) is injective. The map $[n]:M\to M$ factors through a morphism of
$1$-motives $M\to M_0$ with $M_1$ contained in its kernel. It is multiplication by $n$ on the abelian part.
The image of $\sigma$ is zero in $T_\sing(M_0)$ and hence $[n]_*l(\sigma)=0$ in $\Lie(J(C)^\an)$.
The map $[n]_*$ is multiplication by $n$ on $\Lie(J(C))^\an$, hence $[n]_*l(\sigma)=0$ implies
$l(\sigma)=0$.
 We conclude that non-trivial vanishing for $(\sigma,\omega)$ implies that $l(\sigma)=0$, no condition on $\omega$. In this situation this means that $\locus^1\subset \ker(l) \times H^0(C,\Omega^1)$.

From $l(\sigma)=0$ we conclude in particular that 
$P(\sigma)=0$, not only a torsion point. This implies that the points $P(\gamma_i)$ are linearly dependent in $J(C)$. 

Conversely, suppose that $\sigma$ satisfies $l(\sigma)=0$. We introduce the submotive $M_1=[\Z_\sigma\to 0]$ with
$\Z_\sigma$ the sublattice of $L=\Z[D]^0$ generated by $\partial\sigma$.
We map   $L_1=\Z$ to
$L=\Z[D]^0$ by mapping $1$ to
$\sum_{i=1}^n a_i(\gamma_i(1)-\gamma_i(0))\in \Z[D]^0$ and introduce the motive
$M_1=[L_1\to 0]\to [L\to J(C)^\an]$. It is well-defined because $P(\sigma)=0$ and by construction $[\sigma]$ is induced
from a class on $M_1$. 

Morphism of motives induce a morphism of Hodge structures, in particular it respects the Hodge filtration. By assumption $\omega$ is in $H^0(C,\Omega^1)=F^1H^1_\dR(C,D)=F^1\VdR{M}$.
 As a consequence the pull-back of $\omega$ to $M_1$ is in
$F^1\VdR{M_1}=0$ and this gives 
\[
 \int_\sigma\omega= \int_{\iota_*\sigma_1}\omega=\int_{\sigma_1}\iota^*\omega=\int_{\sigma_1}0=0.
 \]
We conclude that every pair $(\sigma,\omega)$ with $l(\sigma)=0$ is contained in $\mathcal N^1$. In other words we have $\locus^1\supset \ker(l) \times H^0(C,\Omega^1)$ and we conclude that  
\[ \locus^1=\ker(l)\times H^0(C,\Omega^1).\]

\begin{summary}\label{sum:1}
Let $J(C)$ be simple and  $\omega$ of the first kind. Then there is non-trivial vanishing for $(\omega,\sigma)$ if and only if $l(\sigma)=0$. We have
\[ \locus^1=\ker(l)\times H^0(C,\Omega^1).\]
 This  implies in particular that the $I(\gamma_i)$ are $\Z$-linearly dependent in $\Lie(J(C)^\an)$. A fortiori the points $P(\gamma_i)$ are linearly dependent in $J(C)$.
\end{summary}

\begin{remark}
The condition looks like trivial vanishing because the class of $\sigma$ vanishes in  $H_1^\sing(C,\Q)$---but it is not. The class
$[\sigma]\in H_1^\sing(C,D;\Q)$ is not in the image   of $H_1^\sing(C,\Q)$ under the natural restriction. Rather
there is a projection $H_1^\sing(C,D;\Q)\to H_1^\sing(C;\Q)$ defined only via the Jacobian. The image of $[\sigma]$
under this projection vanishes.
\end{remark}

\section{Forms of the Second Kind}\label{sssec:2}

Assume that $\omega$ has a non-empty set of poles but no residues. We have
 $[\omega]\in H^1_\dR(C,D)$ and again $M=[\Z[D]^0\to J(C)]$.
Assume that $J(C)$ is simple and that we have non-trivial vanishing induced from a sub-object $M_1\subset [L\to J(C)]$. 
As before we assume that the period pairing has non-trivial vanishing for $(\omega, \sigma)$.
This implies that $\sigma=\iota_*\sigma_1$ and $\omega=p^*\omega_2$. 

\subsection{The Case $M_1=[L_1\to J(C)]$} 

With the same argument as in Section~\ref{ssec:first_1} for $\omega$ of the first kind,
the assumption on $M_1$  leads to $\omega=df$ with $f$ not identically zero
on $D$. In the cocycle computation in Section~\ref{ssec:2deRham} we can choose
all $f_i$ as $f$ and represent the class of  
 $[\omega]$ in $H^1_\dR(C,D)$  by
\[ \underline{\omega}-\partial f=(0,0,f|_D).\]
 The period integral is in this case 
 \begin{align*}
 \alpha&= \langle[\omega],[\sigma]\rangle = \langle f,\partial \sigma\rangle \\
&=\sum_{i=1}^na_i(f(\gamma_i(1))-f(\gamma_i(0))).
\end{align*}
This is a $\Q$-linear combination of algebraic numbers.
It can vanish, and will in examples.

\begin{ex}
\begin{enumerate}
\item
Let  $E\subset\Pe^2$ be the elliptic curve given by 
\[
Y^2Z=X(X-Z)(X-\lambda Z)
\]
for $\lambda\in\Qbar\ohne\{0,1\}$ be an elliptic curve. Consider $\omega=dX$. This is an exact form. For every $x\in\Q$ there is a point $P_x$ in $E(\Qbar)$
with $X(P_x)=x\in \Q$. We choose $x_1,\dots,x_r\in \Q$. There
are coefficients $a_i\in \Q$ such that $\sum a_ix_i=0$. Let $\gamma_i$ be a path
from $P_{\,0}$ to $P_{x_i}$ and define $\sigma=\sum a_i\gamma_i$. Then
$\int_\sigma dX=0$.
\item Let $C$ be a smooth projective curve and $f\in\Qbar(C)$ not constant. 
We view $f$ as a non-constant morphism
$f:C\to \Pe^1$. For  $x\in \Pe^1(\Qbar)$, we  choose $y_0,y_1\in f^{-1}(x)$,  a path $\gamma$ from $y_0$ to $y_1$ in
$C^\an$. Then $\int_\gamma df=f(y_1)-f(y_0)=x-x=0$.
\end{enumerate}
\end{ex}

\begin{summary}\label{sum:ex}
If $\omega=df$ and $J(C)$ is  simple, then there is non-trivial vanishing if and only if
\begin{equation}\label{eq:exact} \langle f,\partial\sigma\rangle=0.
\end{equation}
\end{summary}

If $\sigma=\gamma$ is simple, this is equivalent to $f(\gamma(1))=f(\gamma(0))$.

\subsection{The Case $M_1=[L_1\to 0]$}
The same arguments as for periods of the first kind
imply $l(\sigma)=0$. We show that this is no longer sufficient, but there is also
a condition on $\omega$. Let $\Z_\sigma=\Z\;\partial\sigma\subset \Z[D]^0$ be generated
 by the divisor $\partial\sigma$. It is contained in $L_1$ because $l(\sigma)=0$. Without loss of generality, we take $L_1=\Z_\sigma$.

The condition on $\omega$ to be of the form $p^*\omega_2$ is equivalent to 
\begin{equation}\label{eq:vanish} 
\iota^*([\omega])=[\omega]|_{[\Z_\sigma\to 0]}=0.
\end{equation}

\begin{rem} 
For any $\omega$ of the second kind, the restriction $[\omega]|_{[\Z_\sigma\to 0]}$ is the restriction of an element
of $\VdR{[L\to 0]}$. Translated back to cohomology of curves this means that 
it is in the image of $H^0_\dR(D)$ in $H^1_\dR(C^\circ,D)$. In Lemma~\ref{lem:difff_deRham} the classes in the image are represented by exact differential forms $df$. 
In consequence, we find for every 
every $\omega$ of the second kind an exact form $df$ such that
$\omega-df$ satisfies the vanishing condition. The same is true for the converse: there is always an exact form $df$ such that the vanishing condition is not satisfied.  
\end{rem}

\begin{summary}\label{sum:2}
Let $J(C)$ be simple. Assume that $\omega$ is of the second kind, but neither of the first kind nor exact. Then there is non-trivial vanishing if and only if $l(\sigma)=0$ as well as equation (\ref{eq:vanish}) holds.
\end{summary}

\section{Forms of the Third Kind}\label{ssec:third} 

Assume that $\omega$ is of the third kind but not of the first kind. This means that it has a non-empty set of poles, all of them simple. In the language of Hodge theory this makes it a differential form with log poles. 
 Here we take $M= [L\to J(C^\circ)]$ as our motive. The form
$\omega$ defines a class $[\omega]$ in $H^1_\dR(C^\circ,D)$. It is always
non-zero. Hence the period integral vanishes trivially if $[\sigma]=0$. 
For simplicity, we assume that $J(C)$ is simple.
\subsection{The Case $M_1=[L_1\to 0]$}\label{ssec:third1}
As in the other cases, non-trivial vanishing implies $l^\circ(\sigma)=0$ in $\Lie(J(C^\circ)^\an)$.
We have $\omega\in H^0(C,\Omega^1(\log(S)))= F^1H^1_\dR(C^\circ,D)$, so the discussion is identical to the case of differentials of the first kind. The condition is also sufficient. This means that $l^\circ(\sigma)=0$ implies non-trivial vanishing.

\subsection{The Case $M_1=[L_1\to T_1]\subset [L\to J(C^\circ)]$}\label{ssec:third2}
Here the quotient is $M_2=[L_2 \to J(C^\circ)/T_1]$.  As $\sigma$ is in the image of $\Vsing{M_1}$,  it takes the form $\sigma=\iota_*(\sigma_1)$ for  some $\sigma_1\in \Vsing{M_1}$.  We deduce that   $p_*\sigma=p_*\iota_*\sigma_1=0$ and the same is true for its  image in $\Vsing{[L_2\to J(C)]}$, a quotient of $M_2$.  As in the
case of differentials of the first kind this implies that the image $l(\sigma)=I(\nu \sigma)$ of $l^\circ(\sigma)=I(\nu^\circ \sigma)$
in $\Lie(J(C))^\an$ vanishes. This means that $\sigma$ is closed and  $l^\circ(\sigma)$ is in $\Lie(T)^\an$ where $T$ is the torus part of $J(C^\circ)$  with character group $\Z[S]^0$. In other words, $\sigma$ defines a homology class of $T^\an$.
We choose small enough simple loops $\varepsilon_j$ in $\Lie(J(C^{\circ,\an}))$ around the singularities $Q_j\in S$ with $l^\circ(\varepsilon_i)$  in  $\Lie(T)^\an$ and introduce  $\varepsilon=\sum n_\sigma(Q_i)\varepsilon_i$,   with $n_\sigma(Q_i)$ the winding number of $\sigma$ around $Q_i$. The homology classes of $\varepsilon$ and $\sigma$ agree. As a consequence
 $l^\circ(\sigma)- l^\circ(\varepsilon)$ vanishes in $\Lie(J(C^{\circ,\an}))$. 
It follows that there is non-trivial vanishing for $\sigma-\varepsilon$ by the  case treated in Section~\ref{ssec:third1}. This shows that non-trivial vanishing
for $\sigma$ is equivalent to non-trivial vanishing for $\varepsilon$.
The period integral is up to the factor $2\pi i$ a linear combination of the residues: 
\begin{equation}\label{eq:van_3} 
\int_\sigma\omega=\sum_{i=1}^m 2\pi i \;n_\sigma(Q_i)\mathrm{Res}_{Q_i}(\omega).
\end{equation}
There are cases where this vanishes, but there are also cases where it does not vanish.

\begin{ex}
Choose a small disc $\Delta_i$ around each $Q_i$ and let $\gamma_i=\partial \Delta_i$.
Then $\int_{\sum\gamma_i}\omega=0$ by Cauchy's Residue Theorem.
\end{ex}

\subsection{The Remaining Case}\label{ssec:third3}

Finally we have to consider $M_1=[L_1\to G_1]$ where the abelian part of $G_1$ is
$J(C)$. This implies that $G_2=T_2$ is a torus.
Let $L_2$ be a direct complement of $L_1$ in $L$. Then  $M_2 = [L_2\to T_2]$

By the exact sequence (\ref{eq:fund_seq}) the pull-back of $[\omega]=p^*\omega_2$ to
$M_1$ and further to $\VdR{G_1}$ vanishes.
 If $G_1=J(C^\circ)$, then  $\omega$ would be  exact, but
 $\omega$ is of the third kind and cannot be exact.
This case does not occur. 
Therefore  $G_1\subsetneq J(C^\circ)$ and the quotient $J(C^\circ)/G_1\isom T_2$ is a non-trivial torus. The torus  of $J(C^\circ)$ decomposes up to isogeny as $ T_1\times T_2$ where $T_1$ is the torus part of $G_1$. From the discussion above we get  homomorphisms $J(C^\circ)\to T_2$ and $J(C^\circ) \to G_1$ and as a consequence  an isogeny $J(C^\circ) \cong G_1\times T_2$. The classifying map
$X(G_1\times T_2)\to J(C^\circ)$ 
is the product of the classifying maps of $X(T_1)\to J(C^\circ)$ and $X(T_2)\to 0$.
In particular, $X(T)_\Q\isom X(T_1)_\Qbar\times X(T_2)_\Q$ and the classifying map $X(T)\to J(C)^\vee$ vanishes on $X(T_2)_\Qbar$.

By Lemma~\ref{lem:class_jac}, the classifying map for the semi-abelian variety $J(C^\circ)$ is
\[ X(T)=\Z[S]^0\to J(C)\isom J(C)^\vee.\]
The vanishing of the classifying map on $X(T_2)$ means that there is a linear dependence between the $Q_i$ in $J(C)_\Q\isom J(C)^\vee_\Q$. 

The exact sequence (\ref{eq:fund_seq}) gives that the image of
$[\sigma]\in H_1^\sing(C^\circ,D;\Z)$  in  $H^\sing_1(T_2,D_2;\Z)$ is zero, where $D_2$ is the image of $D$ under the composition of the maps 
$C^\circ\to J(C^\circ)\to T_2$. Moreover, $[\omega]$ has to be a pull-back from
$\VdR{[L_2\to T_2]}$. 
In the identification
\[ \VdR{T}=X(T)_\Qbar\,\isom \,\Qbar[S]^0\] 
the restriction of $\omega$ to  $\VdR{T}$ is mapped to the divisor $\res(\omega)$ given by
\[ \sum \res_{Q_i}(\omega)\,Q_i\in\Qbar[S]^0.\]
This gives the necessary condition, namely that $\res(\omega)$ is in $X(T_2)_\Qbar$, or in other words that  $\sum \res_{Q_i}(\omega)Q_i=0$ in $J(C)_\Qbar$.

\begin{rem}\label{rem:van_nec_3}
The residue condition is necessary, but not sufficient. It is preserved by  adding a form $\omega_1$ of the first kind, but the restriction to
$\VdR{G_1}$ changes. The restriction is in $F^1\VdR{G_1}=F^1\VdR{M_1}$ and 
vanishes in $\VdR{T_1}$. Hence it is induced by a form of the first kind. There is always a choice of $\omega_1$ of the first kind such that $\omega-\omega_1$ satisfies the condition for non-trivial vanishing. Conversely, there is also $\omega_1$ such that it does not.
 \end{rem}

 \begin{summary}\label{sum:3}

Let $J(C)$ be simple, and $\omega$ a differential form of the third kind. Then
there is non-trivial vanishing for $(\omega,\sigma)$ if and only if one of the following conditions is satisfied:
\begin{enumerate}
\item \label{it:1insum3} $l^\circ(\sigma)=0$ in $\Lie(J(C^\circ))^\an$;
\item $l(\sigma)$ in $\Lie(J(C))^\an$ vanishes and
the linear combination in (\ref{eq:van_3}) vanishes;
\item $\res(\omega)$ sums up to $0$ in $J(C)_\Qbar$, and there is a quotient-motive $[L'\to T']$ of $[L\to J(C^\circ)]$ with $T'\subset T$ and 
$X(T') \subset \ker(X(T)\to J(C))$ 
such that
\begin{itemize}
\item $\res(\omega)$ is in $X(T')_\Qbar$,
\item   $[\sigma]$ vanishes in  $H^\sing_1(T',D';\Z)$ where $D'$ denotes  the image of $D$  under the  maps
$C^\circ\to J(C^\circ)\to T'$, 
  
\item $[\omega]$ is a pull-back of some $\omega'\in\VdR{[L'\to T']}$. 
\end{itemize}
\end{enumerate}
\end{summary}

\section{Arbitrary Differential Forms}

So far we have worked out necessary and sufficient conditions for non-trivial vanishing for differential forms of the first, second or third kind. We now turn to the general case
 and write $\omega=\omega_2+\omega_3$ with $\omega_2$ of the second and $\omega_3$ of the third kind. Such a decomposition exists by Lemma~\ref{lem:decomp_23}.

\begin{thm}\label{thm:van_arb}
\index{periods!vanishing}\index{vanishing of periods}
Assume that $J(C)$ is simple and
that we have non-trivial vanishing of $\int_\sigma\omega$.
Then one of the following conditions is satisfied:
\begin{enumerate}
\item \label{it:1arb}$\l^\circ(\sigma)=0$ and the vanishing condition of (\ref{eq:vanish}) holds for $\omega_2$;
\item \label{it:2arb}$l(\sigma)=0$ and for $\varepsilon$ as in Section~\ref{ssec:third2} we have
\[ \int_\varepsilon\omega_3=\sum_{i=1}^m2\pi i n_\varepsilon(Q_i)\res_{Q_i}(\omega_2)=0\]
as well as the vanishing condition of (\ref{eq:vanish}) for $\omega_2$ and $\sigma-\varepsilon$;
\item\label{it:case_exact} $\omega=df$ is exact and the vanishing condition of (\ref{eq:exact}) holds;
\item\label{it:van_arb4}
the divisor $\res(\omega): Q_i\mapsto \res_{Q_i}\omega$ on $S$ sums up  to $0$ in
$J(C)_\Qbar$, 
and there is a quotient-motive $[L'\to T']$ of $[L\to J(C^\circ)]$ with $T'\subset T$ and 
$X(T') \subset \ker(X(T)\to J(C))$ 
such that
\begin{itemize}
\item $\res(\omega)$ is in $X(T')_\Qbar$,
\item   $[\sigma]$ vanishes in  $H^\sing_1(T',D';\Z)$ where $D'$ denotes  the image of $D$  under the  maps
$C^\circ\to J(C^\circ)\to T'$, 
  
\item $[\omega]$ is a pull-back of some $\omega'\in\VdR{[L'\to T']}$. 
\end{itemize}
In this case, $\omega_3$ and $\omega_2$ can be chosen such that
both period integrals vanish.  
\end{enumerate}
\end{thm}
The proof is not fully self-contained, but partially relies on the structural insights of Part~\ref{part4}.
\begin{proof}As discussed in Section~\ref{ssec:trans_motives}, non-trivial vanishing is caused by a short exact sequence
\[ 0\to M_1\xrightarrow{\iota} [\Z[D]^0\to J(C^\circ)]\xrightarrow{p} M_2\to 0.\]
We go through the possible cases for $M_1$.

\textbf{Case $M_1=[L_1\to 0]$} In this case $l^\circ(\sigma)=0$. By the case of differential forms of the third kind the condition implies $\int_\sigma\omega_3=0$ by Summary~\ref{sum:3}~(\ref{it:1insum3})  and  the period integral for $\omega$ depends only on $\omega_2$. We get the vanishing condition formulated for periods of the second kind.

\textbf{Case $M_1=[L_1\to T_1]$}
The discussion starts as for forms of the third kind. We get $l(\sigma)=0$. 
We find $\varepsilon$ consisting of closed loops around the poles such that
$l^\circ(\sigma)=l^\circ(\epsilon)$. The vanishing of period gives an equality
for the periods of $\sigma-\epsilon$ (which only depends on $\omega_2$ by the previous case) and the period $l^\circ(\epsilon)$ (which only depends on
$\omega_3$ and is a linear combination of residues as in the case of periods of the third kind). A sufficient condition of vanishing is
the vanishing of both summands.

For the converse, consider $\int_\varepsilon\omega_3=-\int_{\sigma-\varepsilon}\omega_2$. The left hand side is a $\Qbar$-multiple of $2\pi i$ whereas the right hand side is an incomplete period of the second kind. By the structural results of Chapter~\ref{ch:struct} this makes them linearly independent. Both have to vanish.
To see this, we view the left hand side as a Tate period in the terminology of
Chapter~\ref{ch:struct}. The right hand side is an incomplete period of the second kind. 
By the second part of the linear independence result of Theorem~\ref{thm:disj} both are periods of the first kind with respect to closed paths. By the first linear independence statement of the same theorem, they are in turn linearly independent from Tate periods. This settles the claim.

\textbf{Case $M_1=[L_1\to G_1]$ with abelian part of $G_1$ equal to $J(C)$}
We have $G_1\subset J(C^\circ)$ with quotient a torus $T_2$. The pull-back of $\omega$ to $G_1$ is trivial. If $G_1=J(C^\circ)$, then
this makes $T_2=0$ and $\omega$ exact, a case we have treated in Summary~\ref{sum:ex}.  
 It gives item (\ref{it:case_exact}).

If $T_2\neq 0$, then the argument and the conclusion are the same as for forms of the third kind in Section~\ref{ssec:third3}. 
By Remark~\ref{rem:van_nec_3}, we can choose a form $\omega_1$ such that
there $\int_\sigma(\omega_3-\omega_1)$ vanishes. Hence there is also
vanishing for $\omega_2+\omega_1$.
\end{proof}

\section{Vanishing of Simple Periods}
\index{periods!simple}\index{simple periods}

The discussion simplifies for simple periods. If $\sigma=\gamma$, then
$P(\gamma)=0$ means that the path is closed. The conditions $l(\gamma)=0$  and $l^\circ(\gamma)=0$  mean that $\gamma$
is closed and contractible in $J(C^\an)$ and even $J(C^\circ)^\an$. In many cases this amounts to trivial vanishing: the homology class of $[\gamma]$ vanishes.
We specialise Theorem~\ref{thm:van_arb} to the case of simple periods.

\begin{cor}
Let $J(C)$ be simple, assume that $\sigma=\gamma$  is simple and choose  an arbitrary rational differential form $\omega$. There is non-trivial vanishing of the period $\int_\gamma\omega$ if and only if one of the following condition holds:
\begin{enumerate}
\item $\omega=df$ is exact and $\gamma$ links different points in a fibre of $f:C^\circ\to\A^1$;
\item $\gamma$ is a closed path, homologous to $0$ in $C^\an$ and winding around the poles of $\omega$ in such a way that the linear combination of the residues in (\ref{eq:van_3}) vanishes.
\item Condition (\ref{it:van_arb4}) of Theorem~\ref{thm:van_arb} is satisfied. 
\end{enumerate}
\end{cor}

Using our algebraicity criterion, this immediately translates into a strong transcendence result for simple periods. For closed $\gamma$, we have already
seen in Corollary~\ref{cor:closed} that a period is algebraic if and only if it is vanishes. Together with the above corollary, we now have a complete picture.
We are also ready to treat the case of a non-closed path.

\begin{cor}\label{cor:non-closed}
Let $C$ be a smooth proper curve over $\Qbar$ with simple Jacobian
and $\omega\in\Omega^1_{\Qbar(C)}$ a non-zero rational differential form on
$C$. Let $\gamma$ be a non-closed path on $C^\an$ avoiding the poles of $\omega$
and with endpoints in  $C(\Qbar)$. Then 
\[ \int_\gamma\omega\]
is transcendental, unless one of the following conditions is satisfied:
\begin{enumerate}
\item $\omega=df$ is exact;
\item $\omega=df+\phi$ with $\phi$ of the third kind such that Condition 
(\ref{it:van_arb4}) of Theorem~\ref{thm:van_arb} is satisfied for $\phi$ or, equivalently, for $\omega$.
\end{enumerate}
\end{cor}
\begin{proof}Assume that the period is algebraic. By Theorem~\ref{thm:transc_curves} 
the form $\omega$ can be written as 
$\omega=df+\phi$ with
\[ \int_\gamma\phi=0.\]
We apply what we have learned about the vanishing of simple periods to $\phi$.
As $\gamma$ is not closed, its homology class does not vanish. If $[\phi]=0$, then $\phi$ is exact. This makes $\omega$ exact.

It remains to go through the cases of non-trivial vanishing 
in Theorem~\ref{thm:van_arb}. The cases (\ref{it:1arb}) and (\ref{it:2arb}) are excluded because $\gamma$ is not closed. The case (\ref{it:case_exact}) gives back exact $\phi$. 

If condition ((\ref{it:van_arb4}) of Theorem~\ref{thm:van_arb} is satisfied for
$\phi$, then $\phi=\phi_2+\phi_3$ with $\phi_2$, $\phi_3$ or the second and third kind, respectively, and 
\[ \int_\gamma\phi_2=\int_\gamma\phi_3=0.\] 
By Summary~\ref{sum:ex} and Summary~\ref{sum:2} this can only happen
if either $\phi_2$ is exact or $l(\gamma)=0$. The latter is excluded because
$\gamma$ is not closed. The vanishing for $\phi_3$ is characterised 
in Summary~\ref{sum:3}. This is the same condition as in Theorem~\ref{thm:van_arb}.
\end{proof}

\part{Dimensions of period spaces}\label{part4}


\chapter{Dimension Computations: an Estimate}\label{ch:dim_comp_part1}

The aim   of this chapter is
to establish a formula for the dimension of
of $\Per\langle M\rangle$ for any  $1$-motive $M$.
We have already given a qualitative characterisation in Corollary~\ref{cor:dim_easy}. This characterisation  is now turned into an explicit quantitative version. Explicit formulas are deduced in terms of the constituents of the motive $M$. This will first happen under certain simplifying assumptions. The general case will be considered in the subsequent chapters.

\section{Set-up and Terminology}\label{sec:data}

Throughout this chapter let  $M=[L\to G]$ be in $\onemot_\Qbar$ and let $0\to T\to G\to A\to 0$ be its decomposition
into torus and abelian part. In the arguments below we reserve the letter $B$ for simple
abelian varieties and define $E(B)=\End(B)_\Q$. This is a division algebra of dimension $e(B)=\dim_{\Q}E(B)$.
Let
\[ A\isom B_1^{n_1}\times\dots\dots \times B_m^{n_m}\] 
be the isotypical decomposition of $A$ in the category of abelian varieties up to isogeny.
As usual, we write $X(T)$ for the character lattice of $T$.
If $X$ is a lattice, we denote by $T(X)$  the corresponding torus.
In Section~\ref{sec:sa}  we have shown that the datum of a semi-abelian variety $G$ is equivalent to the datum of a homomorphism $X(T)\to A^\vee(\Qbar)_\Q$.

To establish a formulation of a dimension formula for the period space becomes rather complicated because there is a
non-trivial interplay between the action of $\Hom(A,B)$ on $L$ and the action of $\Hom(A^\vee,B^\vee)$ on the character group of $T$. The problems disappear in special cases listed below.

\begin{defn}\label{defn:red-sat}
We say that the motive $M$ is
\begin{enumerate}
\item
 of \emph{Baker type}\index{Baker type}\index{$1$-motive!of Baker type} if $A=0$; 
\item  of \emph{semi-abelian type} \index{semi-abelian type}\index{$1$-motive!of Baker type}if $L=0$;
\item  of \emph{second kind} \index{$1$-motive!second kind}if $T=0$;
\item 
\emph{reduced}\index{reduced $1$-motive}\index{$1$-motive!reduced} if $L\to A(\Qbar)_\Q$ and $X(T)\to A^\vee(\Qbar)_\Q$ are injective;
\item  \emph{saturated} \index{saturated $1$-motive}\index{$1$-motive!saturated}if it is reduced, 
 and
$\End(M)_\Q=\End(A)_\Q$.
\end{enumerate}
\end{defn}

To state our dimension formulas we need some notation. 
\begin{notation}\label{not:data_all}
\begin{enumerate}
\item Put $\delta(M)=\dim_\Qbar\Per\langle M\rangle$.
\item \label{defn:rk_L}We define the \emph{$L$-rank \index{rank of a motive} $\rk_B(L,M)$ of $M$ with respect to $B$} as the $E(B)$-dimension of
the vector space 
\[ \Hom(A,B)\cdot L:=\sum p(L)\subset B(\Qbar)_\Q\]
for $p\in\Hom(A,B)_\Q$, where $p(L)$ denotes the image of $L$ under the composition of the maps $L\to A(\Qbar)$ and $p$.
\item \label{defn:rk_G}The endomorphism algebra $E(B)$  acts on $B^\vee$ from the right as $(b^\vee,e)\mapsto e^\vee (b^\vee)$ where $e^\vee$  signifies  the isogeny dual to $e$. The \emph{$T$-rank $\rk_B(T,M)$ of $M$ with respect to $B$} is the $E(B)$-dimension of
the right-$E(B)$-vector space  
\[
\Hom(A^\vee,B^\vee)_\Q \cdot X(T):=\sum p(X(T))\subset B^\vee(\Qbar)_\Q
\]
for $p \in \Hom(A^\vee,B^\vee)_\Q$, where $p(X(T))$ denotes
the image of $X(T)$ under the composition of $X(T)\to A^\vee(\Qbar)_\Q$ with $p$.
\item If $[L\to T]$ is a $1$-motive of Baker type, we define the \emph{$L$-rank $\rk_{\Gm}(L,M)$ of $M$ in $\Gm$}  as the rank of
\[ \Hom(T,\Gm)\cdot L:=\sum_{\chi\in X(T)} \chi(L)\]
where $\chi(L)$ denotes the image of $L$ under the composition of
$L\to T$ and $\chi\in X(T)$.
\end{enumerate}
\end{notation}
The main theorem of the present chapter is an estimate from above of the dimension $\delta(M)$.
In many important cases the inequality  is even an equality.                 

\begin{thm}[Dimension estimate]\label{thm:dimension_sat}
\index{dimension estimate}\index{dimension formula}
Let $M$ be a 
$1$-motive with constituents as above.
\begin{enumerate}
\item\label{it:sat_case} If $M$ is the product of a motive of Baker type \index{$1$-motive! of Baker type} $M_0$ and a saturated motive \index{$1$-motive!saturated}$M_1$, then 
\begin{align*} \delta(M)=
\delta(T)&+\sum_B\frac{4g(B)^2}{e(B)} + \delta(L)\\
&+\sum_B \left(2g(B)\rk_{B}(T,M)
+2g(B)\rk_{B}(L,M)\right)\\ 
&+\rk_{\Gm}(L_0,M_0)+\sum_B e(B)g(B)\rk_B(L,M)\rk_B(L,M),
\end{align*}
where all sums are taken over all simple factors of $A$,  without multiplicities and $g(B)=\dim B$, $e(B)=\dim_\Q \End(B)_\Q$.
\item\label{it:contained}For every $1$-motive $M$, there is  a product $\tilde{M}$ of a motive of Baker type and a saturated motive such that $\Per\langle M\rangle \subset \Per\langle \tilde{M}\rangle$. In particular,
\[ \delta(M)\leq \delta(\tilde{M}).\]
The construction is effective (see Lemma~\ref{lem:bk_red} and Lemma~\ref{lem:sat_ok}) in a way that the abelian parts of $M$ and $\tilde{M}$ agree and that
$\rk_B(L,M)=\rk_B(\tilde{L},\tilde{M})$, $\rk_B(T,M)=\rk_B(\tilde{T},\tilde{M})$.

\item\label{it:trivial}
We have $\delta(L)=1$ if $L\neq 0$ and $\delta(L)=0$ if $L=0$, and
 $\delta(T)=1$ if $T\neq 0$ and $\delta(T)=0$ if $T=0$.
\end{enumerate}
\end{thm}

The proof will be given at the end of the chapter.
The following example illustrates the discrepancy between the upper bound and the actual dimension.

\begin{ex}\label{ex:zahlen}
\index{elliptic periods}
We take for $A$ an elliptic curve $E$, a non-trivial
extension $0\to \Gm\to G\to E\to 0$ which is non-split up to isogeny
 and $P\in G(\Qbar)$ a point whose image in $E(\Qbar)$ is not torsion.
We consider the $1$-motive 
\[ M=[\Z \to G]\]
with $1$ mapping to $P$. This is the same motive already considered in Chapter~\ref{ch:ex}.

\begin{description}
\item[Case 1.]
Assume that $E$ does not have complex multiplication, i.e. $\End(E)=\Z$. In this case  $B=E$, $g=1$, $e=e(E)=1$, $\delta(T)=1$, $\delta(L)=1$. Moreover, $\rk_E(T,M)=1$ and $\rk_E(L,M)=1$. The motive is saturated and the theorem predicts
\[ \delta(M)= 11.\]
This has already been verified directly in  Proposition~\ref{prop:ell}.

\item[Case 2.] In the case that $E$ has CM by an imaginary quadratic extension $\Q(\tau)$ of $\Q$, i.e. $\End(E)_\Q=\Q(\tau)$, there is again a single $B=E$.
In this situation, we have $m=1$, $g=1$, $e=e(E)=2$, $\delta(T)=1$, $\delta(L)=1$,  
 $\rk_{E}(L,M)=1$ and $\rk_{E}(T,M)=1$. Going through the construction of $\tilde{M}$ before Lemma~\ref{lem:sat_ok}, we see that we choose the Baker part $M_0$ as $0$ and $\tilde{M}=M_1$ as saturated.  The theorem
gives 
 \[ \delta(M)\leq 10.\] 
By Proposition~\ref{prop:with_CM}, we actually have 
\[ \delta(M)=9.\]
\end{description}
\end{ex}

\begin{rem}
The example shows that we do not have equality in general. 
We will refine the formula in Theorem~\ref{thm:precise1} and completely explicitly in Theorem~\ref{thm:mix_final}.
\end{rem}

\subsection{Outline of the Proof}

In a first and key step in Section~\ref{sec:sat}, we shall prove the dimension formula in the saturated case. In order to simplify notation, we shall first handle in Section~\ref{ssec:s_and_s} the case where the abelian part of $M$ is simple, then upgrade to general $A$. 

Section~\ref{sec:sp_cases} is an interlude: we go through the easier cases of Baker motives, motives of semi-abelian type, motives of the second kind and establish the dimension formulas. None of these need to be saturated, but the formulas can be reduced to the saturated case by constructing a saturation with the same periods and ranks. 

Finally, in Section~\ref{sec:proof_estimate}, we wrap up and establish the dimension estimate announced in Theorem~\ref{thm:dimension_sat}. The case when $M$ is a product of a Baker motive and a saturated motive follows easily with the same type of arguments as in Section~\ref{sec:sat} in the saturated  and in Section~\ref{sec:sp_cases} in the Baker case.
Starting with a general motive $M$, we construct a product of a Baker motive
$M_0$ and a saturated motive $M_1$ such that the space  of periods of $M$ are contained in the space of periods of $M_0\times M_1$. Here we use the ingredients that we identified already in the special case.

\section{The Saturated Case}\label{sec:sat}

We start under restrictive assumptions.
The saturated case is easier, but enough to deduce all structural properties of the period space.

\subsection{Saturated and Simple}\label{ssec:s_and_s}
\index{periods!of saturated $1$-motives}\index{$1$-motive!saturated}\index{saturated $1$-motive}
Assume that $M$ is reduced and saturated (see Definition~\ref{defn:red-sat}) with $A=B$ simple of dimension $g$. We denote by $e$ the $\Q$-dimension of $E=\End(B)_\Q$. Then $\Vsing{M}$ is an $E$-module and we choose an $E$-basis. To be precise,  the inclusions
\[ [0\to T]\subset [0\to G]\to [L\to G]\]
induce injections
\[ \Vsing{T}\hookrightarrow \Vsing{G}\hookrightarrow \Vsing{M},\]
read as inclusions from now on.
Let $\sigma=(\sigma_1,\dots,\sigma_r)$
 be an $E$-basis of $\Vsing{T}$. We write 
$r=|\sigma|$ and use the 
same conventions for all other pieces.
Extend $\sigma$ by $\gamma$ to an $E$-basis of
$\Vsing{G}$ and by $\lambda$ to an $E$-basis of $\Vsing{M}$. 

In the dual setting, we choose $\Qbar$-bases (sic!) of $\VdR{M}$ 
along the inclusions
\[ \VdR{L}\hookrightarrow\VdR{[L\to B]}\hookrightarrow\VdR{M}.\]
To be precise let $u$ be a basis of $\VdR{[L\to 0]}$. We extend $u$ by $\omega$ to a basis of $\VdR{[L\to A]}$ and the resulting basis by $\xi$ to a basis of $\VdR{M}$.
Note that $\VdR{M}$ is also a right $E$-module, but we are not using this structure at this point.

Our discussion shows that the full period matrix \index{period matrix!saturated case}of $M$ with respect to these bases
has the shape
\begin{equation}\label{eq:matrix}
\left(\begin{matrix}\xi(\sigma)&\xi(\gamma)&\xi(\lambda))\\
 \omega(\sigma)&\omega(\gamma)&\omega(\lambda)\\
u(\sigma)&u(\gamma)&u(\lambda)
\end{matrix}\right)
=\left(\begin{matrix}\xi(\sigma)&\xi(\gamma)&\xi(\lambda)\\
 0&\omega(\gamma)&\omega(\lambda)\\
0&0&u(\lambda)
\end{matrix}\right).
\end{equation}

\begin{lemma}\label{lem:compute}
The entries of the matrix above generate $\Per\langle M\rangle$ over $\Qbar$.
The elements of $\xi(\sigma)$ are $\Qbar$-multiples of
$2\pi i$. The elements of $u(\lambda)$ are $\Qbar$-multiples of $1$.
These are the only $\Qbar$-linear relations between the entries.
\end{lemma}

\begin{proof}Let $\alpha_1,\dots,\alpha_e$ be a $\Qbar$-basis
of $E$. Then the tuples 
\[ (\alpha_j\sigma,\alpha_j\gamma,\alpha_j\lambda|j=1,\dots,e)\]
 are
a $\Q$-basis of $\Vsing{M}$, hence the period space is generated by the complex numbers
$\xi_i(\alpha_{j*}\sigma_k), \dots$ over $\Qbar$.  The transformation formula for integrals gives
\[ \xi_i(\alpha_{j*}\sigma_k)=(\alpha^*_j\xi_i)(\sigma_k).\]
The class $\alpha^*_j\xi_i$ is a $\Qbar$-linear combination
of the basis vectors:
\[
 \alpha^*_j\xi_i=\sum b_s\xi_s+\sum c_t\omega_t+d_ru_r
\]
which implies in accordance with the shape of the period matrix that
\[\xi_i(\alpha_{j*}\sigma_k)=\sum b_s\xi_s(\sigma_k).
\]
Similar computations also apply to rest of the basis and we see that
the period space has $\Qbar$-generators as claimed.

Now assume that there is a $\Qbar$-linear relation between the periods. It
has the shape
\begin{equation}\label{eq:relation}
 \sum a\,\xi(\sigma)+\sum b\,\xi(\gamma)
+ \sum c\,\xi(\lambda)+\sum d\,\omega(\gamma)+\sum f\,\omega(\lambda) +\sum gu(\lambda)=0.
\end{equation}
Here $a$ is a matrix  and the first term stands for the sum $\sum_{i,j}a_{ij}\xi_i(\sigma_j)$. The same convention is used in all other places.

We define
 $T_\sigma= \Gm^{|\sigma|}$,  $M_\sigma=[0\to T_\sigma]$, $G_\gamma=G^{|\gamma|}$, $M_\gamma= [0\to G_\gamma], M_\lambda= M^{|\lambda|}$ and  consider  the product 
\[ \tilde{M}=M_\sigma\times M_\gamma\times M_\lambda.\] 
It can be written in the form $\tilde{M}=[\tilde{L}\to \tilde{G}]$  with group part 
\[0\to \tilde{T} \to \tilde{G} \to \tilde{A}\to 0\] 
for 
 a torus $\tilde{T}$  and an abelian variety $\tilde{A}$. We fix elements
\begin{align*} 
\tilde{\gamma}&=(\sigma,\gamma,\lambda)\in\Vsing{\tilde{M}},\\
\tilde{\omega}&=(\phi,\psi,\vartheta)\in\VdR{\tilde{M}}
\end{align*}
where
\begin{align*}
\phi&=\left(\sum a_{ij}\xi_j|i=1,\dots,|\sigma|\right),\\
\psi&=\left(\sum b_{ij}\xi_j+\sum d_{ij}\omega_j| i=1,\dots,|\gamma|\right),\\
\vartheta&= \left(\sum c_{ij}\xi_j+\sum f_{ij}\omega_j +\sum g_{ij}u_j|i=1,\dots,|l|\right)
\end{align*}
using the coefficients of the relation~(\ref{eq:relation}).
Then the $\Qbar$-linear relation~(\ref{eq:relation}) between the periods implies 
\[ \tilde{\omega}(\tilde{\gamma})=0.\]
 By the Subgroup Theorem for $1$-motives, there is a short exact sequence
\[ 0\to M_1\xrightarrow{\nu}\tilde{M}\xrightarrow{p}M_2\to 0\]
with $M_1=[L_1\to G_1]$, $M_2=[L_2\to G_2]$ and $\tilde{\gamma}=\nu_*\gamma_1$, $\tilde{\omega}=p^*\omega_2$.

We analyse $M_2$. In a first step we show that $A_2=0$. Otherwise assume $A_2\neq 0$. Then there is a non-zero homomorphism $A_2\to B$. 
The composition $\tilde{A}= B^{|\gamma|}\times B^{|\lambda|}\to A_2\to B$ is given by a tuple
$(n,m)$ of elements of $E$. We determine the image $\bar{L}_2$ of
$L_2$ in $B$. We have $\bar{L}_2=\sum_{i=1}^{|\lambda|} m_i(L)$
 because the lattice part of $M_\gamma$ is $0$. The image is contained in $L$ because  $L$ is $E$-stable by saturatedness and $L\to G\to B$ injective. Here we use that $M$, being saturated, is reduced.  We have
$\bar{L}_2\neq 0$ if there is some $m_i\neq 0$. In this case $m_i$ is a unit in a division algebra, making the map $m_i:L\to L$ bijective. This implies
$m_i(L)=L$ and then also $\bar{L}_2=L$. The image of  $\tilde{\gamma}$ under
$\bar{p}:\tilde{M}\to M_2\to [L\to B]$ is equal to
\[ \sum n_j\gamma_j+\sum m_i \lambda_i \in \Vsing{[L\to B]}.\]
It vanishes because $p_*\tilde{\gamma}=0$. The image of $(\gamma,\lambda)$ is
an $E$-basis of $\Vsing{[L\to B]}$. Therefore $m_i=0$ for all $i$ and $n_j=0$ for all $j$. This contradicts
the non-triviality of $\tilde{A}\to A_2\to B$. Hence $A_2=0$.

Consider the composition of the inclusion of one of the factors
$[0\to G]$ of $M_\gamma$ into $\tilde{M}$ with $p$. It has the shape \[ [0\to G]\to [L_2\to T_2].\]
If this map was non-zero, its group component would induce a splitting of $G$. This is not possible because we have assumed
that $M$ is reduced. The map has to vanish. As a consequence the pull-back of
$\tilde{\omega}=p^*\omega_2$ to $G$ is equal to 
\[ \tilde{\omega}|_G=(p^*\omega_2)|_G=0^*\omega_2=0.\]
This means $\tilde{\omega}|_{M_\gamma}=0$. In other words, $\psi=0$.
By linear independence of the $\xi_j$ and
$\omega_j$ this implies $b=0$, $d=0$.

We repeat the argument with the inclusion of one of the factors
$M$ of $M_\lambda$ into $\tilde{M}$. It has the shape
\[ [L\to G]\to [L_2\to T_2]\]  
As $M$ is reduced, there is no non-trivial map from $G$ to a torus, so  the group component of this map vanishes. The image of $L$ in $T_2$ has to be $0$  and therefore its image in $L_2$ is in $K_2=\ker(L_2\to T_2)$. This is true
for all factors $M$ of $M_\lambda$, hence the surjection $L^{|\lambda|}\to L_2$ factors via $K_2\subseteq L_2$. 
This implies that  $L_2=K_2$ and the structure map $L_2\to T_2$ of $M_2$ vanishes. The composition
\[ M_\lambda\to \tilde{M}\xrightarrow{p}M_2\isom [L_2\to 0]\times [0\to T_2]\]
has image in $[L_2\to 0]$,
hence the restriction $\vartheta$ of $\tilde{\omega}=p^*\omega_2$ to $\VdR{M_{\lambda}}$ takes values
in $\VdR{[L^{|\lambda|}\to 0]}$.  By linear independence this implies $c=0$,
$f=0$.

This shows that our linear relation has been reduced to
\[ \sum a\xi(\sigma)+\sum gu(\lambda)=0.\]
The terms in the first sum are periods of Tate type, hence  multiples
of $2\pi i$. The terms in the second sum are periods of algebraic type, hence
multiples of $1$. This reduces the proof to the transcendence of $\pi$.
This was shown in  Corollary~\ref{cor:pi}.
\end{proof}

\begin{cor}\label{prop:simple}
\index{dimension formula!saturated, simple case}
Let $M$ be saturated with $A=B$ simple. Then
\begin{multline*} \delta(M)=
\delta(T)+\frac{4g^2}{e} + \delta(L)\\
+2g\,\rk_{B}(G,M)
+2g\,\rk_{B}(L,M)+e\,\rk_B(G,M)\rk_B(L,M).
\end{multline*}
\end{cor}
In particular, the formula in Theorem~\ref{thm:dimension_sat} holds.
\begin{proof}
We read off the numbers from the basis constructed in the proof of the previous lemma.
Accordingly we have 
\begin{gather*}
|\sigma|= \rk_B(G,M),  |\gamma|=2g/e, |\lambda|=\rk_B(L,M),\\
|\xi|=e\,\rk_B(G,M), |\omega|=2g, |u|=e\,\rk_B(L,M).
\end{gather*}
\end{proof}

\subsection{General Saturated Motives}
Assume $M$ is saturated, but $A$ not necessarily simple. 
We have (up to isogeny)
\[ A\isom B_1^{n_1}\times\dots\times B_m^{n_m}\]
with simple, non-isogenous $B_i$. Hence
\[ \End(A)_\Q\isom M_{n_1}(E_1)\times \dots\times M_{n_m}(E_m)\]
with non-isomorphic division algebras $E_i$. 

\begin{lemma}\label{lem:split} 
There is a natural decomposition
\[ M\isom M_1^{n_1}\times\dots \times M_m^{n_m}\]
with each $M_i$ saturated with abelian part given by $B_i$.
\end{lemma}
\begin{proof}
Let $p_i\in\End(A)$ be the projector onto $B_i^{n_i}$. 
This gives a decomposition of the identity
$1=\sum p_i$ into idempotents.

Since $M$ is saturated, we have $\End(A)_\Q=\End(M)$, which implies that $p_i$ can also be viewed as a projector on $M$. We obtain a decomposition 
\begin{equation}\label{isotypical} M\isom \bigoplus_{i=1}^mp_i(M).\end{equation}
into isotypical components.

We now replace $M$ by one of the factors $p_i(M)$ and drop the index $i$. The
abelian part is now $B^{n}$. We write $E=\End(B)_\Q$. Let $q_1,\dots,q_{n}\in M_n(E)$
be the projections to the components. This gives a decomposition of the identity 
$1=q_1+\cdots+q_n$ into idempotents $q_i$. As $\End(M)=M_n(E)$ this
induces a decomposition
\[ M\isom \bigoplus_{i=1}^n q_i(M).\] 
The permutation matrices in $M_n(E)$ induce isomorphisms between
the factors $q_i(M)$, hence we even have
\[ M\isom  (M')^n\]
with abelian part of $M'$ given by $B$.
\end{proof}

\begin{prop}\label{prop:sat}
\index{dimension formula!saturated case}\index{$1$-motive!saturated}\index{saturated $1$-motive}
Consider a saturated $1$-motive $M$.
Then the formula in Theorem~\ref{thm:dimension_sat} holds.
\end{prop}
\begin{proof}
By Lemma~\ref{lem:split} we are dealing with the motive
\[ M_1^{n_1}\times\dots \times M_m^{n_m}\]
with $M_i$ as there. We have
\[ \Per(M_i^{n_i})=\Per(M_i)\] 
hence we may without loss of generality assume $n_i=1$ for $i\geq 1$. We now
repeat the proof of Lemma~\ref{lem:compute} and Corollary~\ref{prop:simple} with an
extra  index $i$.
\end{proof}

\section{Special Cases}\label{sec:sp_cases}

The formulas which we have derived so far simplify if one of the constituents of $M$ vanishes. Indeed, what we have shown so far is enough to give not only estimates but complete formulas. It is worth spelling this out explicitly in the different cases.

\subsection{Motives of Baker Type}
The simplest case is the Baker Motive $M=[L\to T]$. Here we get back Baker's famous theorem on linear forms in logarithms in its qualitative version.

\begin{prop}[Baker's Theorem]\label{prop:baker}
\index{dimension formula!Baker type}\index{$1$-motive!of Baker type}\index{Baker type}
Let $M=[L\to T]$ be of Baker type. Then
\[ \delta(M)=\delta(T)+\delta(L)+\rk_{\Gm}(L,M).\]
\end{prop}
\begin{proof} Let $K$ be the kernel of $L\to T$ and $\bar{L}$ the image (up to torsion). 
Consider the motivic decomposition $[L\to T]=[K\to 0]\oplus [\bar{L}\to T]$ with $\bar{L}\to T$
injective. The periods of $[\bar{L}\to T]$ agree
with the periods of
\[ M'=[\sum_\chi \chi(L)/\mathrm{Torsion}\to \Gm]\]
 with $\chi$ running through $\Hom(T,\Gm)$ as in the definition of $\rk_{\Gm}(L,M)$. The  structure map of $M'$ is injective. The torus
part has rank $\rk_{\Gm}(L,T)$. We now choose bases and proceed as in the proof  of Lemma~\ref{lem:compute}  and Corollary~\ref{prop:simple}, only simpler. 
\end{proof} 

\begin{rem}
This is precisely Baker's theorem, see \cite{baker}, see also
\cite[Theorem~2.3]{baker-wuestholz}.
It can also be deduced directly from the Analytic Subgroup Theorem in its original form applied to a group of the form $V\times T$ for a vector group of dimension equal to the rank of $L$. We can even take $\Ga\times T$. This the line of proof used in  \cite{baker-wuestholz}.
\end{rem}

\subsection{The Semi-abelian Case}
In this section, let $M=[0\to G]$ be of semi-abelian type. The dimension computation
will be achieved by reduction to the saturated case. For later use, we  record
the construction of this saturation.
As always, $G$ is an extension of an abelian variety $A$ by a torus $T$. It is determined by 
\[ X(T)\to A^\vee(\Qbar)_\Q=\Ext^1(A,\Gm)_\Q,\]
see Corollary~\ref{cor:sa_ext}.

\begin{defn}\label{defn:Gsat}
Let $G$ be a semi-abelian variety, $X=X(T)$ the character group and $E=\End(A)_\Q$. We denote by
$X_\red$ the image of $X$ in $A^\vee(\Qbar)_\Q$ and by 
$G_\red$ the semi-abelian variety defined by $X_\red\to A^\vee(\Qbar)_\Q$.
Define 
\[ X_\sat=X_\red E\subset A^\vee(\Qbar)_\Q\]
and $G_\sat$ as  the semi-abelian variety
given by $X_\sat\subset \A^\vee(\Qbar)_\Q$. \end{defn} 
The projection $X\to X_\red$ induces a natural injection
\[ G_\red\to G \] 
and the inclusion
$X_\red\subset X_\red E$  corresponds to a projection 
\[ G_\sat\to G_\red.\]
By construction $\End(G_\sat)=E$.
Our first Lemma relates the spaces $\Per(G)$ and $\Per(G_\sat)$.
\index{saturation}

\begin{lemma}\label{lem:sat_easy}
Let $G$ be semi-abelian. Then
\[ \Per\langle G\rangle=\Per\langle G_\sat\rangle+\Per\langle T\rangle.\]
\end{lemma}
\begin{proof}
Let $X'=\Ker(X\to A^\vee(\Qbar)_\Q)$ and $T'$ the torus corresponding to $X'$. Up to isogeny this gives a decomposition
\[ G\isom T'\times G_\red\]
and deduce that 
\[ \Per\langle G\rangle =\Per\langle T'\rangle +\Per\langle G_\red\rangle.\]
The first space $\Per\langle T'\rangle$ is included in $\Per\langle G_\red\rangle$ unless $X_\red=0$. In this exceptional
case, $G_\red=G_\sat=A$ and the statement is true. In consequence it suffices to 
work in the case $G=G_\red$ and we claim that
\[ \Per\langle G\rangle=\Per \langle G_\sat\rangle.\]
The surjection $G_\sat\to G$ induces an inclusion $\Per\langle G\rangle\subset \Per\langle G_\sat\rangle$ of the period spaces. 
We now establish the converse inclusion. 

Let  $e$ be a $\Q$-basis of  $E=\End(A)$. (As in Section~\ref{ssec:s_and_s} we write $e$ for the array $e_1,\dots,e_d$ etc.) 

Let $\sigma$ be a $\Q$-basis of $\Vsing{T}$ and extend the basis by $\gamma$ to a $\Q$-basis of $\Vsing{G}$.
We write $e_*\sigma$  for the set $e_{i*}\sigma_j$ for $e_i\in e$, $l_j\in l$.
 The tuple $(e_*\sigma,\gamma)$ is a system of generators of $\Vsing{G_\sat}$. 

Let 
$\omega$ be $\Qbar$-basis of $\VdR{A}$ and extend the basis  by $\xi$ to a $\Qbar$-basis of $\VdR{G}$. We denote by  $e^*\xi$  the array
 $e^*_i\xi_j$ for $e_i\in e$, $u_j\in u$. Then 
the array  $(\omega,e^*\xi)$ is a system of generators of $\VdR{G_\sat}$.

All periods of the form $e^*\xi(e_*\sigma)$ 
are multiples of $2\pi i$, hence contained in
$\xi(\sigma)$. The same holds trivially for the periods of the form $\omega(\gamma)$. Consider a period
\[ e_i^*\xi_j(\gamma_k)=\xi_j(e_{i*}\gamma_k)\]
The element $e_{i*}\gamma_k\in \Vsing{M_\sat}$ is a $\Qbar$-linear combination of 
the generators $(e_*\sigma,\gamma)$, hence the period $\xi_j(e_{i*}\gamma_k)$
is a linear combination of periods of the
form  $\xi(e_*\sigma)$ and $\xi(\gamma)$, both
contained in $\Per\langle G\rangle$.
\end{proof}
We are now ready for our dimension computation.

\begin{prop}\label{prop:dim_sa}
\index{dimension formula!semi-abelian type}
\index{$1$-motive!of semi-abelian type}
Let $G$ be semi-abelian, an extension of the abelian variety $A$ by a torus $T$.
Then
\begin{align*} \delta(M)=
\delta(T)&+\sum_B\, \frac{1}{e(B)}\,4g(B)^2+\sum_B \,2g(B)\,\rk_{B}(T,M),
\end{align*}
where the sum is taken over over the simple factors of $A$, without multiplicities.
\end{prop}
\begin{proof}First we consider the exceptional case $G\isom T\times A$, up to isogeny. 
We have 
\[ \delta(A)=\sum_B\, \frac{1}{e(B)}\,4g(B)^2 \]
as a very special case of Proposition~\ref{prop:sat}. In the product case the claim is
the linear independence of the spaces $\Per\langle T\rangle$ and $\Per\langle A\rangle$. This follows from the dimension formula in the case of a non-trivial extension (or its proof), so we may ignore the exception. 

In the non-split case, $\Per\langle T\rangle\subset \Per\langle G_\sat\rangle$ and  Lemma~\ref{lem:sat_easy} gives
\[ \delta(G)=\delta(G_\sat).\]
In the saturated case, the formula is as special instance of Proposition~\ref{prop:sat}.
\end{proof}

\begin{rem}
The statement does not mention $1$-motives, and indeed, the formula can be deduced directly from the Analytic Subgroup Theorem in its original form.
\end{rem}
\subsection{Motives of the Second Kind}
Assume that $M=[L\to A]$ is of second kind. 
We argue  as in  the semi-abelian case.

\begin{defn}\label{defn:Msat2}
Let $M=[L\to A]$ be of second kind, and $E=\End(A)_\Q$.  
We denote by  $L_\red$ be the image of $L$ in $A(\Qbar)_\Q$ and put 
$M_\red=[L_\red\to A]$. 
Define 
\[ L_\sat=EL_\red\subset A(\Qbar)_\Q\]
and write $M_\sat$ for the   motive $[L_\sat\to A]$.\index{saturation}
\end{defn} 
The projection $L\to L_\red$ induces a natural projection
\[ M\to M_\red\] 
and the inclusion
$X_\red\subset X_\red E$ corresponds to an inclusion 
\[ M_\sat\to M_\red.\]
    By construction $\End(M_\sat)$ is $E$.

\begin{lemma}\label{lem:sat_easy2}
Let $M=[L\to A]$ be of the second kind. Then
\[ \Per\langle M\rangle=\Per\langle M_\sat\rangle+\Per\langle L\rangle.\]
\end{lemma}
\begin{proof}
Mutis mutandis, the argument is the same as in in the proof of the semi-abelian case, 
for Lemma~\ref{lem:sat_easy}. 
\end{proof}

\begin{prop}\label{prop:dim_2}
\index{dimension formula!second kind}\index{$1$-motive!of second kind}
Let $[L\to A]$ be a $1$-motive of the second kind.Then
\begin{align*} \delta(M)=
\delta(L)&+
\sum_B\, \frac{1}{e(B)}\,4g(B)^2+\sum_B \,2g(B)\,\rk_{B}(L,M) 
\end{align*}
where the sum is taken over the simple factors of $A$, without multiplicities.
\end{prop}
\begin{proof}Mutis mutandis, the argument is the same as in the semi-abelian case, Proposition~\ref{prop:dim_sa}
\end{proof}

\begin{rem}
In contrast to the semi-abelian case,  the argument here uses the language of $1$-motives. This can be avoided by considering the vector group $M^\natural$ directly. This is not surprising: after all the Analytic Subgroup Theorem for $1$-motives is a consequence of the Analytic Subgroup Theorem in its original form. 

There is an alternative argument for the proof: the Cartier dual of
$M=[L\to A]$ has the shape $[0\to G^\vee]$ where $G^\vee$ is the semi-abelian variety with abelian part $A^\vee$ and defined by the classifying homomorphism $L\to (A^\vee)^\vee(\Qbar)_\Q$. The period spaces of $M$ and $G^\vee$ have the same dimension. Moreover,
$\rk_B(L,M)=\rk_{B^\vee}(T^\vee,G^\vee)$ so that  the formulas for $M$ and
$G^\vee$ match.
\end{rem}


\section{Proof of the Dimension Estimate}\label{sec:proof_estimate}

Theorem~\ref{thm:dimension_sat} states a formula for $\delta(M)$ in the case of a product of a motive of Baker type and a saturated motive. We have already handled each factor separately. Looking a little more carefully at the proof will also give the full result.

\begin{prop}\label{prop:formula_sat_holds}
\index{dimension formula!product of Baker motive and saturated motive}
Let $M$ be a product of a Baker motive and a saturated $1$-motive.
Then the dimension formula of Theorem~\ref{thm:dimension_sat} holds.
\end{prop}
\begin{proof}We have already established the case of saturated motives in Proposition~\ref{prop:sat} and of motives of Baker type, see Proposition~\ref{prop:baker}. With the same argument
as in the proof of Proposition~\ref{prop:sat}, we may assume that the saturated motive is of the form
\[ M_1\times\dots \times M_m\]
where the $M_i$ have simple abelian part $B_i$ and the $B_i$ are pairwise non-isogenous. The periods remain unchanged. As in the proof of Baker's theorem, we further may assume that the motive of Baker type is of the form
\[ M_0=[L\to\Gm]\]
The periods and the ranks in the formula 
remain unchanged by this process.

Going through the argument  in the proof of Proposition~\ref{prop:sat}, but with 
$i=0,1,\dots,m$ instead of $i=1,\dots,m$ proves the Proposition. 
\end{proof}

Since the reduction of the proof of Theorem~\ref{thm:dimension_sat} to the case of Proposition~\ref{prop:formula_sat_holds} is lengthy, we offer a short outline of the proof.  

Given a $1$-motive $M$ we want to find a  $1$-motive $\tilde{M}$ with the same abelian part $A$ such that $E=\End(A)$ operates on all of $\tilde{M}$. The idea is to enlarge both $T$ and $L$ by adding their $E$-translates. In the simplest case the motive is $u:\Z\to A$, which is  made $E$-equivariant by replacing $\Z$ by $\End(A)u$.
We already used this device in Definition~\ref{defn:Msat2}.
However, this needs first to make $M$ reduced before the procedure can be applied. 
If this has been achieved, so that we are in the reduced case, we enlarge $G$ to $G_\sat$ such that $E$ operates on $G_\sat$. To this end,
we have to choose a lift $L\to G$ to $L\to G_\sat$, then enlarge $L$ such that
$E$ operates. In some of the steps the periods remain the same, in other steps, the space of periods is enlarged.
In each step the periods can be controlled. 

To begin with, the first step is to construct in a canonical way a $1$-motive $M_\baker$ of 
Baker type and non-ca\-no\-ni\-cal\-ly a reduced $M_\red$ such that
\begin{equation}\label{eq:reduce} \Per\langle M\rangle =\Per\langle M_\red\rangle+\Per\langle M_\baker\rangle.
\end{equation}

This is done as follows:
The composition $L\to G(\Qbar)_\Q\to A(\Qbar)_\Q$ has a kernel $L'$ and an image $L''$.  We have $[L''\to A]=[L\to A]_\red$ with the notation of Definition~\ref{defn:Msat2}. 
As $L'\to G$ factors via $T$, this defines a motive of Baker type $[L'\to T]$. 

Similarly $X(T)\to A^\vee(\Qbar)$ has a kernel $X(T'')\subset X(T)$
(for some quotient $T\to T''$) and an image $X(T')$ (for some subtorus $T'\subset T$). As in Remark~\ref{rem:decomp_G}
it induces a canonical short exact sequence
\[ 0\to G'\to G\to T''\to 0\]
of semi-abelian groups. The group $G'$ coincides with $G_\red$ in accordance with the notation of Definition~\ref{defn:Gsat}.

With these data we define a Baker type motive by
\[  M_\baker=[L'\to T]\oplus [L\to T''].\]
We now
choose splittings $L\isom L'\times L''$ and 
$X(T)\isom X(T')\times X(T'')$, inducing $G\isom G'\times T''$, all up to isogeny; see Remark~\ref{rem:decomp_G}. The composition of $L''\to G$ with the projection  $G\to G'$ defines (uncanonically, depending on the complement) a reduced motive
\[ M_\red=[L''\to G']\]
with the the $L$-rank and the $T$-rank of $M$ only depending  on $M_\red$. 

\begin{rem}\label{rem:equiv}
For later use, we point out the following: If a semi-simple algebra $E$ operates on $M$, then it will automatically act on
$M_\baker$. Moreover, $M_\red$ can be constructed such that $E$ still operates.
We only have to choose the splittings of $L$ and $X(T)$ equivariantly. 
\end{rem}

\begin{lemma}\label{lem:bk_red}
There exists a decomposition of $\Per(M)$ as    
\[ \Per\langle M\rangle=\Per\langle M_\baker\rangle+\Per\langle M_\red\rangle.\]
\end{lemma}
\begin{proof}Both $M_\baker$ and $M_\red$ are subquotients of $M$, hence their periods
are contained in $\Per(M)$. It remains to check the opposite inclusion.

From Remark~\ref{rem:decomp_G} we know that $G\isom G'\times T''$ with $G'$ reduced as a $1$-motive.
By composition with the projections, the map $L\to G$ induces $L\to G'$ and $L\to T''$. The induced morphism
\[ [L \to G]\to [L\to G']\times [L\to T'']\]
is injective, hence
\[\Per\langle M\rangle\subset \Per\langle[L\to G']\rangle+\Per\langle [L\to T'']\rangle.\]
In the second step, by definition of $M_\baker$, we have $L\isom L'\times L''$ with $L''\to A$ injective and $L'\to G$ factoring via $T$. The map
\[ [L'\to G']\times [L''\to G']\to [L\to G']\]
is surjective, hence
\[ \Per\langle [L\to G']\rangle\subset \Per\langle [L'\to G']\rangle+\Per\langle [L''\to G']\rangle.\]
This is close to the shape we need, but not quite the same. We need to work on the first summand. By assumption the map $L'\to G$ factors via $T$ and hence $L'\to G'$ factors via $T'$. This gives
a well-defined morphism
\[ [L'\to T']\times [0\to G']\to [L'\to G']\]
which is surjective. Hence
\[ \Per\langle [L'\to G']\rangle\subset \Per\langle [L'\to T']\rangle+\Per\langle G'\rangle. \]
Putting these inclusions of period spaces  together we get
\[ \Per\langle M\rangle\subset\Per\langle [L'\to T]\rangle+\Per\langle [L''\to G'\rangle+\Per\langle[L\to T'']\rangle.\]
By definition the first and the last summand add up to $\Per\langle M_\baker\rangle$ whereas the middle term summand equals $\Per \langle M_\red\rangle$.
\end{proof}

Having finished the reduction step, we now consider the case where $M$ is reduced. We keep $E=\End(A)_\Q$. 
Next we shall construct a motive
$M_\sat=M_0\times M_1$ such that $\Per\langle M\rangle\subset \Per\langle M_\sat\rangle$, $M_0$ is of Baker type and $M_1$ is saturated with $\End(M_1)=E$. 
To this end, we use $T_\sat$, $G_\sat$, $L_\sat$ and $L_\sat\to A$ such that the $E$-operation extends to $G_\sat$ and $[L_\sat\to A]$, see Definitions~\ref{defn:Gsat} and~\ref{defn:Msat2}.

It remains to lift $L_\sat\to A$ to a map $L_\sat\to G_\sat$. We choose a lift of $L\to G$ to a morphism $L\to G_\sat$. The image of $L':=E L\subset G_\sat(\Qbar)$ in
$A(\Qbar)$ agrees with $L_\sat$.
By construction $E$ operates on  
\[M'=[L'\to G_\sat].\] 
However,  $M'$ is not necessarily reduced (and then saturated) because $L'\to L_\sat$ is not necessarily injective.
We put
\[ M_0=(M')_\baker,\  M_1=(M')_\red,\  M_\sat=M_0\times M_1.\]
As pointed out in Remark~\ref{rem:equiv}, we can choose $M'_\red$ such that
$E$ still operates. This makes $M_1$ saturated.

\begin{rem}
The construction of $M_\sat$ depends on the choice of a lift
of $L \to G$ to $L\to G_\sat$. We do not know how to do this in a canonical way.
\end{rem}

\begin{lemma}\label{lem:sat_ok}\index{saturation}
Let $M$ be reduced, $M_\sat=M_0\times M_1$ as constructed above. Then $M_0$ is of Baker type, $M_1$ is saturated and
\[ \Per\langle M_\sat\rangle\supset \Per\langle M\rangle.\]
\end{lemma}
\begin{proof}
By definition, $M_0$ is of Baker type and $M_1$ reduced with
$E=\End(A)\subset \End(M_1)\subset \End(A)_\Q$,   hence $M_1$ is saturated.

By construction there is an injection $[L\to G_\sat]\to M'$ (with $M'$ as in the construction of $M_\sat$) and
 a surjection $[L\to G_\sat]\to M$. Together this gives the inclusion
of period spaces
\[ \Per\langle M\rangle\subset \Per\langle M'\rangle.\]
By Lemma~\ref{lem:bk_red} we also have
\[ \Per\langle M'\rangle \subset \Per\langle M'_\baker\times M'_\red\rangle
=\Per\langle M_\sat\rangle.\]
\end{proof}

\begin{proof}[Proof of Theorem~\ref{thm:dimension_sat}.]\index{dimension estimate}
Part (\ref{it:sat_case}) is Proposition~\ref{prop:formula_sat_holds}. 
Part (\ref{it:trivial}) is clear because all periods of $[L\to 0]$ are algebraic and all periods of $T$ are multiples of $2\pi i$. It remains to prove
(\ref{it:contained}).
By the Lemmas~\ref{lem:bk_red} and \ref{lem:sat_ok} we have
\[ \Per\langle M\rangle\subset\Per\langle M_\baker\times M_0\times M_1\rangle\]
with $M_\baker$ and $M_0$ of Baker type and $M_1$ saturated with the
same abelian part as $M$.
\end{proof}


\chapter{Structure of the Period Space}\label{ch:struct}

It is not too difficult to determine the structure  of the space of periods in the general case by going back to its constituents. 
The inclusions
\[ [0\to T]\subset [0\to G]\subset [L\to G]= M\]
induce a filtration
\[ \Vsing{T}\hookrightarrow \Vsing{G}\hookrightarrow \Vsing{M}.\]
and, from the dual point of view, a cofiltration
\[
M=[L\to G]\twoheadrightarrow [L\rightarrow A] \twoheadrightarrow [L\rightarrow 0]
\]
inducing a  filtration 
\[ \VdR{M}\hookleftarrow\VdR{[L\to A]}\hookleftarrow\VdR{[L\to 0]}.\]
Together, they introduce a bifiltration
\[\begin{xy}\xymatrix{
 \Per\langle T\rangle \ar@{^{(}->}[r]&\Per\langle G\rangle\ar@{^{(}->}[r]&\Per\langle M\rangle\\
         &\Per\langle A\rangle\ar@{^{(}->}[r]\ar@{^{(}->}[u]&\Per\langle [L\to A]\rangle\ar@{^{(}->}[u]&\\
         &&\Per\langle [L\to 0]\rangle\ar@{^{(}->}[u]&.
}\end{xy}\]
on $\Per\langle M\rangle $.
We introduce the notation (and terminology):
\begin{align*}
\Per_\tate(M)&=\Per\langle T\rangle&&\text{(Tate periods)}\index{Tate periods}\index{periods!Tate}\\
\Per_\ab(M)&=\Per\langle A\rangle &&\text{(2nd kind wrt closed paths)}\index{periods! second kind}\index{second kind}\\
\Per_\alg(M)&=\Per\langle [L\to 0]\rangle&&\text{(algebraic periods)}\index{periods!algebraic}\index{algebraic periods}\\
\Per_3(M)&=\Per\langle G\rangle/(\Per_\tate(M)+\Per_\ab(M))&&\text{(3rd kind wrt closed paths)}\index{third kind}\index{periods!third kind}\\
\Per_\inc(M)&=\Per\langle[L\to A]\rangle/(\Per_\ab(M)+\Per_\alg(M))&&\text{(2nd kind wrt non-cl. paths)}\\
\Per_\mix(M)&=\Per\langle M\rangle/(\Per_3(M)+\Per_\inc(M))&&\text{(3rd kind wrt non-cl. paths)}\index{periods!incomplete}\index{incomplete periods}
\end{align*} 
If $M$ is of Baker type (i.e. $A=0$) we also use $\Per_\baker(M)=\Per_\mix(M)$.
\index{Baker type}

From the bifiltration scheme we see that for example
that periods of the third kind with respect  to closed paths are only well-defined up to periods of Tate type and periods of the second kind with respect to closed paths.

\begin{defn}\label{defn:delta_list}
In each of the cases $?=\tate, \ab,\alg,3,\inc,\mix$ we put
\[ \delta_?(M)=\dim_\Qbar \Per_?(M).\]
\end{defn}

The dimensions of the various blocks will be determined one by one. 
By
adding up the $\delta_?(M)$ we then get  $\delta(M)$. This works because of the following 
property:

\begin{thm}\label{thm:disj}
\index{period space!structure}\index{structure theorem for period spaces}
The spaces $\Per_\tate(M)$, $\Per_\alg(M)$ and $\Per_\ab(M)$ have mutually trivial intersection. Moreover,
\[ \Per\langle G\rangle\cap \Per\langle[ L\to A]\rangle =\Per\langle A\rangle.\]
\end{thm}
\begin{proof}These are statements about linear independence. We actually determined bases in the case of a product of a motive of Baker type and a saturated motive in order to determine the dimensions. Also by Theorem~\ref{thm:dimension_sat} we find the periods of $M$ in the space of periods of a certain  $\tilde{M}$ of this special shape. The linear independence claims are simply a byproduct.

Alternatively, we can read off the claim from the dimension formulas themselves (rather than their proof). We explain this in detail.

Consider the semi-simple motive
\[ M'=[L\to 0]\times [0\to T]\times [0\to A].\]
It is the product of a motive of Baker type and a saturated motive. Theorem~\ref{thm:dimension_sat} gives
\[ \delta(M')= \delta(T)+\delta(L)+\sum_B \frac{1}{e(B)}\,4g(B)^2=\delta_\alg(M)+ \delta_\tate(M)+\delta_\ab(M).\]
This means that the period spaces $\Per_\tate(M)$, $\Per_\alg(M)$, and $\Per_\ab(M)$ have mutually trivial intersections. 

To prove  the second claim, we consider the motive 
\[ M''=[L\to A]\times [0\to G].\]
The first factor is of second kind, the second factor is of semi-abelian type and we have computed the dimension of their period spaces in Proposition~\ref{prop:dim_sa} and Proposition~\ref{prop:dim_2} as:
\begin{align*} \delta(G)&=
\delta(T)+\delta(A)+\sum_B 2g(B)\,\rk_{B}(T,M),\\
\delta([L\to A]&=\delta(L)+\delta(A)+\sum_B 2g(B)\,\rk_{B}(L,M);
\end{align*}
here we have used that $\rk_B(L,M)=\rk_B(L,M'')$ and $\rk_B(T,M)=\rk_B(T,M'')$. 
On the other hand, combining
Lemma~\ref{lem:sat_easy} and Lemma~\ref{lem:sat_easy2} on the saturations of motives of semi-abelian type or of the second kind, respectively, we get
\[\Per\langle M''\rangle =\Per\langle \tilde{M}''\rangle\]
where
\[ \tilde{M}''=T\times [L\to 0]\times [0\to G_\sat]\times [L_\sat\to A]\]
with $G_\sat$ and $L_\sat$ as in Definition~\ref{defn:Gsat} and \ref{defn:Msat2}. This is a product of a motive of Baker type and a saturated motive, hence
we find the dimension of its period space by Theorem~\ref{thm:dimension_sat} as 
\begin{multline*}
 \delta(M'')=\delta(\tilde{M}'')
\\=
\delta(T)+\delta(L)+\delta(A)+\sum_B 2g(B)\,\rk_{B}(T,M)
+\sum_B 2g(B)\,\rk_{B}(L,M)
\end{multline*}
Moreover, we know that the intersection of $\Per\langle G\rangle$ and $\Per\langle [L\to A]\rangle$ contains at least $\Per\langle A\rangle$. For the numbers to match up, the intersection has to be 
equal to $\Per\langle A\rangle$
\end{proof}

\begin{rem}
The language of $1$-motives has proved useful in the proofs, but it is compulsory for the structural results in the present chapter. The structure of the period space is readily described in terms of the constituents of the $1$-motives, but not in terms of finitely generated subgroups of connected commutative algebraic groups and their constituents.
\end{rem}

\begin{cor}\label{cor:sum}\index{dimension formula!precise}
We always have
\[ \delta(M)=\delta_\tate(M)+\delta_\ab(M)+\delta_\alg(M)+\delta_3(M)+\delta_\inc(M)+\delta_\mix(M).\]
\end{cor}

This can be compared with our results for motives of semi-abelian type, of second kind or of Baker type. This is what we know so far:

\begin{prop}\label{prop:list}
\begin{enumerate}
\item 
All Tate periods are $\Qbar$-multiples of $2\pi i$, all algebraic 
periods are in $\Qbar$. In particular $\delta_\tate(M)$ and $\delta_\alg(M)$ take the values $0$ or $1$, depending on whether $T$ or  $L$ are trivial.
\item We have
\[ \delta_\ab(M)=\sum_B \frac{1}{e(B)}\,4g(B)^2\]
where the sum is taken over all simple factors of $A$, without multiplicities.
\item We have
\begin{align*}
 \delta_3(M)&=\sum_B 2g(B)\,\rk_{B}(T,M)\\[2ex]
  \delta_\inc(M)&=\sum_B2g(B)\,\rk_B(L,M)
\end{align*}
\item If $M$ is of the Baker type, then
\[ \delta_\mix(M)=\delta_\baker(M)=\rk_{\Gm}(L,M).\]
\item \label{it:list6}If $M$ is saturated, then
\[ \delta_\mix(M)=\sum_Be(B)\,\rk_B(G,M)\,\rk_B(L,M).\]
\item If $M$ is the product of a motive of Baker type and a saturated motive, then the contributions add up.
\end{enumerate}
\end{prop}

\begin{rem}
It remains to determine the precise value of $\delta_\mix(M)$ in the general case. This will happen in the next chapter. In contrast to the other entries, there does not seem to be an easy and clean answer. The classical cases are misleading in this respect. The problem is the subtle interplay between the lattice and the torus part as well as with non-trivial endomorphisms of the abelian part.   

Translated to algebraic varieties, these are the periods of the third kind with respect to non-closed paths on algebraic curves of genus bigger than $0$. Next to nothing was known about them before our monograph. Indeed, as already mentioned before it is here were the point of view of $1$-motives really is needed. 
\end{rem}

\chapter{Incomplete Periods of the Third Kind}\label{ch:precise}
\index{third kind}\index{incomplete periods}\index{periods!incomplete}

In this chapter we develop a precise formula for the dimension  $\delta_\mix(M)$
of the space of periods of the third kind with respect to non-closed paths given by
\[ \Per_\mix(M)=\Per\langle M\rangle/(\Per\langle G\rangle +\Per\langle [L\to A]\rangle)\]
where $M=[L\to G]$ and $G$ is an extension of the abelian variety $A$ by the torus $T$.
It is the most complex part of the picture. 

\section{Relation Spaces}

The assignment $M\mapsto \Per_\mix(M)$ only has a weak functoriality. 
If $M'\hookrightarrow M$ is injective or $M\twoheadrightarrow M''$ surjective,
the inclusions $\Per(M'),\Per(M'')\subset \Per(M)$ also induce maps
\[  \Per_\mix(M')\to \Per_\mix(M),\hspace{2ex}\Per_\mix(M'')\to \Per_\mix(M).\]

In terms of the filtration on $\Per\langle M\rangle$ we are now
interested in the periods of the associated gradeds 
\begin{gather*}\Vsing{M}/\Vsing{[0\to G]}\isom \Vsing{[L\to 0]},\\
\VdR{M}/\VdR{[L\to A]}\isom \VdR{T}
\end{gather*}
of the filtrations in Chapter\ref{ch:struct}  of
highest degree. 
We make the identifications 
\begin{gather*}
 \VdR{T}=X(T)\tensor \VdR{\Gm}\isom X(T)_\Qbar,\\  
\Vsing{[L\to 0]}=L\tensor \Vsing{[\Z\to 0]}\isom L_\Q.
\end{gather*}
Given an elementary tensor $l\otimes x\in L\tensor X(T)$ we choose $\xi\in \VdR{M}$ and $\lambda\in\Vsing{M}$ with image $x$ in $\VdR{T}$  and $l$ in $\Vsing{[L\to 0]}$. Then we define a map
\begin{eqnarray*}\Phi: L_\Q\tensor X(T)_\Q&\to & \Per_\mix(M)\\
l\otimes x& \mapsto & \xi(\lambda).
\end{eqnarray*} 

\begin{lemma}
The map $\Phi$  is well-defined and surjective after extension of scalars to $\Qbar$.
\end{lemma}

\begin{proof}
Two choices of $\lambda$ differ by an element $\delta$ of $\Vsing{G}$. The period
$\xi(\delta)$ only depends on the image of $\xi$ in $\VdR{G}$. Hence it is an element of $\Per\langle G\rangle$. The same reasoning works for $\xi$.
Every element of $\Per\langle M\rangle$ is a linear combination of elements
of the form $\xi(\lambda)$. By definition  we have
\[ \xi(\lambda)=\Phi(x\tensor l)\]
where $x$ is the image of $\xi$ in $\Vsing{[L\to 0]}$ and $l$ the image of
$\lambda$ in $\VdR{T}$. This makes $\phi$ surjective.
\end{proof}
The surjectivity of $\Phi$ gives some first information for $ \delta_\mix(M)$. Indeed we have:

\begin{cor}
\[ \delta_\mix(M)\leq (\rk L)(\dim T).\]
\end{cor}

Also we note that
the map $\Phi$ is compatible
with the weak functoriality of $\Per_\mix$. This can be expressed by the commutative diagram
\[\begin{xy}\xymatrix{
&L'\tensor X(T')\ar[r]^{\phi_{M'}}&\Per_\mix(M')\ar[dd]\\
L'\tensor X(T)\ar[ru]\ar[rd]\\
&L\tensor X(T)\ar[r]^{\phi_M}&\Per_\mix(M)
}\end{xy}\]
for $M'\hookrightarrow M$
and the same for $M\twoheadrightarrow M''$.
In order to determine $\delta_\mix(M)$ we need to describe the kernel of
$\Phi$.

\subsection{Structure of Relation Spaces}
In a first step we describe the obvious relations and in a second step we verify that they are sufficient.

A morphism of iso-$1$-motives
$\alpha:M_1\to M$ induces morphisms 
\[
\alpha_*:A_1\to A, \quad \alpha_*:L_{1,\Q}\to L_\Q\quad \text{and} \quad\alpha_*:T_1\to T.
\]
By duality we also get a morphism $\alpha^*:X(T)_\Q\to X(T_1)_\Q$. 

Consider an exact sequence  
\[ M_1\xrightarrow{\alpha}M\xrightarrow{\beta}M_2,\]
For $l_1\in L_1$, $x_2\in X(T_2)$ the period class $\phi_{M}(\alpha_*(l_{1})\tensor\beta^*(x_2))$  agrees with the image of
\[ 
\phi_{M'}(\alpha(l_1)\otimes \beta^*x_2|_{M'})=\phi_{M'}(\alpha(l_1)\otimes 0)=0\in\Per_\mix(M')\]
by the functiorility relation  for $M'=\alpha(M_1)\subset M$.
This shows that
the elements of $\alpha_*(L_{1})_\Q\tensor\beta^*(X(T_2))_\Q$ have
image $0$ in in $\Per_\mix(M)$, which gives us some first relations. They depend on the exact sequence which determines $\alpha$ and $\beta$. This suggests to go further and take the sum over all short exact sequences as first approximation of the relation space.

\begin{defn}\index{relation space}
We define the space 
\[ R_1(M):= \sum_{\alpha,\beta}\alpha_*(L_{1\Q})\tensor \beta^*(X(T_2)_\Q)\subset L_\Q\tensor X(T)_\Q\]
where the sum is with respect to all exact sequences
\[ M_1\xrightarrow{\alpha}M\xrightarrow{\beta} M_2.\]
Such relations are called \it{primitive}.\index{primitive relations}
\end{defn}

For $n\geq 1$ we extend the definition and introduce the summation maps
\[
s_n = \sum p_{i}\tensor q_i: L_\Q^n\tensor X(T^n)_\Q\to L_\Q\tensor X(T)_\Q
\]
where $p_i:\Vsing{M^n}\to \Vsing{M}$ and $q_i:\VdR{M^n}\to \VdR{M}$ are the projections to the
$i$th components, respectively,  induced by the projection
$\pi_i:M^n\to M$ to the $i$th component and  the inclusion $\iota_j:M\to M^n$ into the $j$-component. For $i\neq j$ they satisfy $\pi_i\circ\iota_j=0$.

\begin{defn}
For $n\geq 1$, we put
\[ R_n(M)=s_n(R_1(M^n)).\]
\end{defn}
This is justified by the following lemma.

\begin{lemma}
The space
$R_n(M)$ is contained in the kernel of $\Phi$.
\end{lemma}
\begin{proof} We explained that $R_1(M)$ is contained in the kernel of $\Phi$ before we introduced $R_1(M)$. Applying this to the motive $M^n$ we see that 
\[ \sum_{i,j}(p_i\tensor q_j) (R_1(M^n))\]
is in the kernel of $\Phi$. 
For $i\neq j$, the orthogonality relation implies that all elements in the image of $p_i\tensor q_j$ are in
$\Ker(\Phi)$ themselves. Dropping them from the sum we are still in the kernel.
\end{proof}
\begin{rem}
The space of primitive relations remains unchanged if
we restrict to injective $\alpha$ and surjective $\beta$. However, we find the extra flexibility useful.
Observe that 
\begin{align*}
\beta^*(X(T_2)_\Q)&=\Ker(\alpha_*: X(T)_\Q\to X(T_1)_\Q),\\
\alpha_*(L_{1,\Q})&=\Ker(\beta_*: L_\Q\to L_{2,\Q}),
\end{align*}
where $M_1=[L_1\to G_1]$ with torus part $T_1$ and $M_2=[L_2\to G_2]$ with torus part $T_2$.
This gives a less redundant, but also less symmetric description
\begin{align*}
R_1(M)
&=\sum_{\alpha:M_1\to M}\im(\alpha_*)\tensor\Ker(\alpha^*)\\
&=\sum_{\beta:M\to M_2} \Ker(\beta_*)\tensor \im(\beta^*)
\end{align*}
\end{rem}
There are trivial inclusions $R_n(M)\subset R_{n+1}(M)$. This suggests that we make the following definition.
\begin{defn}
We put\index{relation space}
\[ R_\mix(M)=\bigcup_{n=1}^\infty R_n(M).\]
\end{defn}
The hope is that this makes up all relations. This is the statement of the following theorem.

\begin{thm}\label{thm:precise1}
\index{dimension formula!precise}\index{periods!incomplete}\index{relation space}
The map
\[ \left(L_\Q\tensor X(T)_\Q/R_\mix(M)\right)_\Qbar\to \Per_\mix(M)\]
induced by $\Phi$ is an isomorphism. In particular, we have
\[ \delta_\mix(M)=(\rk L)(\dim T)-\dim(R_\mix(M)).\]
\end{thm}

As $L_\Q\tensor X(T)_\Q$ is finite dimensional, the system  $R_n(M)$
has to stabilise. Actually, the proof below will show that taking $\rk(L)$
or $\rk(X(T))$ for $n$ suffices. The argument has two steps: first we deal in Section~\ref{ssec:L=Z} with  the special case $L=\Z$, then in Section~\ref{ssec:reduce} we give the reduction argument. 

\subsection{The Case $L=\Z$}\label{ssec:L=Z}
In this section we make the hypothesis that $L=\Z$.
In this case $L\tensor X(T)$ is identified with $X(T)$ and with this identification the map $\Phi$ simplifies to 
\[ \Phi: X(T)_\Q\to \Per_\mix(M).\]

\begin{lemma}\label{lem:Rrk1}
Under this assumption
\[ R_1(M)=\sum_{\alpha} \alpha^*X(T')\]
where the sum is with respect to all $M'=[L'\to G']$  with torus part $T'$ and all $\alpha:M\to M'$ such that $\alpha_*(l)=0$ for $l$
the image of $1$ in $G(\Qbar)_\Q$. 
\end{lemma}
\begin{proof}We use the description
\[ R_1(M)=\sum_\alpha \Ker(\alpha_*)\tensor \im(\alpha^*)\]
with respect to all $\alpha:M\to M'$. Clearly we have $\ker(\alpha_*)\neq 0$ if and only if $l\in\ker(\alpha_*)$ and in this case we have identified
$\Z\tensor  \alpha^*X(T')$ with $\im(\alpha^*)$.
\end{proof}

\begin{prop}\label{prop:rk1}
Let $M=[\Z\to G]$ be a $1$-motive. Then the map 
\[ X(T)_\Qbar/R_1(M)_\Qbar\to \Per_\mix(M)\]
is an isomorphism. In particular,
\[ R_\mix(M)=R_1(M).\] 
\end{prop}
\begin{proof}
We have already discussed surjectivity, so it remains to show
\[ \Ker(\Phi)\subset R_1(M).\]
We choose bases along the same pattern and with the same notation as for the proof of  Lemma~\ref{lem:compute}, but without taking the $\End(A)$-action into account.
This means  that $(\sigma,\gamma,\lambda)$ is a $\Q$-basis of $\Vsing{M}$ and
$(u,\omega,\xi)$ is a $\Qbar$-basis of $\VdR{M}$,
and moreover, we choose $\lambda$ such that its image
in $\Vsing{[\Z\to 0]}$ is the basis $1$. Also we choose $\xi$ such that its image $\bar{\xi}$ in $\VdR{T}$ is a $\Q$-basis of $X(T)_\Q$.

Suppose we have an element $\psi$ in the kernel of $\Phi_M$. It is of the form $\psi=\sum c_{j}\xi_j$ such that its period
is contained in $\Per\langle[L\to A]\rangle+\Per\langle G\rangle$ by the definition of $\Per_\mix(M)$.
Unwinding this we see that there is a $\Qbar$-linear relation of the form
\begin{multline*}
 \sum a\,\xi(\sigma)+\sum b\,\xi(\gamma)
+ \sum c\,\xi(\lambda)+\sum d\,\omega(\gamma)+\sum f\,\omega(\lambda) +\sum g u(\lambda)=0.
\end{multline*}
This gives a non-trivial period relation and  the Subgroup Theorem for $1$-motives
can be applied.
We introduce the motives, which can be read off the relation above. They are 
 \begin{align*}
 M_\sigma&=[0\to T_\sigma], \quad T_\sigma=\Gm^{|\sigma|}, \\
 M_\gamma&=[0\to G_\gamma],\quad   G_\gamma=G^{|\gamma|}, 
\\ 
M_\lambda&=M.
 \end{align*}
 We apply the Subgroup Theorem to the motive 
 \[ 
 \tilde{M}=M_\sigma\times M_\gamma\times M_\lambda=[\tilde{L}\to\tilde{G}].
\] 
 In order to write the above relation as a period relation on $\tilde M$ we introduce
\begin{align*} 
\tilde{\gamma}&=(\sigma,\gamma,\lambda)\in\Vsing{\tilde{M}},\\
\tilde{\omega}&=\left(\sum a_{ij}\xi_j|i=1,\dots,|\sigma|\right)\times \left(\sum b_{ij}\xi_j+\sum d_{ij}\omega_j| i=1,\dots,|\gamma|\right)\\
  &\times \left(\sum c_{j}\xi_j+\sum f_{j}\omega_j +\sum g_{j}u_j\right).
\end{align*}
Then the dependence relation above says that 
\[ \tilde{\omega}(\tilde{\gamma})=0.\]
The Subgroup Theorem for $1$-motives gives a short exact sequence
\[ 0\to M_1\xrightarrow{\nu}\tilde{M}\xrightarrow{p}M_2\to 0\]
such that $\tilde{\gamma}=\nu_*\gamma_1$ and $\tilde{\omega}=p^*\omega_2$.
In the next step we unwind what we have obtained so far.

By assumption the push-forward $p_*\tilde\gamma$ 
of $\tilde{\gamma}$ in $\Vsing{M_2}=\Vsing{[L_2\to G_2]}$ vanishes. Further its image  
in $\Vsing{[L_2\to 0]}\isom L_{2,\Q}$ coincides with the image of the generator of
$\tilde{L}\isom \Z$, which vanishes. This implies that $L_2=0$, and $M_2=[0\to G_2]$.

We now construct the $1$-motive $M'$ and $\alpha:M\to M'$ as needed in
the description of $R_1(M)$ in Lemma~\ref{lem:Rrk1}.
Let $G'$ be the image of $G_\lambda\subset \tilde{G}$ in $G_2$ under $p:\tilde{G}\to G_2$. We write $p_\lambda$ for the restriction of $p$ to $M_\lambda$. The motive $M'=[0\to G']$ and $\alpha=p_\lambda:M=M_\lambda\to M'$ satisfy the conditions for $R_1(M)$. It remains to relate $\psi$ to an element in $\alpha^* X(T')$.

The restriction of $\tilde{\omega}$ to $M_\lambda$ only depends on the map $\alpha: M_\lambda \to [0\to G']$ and can be expressed as  $\alpha^*\omega'$ for the restriction $\omega'$ of $\omega_2$ to $\VdR{G'}$.
A further restriction  to $G_\lambda$ gives 
\[ \tilde{\omega}|_{G_\lambda}=  \sum c_{j}\xi_j+\sum f_{j}\omega_j =    p_\lambda^*\omega'\in \VdR{G_\lambda}.\]
Restricting to $T_\lambda=T$ gives
\[ \tilde{\omega}|_{T_\lambda}=\sum c_{j}\xi_j=\alpha^*(\omega'|_{T'})\in \VdR{T_\lambda}=X(T)_\Qbar.\]
In conclusion, we get 
\[ \psi=\sum_j c_{j}\bar{\xi}_j\in \alpha^*X(T')_\Qbar\subset R_1(M)\tensor_\Q\Qbar.\]
and this proves our claim. 
\end{proof}

\subsection{The Reduction Argument}\label{ssec:reduce}
Given a $1$-motive $M=[L\xrightarrow{u}G]$, there is a $1$-motive $\tilde{G}$
with the same periods as $M$, but a lattice of rank $1$: let $l_1,\dots,l_r$ be
a basis of $L$ and put
\[ \tilde{M}=[\Z\to G^r]\]
with structure map $1\mapsto (u(l_1),\dots,u(l_r))$.
The same choice of basis also induces an identification
\[ L\tensor X(T)\isom X(T)^r\]
compatible with the period map $\Phi$.

\begin{lemma}
In this situation,
\[ R_1(\tilde{M})\subset R_1(M^r).\]
\end{lemma}
\begin{proof}
Let $\alpha:\tilde{M}\to \tilde{M}'$ be as in the definition of
$R_1(\tilde{M})$ with $M$ replaced by $\tilde M$, accordingly the image of the basis element $1\in\Z$ in the lattice part of $\tilde{M}'$ vanishes. In other words, the map on the group part
\[ \alpha:G^r\to G'\]
extends to a morphism of motives
\[ \alpha: [\Z\to G^r]\to [0\to G'].\]
Let $L'$ be the image of $L^r$ in $G'$ and $M'=[L'\to G']$. Then the induced
\[ M^r=[L^r\to G^r]\xrightarrow{\alpha'}M'\]
is as in the definition of $R_1(M^r)$. Moreover, the element $(l_1,\dots,l_r)$ is in the kernel of $\alpha'$
because $\alpha(u(l_1),\dots,u(l_r))=0$. 
Let $(t_1,\dots,t_r)$ be an arbitrary element of $\alpha^*(X(T'))\subset R_1(\tilde{M})$. Then $(l_1,\dots,l_r)\tensor(t_1,\dots,t_r))$ fulfills the requirements for being in $R_1(M^r)$.
\end{proof}

\begin{proof}[Proof of Theorem~\ref{thm:precise1}.]
The periods of $M$, $\tilde{M}$ and $M^r$ agree. From Proposition~\ref{prop:rk1} and the last lemma we know that
\[
 \ker(\Phi_{\tilde{M}})=R_1(\tilde{M})\subset R_1(M^r)\subset\ker(\Phi_{M^r}).
\]
In consequence, equality holds everywhere. 
The map $\Phi_{M^r}$ factors through  the summation map $s_r$ and, by definition, 
$R_r(M)=s(R_1(M^r))$. This implies
\[ \ker(\Phi_M)=R_r(M).\]
As the $R_n(M)$ are nested, this means
\[ R_\mix(M)=R_r(M)=\ker(\Phi_M).\]
\end{proof}

\begin{cor}
We have
\[ R_\mix(M)=R_n(M)\]
where $n$ is the minimum of $\rk L$ and $\dim T$.
\end{cor}
\begin{proof}The proof of the theorem gave equality for  $n=r$. The dual arguments allows to reduce to the case where the torus is of dimension $1$ and hence
$n=\dim T$ is also enough.
\end{proof}

\section{Alternative Description of $\delta_\mix(M)$}\label{sec:precise}

In the last section, we used a trick to reduce the computation of $\delta_\mix(M)$ to the case of motives with lattice part of rank $1$. The trick has a canonical interpretation that will lead to an alternative description of
$R_\mix(M)$.

\begin{lemma}
Let $\Ah$ be an additive category, $\Zfree$ the category of finitely generated free abelian groups of finite rank. There is an additive bifunctor, the \emph{external tensor product}
\[ \tensor:\Zfree\times \Ah\to\Ah,\]  
uniquely determined by $(\Z,X)\mapsto X$.   
\end{lemma}
\begin{proof}Let $\Lambda$ be a free $\Z$-module of rank $r$ and choose a basis $\lambda_1,\dots,\lambda_r$. For $X\in\Ah$ we put
\[ \Lambda\tensor X:=X^r.\]
 Let $f:\Lambda\to\Lambda'$ be a  $\Z$-linear map. In our chosen bases it is given by a matrix $(a_{ij})_{i,j}$ with entries in $\Z$. We define
\[ f_*:\Lambda\tensor X\to \Lambda'\tensor X\]
as the map $X^r\to X^{r'}$ defined by the matrix. 
\end{proof}
The construction is applied to the categories of abelian or semi-abelian varieties. Given a $1$-motive $M=[L\xrightarrow{u}G]$ we can consider
$\Lambda=L^\vee$ and the semi-abelian variety $L^\vee\tensor G\isom G^r$.
The structure map $u:L\to G$ of $M$ induces a canonical homomorphism
$\Z\to L^\vee\tensor G$ as follows:
if $l_1,\dots,l_r$ is a basis of $L$, then 
\[ 1\mapsto c:=\sum_i l_i^\vee\tensor u(l_i)\]
 is the image of $1$. We refer to $c$ as the \emph{tautological element}. This map is independent of the choice of a basis and
has the properties of an adjoint of $u$. 

\begin{rem}
The $1$-motive $[\Z\to L^\vee\tensor G]$ agrees with the $1$- motive $\tilde{M}$
considered in Section~\ref{ssec:reduce}.
\end{rem}
The abelian part of $[\Z\to L^\vee\tensor G]$ is $L^\vee\tensor A$, the torus part is $L^\vee\tensor T$ with character group
$X(L^\vee \tensor T)=L\tensor X(T)$.

As pointed out before, the periods of $M$ agree with the periods of our adjoint. Moreover, the map $\Phi$ for $[\Z\to L^\vee\tensor G]$ given by
\[ \Phi: \Z\tensor X(L^\vee\tensor T)=L\tensor X(T)\to \Per_\mix(M)\]
agrees with the map $\Phi_M$. We have shown in Proposition~\ref{prop:rk1} 
that 
\[ \ker(\Phi)=R_1([\Z\to L^\vee\tensor G]).\]
Our objective is to give an explicit description of this space.

The algebra $E=\End(L^\vee\tensor A)_\Q$
operates
(up to isogeny) from the right on $(L^\vee\tensor A)^\vee\isom L\tensor A^\vee$.
On the other hand, the semi-abelian variety $L^\vee\tensor G$ is characterised by a homomorphism 
$[L^\vee\tensor G]:X(L^\vee\tensor T)\to (L^\vee\tensor A)^\vee(\Qbar)_\Qbar$.
A choice of elements $\alpha\in E$, $y\in L\tensor X(T)$,  $x\in L\tensor A^\vee(\Qbar)_\Q$ determines extensions $G_x$, $G_{\alpha^\vee(x)}\isom\alpha^*G_x$ and $G_{[L^\vee\tensor G](y)}$ in $\Ext^1(L^\vee\tensor A,\Gm)$.

We require that the following compatibility conditions are satisfied:
\begin{enumerate}
\item[(A)] $\alpha^\vee(x)= [L^\vee\tensor G](y)$; in other words, the diagram  \[\begin{xy}\xymatrix{
X(L^\vee\tensor T)\ar[d]_{[L^\vee\tensor G]}&\Z=X(\Gm)\ar[l]_{y\leftmapsto 1}\ar[d]^{[G_x]}\\ 
L\tensor A^\vee(\Qbar)&L\tensor A^\vee(\Qbar)\ar[l]^{\alpha^\vee}
}\end{xy}\]
is commutative, which means that  
$\alpha$ extends to a morphism
\[\alpha_y: L^\vee\tensor G\to G_x.
\]
\item[(B)] 
 $\alpha_{y}(c)=0$ in $G_x$ and again as a consequence $\alpha_y$ defines a morphism
\[ [\Z\to L^\vee\tensor G]\to [0\to G_x].\]
\end{enumerate}


\begin{thm}\label{thm:mix_final}
\index{relation space}\index{periods!precise}\index{periods!incomplete}
Let $M=[L\xrightarrow{u}G]$ be a $1$-motive.
Then
\begin{multline*} R_\mix(M)=
\big\langle y\in L_\Q\tensor X(T)_\Q\;|\\
 \exists \alpha\in \End(L^\vee\tensor A)_\Q,  \exists x\in (L\tensor A^\vee)(\Qbar)_\Qbar,  (A), (B)\big\rangle 
\end{multline*}
\end{thm}   
\begin{proof}
We have to check that the set on the right coincides with $R_1([\Z\to L^\vee\tensor G])$.
Given $\alpha\in E$ and $x$ with (A), the morphism $\alpha$ extends to
The element $y\in X(L^\vee\tensor T)$ is in the image of $\tilde{\alpha}^*$ by (A),
hence a primitive relation.

Conversely, take $y\in R_1([\Z\to L^\vee\tensor G])$. By Lemma~\ref{lem:Rrk1} 
there are $\alpha:[\Z\to L^\vee\tensor G]\to M'$ and $x\in X(T')$ such that
$y=\alpha^*(x)$ (this equality is property (A)) and and the tautological element $c$ is in $\ker(\alpha_*)$. A fortiori, the image of
$c$ vanishes in $X(G_x)$ (the vanishing is property (B)).

The abelian part $A'$ of $M'$ is a quotient of $L^\vee\tensor A$. By semi-simplicity, we can choose a direct complement $A''$. By abuse of notation, we also denote the projector $L^\vee\tensor A\to A'\to L^\vee\tensor A$ obtained this way by $\alpha$. We may now
replace $M'$ by $M'\times [0\to A'']$. The new data defines an element of
the right hand side.
\end{proof}

\begin{ex}\label{ex:delta_1}\index{elliptic periods}
We go back to Example~\ref{ex:zahlen} and take for $A$ an elliptic curve $E$, $0\to\Gm\to G\to E\to 0$ a non-trivial extension (non-split, even up to isogeny) and $P\in G(\Qbar)$ a point whose image in $E(\Qbar)$ is not torsion. We consider
\[ M=[\Z\xrightarrow{u} G]\]
with $u(1)=P$. The surjection
\[ \Phi:\Z\tensor X(\Gm)=\Z\to\Per_\mix(M)\]
gives an upper bound of $\delta_\mix(M)\leq 1$. It remains to compute
the relations. Assume that $0\neq y\in R_\mix(M)$.
By Theorem~\ref{thm:mix_final} this means there are $\alpha\in\End_\Q(E)_\Q$ and $x\in E^\vee(\Qbar)_\Qbar$ such that 
$\alpha^\vee(x)=[G](y)$ and $\alpha_y(c)=0$ in $G_x$. The first condition implies that $\alpha$ and $x$ are different from zero. This makes $\alpha$ invertible and we get an isomorphism $G\to G_x$. 
The tautological class $c\in G(\Qbar)_\Q$ is equal to $P$ in our case. It was assumed to be non-torsion in $G(\Qbar)$, and the same remains true after applying the isomorphism $\alpha$. We get a contradiction to (B), and
this shows that $R_\mix(M)=0$ and, 
\[ \delta_\mix(M)=1.\]
This fits with our explicit computation in Chapter~\ref{ch:ex}, both without  CM  and with CM.
\end{ex}

\begin{rem}
We had the suspicion that there might be a better description of $R_\mix(M)$ in the language of biextensions. We now tend to think that this is not the case. The period pairing is not related to the pairing between a $1$-motive and its Cartier-dual. 
\end{rem}

\chapter{Elliptic Curves}\label{sec:elliptic}
In some sense, Baker's theory of linear forms in logarithms can be seen as an intermezzo, although one of the most influential, in establishing a modern  theory of periods. We shall now describe a second very important aspect of the theory, namely elliptic periods,  which has been  developed in the last one hundred years by many authors starting with Siegel and Schneider. We shall describe it first in a classical way as has been understood by these authors and then give  the translation into our modern language of $1$-motives. For more details about the history see the introduction, in particular
Section \ref{background}.

\section{Classical Theory of Periods}\label{ssec:class}
\index{elliptic curves}

We review the classical theory of elliptic curves from an algebraic and analytic point of view. This will be used for an application of our abstract results about $1$-motives to explicit  transcendence results in the elliptic case.
For basics on elliptic curves and functions we refer to \cite[Chapter 7]{ahlfors} or \cite[Chapters III, IV]{chandra}.

Let $E$ be an elliptic curve given in the projective plane $\Pe\,^2$ by an equation of the form as
\begin{eqnarray}\label{PEC}
y^2w=4x^3-g_2 \,xw^2-g_3\,w^3
\end{eqnarray}
with complex parameters $g_2$ and $g_3$ such that the discriminant $\Delta=g_2^3-27g_3^2\neq 0$. 

On $E$ addition can be defined and one obtains a projective commutative  algebraic curve
with unit element $e_\infty=[0:0:1]$, one of the four zeroes $e_\infty, e_1, e_2, e_3$ of the right hand side of $\ref{PEC}$. These are the four Weierstraß points, the $2$-torsion points  on the curve.

The associated complex manifold $E^\an$ becomes a complex Lie group with Lie algebra $\Lie(E)$ and exponential map
\[
\exp_E: \, \Lie (E) \rightarrow E^\an
\]   
with kernel a lattice $\Lambda$, which leads to an exact sequence
\begin{eqnarray*}
0\rightarrow \Lambda \rightarrow \C \xrightarrow{\exp_E} E^\an \xrightarrow{} 0
\end{eqnarray*}
and which shows that $E^\an\simeq \C/\Lambda$.

In terms of complex analysis this can be described as follows.
For a pair of complex numbers $\omega_1$, $\omega_2$ with $\tau=\frac{\omega_2}{\omega_1}$ in the upper half plane $\Im \tau >0$ we write $\Lambda=\Z\omega_1+\Z\omega_2$ and consider the \emph{Weierstraß elliptic function} 
\[
\wp(z;\Lambda) = \frac{1}{z^2}+ \sum_{0\neq\omega\in\Lambda}\left[\frac{1}{(z-\omega)^2}-\frac{1}{\omega^2}\right].
\]
The Weierstraß \index{Weierstraß $\wp$-, $\zeta$-, $\sigma$-functions} elliptic function is meromorphic with poles of order two on the lattice $\Lambda$ and  periodic with period lattice  $\Lambda$. For $z\in \C$ the triple   $(w,x,y)=(1,\wp(z;\Lambda),\wp'(z;\Lambda))$  satisfies the equation $\ref{PEC}$ with
\[
g_2=g_2(\Lambda)= 60 \sum_{\omega\neq 0}\frac{1}{\omega^2} \quad \text{and} \quad g_3= g_3(\Lambda) = 140 \sum_{\omega\neq 0}\frac{1}{\omega^3}.
\]
For $z \notin \Lambda$ in $\C$,  the exponential map $\exp_E$ can be written as 
\[
\exp_E(z) = [1:\wp(z):\wp'(z)]
\]
in terms of the Weierstraß $\wp$-function.
It  parametrises the plane algebraic curve 
\begin{eqnarray*}
y^2=4x^3-g_2 \,x-g_3.
\end{eqnarray*}
In terms of the uniformisation by the exponential map, the Weierstraß points are $e_\infty=\exp_E(0))$, 
$e_1=\exp_E(\omega_1/2)$, $e_2=\exp_E(\omega_2/2)$ and $e_3=\exp_E((\omega_1+\omega_2)/2)$ for a basis $\omega_1,\omega_2$ of $\Lambda$.

There are two more classical Weierstraß  functions which are derived from the \emph{Weierstraß $\wp$-function}. The first is  the \emph{Weierstraß $\zeta$- function}
\[
\zeta(z;\Lambda) = \frac{1}{z}+ \sum_{0\neq\omega\in\Lambda}\left[\frac{1}{(z-\omega)}+\frac{1}{\omega} + \frac{z}{\omega^2}\right]
\]
and the second the \emph{Weierstraß $\sigma$-function}
\[
\sigma(z;\Lambda)= z\prod_{0\neq \omega\in \Lambda}(1-\frac{z}{\omega})\;e^{\frac{z}{\omega}+\frac{1}{2}\bigl(\frac{z}{\omega}\bigr)^2}.
\]
The latter can be seen as a variant of the Jacobi Theta-function, which makes up Jacobi's theory of elliptic functions. The three functions are related by the differential equations
\begin{eqnarray*}
\frac{d}{dz}\log(\sigma(z;\Lambda))=\zeta(z;\Lambda)\quad \text{and} \quad \frac{d}{dz}\zeta(z,\Lambda)=-\,\wp(z;\Lambda).
\end{eqnarray*}
For $u\in\C$ fixed, we put 
\[
F(z;u)=\frac{\sigma(z-u)}{\sigma(z)\sigma(u)}\,e^{\zeta(u)z}
\]
and compute that
\[
\frac{d\log F(z;u)}{dz}= \zeta(z-u) - \zeta(z) + \zeta(u).
\]

The classical differential forms of the first, second and third kind are
\begin{equation}\label{123}
\omega=\frac{dx}{y}, \quad \eta=\frac{x\,dx}{y} \quad  \text{and} \quad \xi_P=\frac{y+y(P)}{x-x(P)}\frac{dx}{y}
\end{equation}
where $P=\exp_E(u)$ is fixed. 

The functions $x$ and $y$ have polar divisor $2(e_\infty)$ and $3\,(e_\infty)$ respectively, so that
the differential form $\omega$ is regular, $\eta$ has a pole of order $2$ at $e_\infty$. 
The function $x-x(P)$ has a zero at $P$ and $-P$ because it is even, and  $y+y(P)$ is zero at $(-P)$ together with two other points $P_1$, $P_2$.
This shows that the divisor of $\xi_P$ is 
\[
(\xi_P)=((P_1)+(P_2)) -\left((e_\infty)+(P)\right).
\]
 We conclude that the polar divisor of $\xi_P$ is $(e_\infty)+(P)$ with residue $-2$ at $(e_\infty)$ and $2$ at $(P)$.

Using \cite[Chapter IV, \S3, equation (3.6)]{chandra}, one  verifies that
\begin{eqnarray*}
\begin{aligned}
d\log F(z;u)&= \left(\zeta(z-u) - \zeta(z) + \zeta(u)\right) dz\\
&=\frac{1}{2} \frac{\wp'(z)+\wp'(u)}{\wp(z)-\wp(u)}dz=\exp_E^* \xi_P.
\end{aligned}
.
\end{eqnarray*}
This gives
\begin{eqnarray*}
\exp_E^*\omega=dz, \quad 
\exp_E^*\eta=-d\zeta(z) \quad \text{and} \quad
\exp_E^*\xi_P=d\log F(z,u).
\end{eqnarray*} 

The Weierstraß $\zeta$-function is quasi-periodic with  quasi-periods $\eta_i=2\zeta(\omega_i/2)$
as is readily seen. Note that this agrees with the normalisation of
\cite[Section~3.2]{ahlfors} and \cite{fricke1} but differs from \cite[Chapter~IV.1 Theorem~1, Theorem~3]{chandra}.  
The  Weierstraß functions transform as 
\begin{eqnarray*}
\begin{aligned}
\wp(z+\omega_i)&=\wp(z),\\
\zeta(z+\omega_i)&=\zeta(z)+\eta_i,\\ 
\sigma(z+\omega_i)& =- \sigma(z)\,e^{\eta_i\left(z+\frac{\omega_i}{2}\right)},\\
F(z+\omega_i,u)&=F(z,u)\,e^{-\eta_i u+\zeta(u)\omega_i} = F(z,u)\,e^{\lambda(u,\omega_i)}
\end{aligned}
\end{eqnarray*}
where $\lambda(u,\omega_i)=\zeta(u)\omega_i-\eta_iu$ for $u\notin \Lambda$. One reads off again that that the function $\wp$ is periodic, as already mentioned, and that the functions $\zeta$, $\sigma$ and $F$ are quasi-periodic with periods $\eta_i$,
$e^{\eta_i\left(z+\frac{\omega_i}{2}\right)}$ and $e^{\lambda(u,\omega_i)}$. 
The function $\lambda$ can be extended additively to a function on $\C\times\Lambda$.
Note that we have
$\lambda(\omega_1/2,\omega_2)=\pi i$ by the Legendre relation.

\section{Elliptic Periods}\label{sec:EP}
\index{elliptic periods}

For $Q=\exp_E(v)\in E^\an$,  and $\gamma$ a path from $e_\infty$ to $Q$ the integral 
\[
\omega(\gamma)=\int^Q_{e_{\infty}}\omega= \int_0^v\exp_E^* \omega=\int_0^vdz=v(\gamma)
\]
defines a multivalued map from $E^\an$ to $\C$. For different paths from $e_\infty$ to $Q$ the integrals differ by a period $\omega\in \Lambda$.  We get back the generators  of $\Lambda$ as the periods               
\[
\omega_1=\omega(\varepsilon_1)= \int_{\varepsilon_1}\omega \quad \text{and} \quad  \omega_2=\omega(\varepsilon_2)=\int_{\varepsilon_2}\omega
\]
taken along the basis $\varepsilon_1,\varepsilon_2$ of $H_1^\sing(E^\an,\Z)$ defined as the image of the straight paths $[0,\omega_i]$ in $\C$.
The integral $\omega(\gamma)$ is called \emph{incomplete period} of the 
first kind and  becomes a period (i.e. an element of $\Lambda$) if $\gamma(1)=e_\infty$.

In the case of periods of the second kind, the path $\gamma$ must not contain the pole of $\eta$, i.e. the Weierstraß point $0=e_\infty$.  For the closed paths $\varepsilon_1$ and $\varepsilon_2$ we get back the quasiperiods $\eta_i=\eta(\varepsilon_i)$ from above.

For the differentials of the third kind  $\xi_P$ as above, with polar divisor $(e_\infty)+(P)$, we have to consider in addition a closed path $\varepsilon_0$  going once counterclockwise around $0$ with no other singularities inside and then $\lambda(u,\omega_i)$  and $2\pi i$ become complete  periods  of the third kind.

In the case of incomplete periods of the second kind, we take 
a path $\gamma$ with $\gamma(0)=\exp_E(v)$ and $\gamma(1)=\exp_E(v+w)$ which does not pass through $e_\infty$ and obtain by \cite[(2)~on~p.~202]{fricke1}
\begin{eqnarray*}
\begin{aligned}
\eta(\gamma) = \int_\gamma\eta&= \int_v^{v+w}\exp_E^*(\omega)=\int_v^{v+w}\wp(z)dz
=-\zeta(w+v)+\zeta(v)\\
&=-\zeta(w)-\frac{1}{2} \frac{\wp'(w)-\wp'(v)}{\wp(w)-\wp(v)}\\
&=-\zeta(w)+\alpha(\gamma). 
\end{aligned}
\end{eqnarray*}

When $E$ is defined over $\Qbar$, then
    $\alpha(\gamma)$ is  algebraic.

Let $\gamma: [0,1] \rightarrow E^\an$ be a path of the form $\exp_E\circ \delta$ with $\delta: [0,1] \rightarrow \Lie \,E^\an$ from $v$ to $w$ continuously differentiable  and not containing any of the two poles  of the differential of the third kind $\xi_P$. Then its period along $\gamma$ is
\begin{align*}
\xi_P(\gamma)&=\int_\gamma\xi_P=\int_{\exp\circ\delta}\xi_P=\int_\delta\exp^*\xi_P
=\int_\delta\frac{F'(z,u)}{F(z,u)}dz\\
&=\int_{\delta}\frac{1}{F(z,u)}d F(z,u)=\int_{F\circ \delta}\frac{dt}{t}\\
&=\log_{F\circ \delta}(F\circ\delta(1))-\log(F\circ\delta(0))
\end{align*}
where $\log_{F\circ\delta}$ is the branch of the logarithm defined by analytically
continuing the function $\log$ from the starting point $F\circ\delta(0)=F(v,u)$ along $F\circ\delta$. Any two branches
of the logarithm  differ by an integral multiple of $2\pi i$. Up to such multiples
$\log$ satisfies the usual functional equation, hence
\[ \xi_P(\gamma)=\log\frac{F(w,u)}{F(v,u)}+2\pi i \nu\]
for some $\nu\in\Z$.

In both cases, we see sees that if $\gamma$ is closed with period $\omega$, we get complete periods $\eta(\gamma)$ and $\xi_P(\gamma)=\lambda(u,\omega)+2\pi i\nu$.

\begin{summary}\label{sum:explicit}
Let $\gamma:[0,1]\to E^\an$ be a path. We write it as
$\exp_E\circ \delta$ with $\delta$ a path in $\C\isom\Lie(E)^\an$ with
$\delta(0)=v$, $\delta(1)=v+w$.
For $\omega, \eta, \xi_P$ as in (\ref{123}) and $P=\exp_E(u)$  we get
\begin{align*}
 \omega(\gamma)&=w\\
 \eta(\gamma)&=-\zeta(w)-\frac{1}{2}\frac{\wp'(w)-\wp'(v)}{\wp(w)-\wp(v)}\\
  \xi_P(\gamma)&=\log\frac{F(v+w,u)}{F(v,u)} +2\pi i\nu(\gamma).
\end{align*}
Let $\varepsilon_1,\varepsilon_2$ be the generators of $H_1(E^\an,\Z)$. Then
$\omega(\varepsilon_i)=\omega_i$ are the generators of the period lattice of $E^\an$, $\eta(\varepsilon_i)=\eta_i$ are the quasi-periods of $E^\an$ and
$\xi_P(\varepsilon_i)=\lambda(u,\omega_i)+2\pi i\nu$, for some $\nu(\gamma)\in\Z$, are the periods of the third kind.
\end{summary}
The period computation for these special differential forms extends to all differential forms.

\begin{lemma}
Every meromorphic differential form $\vartheta$ on $E$ can be written as
\[ \vartheta= a\omega+b\eta+\sum_i c_i\xi_{P_i}+df\]
with complex coefficients,  $P_i\in E$ and elliptic $f$. \end{lemma}
\begin{proof}
Let $P_1,\dots,P_k$ be the points different from $e_\infty$ where $\vartheta$ has a non-vanishing residue $c_i$. Then $\vartheta'=\vartheta-\sum_i c_i\xi_{P_i}$ has vanishing residues in all points different from $e_\infty$. As the sum of all residues vanishes, it is even of the second kind. As spelled out in Section~\ref{ssec:2deRham} it defines a class in
$H^1_\dR(C)$. On the other hand the classes of $\omega,\vartheta$ are a basis
of the same cohomology group. This gives rise to the identity 
\[[\vartheta']=a[\omega]+b[\eta]\]
for suitable $a,b\in\C$ and makes $\vartheta'-a\omega-b\eta$ exact.
\end{proof}

Clearly if $E$ and $\vartheta$ are defined over $\Qbar$, then everything can also be chosen over $\Qbar$.
By the lemma, the above formulae can be put together to a computation of $\vartheta(\sigma)$ for any
$\vartheta$ and chain $\sigma$.

\section{A Calculation}

In this section we are concerned with periods of the form $2\log \sigma(u)-\zeta(u)u$. This is an incomplete period of the third kind. Such a period appears the first time in transcendence theory and it is natural to ask whether it is transcendental. However there is no way to answer this directly,  it has to be isolated from a linear form
in incomplete elliptic periods. To achieve this we definitely need our results about linear independence of periods given in Chapter~\ref{ch:ex} .

\begin{prop}\label{prop:calc_sigma}
Let $P=\exp_E(u)$ and $Q=\exp_E(w)$ be distinct and non-zero points on $E$ and let $\delta:[0,1]\to\C$ be a path from $-w$ to $w$ such that $\gamma=\exp_E\circ \delta$ does not pass through $P$. Then 
\[
\int_{-Q}^Q\xi_P:=\int_\gamma\xi_P= 
2\log\frac{\sigma(u)\sigma(w)}{\sigma(w+u)}+2\zeta(u)w+\log\left(-\wp(w)+\wp(u)\right) +2\pi i\nu
\] 
for some $\nu\in\Z$.  In the case  $w=-u/2$ we write $P/2=\exp_E(u/2)$ and then
\[ \int_{P/2}^{-P/2}\xi_P:=\int_\gamma\xi_P=2\log \sigma(u)-\zeta(u)u+\log \left(-\wp\left(\frac{u}{2}\right)+\wp(u)\right) +2\pi i\nu.\]
\end{prop}
\begin{proof}
Incomplete periods of elliptic integrals are up to integer multiples of $2\pi i$ of the form $\log\frac{F(w,u)}{F(v,u)}$. We specialise to $v=-w$. Going back to the definition of $F(w,u)$ we calculate
\begin{align*}
\frac{F(w,u)}{F(-w,u)}&=\frac{\sigma(w-u)}{\sigma(-w-u)}\frac{\sigma(-w)}{\sigma(w)}e^{2\zeta(u)w}\\
&=\frac{\sigma(w-u)}{\sigma(w+u)}e^{2\zeta(u)w}\\
&=\frac{\sigma(w-u)\sigma(w+u)}{\sigma(w+u)^2}e^{2\zeta(u)w}\\
&=\frac{\sigma(u)^2\sigma(w)^2}{\sigma(w+u)^2}\left(-\wp(w)+\wp(u)\right) e^{2\zeta(u)w}
\end{align*}
using the identity
\[
\frac{\sigma(v+u)\sigma(v-u)}{\sigma(v)^2\sigma(u)^2}=-\wp(v)+\wp(u)
\]
which we take from  (14) in \cite[p. 217]{fricke1}. This proves our first formula.

We continue with the choice $w=-\frac{-u}{2}$ to get
\begin{eqnarray*}
\frac{F(-\frac{u}{2},u)}{F(\frac{u}{2},u)}
&=\sigma(u)^2 e^{-\zeta(u)u}\left(-\wp(\frac{u}{2}))+\wp(u)\right).
\end{eqnarray*}
proving the second formula.
\end{proof}
In the case of interest for us, $E$ is defined over $\Qbar$ and the points
$P$ and $Q$ are chosen in $E(\Qbar)$. Then $\wp(u)$ and $\wp(u/2)$ are algebraic and hence $\int_{-P/2}^{P/2}\xi_P$ 
is equal to $2\log \sigma(u)-\zeta(u)u$ modulo Baker periods and multiples of $2\pi i$.

\section{Transcendence of Incomplete Periods}

We now come back to Schneider's Problem 3 mentioned in the introduction, see \cite[p.~138]{S}.

\begin{theorem}\label{prop:E_transc}
\index{elliptic periods}\index{transcendence!of elliptic periods}
Let $E$ be an elliptic curve over $\Qbar$, $P$ a point in $E(\Qbar)$,
$\omega,\eta,\xi_P$ the differential forms of the first, second and third kind from above. Assume that $P$ is non-torsion in $E(\Qbar)$. Let
$\kappa=\sum_{i=1}^na_i\gamma_i$ be a chain in $E^\an$ with boundary in $E(\Qbar)$  avoiding the points $e_\infty$ and $P$  and 
 such that $P(\kappa)=\sum_{i=1}^na_i(\gamma_i(1)-\gamma_i(0))\in E(\Qbar)$ is not a torsion point.  Then
\[ 1, 2\pi i, \omega(\kappa), \eta(\kappa), \xi_P(\kappa)\]
are $\Qbar$-linearly independent, in particular $2\pi i$,  $\omega(\kappa), \eta(\kappa), \xi_P(\kappa)$ are transcendental. 
\end{theorem}
\begin{proof}
We define  $E^\circ=E\setminus\{e_\infty,P\}$. Let $D\subset E$ be the support
 of the boundary $\partial \kappa$. The chain $\kappa$ defines
a homology class $[\kappa]\in H_1^\sing(E^{\circ,\an},D;\Q)$, and the forms $\omega,\eta,\xi_P$ define classes in $H^1_\dR(E^{\circ},D)$. This shows that we we may view our periods as cohomological periods for  $H^1(E^\circ,D)$ in the sense of Definition~\ref{defn:coh_periods}.

We choose an embedding $\nu^\circ:E^\circ\to J(E^\circ)$ into the generalised Jacobian $J(E^\circ)$ introduced in Section~\ref{app:B}
via a base point $P_0\neq e_\infty,P$. This is an extension of $E$ by $\Gm$. Our  assumption on $P$ ensures  that it is non-split up to isogeny. The induced
map $\nu:E\to J(E)\isom E$ is $Q\mapsto Q+P_0$.

By Lemma~\ref{lem:per_J}, the periods of $H^1(E^\circ,D)$ agree with the
periods of the $1$-motive $[\Z[D]^0\to J(E^\circ)]$. Actually, the submotive
$M=[\Z\to J(E^\circ)]$ with $1\mapsto  P^\circ(\kappa):=\sum_i a_i\nu^\circ(\gamma_i(1)-\gamma_i(0)))$  suffices. Note that the image of $P^\circ(\kappa)$ in $E(\Qbar)$ is $P(\kappa)$, which is independent of the choice of $P_0$. By assumption it is non-torsion. Hence $M$  is reduced as a $1$-motive  (see Definition~\ref{defn:red-sat}) and of the form as considered  in Chapter~\ref{ch:ex}.
The class $[\kappa]$ is can be identified
with an element $\lambda\in \Vsing{M}$ as in the  proof of Proposition~\ref{prop:ell}, the non-CM case,  or Proposition~\ref{prop:with_CM}, the CM case.
Accordingly the periods of $\lambda$ agree with the period integral of $\kappa$. 
As a consequence 
the elements
$1, 2\pi i, \omega(\kappa),\eta(\kappa), \xi_P(\kappa)$ are a subset of the basis
considered there, in particular they are linearly independent.  
\end{proof}

\begin{rem}
\begin{enumerate}
\item The assumption is satisfied for if $\kappa$ is a single non-closed
path $\gamma$ with $\gamma(1)-\gamma(0)$ non-torsion. In this case the period
numbers were computed explicitly in Summary~\ref{sum:explicit} in terms
of the Weierstraß functions. The cases of integrals of the first and second kind are actually already due to Schneider; see \cite[Satz 15, p.~60]{S}. The result is new for integrals of the third kind.
\item It is possible to extend the considerations to the case when $P(\kappa)$ is torsion. We do not go into details here.
\end{enumerate}
\end{rem}

By specialising further we obtain the following explicit transcendence result.

\begin{theorem}\label{thm:sigma}
\index{transcendence!of an explicit period}
Let $u\in\C$ be such that $\wp(u)\in \Qbar$ and $\exp_E(u)$ is non-torsion in $E(\Qbar)$. Then
\[ u\zeta(u)-2\log \sigma(u)\]
is transcendental.
\end{theorem}
\begin{proof}For $\exp_E(u)$ we simply write $P$. We choose a path $\delta$ from $u/2$ to $-u/2$ which avoids  the singularities of $\xi_P$ and put $\gamma=\exp_E\circ\delta$. By Proposition~\ref{prop:calc_sigma} the period has the form
\[\xi_P(\gamma)=2\log \sigma(u)-\zeta(u)u+\log (\alpha) +2\pi i\nu\]
 for some algebraic $\alpha$.  By Theorem~\ref{prop:E_transc} it is transcendental, but this is not enough. If
\[ 2\log\sigma(u)-\zeta(u)u=\xi_P(u)-\log\alpha-2\pi  i\nu\]
was algebraic, then we would have a linear dependence relation between the numbers $1, \xi_P(\gamma), \log\alpha, 2\pi i$. 
To obtain a contradiction, it suffices to show that they are
$\Qbar$-linearly independent. (Except when $\alpha$ is a root of unity and $\log\alpha$ a rational multiple of $2\pi i$. Then the element $\log\alpha$ can be dropped from the list and the linear independence is already shown in Theorem~\ref{prop:E_transc}.)

Note that the term $\log\alpha$ is the period of a Kummer 
motive $M_0=[\Z\to \Gm]$ with $1\mapsto \alpha$ and $\xi_P(\gamma)$ is an incomplete period of the third kind of
$M_1=[\Z\to J(E^\circ)]$ as in the proof of Theorem~\ref{prop:E_transc}.

Linear independence 
could be addressed by applying the techniques
of Chapter~\ref{ch:dim_comp_part1}
directly to $M_0\times M_1$. Instead we explain the deduction from the general results proved earlier. 
 In fact, as shown in Example~\ref{ex:delta_1}, we have $\delta_\mix(M_1)=1$. This means that $\xi_P(\gamma)$ is a non-zero element of $\Per_\mix(M_1)$ (see Chapter~\ref{ch:struct}). By Lemma~\ref{lem:sat_ok} we have
$\Per_\mix(M_1)\subset\Per_\mix(M_1^\sat)$ for $M_1^\sat$  a saturation of
$M_1$ as constructed there.
Proposition~\ref{prop:list}, item (\ref{it:list6})  shows that  $\log\alpha$ and $\xi_P(\gamma)$ are
linearly independent in 
\[ \Per_\mix(M_0\times M_1^\sat)=\Per(M_0\times M_1^\sat) /\left(\Per\langle \Gm\times J(E^\circ)\rangle + \Per\langle [\Z\to E]\rangle\right).\]
The terms $1,2\pi i$ are in $\Per\langle \Gm\times J(E^\circ)\rangle \rangle+\Per\langle [\Z\to E]\rangle$ and linearly independent by Lindemann's result;see  Corollary~\ref{cor:pi} or Theorem~\ref{thm:dimension_sat}.
We deduce that all four are
linearly independent and $u\zeta(u)-2\log\sigma(u)$ is transcendental.
\end{proof}

\section{Elliptic Period Space}

In \cite{wuestholz-ell} period spaces for elliptic curves and abelian varieties of dimension $2$  were considered and their dimension was determined, see 
\cite[Theorem 2 and Theorem 3]{wuestholz-ell}, respectively. The motivation was billiards on the ellipsoids where irrationality of elliptic and abelian periods give an answer to the question whether  curvature lines or geodesics are closed or not. The proofs rely on the Analytic Subgroup Theorem in its original version.

 In this section we come back to this problem in a more general setting using the language of $1$-motives. We confine ourselves to the elliptic case  but we go further and determine the dimension of an extended period space: a period space: a period space where instead of considering
 a single differential of the third kind we consider a finite number. In addition we do not restrict to closed paths as in loc.cit. but allow that the paths need not to be closed. This was mentioned as a problem in \cite{wuestholz-ell}. It turns out that 
the presence of several differentials of the third kind and the additonal generality of allowing non-closed paths 
makes the problem much more difficult and givers a rather unexpected answer for specialists.

Suppose now that $E$ is an elliptic curve defined over $\Qbar$  and $\omega$, $\eta$ are the differentials of the first and the second kind, and $\xi_k=\xi_{P_k}, 1\le k \le n$ differentials of the third kind on $E$ with $P_k=\exp_E(u_k)\in E(\Qbar)$ for $u_k\in\Lie(E)_\C$; see Section~\ref{ssec:class} equation (\ref{123}). We denote by $\omega_1$, $\omega_2$, $\eta_1$ and $\eta_2$ the periods and quasi-periods of $E$ in the classical sense, i.e. the integrals of
$\omega$ and $\eta$ with respect to a pair of basis vectors of $H_1(E^\an,\Z)$.
We choose non-closed paths $\gamma_i: [0,1] \rightarrow E^\an$, $1\le i \le m$  with $\gamma_i(0), \gamma_i(1) \in E(\Qbar)$. 

\begin{defn}
\index{elliptic period space}
The period space $W=W(E,P_k,\gamma_i)$ is generated over $\Qbar$ by 
\[
1, 2\pi i, \omega_1, \omega_2, \eta_1, \eta_2, \lambda(u_k,\omega_1), \lambda(u_k,\omega_2),  \omega(\gamma_i), \eta(\gamma_i), \xi_{P_k}(\gamma_i)
\]
for $1\le i \le m$ and $1\le k \le n$. 
\end{defn}
We shall show that $W$ can be identified with 
the period space of a $1$-motive. Let  $S$ be the set $S=\{0,P_1,\dots,P_n\}\subset E$ and $D$ the union of the supports of  $\partial \gamma_i$ for $i=1,\dots,m$. Put $E^\circ= E\setminus S$. Consider the object $H^1(E^\circ,D)$ in the category $\VVarg{\Qbar}{\Q}$. By 
Lemma~\ref{lem:per_J}, its periods agree with the periods of the $1$-motive
\[ M'=[\Z[D]^0\to J(E^\circ)].\]
 Let $L\subset \Z[D]^0$ be generated by
$\partial\gamma_i$ for $i=1,\dots,m$ and put
\[ M=[L\to J(E^\circ)].\]
The lattice $L$ has rank at most $m$. We write $G=J(E^\circ)$ for short. It has abelian part $E$ and a torus part $T$ of rank at most $n$.

\begin{prop} 
The period space $W$ coincides with $\Per\langle M\rangle$.
\end{prop}
\begin{proof}
We choose generators for 
\[ \Vsing{M}\subset \Vsing{M'}\isom H_1^\sing(E^{\circ,\an},D;\Q)\]
in the following way.
Take small loops $\sigma_0,\sigma_1,\dots,\sigma_n$ around the points
in $S$. They generate $\Vsing{[0\to T]}$. Choose next loops $\varepsilon_1,\varepsilon_2$ 
in $E^{\circ,\an}$ whose images in $E^\an$ generate $H_1^\sing(E^\an,\Z)$.
Finally, view $\gamma_1,\dots,\gamma_m$ as paths in $E^{\circ,\an}$. By definition,
\[ \sigma_0,\dots,\sigma_n,\varepsilon_1,\varepsilon_2,\gamma_1,\dots,\gamma_m\]
generate $\Vsing{M}$.

We turn to $\VdR{M}$, a quotient of  $H^1_\dR(E^\circ,D)$. 
We take exact differential form $u_1,\dots,u_r$ which generate the kernel
of $H^1_\dR(E^\circ,D)\to H^1_\dR(E^\circ)$. The differential forms
$\omega,\eta,\xi_{P_1},\dots,\xi_{P_n}$ can be seen as elements of $H^1_\dR(E^\circ,D)$. As we shall prove the set $u_1,\dots,u_m,\omega,\eta,\xi_{P_1},\dots\xi_{P_n}$ generates the whole cohomology. Obviously $u_1,\dots,u_m,\omega,\eta$ generate the subspace
$H^1_\dR(E,D)$, and hence it remains to show that the $\xi_{P_i}$ generate
\[ H^1_\dR(E^\circ,D)/H^1_\dR(E,D)\isom H^1_\dR(E^\circ)/H^1_\dR(E)\isom \ker(H^0_\dR(S)\to H^0_\dR(E)).\]
The composition of the two isomorphisms maps a logarithmic form $\vartheta$ with polar divisor included in $S$ to its residue vector $(\res_0\vartheta,\res_{P_1}\vartheta,\dots,\res_{P_n}\vartheta)$.
The image of $\xi_{P_i}$  is $(-2,0,\dots,2,0\dots)$ 
 with $2$ in place $i$. Together they generate the kernel as claimed.
 
As a consequence the period matrix of $M$ has the shape
\[
\left(\begin{matrix}\xi_{P_i}(\sigma_j)&\xi_{P_i}(\varepsilon_j)&\xi_{P_i}(\gamma_i)\\
 0&\omega(\varepsilon_i)&\omega(\gamma_j)\\
 0&\eta(\varepsilon_i)&\eta(\gamma_j)\\
0&0&u_i(\gamma_j)
\end{matrix}\right)
=
\left(\begin{matrix}2\pi i\alpha_{ij}&\lambda(u_i,\omega_j)+2\pi \nu_j&\xi_{P_i}(\gamma_i)\\
 0&\omega_i&\omega(\gamma_j)\\
 0&\eta_i&\eta(\gamma_j)\\
0&0&\beta_{ij}
\end{matrix}\right).
\]
with $\alpha_{ij},\beta_{ij}\in\Qbar$, $\nu_j\in\Z$.
Its  entries generate the vector space $W$.
\end{proof}

Corollary~\ref{cor:sum} together with Proposition~\ref{prop:list} give a formula for the dimension of $W$. With the notation introduced earlier, we state the following result.

\begin{thm}
\index{elliptic periods!dimension formula}\index{dimension formula!for elliptic periods}
Let $E/\Qbar$ elliptic with $e=\dim_\Q\End(E)_\Q$. The dimension of $W$ over $\Qbar$ is given by
\begin{align*} \dim_\Qbar W=  2+\frac{4}{e} 
+2\rk_{E}(T,M)
+2\rk_{E}(L,M) 
+\delta_\mix(M)
\end{align*}
If $M$ is saturated, then 
 $\delta_\mix(M)=e\cdot\rk_E(T,M)\cdot\rk_E(L,M)$. 
\end{thm}

Note that $M$ is \emph{not} reduced in  general. If $L\to E(\Qbar)_\Q$
or $X(T)\to E(\Qbar)_\Q$ have a kernel, this implies that suitable Baker periods (i.e. values of $\log$ in algebraic numbers) are contained in $W$. 
This happens for example if the end point of a path or one of the $P_k$ are torsion points in $E(\Qbar)$. The situation simplifies if we exclude this case.

\begin{cor}
Assume that $E$ does not have CM, that $n=\rk\langle P_1,\dots,P_n\rangle$, and that $m=\rk \langle \gamma_i(1)-\gamma_i(0)|i=1,\dots,m\rangle$ as subgroups of $E(\Qbar)$. Then
\[\dim W=6+2(n+m)+nm.\]
\end{cor}
\begin{proof}The assumptions imply that $e(E)=1$, that $M$ is  saturated and that 
$n=\rk_E(T,M)$, $m=\rk_E(L,M)$.
\end{proof}

\subsection{With CM}
The CM case is a lot more complicated, even if $M$ is reduced. We consider 
an example with small rank $n=m=2$, hence $M=[L\to G]$ reduced, with $L\isom \Z^2$, $G$ an extension of
an elliptic curve $E$ by the torus $T=\Gm^2$ characterised by its classifying map $X(T)\to E^\vee$. 

Assume that $K=\End_\Q(E)$ is an imaginary quadratic field, which gives $e=2$. In this situation $\rk_E(L,M)$ and $\rk_E(T,M)$ can take the values $1$ and $2$ and then and then we can state the following computation.

\begin{lemma}
\[ \delta_\mix(M)=\begin{cases}2&\text{$M$ saturated,} \\
                               4&\text{otherwise.}
\end{cases}\]
\end{lemma}
\begin{proof} We show that if there is a non-zero element in $R_\mix(M)$ then $M$ is saturated. We go back to the characterisation of Theorem~\ref{thm:mix_final} and choose a  $\Z$-basis $l_1,l_2$ of $L\subset E(\Qbar)$
which is used to to make the identification  $L^\vee\tensor E\isom E^2$ as before.
This gives $c=(l_1,l_2)$. 
Let
$(\alpha,y,x)$ be  a triple satisfying the conditions in Theorem~\ref{thm:mix_final} with $y\neq 0$. First consider the morphism $\alpha:E^2\to E^2$. We may view it
as an element of $M_2(K)$. Condition (B) 
implies that $\alpha(c)=\alpha(l_1,l_2)=0$. There are three possible cases, depending
on the rank of $\alpha$.

If $\alpha$ is invertible, then the non-zero vector $(l_1,l_2)$ cannot be
mapped to $0$. This case does not occur.

If $\alpha$ has rank $1$, it suffices to consider 
$\alpha':E^2\to \alpha(E)\isom E$. We replace $\alpha$ by $\alpha'$ in the arguments and $x,y$ by their images in $E$ and $X(T)$
.
The new $\alpha$ has shape $(m,n)$ for $m,n\in K$. By assumption the image
of $c$ in $E$ is 
\[ \alpha'(c)=m\alpha'(l_1) + n\alpha'(l_2) = ml_1'+nl_2'=0. \]
Without loss of generality, $m\neq 0$, hence it is invertible. We replace $\alpha$ by $m^{-1}\circ \alpha$ and then have $m=1$, $l'_1+nl'_2=0$  with $n$ replaced by $m^{-1}n$. As $l'_1$ and $l'_2$ are $\Z$-linearly independent, this implies $n\in K\setminus \Q$. The image of $L$ in $E(\Qbar)$ contains  $l'_2$ and  $-nl'_1$.
As $1,-n$ are a $\Q$-basis of $K$, this implies that $Kl'_2\subset L$, and since $L$ has $\Q$-rank $2$, this even implies
$L\isom Kl'_2$. In other words: the lattice $L$ is $K$-stable in $E(\Qbar)_\Q$ and
we have $\rk_E(L,M)=1$. 

We continue with the diagram (A). As $M$ is reduced, the classifying map $X(T)\to E^\vee(\Qbar)$ is injective and we use it to identify elements of
$X(T)$ with elements of $E^\vee(\Qbar)$.
The adjoint of $\alpha$ is of the form
$(1,n^\vee):E^\vee\to {E^\vee}^2$. This gives  $y=\alpha^\vee(x)=(x,n^\vee(x))\in X(T)^2$. 
In particular, $x$ is a non-zero element of $X(T)$ and its image under $n^\vee\in K$ is
again in $X(T)$.  Hence  $X(T)$ is also $K$-stable in $E^\vee(\Qbar)$ and  $\rk_E(T,M)=1$. This implies that the classifying map of $G$ is $K$-equivariant and 
that the operation of $K$ on $E$ extends to an action on
$G$. The map in (A) is 
$G^2\to G$,  given by $(1,n)$. Condition (B) is $l_1+nl_2=0$, and if this is satisfied, then $L\subset G(\Qbar)$ is $K$-stable, which means that $M$ is saturated. In this case 
\[ \delta_\mix(M)=e\rk_E(L,M)\rk_E(T,M)=2.\]

It remains to consider the case $\alpha=0$, in which the condition $[L^\vee\tensor G]^*(y)=\alpha^*(x)=0$ implies
$y=0$ and then $\delta_\mix(M)=4$.
\end{proof}

\begin{cor}
The possible values for $\delta(M)$ in the CM-case are
$16$ when neither neither $L$ nor $X(T)$ are $K$-stable, $14$ if one of $L$ and $X(T)$ is
$K$-stable, $12$ if both are $K$-stable, but $M$  is not saturated, and $10$ if $M$ is saturated.
\end{cor}

\chapter{Values of Hypergeometric Functions}\label{ch:hypergeom}

We review how our knowledge on periods of curves and $1$-motives can also
be used to deduce transcendence results for certain values of hypergeometric functions. The result in the elliptic case can be found as a special case of Wolfart's publication
\cite{wolfart} or in \cite{chudchud} by Chudnovsky--Chudnovsky. 

More generally it is well-known that values of hypergeometric functions can be expressed as quotient of two abelian integrals, in general of the second kind. This leads to a period relation between the two periods of the second kind with the hypergeometric function as coefficient. Algebraic values of the hypergeometric function provide  linear relations between the two periods periods with algebraic coefficients. This cannot be true in general and leads to special  points on certain Shimura varietiea as explained very carefully in Chapter~5 of Tretkoff's beautiful monograph \cite{tretkoff}.  

We explain the method because we think that it should generalise to many other interesting cases.

\section{Elliptic Integrals}\label{sec:ell_int}
\index{Euler integral}\index{elliptic integral}

We fix a parameter $\lambda\in\C\ohne\{0,1\}$. 
The famous differential form
\[
\xi(\lambda) = \frac{du}{\sqrt{u(1-u)(1-u\lambda)}}
\]
on the compactification $\hat{\C}$ of the complex plane is multivalued with branch points $ 0,1,\infty,\lambda^{-1}$, the so-called Weierstraß points.
Locally, the branches differ by a sign. 
  
Integration of  differential forms of the type $\xi(\lambda)$ over paths leads to  so-called Euler integrals which were introduced and studied by Euler in connection with his work on hypergeometric functions. 

In our special case of the differential form $\xi(\lambda)$, we take the integrals over arcs $\gamma_{p,q}$ with loose end   at two of the four branch points $p$ and $q$,  respectively, but not passing through branch points. 
They are given by the improper integrals
\[ 
I_{p,q}(\lambda)=\int_{\gamma_{p,q}}\frac{1}{\sqrt{u(1-u)(1-\lambda u)}}\,du
\]
and there are $6$ choices of pairs. Each choice of such a path determines the value up to sign. The convergence of these integrals can be seen as follows: 
We take the path $\gamma_{0,1}$ and have to show that the integral converges locally at $0$ and $1$. The change of variables $t^2=u$ gives convergence close to $0$ and the change of variables $t^2=1-u$ gives convergence close to $1$. The remaining integrals can be transformed by applying the Moebius transformations
\[ u \mapsto \frac{1}{u}, u-1, \frac{1}{u-1},
\frac{(u-1)}{u}, \frac{u}{1-u} .\]

\begin{ex}\label{ex:branch}
If $\lambda\notin [1,\infty)$, we  choose $\gamma_{0,1}$
as the straight path from $0$ to $1$ in $\C$.  By definition the function $u^a$ for complex $a\in \C$ is given by $\exp(a(\log|u| + i\arg u))$.
We take the branch of the integrand determined by the following assignment of the arguments:
\[
\arg\,u=0, \quad\arg\,(1- u)=0, \quad\vert\arg(1-\lambda u)\vert < \pi/2
\]
and then, as we have indicated above, the integral is convergent.
\end{ex}

The Euler integral can be viewed  as a period in our sense.
To see this, let $C_\lambda$ be the curve of genus $1$ in Legendre form with affine equation
\[ y^2=u(1-u)(1-\lambda u).\]
Since the point $\infty$ is rational the curve has a group structure and becomes an elliptic curve.\index{elliptic curve}

The projection $\pi:C_\lambda\to\Pe^1$ with $(u,y) \mapsto u$ 
is a $2$-fold
cover ramified in the $2$-torsion points  $u=0,1,\lambda^{-1},\infty$. 
The multivalued differential form $\xi(\lambda)$ lifts from $\hat\C$ to the single valued 
form  and we have
\[ 
\omega(\lambda)= \frac{du}{y}
\]
 on $C_\lambda$. The closure of our arcs $\gamma_{p,q}$ in $\Pe^1$  lift to paths $\tilde{\gamma}_{p,q}$ on $C_{\lambda}^\an$ and then
 \[
  I_{p,q}(\lambda)=\int_{\tilde{\gamma}_{p,q}}\omega(\lambda).
  \]

There are two choices for the lift and the choice of the branch over $\gamma_{p,q}$ in the original definition is replaced by the choice of the lift $\tilde{\gamma}_{p,q}$. Let $\tilde{\gamma}^-_{p,q}$ be the other lift. Then $[\tilde{\gamma}^-_{p,q}]=-[\tilde{\gamma}_{p,q}]$. For algebraic $\lambda$, this description makes $I_{p,q}(\lambda)$ an incomplete elliptic period of the first kind corresponding to a motive of the form $[\Z\to C_\lambda]$. This insight is enough for the application to values of the hypergeometric function, but we can be more precise.
As
the end points are $2$-torsion points, the image of $\tilde{\gamma}_{p,q}$ is a closed path in $C_\lambda/C_\lambda[2]$. 
The elliptic integral $I_{p,q}(\lambda)$ agrees with a complete elliptic period of the first kind there.  
Expressed in terms of $1$-motives:
$[\Z\to C_\lambda]$ is isomorphic to $[\Z\to 0]\oplus [0\to C_\lambda]$ in the isogeny category $\onemot_\Qbar$ and the integral $I_{p,q}(\lambda)$ is a period of the second factor.

Another way to see this is by complex analysis.
The standard way in complex analysis is to replace the loose ends by small circles with radius $\epsilon$. We get a closed path $c_{p,q}$ around $p$ and $q$ by going forth and back along the arc from $p$ to $q$ and around the circles.

\begin{center}
\includegraphics[scale=0.3]{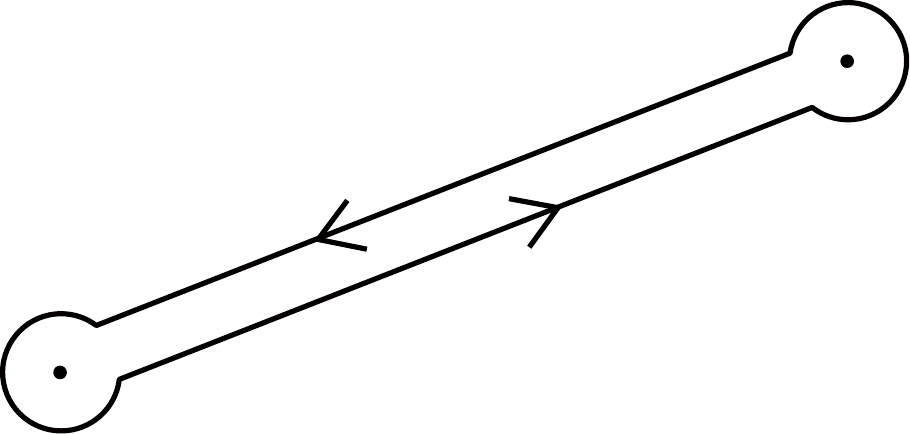}
\end{center}
Since the differential form is multivalued one has to be careful: going once around the circle counterclockwise changes the value of $\xi(\lambda)$ by $-1$  along the arc. Going once around the second circle again counterclockwise does the same and we get back the original determination of the value of the differential form along the arc but we are passing in the opposite direction. The value of the integral is independent of $\varepsilon$. After letting $\epsilon$ tend to zero we obtain   $\int_{c_{p,q}}\xi(\lambda)=2\,I_{p,q}(\lambda)$. 
 Our paths $c_{p,q}$ lift to closed paths $\tilde c_{p,q}$ on $C_\lambda^\an$ such that 
 \[ [\tilde{c}_{p,q}]=[\tilde{\gamma}_{p,q}]-[\tilde{\gamma}^-_{p,q}]=2[\tilde{\gamma}_{p,q}]\in H_1(C_\lambda^\an,\Q),\]
and 
 \[
2\,I_{p,q} (\lambda)= \int_{\tilde c_{p,q}}\omega(\lambda).
 \]
For algebraic $\lambda$, this description makes $I_{p,q}(\lambda)$ a complete elliptic period of the first kind corresponding.\index{elliptic period}

In toto:

\begin{lemma}
Let $p,q,r$ be  distinct in $\{0,1,\infty,\lambda^{-1}\}$. Then $\tilde c_{p,q}$ and $\tilde c_{p,r}$ form  a basis of $H_1^\sing(C_\lambda^\an,\Q)$.
\end{lemma}

\begin{proof}As a topological space we may identify $C_\lambda^\an$ with $(\R/2\Z)^2$. Our exceptional points are the classes of $(0,0)$, $(1,0)$, $(0,1)$, $(1,1)$. We lift the paths $\gamma_{p,q}$ to the universal cover $\R^2$ where they start in a lift $\tilde{p}$ of $p$ and end in lifts $\tilde{q}$ and $\tilde{r}$ of $q$ and $r$, respectively.  The alternative lift
$\tilde{\gamma}^-_{p,q}$ connects $-\tilde{p}$ to $-\tilde{q}$. Up to homotopy
the lift of the closed loop $c_{p,q}$ connects $\tilde{p}$ to $\tilde{p}+2(\tilde{q}-\tilde{p})$. Its homology class is
$2(\tilde{q}-\tilde{p})\in (2\Z)^2=H_1^\sing((\R/2\Z)^2,\Z)$.  
 The vectors $\tilde{q}-\tilde{p},\tilde{r}-\tilde{p}$ are non-zero and distinct in $(\Z/2\Z)^2$, hence linear independent. This makes $2(\tilde{q}-\tilde{p}), 2(\tilde{r}-\tilde{p})$ linear independent in $(2\Z)^2$.
\end{proof}
 
\subsection{A Hypergeometric Function}\label{sec:hyper}

For $|\lambda|<1$ the hypergeometric function \index{hypergeometric function}
with 
parameters $(1/2,1/2,1)$ is defined by  the power series
\[
F\left(\frac{1}{2},\frac{1}{2},1;\lambda\right)=\sum_{n=0}^\infty\frac{(\frac{1}{2})_n(\frac{1}{2})_n}{n!}\frac{\lambda^n}{n!},
\]
which is convergent for $|\lambda|<1$. Here  $(a)_n=a(a+1)+\cdots+(a+n-1)$ for $n>0$ and $1$ for $n=0$ are the Pochhammer symbols.\index{Pochhammer symbol}
The hypergeometric function is a solution of the hypergeometric differential equation
\begin{equation}
\lambda(\lambda-1)\phi''+(2\lambda-1)\phi'+\frac{1}{4}\phi=0.
\end{equation}
As such it extends to a meromorphic function on $\C\ohne\{0,1\}$. The differential equation is a second order  equation of Fuchsian type with regular singular points at $0, 1, \infty$. Its solution space $W$  has dimension 2. A basis for $W$  is given by the pair
\begin{eqnarray*}
z_0(\lambda)&:=&\quad F\left(\frac{1}{2},\frac{1}{2},1;\lambda\right)\\
z_1(\lambda)&:=&-iF\left(\frac{1}{2},\frac{1}{2},1;1-\lambda\right).
\end{eqnarray*}
Each solution extends to a meromorphic function on $\C\setminus \{0, 1\}$.

A classic computation due to Euler (see \cite[p.7]{klein-hyper}; see also \cite[Chapter~2.3.2]{iwasaki-et-al}) shows that
for $|\lambda|<1$
\begin{equation}\label{eq:euler} I_{0,1}(\lambda)=\B\left(\frac{1}{2},\frac{1}{2}\right)F\left(\frac{1}{2},\frac{1}{2},1;\lambda\right)\end{equation}
where $\B(p,q)$ is Euler's beta function. 
From
\[ \B\left(\frac{1}{2},\frac{1}{2})\right)=\int_0^1u^{-1/2}(1-u)^{-1/2}du=\frac{\Gamma(\frac{1}{2})\Gamma(\frac{1}{2})}{\Gamma(1)}=\pi.\]
we conclude that for all $\lambda\in\C\ohne\{0,1\}$. 
\begin{equation}\label{eq:euler_2} I_{\,0,1}(\lambda)=\pi F\left(\frac{1}{2},\frac{1}{2},1;\lambda\right).\end{equation}

We know that $\pi$ and $2\,I_{0,1}(\lambda)$ are periods. For $\pi$ this is clear and for
$I_{0,1}(\lambda)$ we discussed it at length in Section~\ref{sec:ell_int}.  Summary~\ref{sum:1} implies that $I_{0,1}(\lambda)$ 
is non-zero if  $\lambda$ is algebraic. 
Equation~(\ref{eq:euler_2}) is a $\C$-linear relation between
the elliptic period $I_{0,1}(\lambda)$ and $\pi$, which by Theorem~\ref{thm:disj} are linearly independent
over $\Qbar$.  
These considerations prove the following result. 

\begin{prop}[Wolfart,  Chudnovsky-Chudnovsky]\label{prop:1/2} \index{transcendence!of values of the hypergeometric function}\index{hypergeometric function!transcendence}
For $z\in\Qbar\ohne\{0,1\}$, the value
$F(1/2,1/2,1;z)$ of the hypergeometric function is transcendental.
\end{prop}

\begin{rem}
The above fact is pointed out by Wolfart \cite[\S~3,~Fall 3]{wolfart} as a consequence of \cite[Satz~2]{wolfart-wuestholz}. Chudnovsky-Chudnovsky mention it in
\cite[p.~426]{chudchud} as a corollary of Chudnovsky's Theorem on the algebraic independence of elliptic periods. Both references study more generally values of hypergeometric functions from different angles. Andr\'e \cite{andre-crelle} has an alternative approach. But actually, the necessary transcendence result is much older: Schneider \cite[Satz~IIIa]{S3} proved in 1936 the transcendence of
$\pi/\omega$ where $\omega$ is a period of an elliptic curve defined over $\Qbar$. As the relation to values of the hypergeometric function that we described is classical, we do not know who was the first to make the connection to their transcendence. 

An alternative 'modular' proof is suggested by the explicit computation
of $F(1/2/,1/2,1;z)$ given in \cite[Remark~6]{archinard}, corrected by
Bostan, see \cite{bostan}. We have
\begin{equation}\label{eq:bostan}
F\left(\frac{1}{2},\frac{1}{2},1;\lambda(\tau)\right)=E_4(\tau)^{\frac{1}{4}}(\lambda^2(\tau)-\lambda(\tau)+1)^{-\frac{1}{4}}.
\end{equation}
where $\lambda(\tau)$ is the Legendre modular function and $E_4(\tau) =\frac{3}{4\pi^4}g_2(\tau) $ the Eisenstein modular form of weight $4$. The difference to Archinard's formula is the inverse on the left hand side. Assume $\lambda(\tau)$ and $F(1/2,1/2,1;\lambda(\tau))$ to be algebraic. This makes $j(\tau)$ and
$E_4(\tau)$ algebraic. The identity
\[ j(\tau)=1728\frac{E_4(\tau)^3}{E_4(\tau)^3-E_6(\tau)^2}\]
implies that $E_6(\tau)$ is algebraic as well. This is a contradiction to what Bertrand \cite[\S~1.1]{bertrand} showed. As a consequence, $F(1/2,1/2,1;\lambda(\tau))$ is transcendental for algebraic $\lambda(\tau)$. We thank Bostan for pointing out this argument. 
Looking more closely, this is actually the same as before: Bertrand's proof relies on the 
$\Qbar$-linear dependence between $\pi$ and elliptic periods.
\end{rem}

\subsection{The Legendre Family}
\index{Legendre family}\index{elliptic curve}

We return to the integrals $I_{p,q}(\lambda)$ as functions in the variable $\lambda\in S^\an:=\hat{\C}\ohne\{0,1,\infty\}$. The functions are holomorphic and multivalued. We concentrate on $I_{0,1}(\lambda)$. Different  branches correspond to different choices of homotopy classes of
paths $\gamma_{0,1}$.

\begin{ex}
On $\hat{\C}\ohne[1,\infty)$, we use the principal branch of the function normalised as in Example~\ref{ex:branch}. It does not extend to 
$\lambda\in [1,\infty)$ because the straight path $[0,1 ]$ would pass through the branch point $\lambda^{-1}$ and we would  get two possible values for the integral depending on the choice of the branch.  Instead, choose $\lambda_0\in [1,\infty)$ and replace the straight path via one of the two semi-circles sufficiently small radius around $\lambda_0^{-1}$.
\begin{center}
\includegraphics[scale=0.3]{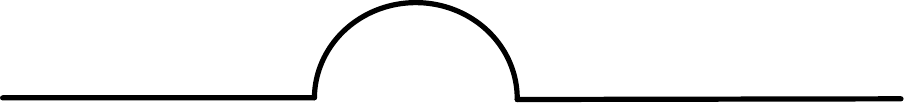}
\end{center}
 some fixed $\lambda_0^{-1}$. The choice of the semi-circle replaces the choice of branch and we take the one on which  $\Im(z) >0$. 
This point of view gives a description of the analytic continuation of $I_{0,1}(\lambda)$. Let $\gamma$ be the new path.
The integral $\int_\gamma\omega(\lambda)$ is well-defined
for $\lambda\in\C\ohne\{\gamma(t)^{-1}| t\in [0,1]\}$, as can be verified, and defines a holomorphic function. In particular, it is well-defined at $\lambda_0$. By the Monodromy Theorem, the modified function agrees with $I_{0,1}(\lambda)$ for $\lambda^{-1}$ in the lower half plane because $\omega(\lambda)$ remains regular between the two paths. It furnishes an analytic continuation of $I_{0,1}(\lambda)$, depending on our choice of $\gamma$. If instead we take the second semi-circle, we obtain a second analytic continuation. Going through all possible paths from $0$ to $1$ gives the full analytic continuation of $I_{0,1}(\lambda)$ as a multi-valued function on $\C\ohne\{0,1\}$. 
In particular, all values of the analytic continuation are periods for different choices of path from $0$ to $1$.
\end{ex} 

The Euler integrals are solutions of the hypergeometric equation
\begin{equation}\label{eq:hypergeom}
\lambda(\lambda-1)\phi''+(2\lambda-1)\phi'+\frac{1}{4}\phi=0
\end{equation}
We refer to \cite[{\textsection 16}]{klein-hyper} or \cite{iwasaki-et-al} for the explicit computation. 
This is a second order differential equation of Fuchsian type with regular singular points at $0, 1, \infty$. The fundamental group of $S^\an=\Pe^1\setminus \{0,1,\infty\}$ is a free group $\Gamma$ with generators $\gamma_0, \gamma_1$ and $\gamma_\infty$, which are loops with base point $b=1/2\in S^\an$ around the points $0,1,\infty$ with $\gamma_0^{-1} = \gamma_1\gamma_0$. The solution space $W$ of the differential equation is a local system of rank $2$. The group $\Gamma$ has a representation $\rho$,  the monodromy representation, in the solution space $W$. Accordingly the solutions are multi-valued functions on $S^\an$ and holomorphic on the universal cover. 
 
A more conceptual interpretation of the situation is to consider the Legendre family 
\[
p: C\rightarrow \mathbb P^1
\]
with fibre $C_\lambda$ at $\lambda$. 
The family has degenerate fibers at $0$, $1$, $\infty$.  Over  $S=\Pe^1\ohne\{0,1,\infty\}$ there is a global
basis of  $H^1_\dR(C_\lambda)$ given by the two differential forms
\begin{eqnarray*}
\omega(\lambda)=\frac{du}{y}, \;  \eta(\lambda)=\frac{udu}{y}
\end{eqnarray*}
The form $\omega(\lambda)$ is holomorphic and thus of the first kind whereas $\eta(\lambda)$ has a pole of order  $2$ at $u=\infty$ and this means that it is of the second kind. 

The homology groups $H_1^\sing(C_\lambda^\an,\Z)$ organise as a
local system of rank $2$ on $S^\an$. Abstractly, it is the dual of
 $p_{S*}\Z$ where $p_S$ is the Legendre family over $S$. For an explicit description
let $s\in S^\an$ be fixed and let  $\gamma_s$   be a cycle in $C_s$. By parallel transport we obtain a horizontal lifting $\gamma(\lambda)$ of the cycle $\gamma_s$. It depends on the choice of an Ehresmann connection. But the homology class of the cycle $\gamma(\lambda)$ is independent of the choice. This gives a horizontal family of homology classes of cycles. The periods
\[
z(\lambda):=\int_{\gamma(\lambda)}\omega(\lambda).
\]
along horizontal cycles give multi-valued analytic functions on $S$. They are solutions of the Gauß--Manin connection  on $H^1_\dR(C_\lambda/S)$, which leads to a differential equation with regular singularities  in $\{0, 1, \infty\}$.
The differential equation so obtained coincides with the differential equation (\ref{eq:hypergeom}).

\section{Abelian Integrals}

In the previous section we introduced the special hypergeometric function $F(1/2,1/2,1;\lambda)$ and showed that for $\lambda \neq 0$ it takes transcendental values. Our proof bases on 
the comparison of the periods of the Legendre curve and the period of the curve $y^2=x(1-x)$ obtained from the Legendre curve when $\lambda=0$. This will now be studied for superelliptic generalisations of the Legendre curve.

\subsection{Euler Integral}

The hypergeometric  function $F(1/2,1/2,1;\lambda)$  is only a very particular case of the hypergeometric function $F(a,b,c;\lambda)$ \index{hypergeometric function}
with expansion
\[
F(a,b,c;\lambda)=\sum_{n=0}^{\infty} \frac{(a)_n(b)_n}{(c)_n}\frac{\lambda^n}{n!}
\]
and \index{hypergeometric function}
convergent for $|\lambda|<1$. As before,
$(a)_n=a(a+1)\dots (a+n-1)$ are the Pochhammer symbols.\index{Pochhammer symbol}
In the most general case the arguments $a$, $b$ and $c$ are complex numbers with $c$ neither zero nor a negative integer.
It satisfies the differential equation
\[
\lambda(1-\lambda) F''+(c-(a+b+1)\lambda)F'-abF=0.
\]
One sees that the differential equation specialises to the differential equation~(\ref{eq:hypergeom}) from above when taking $a=b=1/2$.
Also for this more general hypergeometric function an integral representation can be derived. 
We consider the differential forms
\begin{equation}\label{eq:omega}
\omega(a,b,c;\lambda)=u^{b-1}(1-u)^{c-b-1}(1-\lambda u)^{-a}du
\end{equation}
and
\[ \omega(b,c-b)=u^{b-1}(1-u)^{c-b-1}du.\]
The Euler integral\index{Euler integral}
\[
\Omega(a,b,c;\lambda)=\int_0^1u^{b-1}(1-u)^{c-b-1}(1-\lambda u)^{-a}du
\]
can be expressed in terms of $F(a,b,c;\lambda)$ and the Euler Beta-function. The latter is  usually written as  
\[\B(b,c-b)=\int_0^1 u^{b-1}(1-u)^{c-b-1}du.
\]
\index{Beta integral}
It is obtained from the degeneration of $\omega(a,b,c;\lambda)$ at $\lambda=0$.

\begin{prop}[{ \cite[S.~7]{klein-hyper}, \cite[Chapter~2.3.2]{iwasaki-et-al})}]\label{prop:hyper_gen}
The Euler integral and the hypergeometric function $F(a,b,c;\lambda)$ are related by the equation
\[
\Omega(a,b,c;\lambda) =\B(b,c-b) F(a,b,c;\lambda).
\]
\end{prop}

If $a$, $b$ and $c$ are rational numbers with smallest common denominator $N$, then $\Omega(a,b,c;\lambda)$ can be interpreted as a period on the algebraic curve $C_N(\lambda)$ of the form
\[ y^N=x^r(1-x)^s(1-\lambda x)^t\]
for suitable $r,s,t$. The degeneration $C_N(0)$ with affine equation
\[ y^N=x^r(1-x)^s,\]
has the same property with respect to $\B(b,c-b)$. As in the elliptic case, knowledge about
linear independence of periods leads to transcendence results for
$F(a,b,c;\lambda)$. This connection was already exploited by Wolfart in \cite{wolfart}. We explain an approach from a different angle. In order to simplify the exposition, we restrict to the case where $N=p$ is an odd  prime, $0<r,s,t<p$ and require that $p$ does not divide $r+s+t$ or $r+s$. 

\subsection{Geometry of $C_p(\lambda)$}
Let $\lambda\neq 0,1$ and $p,r,s,t$ as specified above.
The curve $C_p(\lambda)$ was studied first by Wolfart, then very detailed by Archinard  \cite{arch2}, who corrected some errors by Wolfart. Quite  recently Archinard's paper was  extended and corrected by Asakura and Otsubo \cite{AO}. We briefly sketch the results which are relevant for us.

The curves $C_p(\lambda)$ and $C_p=C_p(0)$ are singular in general. Let $X_p(\lambda)$ and $X_p$ be their normalisations, $J_p(\lambda)$ and $J_p$  the Jacobians of $X_p(\lambda)$ and $X_p$, respectively.  The desingularisation is computed in detail by Archinard. As $r,s$ and $t$ are prime to $p$, the branch points have exactly one preimage in the desingularisation by \cite[Remark~3]{arch2}. This  makes the maps $X(\lambda)^\an\to C_p(\lambda)^\an$ and $X_p^\an\to C_p^\an$ homeomorphisms.

\begin{rem}By replacing $y$ by $\pm \lambda^{-t/p}y$ we get the equation for $C_p(\lambda)$ in the form
\[ y^p=x^r(x-1)^s(x-\lambda^{-1})^t\]
considered in \cite{arch2}. 
\end{rem}
\begin{lemma}The genus of $X_p(\lambda)$ and $X_p$ is $p-1$ and $(p-1)/2$, respectively.
\end{lemma}
\begin{proof} Apply \cite[Theorem~4.1]{arch2} to our case.\end{proof}

The group $\mu_p$ of roots of unity operates on $C_p(\lambda)$ and $C_p$ via
$\sigma(x,y)=(x,\zeta^{-1}y)$ for $\zeta\in \mu_p$. This operation induces
an operation of $\mu_p$ on $X_p(\lambda)$ and defines an embedding
$\Q(\mu_p)\to \End_\Q(J_p(\lambda))$. 

\begin{lemma}\label{lem:isotypical}
The abelian variety $J_p(\lambda)$ has at most two isotypical components. If it has two components, both have complex multiplication \index{complex multiplication} by $\Q(\mu_p)$. The abelian variety $J_p$ has complex multiplication.
\end{lemma}
\begin{proof}For $J_p$, we have $2\dim(J_p)=[\Q(\mu_p):\Q]$, making it CM.

For $J_P(\lambda)$, the CM-field $\Q(\mu_p)$ operates on each isotypical component.
The dimension of such a component is at least $(p-1)/2$. This implies that there is either a single isotypical component or there are two. If there are two, both factors have dimension $(p-1)/2$, making them CM.
\end{proof}

\begin{rem}The curve $C_p$ has a cover by the Fermat curve \index{Fermat curve} with affine equation
\[ x_1^p+x_2^p=1\]
via $(x_1,x_2)\mapsto (x_1^p,x_1^rx_2^s)$, see Gross \cite{gross-iccm}. This makes $J_p$ (up to isogeny) a direct factor of the Jacobian of the Fermat curve.
The latter has been studied intensely by Gross and Rohrlich in \cite{gross}.
\end{rem}

\subsection{Differentials on $C_p(\lambda)$}

As the genus is $p-1$, the space $V=\Omega(X_p(\lambda))$ of global differential forms has dimension $p-1$ and $H^1_\dR(X_p(\lambda))$ has dimension $2(p-1)$. Recall that the latter has a description in
terms of differentials of the second kind, see Lemma~\ref{lem:difff_deRham}. 
For $X_p$, the numbers have to be divided by $2$.
\begin{lemma} \label{lem:first}
\index{differential forms!second kind}
\begin{enumerate}
\item 
Suppose that $(p, r+s+t)=1$. For $1\leq n \leq p-1$ and $0\leq u,v,w$ the differential forms on $X_p(\lambda)$, 
\[
\omega_n^{u,v,w}= \frac{x^u(1-x)^v(1-\lambda x)^w}{y^n}dx
\]
are of the second kind. They are holomorphic if and only if
\begin{gather*}
u\geq \left[\frac{rn}{p}\right],\quad v\geq \left[\frac{sn}{p}\right],\quad w\geq \left[\frac{tn}{p}\right],\\
u+v+w\leq \frac{n(r+s+t)-1}{p}-1
\end{gather*}

\item Assume that $(p,r+s)=1$. For $1\leq n\leq p-1$ and $u, v \geq0$ the differential forms on $X_p$,
\[\omega_n^{u,v}=\frac{x^u(1-x)^v}{y^n}dx\]
are of the second kind. They are holomorphic if and only if
\begin{gather*}
u\geq \left[\frac{rn}{p}\right],\quad v\geq \left[\frac{sn}{p}\right],\\
u+v\leq\frac{n(r+s)-1}{p}-1
\end{gather*}
\end{enumerate}
\end{lemma}
\begin{proof}
This is Remark 12 in \cite{arch2} see also Lemma 2.2 in \cite{AO}. In these references the lower bound is given in the form
$\frac{rn+1}{p}-1$ etc. As $u$ is an integer, we can replace the bound by
$\lceil\frac{rn+1}{p}\rceil-1$. Let $rn=kp+e$ with $0\leq e<p$.  This gives
\[ \left\lceil\frac{rn+1}{p}\right\rceil-1=k+\left\lceil\frac{e+1}{p}\right\rceil-1\]
As $1\leq e+1\leq p$, we have $\lceil\frac{e+1}{p}\rceil=1$. We may write the bound in the shape that we used.
\end{proof}
The forms $\omega_n^{u,v,w}$  and $\omega^{u,v}_n$ are $\zeta^n$-eigenvalues for the operation of $\sigma$. 
\begin{cor}[{\cite[Proposition~2.3]{AO}}]\label{cor:diff_einfach}
If $s+t=p$, then on $X_p(\lambda)$,
\[ \omega_n:=\omega_n^{u,v,w}\]
with 
\[
u=\left[\frac{nr}{p}\right], \;
v=\left[\frac{ns}{p}\right], \; w= \left[\frac{nt}{p}\right].
\]
is  of the first kind and
\[ \eta_n:=\omega_n^{u,v+1,w}\]
is of the second kind, but not of first kind. Moreover, the tuple $(\omega_n|n=1,\dots,p-\nobreak1)$ is a basis of $\Omega(X_p(\lambda))$ and 
$(\omega_n,\eta_n\;|\; n=1,\dots,p-1)$ is a basis for $H^1_\dR(X_p(\lambda))$.
\end{cor} 
\begin{proof}
We show that these choices are the only ones giving a holomorphic form.
Division with remainder gives  $nr=kp+e$, $ns=lp+f$ and $nt=mp+g$. It is required that $u\geq m$, $v\geq l$, $w\geq m$.
In conclusion,
\[
u+v+w \geq k+l+m.
\]

Furthermore it is also  required that 
\[ u+v+w \leq \frac{n(r+s+t)-1}{p} -1.\]
 The condition  that $r+s=p$ forces that $p$ divides $f+g$ and leads to $f+g=p$. This implies that
\[
u+v+w \leq k + l+ m +1 +\frac{e-1}{p} -1
\]
and since $u+v+w$ is an integer we conclude that 
\[
u+v+w \leq k+l+m.
\]
It follows that $u+v+w=k+l+m$, $u=k$, $v=l$, $w=m$  and this is what was stated.
\end{proof}

As the forms are eigenforms for the $\mu_p$-operation, this is even an eigenbasis. The forms $\omega_n,\eta_n$ are a basis for the
$\zeta^n$-eigenspace.
 
\begin{cor}\label{cor:diff_X_p}
For $1\leq n\leq p-1$, the differentials
\[ \omega_n:=\frac{x^u(1-x)^v}{y^n}dx\]
with
\[
u=\left[\frac{nr}{p}\right], \;
v=\left[\frac{ns+1}{p}\right], 
\]
are a basis for $H^1_\dR(X_p)$.  The expression 
\begin{equation}\label{eq:expression}
 \left\langle \frac{nr}{p}\right\rangle +\left\langle \frac{ns}{p}\right\rangle-\left\langle \frac{n(r+s)}{p}\right\rangle,
\end{equation}
where  $\langle x\rangle$ denotes the fractional part of a rational number $x$,
takes the value $1$ if $\omega_n$ is holomorphic and the value $0$ if it is not.

Precisely one of the forms $\omega_n$ and $\omega_{p-n}$ is holomorphic. 
\end{cor}

\begin{proof}
They are linearly independent because they have different eigenvalues for the
$\mu_p$-operation. They span a subspace of $H^1_\dR(X_p)$ of dimension $p-1$, hence they even generate it.

The criterion for being holomorphic is  the  computation
of the dimension of $\zeta^n$-eigenspace of $\Omega(X_p)$ in \cite[Theorem~6.7]{arch2}. We check it by hand.

Euclidean division gives $nr=kp+e$, $ns=lp+f$, and then  $n(r+s)=(k+l)p+e+f$. We have
$0<e+f<2p$ and excluded $p\mid e+f$, hence $\langle \frac{n(r+s)}{p}\rangle$ takes the value $\frac{e+f}{p}$ if $e+f<p$ and the value $\frac{e+f-p}{p}$ if
$e+f>p$. This gives
\[ \left\langle \frac{nr}{p}\right\rangle +\left\langle \frac{ns}{p}\right\rangle-\left\langle \frac{n(r+s)}{p}\right\rangle=\begin{cases} 0&e+f<p\\
                          1&e+f>p\end{cases}\]
On the other hand, the condition in Lemma~\ref{lem:first} contains the upper bound
\[ \frac{n(r+s)-1}{p}-1=\left[\frac{rn}{p}\right]+\left[\frac{rs}{p}\right]+\frac{e+f-p-1}{p}\]
We have $e+f-p-1\geq 0$ if and only if $e+f>p$. This makes $\omega_n^{u,v}$ holomorphic if and only if (\ref{eq:expression}) takes the value $1$.

The last statement is \cite[Theorem~6.8]{arch2}; see also \cite[Remark~13]{arch2}. It is also obvious from the above computation: the remainder of $(p-n)r$ is $p-e$ and
the remainder of $(p-n)s$ is $p-f$.
\end{proof}

A CM-pair $(J,\iota)$ consisting of an abelian variety $J$ and an embedding 
$\iota:F\to \End_\Q(J)$ of a CM-field is uniquely determined up to isogeny by the pair
$(J,\Phi)$, where $\Phi$ is the set of eigenvalues for the operation of
$F$ on $\Omega(J)$, the \index{CM type}\index{complex multiplication}CM-type. As a byproduct, the corollary describes the CM-type $\Phi_p$ of $J_p$. In detail:
We introduce 
\[ 
H=\left\{n\in(\Z/p\Z)^*\;\middle\vert\; \left\langle \frac{nr}{p}\right\rangle +\left\langle \frac{ns}{p}\right\rangle-\left\langle \frac{n(r+s)}{p}\right\rangle=1\right\}
\]
and
\[ W=\left\{a\in (\Z/p\Z)^*\mid aH=H\right\}.\]
For a given $n\in H$, the condition in the corollary is satisfied. This means that $\omega_n$ is holomorphic and that $\zeta^n$ appears as an eigenvalue in the operation of $\mu_p$ on $\Omega(J_p)$. We can identify $H$ and $\Phi_p$.
The group $W$ is its stabiliser under the identification of the Galois group $\Gal(\Q(\mu_p)/\Q)$ with $(\Z/p\Z)^*$. 

\begin{rem}Following Gross and Rohrlich in \cite{gross-rohrlich}, we may introduce a fake variable
$t$ with $r+s+t=p$ (unrelated to the $t$ appearing $C_p(\lambda)$). Then
\begin{align*}
 H&=\left\{n\in(\Z/p\Z)^*\;\middle\vert\; \left\langle \frac{nr}{p}\right\rangle +\left\langle \frac{ns}{p}\right\rangle-\left\langle \frac{-nt}{p}\right\rangle=1\right\}\\
&=\left\{n\in(\Z/p\Z)^*\;\middle\vert\;  \left\langle \frac{nr}{p}\right\rangle +\left\langle \frac{ns}{p}\right\rangle+\left\langle \frac{nt}{p}\right\rangle=2\right\}
\end{align*}
because $\langle -x\rangle=1-\langle x\rangle$. This is the complement in $(\Z/p\Z)^*$ of the set
of
\[ H_{r,s,t}=\left\{n\in(\Z/p\Z)^*\;\middle\vert\; \left\langle \frac{nr}{p}\right\rangle +\left\langle \frac{ns}{p}\right\rangle+\left\langle \frac{nt}{p}\right\rangle=1\right\}\]
appearing in \cite{gross-rohrlich} because the expression takes the values $1$ and $2$. 
The normalisation of the operation of $\mu_p$ on $C_p$ in \cite{gross-rohrlich} is complex conjugate to ours, hence they describe the CM type by $H_{r,s,t}$ rather than by our $H$. Note that an $a\in (\Z/pZ)^*$ which stabilises $H$ also stabilises
$H_{r,s,t}=(\Z/p\Z)^*\ohne H$ and conversely. So we actually have
\[ W=W_{r,s,t}\]
as defined in \cite[Lemma~1.6]{gross-rohrlich}.
\end{rem}

\begin{lemma}[Gross and Rohrlich {\cite[Lemma~1.6]{gross-rohrlich}}]\label{lem:W}
The group $W$ is trivial unless $r^3\equiv s^3\equiv(-r-s)^3\mod p$, in which case it is the group of cube roots of unity (modulo $p$).
\end{lemma}
\begin{cor}\label{cor:crit_simple}
The CM-abelian variety $J_p$ is simple if and only if the group $W$ is trivial. In particular, this is the case if $p\nequiv 1\mod 3$. 
\end{cor}
\begin{proof} We identity $H$ with the CM-type, where $n\in(\Z/p\Z)^*$ stands for the $\zeta^n$-eigenspace of the operation of $\zeta\in\mu_p$ on $\Omega(J_p)$. The operation of the Galois group  $\Gal(\Q(\mu_p)/\Q)=(\Z/p\Z)^*$ on the eigenvalues is identified with the left multiplication of $(\Z/p\Z)^*$ on itself.
The condition in the corollary means that the stabiliser of $H$ is trivial, i.e. the CM-type is primitive. Being primitive is equivalent to the abelian variety being simple.

If $J_p$ is not simple, then Lemma~\ref{lem:W} implies that $3$ divides $p-1$.
\end{proof}

\begin{ex}\label{ex:is_simple}
Let $p=11$, $r=s=2$. Then $H=\{3,4,5,9,10\}$ and $J_p$ is simple. 
For $p=7$, $r=2$, $s=4$, we have $H=\{3,5,6\}$ and $J_p$ is not simple because 
$W=\{1,2,4\}$. For $p=7$, $r=s=2$, we have $H=\{2,3,6\}$ and $J_p$ is again simple.
These examples are compatible the criterion of Lemma~\ref{lem:W}.
\end{ex}

\begin{rem}It can happen that $\omega_n^{u,v,w}$ is of the first kind, but $\omega_n^{u,v}$ (for the same parameters $r,s,t,u,v$) is not. Indeed, it must happen because the dimension of the space of differential forms of the first kind goes down.
\end{rem}

\subsection{Transcendence}

We can now connect the periods of $X_p(\lambda)$ to our Euler integrals.
Recall the differential forms $\omega(a,b,c;\lambda)$ in
the complex plane, see (\ref{eq:omega})  

\begin{prop}\label{prop:abc}
Let $p$ be an odd prime, $0<r,s,t<p$ and $p\nmid r+s+t$. 
Choose $1\leq n\leq p-1$, $u,v,w\geq 0$ and introduce 
\[ a=-w+\frac{nt}{p},\quad b=u+1-\frac{nr}{p},\quad c=b+1+v-\frac{ns}{p}=u+v+2-\frac{n(r+s)}{p}.\]
\begin{enumerate}
\item
The form $\omega_n^{u,v,w}$ is the pull-back of $\omega(a,b,c;\lambda)$ 
to $X_p(\lambda)$ under the projection $(x,y)\mapsto x$.
The Euler integral $\Omega(a,b,c;\lambda)$ is a complete period of the second kind for $J_p(\lambda)$. 
\item
The form $\omega^{u,v}_n$ is the pull-back of $\omega(b,c-b)$ to $X_p$ under the projection $(x,y)\mapsto x$.
The Beta integral $\B(b,c-b)$ is a complete period of the second kind for $J_p$.
\end{enumerate}
\end{prop}
\begin{proof}
The first claim is straightforward by replacing $y$ in $\omega_n^{u,v,w}$ by
\[ y=x^{r/p}(1-x)^{s/p}(1-\lambda x)^{t/p}.\]

We write $\omega$ for our meromorphic differential form.
A chosen generator $\zeta\in\mu_p$ operates via $\sigma$ on $C_p(\lambda)$ and (by functoriality) on $X_p(\lambda)$.
Let $\gamma$ be the lift of $[0,1]$ to 
$X_p(\lambda)$ corresponding to the choice of branch in the Euler integral. Then $(\sigma_*\gamma)^{-1}\gamma$ is a closed loop in $X_p(\lambda)^\an$. In computing the integral over the closed path we get
\[ \int_{\gamma}\omega-\int_{\sigma_*\gamma}\omega=\int_\gamma(\omega-\sigma^*\omega)=(1-\zeta^n)\int_\gamma\omega\]
This makes the Euler integral a closed period of the second kind. The argument for the Beta integral is analogous. The period is of the second kind because the differential form $\omega_n^{u,v}$ is of the second kind in general.
\end{proof}

By Proposition~\ref{prop:abc}, the formula
\[ \Omega(a,b,c;\lambda)=F(a,b,c;\lambda)\B(b,c-b)\]
of Proposition~\ref{prop:hyper_gen} can be regarded as a linear relation between periods of algebraic curves. The dimension computations in Chapter~\ref{ch:dim_comp_part1} tell us about $\Qbar$-linear independence of period numbers.

Recall that $\Phi_p$ is the CM-type of $J_p$ with the operation of $\Q(\mu_p)$ induced from the operation of $\mu_p$ on $C_p$. For $\alpha\in\Gal(\Q(\mu_p)/\Q)$ we write $J_p^\alpha=(J_p,\alpha\Phi_p)$ and $\overline{J_p^\alpha}=(J_p,\overline{\alpha\Phi_p})$ (the complementary CM-type).
\index{CM-type}

\begin{thm}\label{thm:hypergeom}
\index{hypergeometric function!transcendence}\index{complex multiplication}
Let $p$ be an odd prime, $0<r,s<p$ such that $p$ does not divide  $r+s$ and put $t=p-s$. We assume that $J_p$ is simple and take $0\leq u,v$  and $a,b,c$ as in Proposition~\ref{prop:abc}, $\lambda\neq 0,1$ algebraic. If $F(a,b,c;\lambda)$  is algebraic and not zero then   (up to isogeny)
\[ J_p(\lambda)\isom J_p^\alpha\times \overline{J_p^\alpha}\]
in the category of abelian varieties with $\mu_p$-action. 
\end{thm}

\begin{proof}The assumption $t+s=p$ ensures that $p\nmid r+s+t$ is satisfied.
The period $\B(b,c-d)$ does not vanish, by the criterion of Summary~\ref{sum:1}, because the lift  $\gamma$ of $[0,1]$ to $X_p^\an$ is not closed.  As $F(a,b,c;\lambda)$ is assumed non-zero, this also makes $\Omega(a,b,c;\lambda)$ non-zero.

We consider $A=J_p\times J_p(\lambda)$. 
The dimension formula in Theorem~\ref{thm:dimension_sat} implies that there are no $\Qbar$-linear relations between non-trivial periods of different isotypical components. This implies
that $J_p$ and $J_p(\lambda)$ share a simple factor. As $J_p$ is simple, this means
that 
\[ J_p(\lambda)\isom J_1\times J_2\]
with 
\[ J_p\isom J_1\]
up to isogeny. Note that we do not know if the isomorphism is compatible with the
$\mu_p$-operation. 

By Lemma~\ref{lem:isotypical} both factors have CM by $\Q(\mu_p)$ (because either $J_p(\lambda)=J_p^2$ or it has two isotypical components) and the eigenbasis 
 computation in Corollary~\ref{cor:diff_einfach} shows  that $J_1$ and $J_2$ have
complex conjugate CM-types $\Phi$ and $\bar \Phi$.  As $J_1$ is simple, the simplicity criterion of Corollary~\ref{cor:crit_simple} applies to $\Phi$  and then also to $\bar\Phi$. 
Primitive CM-types classify pairs $(B,\iota)$ for $B$ a simple abelian variety and
$\iota:\Q(\mu_p)\to\End_\Q(B)$. The CM-type of $J_2$ is realised by $(J_1,\bar{\iota})$, hence $J_2$ and $J_1$ are isogenous after all. Here $\bar{\iota}$ means complex conjugation on $\Q(\mu_p)$ followed by $\iota$.

We now have $J_p(\lambda)\isom J_p^2$. Let $\zeta$ be a generator of $\mu_p$ and
$\sigma$ the corresponding automorphism of $J_p(\lambda)$. It operates on
$J_p^2$ by a ($2\times 2$)-matrix with entries in $\Q(\mu_p)=\End_\Q(J_p)$. As
$\sigma^p=\id$, it is diagonalisable and its eigenvalues are $p$th roots of unity. Without loss of generality, it has the shape
\[
\left(\begin{matrix} \zeta^\alpha&0\\ 0&\zeta^{\alpha'}\end{matrix}\right).
\]
The induced operation on $\Omega(J_p(\lambda))$ is diagonalisable with
$p-1$ distinct eigenvalues. This implies that $\alpha=-\alpha'\mod p$. 
In other words, the isomorphism $J_p(\lambda)\isom J_p^2$ can be chosen such that $\sigma$ operates via $\zeta^\alpha$ on the first factor and
via $\zeta^{-\alpha}$ on the second. This determines the 
CM-types of the two factors. They are complex conjugate to each other.
\end{proof}

\begin{rem}As a consequence of the theorem, the argument $\lambda$ has to be \emph{special} in the sense of Shimura varieties. This could certainly be investigated in more detail. However, the subtelties of Shimura varieties and the Andr\'e-Oort conjecture are beyond the scope of our book.
\end{rem}

\begin{cor}\label{cor:jetzt_aber}
Let $p,r,s,t$ be as in the theorem.
Let $1\leq n\geq p-1$ 
and $u,v,w$ as in Corollary~\ref{cor:diff_einfach} and $a,b,c$ as in Proposition~\ref{prop:abc}. If
\[ \left\langle \frac{nr}{p}\right\rangle +\left\langle \frac{ns}{p}\right\rangle-\left\langle \frac{n(r+s)}{p}\right\rangle\neq 1\]
then $F(a,b,c;\lambda)$ is zero or transcendental \index{hypergeometric function!transcendence}
for all algebraic $\lambda\neq 0,1$.
\end{cor}
\begin{proof}In this case, the form $\omega_n^{u,v,w}$ on $X_p(\lambda)$ is of the first kind, whereas
$\omega_n^{u,v}$ on $X_p$ is not. 

Assume $F=F(a,b,c;\lambda)$ is not zero and algebraic for algebraic $\lambda\neq 0,1$. 
As the CM-types of the two factors of $J_p(\lambda)$ are distinct,
the form $\omega^{u,v,w}_n$ restricts to $0$ on one of them, say the second.
We may view it as a differential form of the first kind on $J_p$. As an element of $H^1_\dR(J_p)$ it is $\Qbar$-linearly independent of the form $\omega_n^{u,w}$, which is not of the first kind by our choice of $n$. By Lemma~\ref{lem:compute}
this makes their periods $\Qbar$-linearly independent.

\end{proof}

Recall that we have an easy numerical criterion to check whether $J_p$ is simple, see Lemma~\ref{lem:W}.

\begin{ex}We choose $p=11$, $r=s=2$, $r+s=4$, $t=9$. As pointed out in Example~\ref{ex:is_simple} this makes $J_p$ simple with CM-type given by $H=\{3,4,5,9,10\}$. Conversely, $\omega_n$ is not holomorphic for $n=1,2,6,7,8$. We have
$a=\langle\frac{9n}{11}\rangle$, $b=1-\langle\frac{2n}{11}\rangle$, $c=2b$
\begin{center}\begin{tabular}{|c|c|c|c|c|c|c|}
$n$&$u$&$v$&$w$&$a$&$b$&$c$\\
$1$&$0$&$0$&$0$& $9/11$ & $9/11$ & $18/11$\\
$2$&$0$&$0$&$1$& $7/11$ & $7/11$ & $14/11$\\
$6$&$1$&$1$&$4$& $10/11$ & $10/11$ & $20/11$\\
$7$&$1$&$1$&$5$& $8/11$ & $8/11$ & $16/11$ \\
$8$&$1$&$1$&$6$& $6/11$ & $6/11$ & $12/11$
\end{tabular}
\end{center}

\noindent The corresponding values of the hypergeometric function
 are zero or transcendental for all algebraic $\lambda\neq 0,1$. 
For $\lambda\in (0,1)$, the Euler integral is non-zero and $F(10/11,10/11,20/11;\lambda)$ is a transcendental number.
\end{ex}


\part{Appendices}

\begin{appendix}

\chapter{Nori Motives}\label{sec:app_nori}
\index{Nori motives}\index{motive!Nori}

In this section a bare minimum of Nori's theory of motives is reviewed, to the extent
needed in the main text. For a more complete picture, see
\cite[Chapter~9.1]{period-buch}.

Our base field will be $k\subset \C$ and we shall work  with $\Q$-coefficients throughout.
We denote by $\Q\Vect$ the category of finite dimensional $\Q$-vector spaces
and more generally by $E\Mod$ the category of finitely generated $E$-left modules
for a finite dimensional $\Q$-algebra $E$.

\section{Effective Motives and Realisations}\label{sec:eff}

A \emph{diagram} $D$ is an oriented graph. A \emph{representation} of $D$ is a map of oriented graphs $T:D\to\Ah$ into an abelian category $\Ah$. It assigns an object to every vertex and a morphism to every edge. There is an abstract construction due to Nori that attaches to every representation $T:D\to\Q\Vect$
a $\Q$-linear abelian category. It should be thought of as the abelian category generated by $D$ inside the category $\Q\Vect$. For a particular choice of the diagram and the representation we obtain the category of motives.

\begin{defn}[{\cite[Definition~9.1.1]{period-buch}}]\label{defn:pairseff}
Let $\pairseff$ be the diagram for which
\begin{enumerate}
\item the vertices are triples
$(X,D,i)$ where $X$ is an algebraic variety over $k$, $D\subset X$
is a closed subvariety and $i\in\Na_0$;
\item two types of edges:
\begin{itemize}
\item (functoriality) for every morphism of varieties $X\to X'$ mapping a subvariety
$D\subset X$ to $D'\subset X'$ an edge \[ f^*: (X',D',i)\to (X,D,i);\]
\item (coboundary) for every triple $X\supset Y\supset Z$ an edge
\[ \partial: (X,Y,i)\to (Y,Z,i+1).\]
\end{itemize}
\end{enumerate}
We define the \emph{singular realisation}\index{singular realisation!of a Nori motive}
\[ H_\sing: \pairseff\to\Q\Vect\]
 by mapping a vertex to the singular cohomology 
\[ (X,D,i)\mapsto H^i_\sing(X^\an,D^\an;\Q)\]
of the datum 
and edges of type $f^*$ to pull-back on cohomology and edges of type coboundary
to the coboundary map in the long exact sequence in cohomology.
\end{defn}

\begin{thmdefn}[Nori, {\cite[Definition~9.1.3, Theorem~9.1.10]{period-buch}}]
\label{thmdefn:Nori}
There is are an abelian $\Q$-linear category $\MMNeff(k,\Q)$, the \emph{category of
effective Nori motives over $k$},\index{effective Nori motives}\index{Nori motives!effective} a faithful exact functor
\[ H_\sing:\MMNeff(k,\Q) \to \Q\Vect\]
and a representation 
\[ H_\Nori:\pairseff \to \MMNeff(k,\Q)\]
such that
\[ H_\sing\circ H_\Nori=H_\sing,\]
is an isomorphism of functors,
in particular
\[ H_\sing\circ H_\Nori(X,D,i)=H_\sing(X,D,i)=H^i_\sing(X,Y;\Q).\]

This triple ($\MMNeff(k,\Q)$, $ H_\sing$, $H_\Nori$)  is uniquely determined by the following universal property:

For any abelian $\Q$-linear category $\Ah$, together with a $\Q$-linear faithful exact functor $f:\Ah\to \Q\Vect$ and a representation $T:\pairseff\to\Ah$
such that
\[ f\circ T\isom H_\sing\]
there is
there is a $\Q$-linear
exact functor
\[ \tilde{T}:\MMNeff(k,\Q)\to\Ah\]
and an isomorphism of functors  $f\circ \tilde{T}\to H^*$ which extends the isomorphism on $\pairseff$.
\end{thmdefn}

The universal property can be summed up in a diagram:
\[\begin{xy}\xymatrix{
&\MMNeff(k,\Q)\ar[rd]^{H_\sing}\ar@{-->}[dd]_{\tilde{T}}\\
D\ar[ru]^{H_\sing}\ar[rd]_T&&\Q\Vect\\
&\Ah\ar[ru]_f
}\end{xy}\]
In this very precise sense, $\MMNeff(k,\Q)$ is the abelian category
generated by $\pairseff$.

We also use the notation
\[ H^i_\Nori(X,D)=H_\Nori(X,D,i)\]
and call it the $i$th Nori motive of $(X,D)$.

 For our purposes, the most important choices for $\Ah$ are
the category $\MHS_k$ of mixed $\Q$-Hodge structures over $k$, see Definition~\ref{defn:mhs}, and the category $\VVarg{k}{\Q}$, see Definition~\ref{defn:VV}.
Deligne constructed in \cite{hodge3} a functor $H^*_\hodge=(H^*_\dR,H^*_\sing,\phi)$ from the category of $k$-varieties
to $\MHS_k$. This has been extended to the diagram $\pairseff$ (e.g. 
in \cite{huber_1604} or as a by-product of \cite{Hreal, Hreal2}).  
The situation can be summed up by the commutative  diagram
\[\begin{xy}\xymatrix{
&\MMNeff(k,\Q)\ar[rd]^{H_\sing}\ar@{-->}[dd]_{H_\hodge}\\
D\ar[ru]^{H_\sing}\ar[rd]_{H_\hodge}&&\Q\Vect\\
&\MHS_k\ar[ru]_f
}\end{xy}\]
where $f$ is the forgetful functor $(V_\dR,V_\sing,\phi)\mapsto V_\sing$.

By the universal property of Nori motives, the representation $H_\hodge$ extends to  a functor 
\[ H_\hodge:\MMNeff(k,\Q)\to \MHS_k\]
on Nori motives. 
We sum up: 

\begin{defn}
\begin{enumerate}
\item
The \emph{Hodge realisation}\index{mixed Hodge structure!of a Nori motive}
\[ H_\hodge: \MMNeff(k,\Q)\to \MHS_k\]
is the canonical extension of the representation of $\pairseff$ in $\MHS_k$ compatible with the singular realisation. 
\item The \emph{period realisation}\index{period realisation of a Nori motive}
\[ H:\MMNeff(k,\Q)\to \VVarg{k}{\Q}\]
is defined by forgetting the filtrations (i.e. composing with the faithful exact functor 
$\MHS_k\to \VVarg{k}{\Q}$).
\item The \emph{de Rham realisation}\index{de Rham realisation!of a Nori motive}
\[ H_\dR:\MMNeff(k,\Q)\to k\Vect\]
is defined
by projecting to the $k$-component (i.e. composing with the faithful exact functor $\VVarg{k}{\Q}\to k\Vect$).
\end{enumerate}
\end{defn}

\begin{rem}\label{rem:nach_Nori}
\begin{enumerate}
\item Every effective Nori motive $M$ over $\Qbar$ has a well-defined set of periods given by the periods of  $H(M)\in\VVarg{k}{\Q}$, i.e. the image of the period pairing
\[ H_\dR(M)\times H_\sing(M)^\vee\to \C,\]
see \cite[Section~11.2]{period-buch}.
\item \label{it:subquotient_Nori} By the universal property of $\MMNeff(k,\Q)$, every object in the category is a subquotient
of a motive of the form $H^i_\Nori(X,D)$. This will allow us to reduce questions on periods of motives to motives of the special shape.
\item In our monograph the theory of motives has been set up to be \emph{contravariant} on the category of varieties. This follows the convention of \cite{period-buch}, but differs from Nori's original approach.
\end{enumerate}
\end{rem}

\section{Filtration by Degree}\label{sec:filt}

Following Ayoub and Barbieri-Viale in \cite{ayoub-barbieri} we concentrate on the subcategories generated by motives of bounded cohomological degree or dimension. 

\begin{defn}
For $n\geq 0$ let
$d_n\MMN(k,\Q)\subset \MMNeff(k,\Q)$ be the thick abelian subcategory (i.e. full and closed under extensions and subquotients) generated by the
objects $H^i_\Nori(X,D)$, with $X$ a $k$-variety, $D\subset X$ a closed subvariety and $i\leq n$.\index{degree filtration of Nori motives}\index{filtration by degree of a Nori motive}
\end{defn}

\begin{rem}
Our definition  is the contravariant analogue of \cite[Definition~3.1]{ayoub-barbieri}. By \cite[Proposition~3.2]{ayoub-barbieri} it suffices to deal with the case when $X$ has  dimension at most $n$. The period version of the argument for $n=1$,  which is the case of our interest    is given in Proposition~\ref{prop:reduce_C}.
\end{rem}

For $n=0$, the category $d_0\MMN(k,\Q)$ is the category of Artin motives; see
\cite[Theorem~4.3]{ayoub-barbieri}. 
The following theorem discusses the case $n=1$, which is of direct relevance for us.

\begin{thm}[Ayoub and Barbieri-Viale {\cite[Sections~5, 6]{ayoub-barbieri}}]\label{thm:abv} The following hold.
\begin{enumerate} 
\item
The inclusion 
\[ d_1\MMNeff(k,\Q)\to \MMNeff(k,\Q)\]
 has a left-adjoint. 
\item There is an anti-equivalence of categories 
\[ \onemot_k\to d_1\MMNeff(k,\Q).\]
\item The abelian category $d_1\MMNeff(k,\Q)$ can be described as the diagram category  in the sense of Nori  (see \cite[Theorem~7.1.13]{period-buch}),
defined by
the diagram with vertices $(C,D)$, where $C$ is a smooth affine curve and $D$ a collection of points on $C$, and edges given by morphisms of pairs together with singular cohomology as a representation.
\end{enumerate}
\end{thm}

\section{Non-effective Motives}\label{sec:non-eff}

The full category of motives is constructed from the category of effective motives by inverting the  \emph{Lefschetz motive}
 \[ \Q(-1)=H^2_\Nori(\Pe^1)=H^*_\Nori(\Gm,\{1\}).\]
We explain the construction.

\begin{rem}
Nori constructs a tensor structure on the category $\MMNeff(k,\Q)$. However, the resulting tensor category is not rigid. This defect can resolved by passing to $\MMN(k,\Q)$, which turns out to be rigid. 
We do not need the tensor structure, so we do not go into details; instead see \cite[Section~9.3]{period-buch}.
\end{rem}

The map of diagrams
\[ \pairseff\to\pairseff\]
given by
\[ (X,D,i)\mapsto (X\times\Gm,X\times \{1\}\amalg D\times\Gm,i+1)\]
is compatible with the singular realisation because
\begin{multline*}
 H_\sing^{i+1}(X\times\Gm,X\times\{1\}\amalg D\times\Gm;\Q)\\
\isom H_\sing^i(X,D;\Q)\tensor H^1(\Gm,\{1\};\Q)\isom H^i_\sing(X,D;\Q).
\end{multline*}
Note that this stupid version of the Künneth formula actually holds true because the whole cohomology of $(\Gm,\{1\})$ is concentrated in degree $1$. The datum $(\Gm,\{1,\},1)$ is what Nori calls a \emph{good pair}. 

By the universal property of  the category of effective Nori motives, this induces a faithful exact functor
\[ (-1):\MMNeff(k,\Q)\to\MMNeff(k,\Q),\]
the \emph{Tate twist}.

\begin{defn}
The category of \emph{Nori motives}\index{Nori motives!non-effective} $\MMN(k,\Q)$ over $k$ with coefficients in 
$\Q$ is defined as the localisation  of $\MMNeff(k,\Q)$ with respect to
the twist functor $(-1)$: objects in $\MMN(k,\Q)$ are of the form
$M(i)$ for $M\in\MMNeff(k,\Q)$ and $i\in\Z$ and
\[ \Hom_{\MMN(k,\Q)}( M(i),N(j))=\lim_{n\to\infty}\Hom_{\MMNeff(k,\Q)}(M(i+n),N(j+n))\]
and
 $(-1)$ induces an equivalence of categories
\[ (-1):\MMN(k,\Q)\to\MMN(k,\Q).\]
\end{defn}
See \cite[Section~8.2]{period-buch} for a construction of the same localisation in terms of diagrams. The natural functor
\[ \MMNeff(k,\Q)\to\MMN(k,\Q)\]
is faithful because both categories are equipped with forgetful functors into the category of $\Q$-vector spaces.. However, we do not know if it is full.

\begin{rem}
It is also an open question whether the inclusion 
\[ d_1\MMNeff(k,\Q)\hookrightarrow \MMN(k,\Q)\]
 into the category of all motives is full. We give a positive answer for
$k=\Qbar$ in Theorem~\ref{thm:ff}.
\end{rem}

\section{The Period Conjecture}
\index{Period Conjecture!for Nori motives}

In this section we restrict to $k=\Qbar$.

The formalism of Chapter~\ref{ch:formalism} can be applied to
the  diagram $\pairseff$ of Definition~\ref{defn:pairseff} or the additive category $\MMNeff(\Qbar,\Q)$ and the representation/functor $H$ with values
in $\VVarg{\Qbar}{\Q}$. We make this explicit.

 The periods of $\pairseff$ are the period numbers 
\[ \Per(\pairseff)=\bigcup_{i=0}^\infty\Per^i=:\Per^\eff\]
 of
Definition~\ref{defn:coh_periods}
On the other hand, we have
\[ \Per(\pairseff)=\Per(\MMNeff(\Qbar,\Q))\]
because any object of $\MMNeff(\Qbar,\Q)$ is a subquotient of an object of the form $H^i_\Nori(X,D)$ see Remark~\ref{rem:nach_Nori} (\ref{it:subquotient_Nori}).
In Definition~\ref{defn:formal_abstract}, we also introduced the notion of a vector space of formal periods attached to an additive category or a diagram. The notion aims at a description of actual period spaces in terms of generators and relations. A priori, passing from the diagram $\pairseff$ to the abelian category $\MMNeff(\Qbar,\Q)$ might introduce new generators and more relations. However, this is not the case. 
We also have
\[ \Perform(\pairseff)=\Perform(\MMNeff(\Qbar,\Q));\]
this holds true for any representation of a diagram and its diagram category. 
The fact is implicit in \cite{period-buch}; see \cite[Theorem~3.7]{huber_galois}
for full details. 
Conjecture~\ref{conj:kontsevich_full} is the 
Period Conjecture for $\pairseff$ in the sense of Definition~\ref{defn:conj_abstract}, i.e.  the question about  injectivity of the map 
\[ \Perform(\pairseff)\to \Per^\eff. \]
Up to a minor changes the conjecture was formulated by Kontsevich in \cite{kontsevich}. We refer the reader to in \cite[Remark~13.1.8]{period-buch}
for a detailed discussion. Injectivity is equivalent to the Period Conjecture for
effective Nori motives over $\Qbar$, i.e. the injectivity of
\[ \Perform(\MMNeff(\Qbar,\Q))\to \Per^\eff.\]
By Corollary~\ref{cor:fullness_abstract}, the Period Conjecture implies fullness of $\MMNeff(\Qbar,\Q)$.
By Lemma~\ref{lem:conj_single}, the Period Conjecture for $\MMNeff(\Qbar,\Q)$ is equivalent to the Period Conjecture for $\langle M\rangle $ for all $M\in\MMNeff(\Qbar,\Q)$. Here, as in Definition~\ref{defn:single}, we denote by $\langle M\rangle\subset\MMNeff(\Qbar,\Q)$ the smallest full subcategory containing $X$ which is closed under subquotients. Moreover, the Period Conjecture for $M$ can be reformulated as asking for
\[ \dim_\Qbar\Per\langle M\rangle=\dim_\Q E(M)\]
with
\[ E(M):=\End(H_\sing|_{\langle M\rangle}),\]
as in Definition~\ref{defn:End}.
Unconditionally, we get the estimate
\[ \dim_\Qbar\Per\langle M\rangle\leq \dim_\Q E(M).\]

In Section~\ref{sec:filt} we have introduced the filtration of $\MMNeff(\Qbar,\Q)$ by degree.

\begin{cor}
The Period Conjecture for all motives is equivalent to the Period Conjecture for
$d_{n}\MMNeff(\Qbar,\Q)$ for all $n\geq 0$.
\end{cor}
\begin{proof}If $M\in d_{n}\MMNeff(\Qbar,\Q)$, then even
$\langle M\rangle\subset\MMNeff(\Qbar,\Q)$. Moreover, every object of
$\MMNeff(\Qbar,\Q)$ is contained in some $d_n\MMNeff(\Qbar,\Q)$. If the Period Conjecture holds for $\MMNeff(\Qbar,\Q)$, then it holds for all
$\langle M\rangle$ and hence for all $d_n\MMNeff(\Qbar,\Q)$, and conversely.
\end{proof}

One of the main results of the present monograph is the validity of
the Period Conjecture for $d_1\MMNeff(\Qbar,\Q)$; see Theorem~\ref{thm:main_kontsevich}. This implies the dimension formula for all $1$-motives; see Corollary~\ref{cor:dim_easy}.

The Period Conjecture can also be formulated for non-effective Nori motives, see Section~\ref{sec:non-eff}.
We have 
\[ \Per(\MMN(\Qbar,\Q))= \Per^\eff[1/2\pi i]\]
because $2\pi i$ is the period of $\Q(-1)$. As we do not know if
$\MMNeff(\Qbar,\Q)\to \MMN(\Qbar,\Q)$ is full, it is also an open question whether
$\Perform(\MMNeff(\Qbar,\Q))\to \Perform(\MMN(\Qbar,\Q))$ is injective.
By Proposition~\ref{prop:is_full} this injectivity is a consequence of the Period Conjecture for $\MMN(\Qbar,\Q)$. We deduce it for the category $d_1\MMN(\Qbar,\Q)$ in
Theorem~\ref{thm:ff}.

\chapter{Voevodsky Motives}\label{sec:voe}
\index{geometric motives}\index{motive!geometric}

An alternative approach to the theory of motives starts out with a triangulated category of motives, thought of as the bounded derived category of motives. We use Voevodksy's approach. His category and the results known about its relation to $1$-motives are only used in the proof of Theorem~\ref{thm:main_kontsevich} (2).
We refer to Voevodsky's survey in \cite{voevodsky-motives} for explicit constructions and proofs.

\section{Geometric Motives and the Singular Realisation}

Let $k$ be a field of characteristic $0$. We work with $\Q$-coefficients throughout. This assumption makes most of subtleties of \cite{voevodsky-motives}, like finite correspondences and the distinction between the Nisnevich and the \'etale topology unnecessary. There are alternative constructions by different authors which yield the same result. For  example, Ayoub's category $\DA_c^\eff(k,\Q)$ built with the etale topology and without transfers  (see \cite{ayoub_etale}) is equivalent to Voevodksy's $\DMgmeff(k,\Q)$ built with the Nisnevich topology and correspondences. 

\begin{defn}[{\cite[Definition~2.1.1]{voevodsky-motives}}]\label{defn:geom_motives}
We denote by $\DMgmeff(k,\Q)$ the triangulated category of \emph{effective geometric motives} and by
\[ M:\mathrm{Var}_k\to \DMgmeff(k,\Q) \]
the covariant functor which attaches to an algebraic $k$-variety its \emph{(Voevodsky) motive}. 
\end{defn}

Singular cohomology \index{singular cohomology}\index{singular realisation!of a geometric motive}extends to a contravariant functor
\[ H_\sing:\DMgmeff(k,\Q)\to \Q\Vect\]
such that
\[ H_\sing(M(X)[i])=H^i_\sing(X,\Q).\]
This is a consequence and easy special case of \cite{Hreal, Hreal2}.

\begin{rem}
From the point of view of Voevodsky motives as well as from the point of view of $1$-motives, it can be argued that it would be more natural to work with singular \emph{homology} instead. Indeed, this is the point of view taken by Ayoub in his construction of the Betti realisation not only of motives over a field, but over general base; see \cite{ayoub-betti}. On the other hand, de Rham cohomology is a lot more natural than de Rham homology. 
We stick to the conventions set up in \cite{period-buch}.
\end{rem}

Let $\pairseff$ be the diagram considered in Definition~\ref{defn:pairseff} in order to define Nori motives.

\begin{lemma}
The functor $M$ extends to a representation
\[ M:\pairseff\to \DMgm(k,\Q)\]
such that $H_\sing\circ M$  agrees with the singular realisation of
Section~\ref{sec:eff}.
\end{lemma}
\begin{proof}
Let $(X,Y,i)$ be a vertex of $\pairseff$, i.e.  $X$ an algebraic variety, $Y$ a closed subvariety and $i\geq 0$. We define
\[ M(X,Y,i)=M(Y\to X)[i]\]
where $M(Y\to X)$ is the object of $\DMgmeff(k,\Q)$ corresponding to
the bounded complex $ [Y\to X]$ in the additive category $\mathrm{SmCor}$ of finite correspondences; see \cite[Section~2.1]{voevodsky-motives}. Its singular realisation is $H^i_\sing(X,Y;\Q)$.

A morphism $f:X\to X'$ mapping $Y$ to $Y'$ induces a natural morphism $f^*:M(X',Y',i)\to M(X,Y,i)$. A triple $Z\subset Y\subset X$ induces
\[ \partial: M(X,Y,i)\to M(Y,Z,i+1).\]
Both are mapped to the correct map on singular cohomology.
\end{proof}

\section{Filtration by Dimension}

\begin{defn}[{\cite[Section~3.4]{voevodsky-motives}}]\label{defn:dim_filt}
\index{dimension filtration for geometric motives}\index{filtration by dimension for geometric motives}
For $n\geq 0$ let
$d_n\DMgmeff(k, \Q)$ be the full thick subcategory generated by the motives of the form  $M (X)$ for  a smooth variety $X$ of dimension at most $n$.
\end{defn}

The case $n=0$ is easy. Voevodsky showed that $d_0\DMgmeff(k,\Q)$ is equivalent to the bounded derived category of finite dimensional continuous representations of $\Gal(\bar{k}/k)$. If $k$ is algebraically closed, this is simply the
bounded derived category of $\Q\Vect$.

\begin{thm}[Orgogozo \cite{orgogozo}, Barbieri-Viale--Kahn \cite{BVK}]\label{thm:orgo_app}
There is a natural equivalence of triangulated categories
\[ D^b(\onemot_k)\to d_1\DMgm(k,\Q)\]
from the derived category $D^b(\onemot_k)$ of the abelian category of iso-$1$-motives to 
$d_1\DMgm(k,\Q)$.
The inclusion $d_1\DMgm(k,\Q)\to \DMgmeff(k,\Q)$ has a left adjoint which is a section.
\end{thm}

For $n\geq 2$, the subcategories $d_n\DMgmeff(k,\Q)$ remain mysterious.

\section{Relation to Nori Motives}

We have now seen two approaches to a theory of motives. They are related.
\begin{thm}[Nori, Harrer \cite{harrer}]\label{thm:harrer}
\index{Nori motives}
There is a triangulated functor
\[ \DMgmeff(k,\Q)\to D^b(\MMNeff(k,\Q))\]
between triangulated categories compatible with the singular realisation
 into the derived category of
$\Q$-vector spaces. 
\end{thm}
\begin{proof} The existence of the functor is due to Harrer, based on a construction of Nori. 
\end{proof}

\begin{prop}\label{prop:harrer_dim}
The functor of Theorem~\ref{thm:harrer} maps $d_n\DMgmeff(k,\Q)$ to 
the full subcategory of $D^b(\MMNeff(k,\Q))$ consisting
of objects with cohomology in $d_n\MMNeff(k,\Q)$.
\end{prop}
\begin{proof}
As $d_n\MMNeff(k,\Q)$ is closed under subquotients and extensions, it suffices to check the claim for a system of generators for $d_n\DMgmeff(k,\Q)$ as a triangulated category.
We first do the case $n=1$, which is of most significance for
us. 

Let $C$ be a connected smooth variety of dimension $0$ or $1$ and $M(C)$ be the
corresponding object in $\DMgm(k,\Q)$. Its image in $D^b(\MMN(k,\Q))$ has
cohomology in degrees at most $2$. Cohomology in degree $2$ only occurs
if $C$ is a smooth proper curve. In this case
\[ H^2_\Nori(C)\isom H^2_\Nori(\Pe^1)\isom H^1_\Nori(\Gm),\]
hence it is also in $d_1\MMN(k, \Q)$.

In the general case, we need to consider
$M(X)$ with $X$ a smooth variety of dimension at most $n$. It remains 
to show that $H^i_\Nori(X)$ is in $d_n\MMNeff(k,\Q)$.
By the Mayer-Vietoris property we may assume that $X$ is affine. 
We then follow the construction of $H^i_\Nori(X)$ in \cite{period-buch}.
We choose a \emph{good filtration}\index{good filtration} 
\[ X_0\subset X_1\subset \dots\subset X_n=X\]
by closed subvarieties such that $H^i_\sing(X_j,X_{j-1};\Q)$ is concentrated in
degree $j$. (The existence of such a filtration is guaranteed by Nori's Basic Lemma; see~\cite[Proposition~9.2.3]{period-buch}.) This yields the complex
\[ H^0_\Nori(X_0)\to H^1_\Nori(X_1,X_0)\to\dots\to H^n_\Nori(X_n,X_{n-1}),\]
whose cohomology is equal to $H^*_\Nori(X)$. As the complex is in
$d_n\MMNeff(k,\Q)$, so is its cohomology.
\end{proof}


\chapter{Comparison of Realisations}\label{ch:comp_real}

In this appendix, we identify Deligne's explicit realisations of $1$-motives
with realisations of motives of algebraic varieties. 

We work over an algebraically closed field $k$ with a fixed embedding into $\C$.
Recall the functor
\[ V:\onemot_k\to \MHS_k\to \VVarg{\Q}{k}.\]

It maps the iso-$1$-motive $M=[L\to G]$ to the triple consisting of its singular realisation $V_\sing(M)$, its de Rham realisation $V_\dR(M)$ and the period isomorphism.
We refer to Chapter~\ref{sec:one-mot} for their construction.\index{period realisation!of a $1$-motives}\index{singular realisation!of a $1$-motive}\index{de Rham realisation!of a $1$-motive}

On the other hand, there is a functor
\[ (X,Y)\mapsto H^1(X,Y)\in\VVarg{k}{\Q}\]
for pairs \index{singular realisation!of a Nori motive}of a variety $X$ and a closed subvariety. A complete reference for its construction is \cite{period-buch} and this is what we are going to rely on. We have already discussed it in Appendix~\ref{sec:app_nori} in the context of Nori motives. Alternatively, we can also construct it from the theory of Voevodsky motives, see Appendix~\ref{sec:voe}: every morphism
 $Y\to X$ of smooth $k$-varieties gives rise to
an object $M(Y\to X)\in\DMgm(k)$. Let $H^0:\DMgm(k)\to\VVarg{k}{\Q}$  be the standard cohomological functor of \cite{Hreal, Hreal2} and $H^1=H^0\circ [1]$. By composition
we define
\[ H^1(Y\to X)\in \VVarg{k}{\Q}. \]
Its de Rham component is defined explicitly as  $H^1$ of
\[ \RGamma_\dR(X,Y):=\cone(\RGamma_\dR(Y)\to \RGamma_\dR(X))[-1]\]
and its singular component as $H^1$ of
\[ \RGamma_\sing(X,Y):=\cone(\RGamma_\sing(Y)\to \RGamma_\sing(X))[-1],\]
where $\RGamma_\dR$ and $\RGamma_\sing$ are functorial complexes computing de Rham and singular cohomology respectively; see \cite[Section 3.3.3, Section~5.5]{period-buch}. They are connected by a period isomorphism, see \cite[Corollary~5.52]{period-buch}. In the special case $Y\subset X$, we get back relative cohomology.

\section{The de Rham Realisation Revisited}\label{ssec:de_rham}

In Chapter~\ref{ch:homology} we gave an explicit definition of relative algebraic de Rham cohomology in the smooth case. We need to relate it to the one in our main reference \cite{period-buch}
\index{de Rham cohomology}\index{de Rham realisation!of a pair}

\begin{lemma}
Let $X$ be a smooth variety and $\Uf$ a finite open affine cover of $X$. Then there is a natural isomorphism of complexes in the derived category
\[ \RGammatilde_\dR(X,\Uf)\to \RGamma_\dR(X),\]
functorial for affine maps.
If $Y\subset X$ is a smooth closed subvariety, then 
\[ \RGammatilde_\dR(X,Y,\Uf):=\cone(\RGamma_\dR(Y,\Uf\cap Y)\to \RGamma_\dR(X,\Uf))[-1]\]
is naturally isomorphic to $\RGamma_\dR(X,Y)$.
\end{lemma}
\begin{proof}
The definition of $R\Gamma_\dR$ for complexes of smooth varieties
is given in \cite[Definition~3.3.1]{period-buch}. For a single smooth
variety, by \cite[Definition~3.3.14]{period-buch}, it is given as global sections
of the Godement resolution (\cite[Section~1.4.2]{period-buch})
\[ \RGamma_\dR(X)=\Gd_X\Omega^*_X(X).\]
The \v{C}ech-complex 
\[ \mathrm{tot}(C^*(\Uf,\Gd_X\Omega^*_X))\] 
 receives natural quasi-isomorphisms both from
$\RGammatilde(X, \Uf)$ (because the natural map
$\Omega^*_X\to\Gd_X\Omega^*_X$ is quasi-isomorphism of complexes of sheaves)
and from $\RGamma_\dR(X)$ (because the cover $\Uf$ refines
$X$). Together they define an isomorphism in the derived category.

The construction extends to complexes $Y\to X$.
\end{proof}

\section{The Comparison Result}

In Definition~\ref{defn:external}, we introduced the external duality functor
\[ \cdot^\vee:\VVarg{\Q}{k}\to \VVarg{k}{\Q}\]
mapping the triple $(V_k,V_\Q,\phi)$ to $(V_\Q^\vee, V_k^\vee,\phi^\vee)$.

\begin{prop}\label{prop:compare}
Let $M=[L\xrightarrow{f} G]$ be a $1$-motive and $e_1,\dots,e_r$ be a basis of
$L$ and put $e_0=0$. We put $Z=\coprod_{i=0}^r\Spec(k)=\{P_0,\dots,P_r\}$
and $\tilde{f}:Z\to G$ given by $\tilde{f}(P_i)=f(e_i)$.
Then\index{$1$-motive!comparison of realisations} 
\[ V(M)^\vee\isom H^1(Z\xrightarrow{\tilde{f}}G).\]
\end{prop}
The proof will take the rest of the appendix. Before going into it, we want to record a consequence. Let $M=[L\to G]$ be a $1$-motive. The assignment
\[ S\mapsto \left[L\tensor\Q\to G(S)\tensor_\Z\Q\right]\]
defines a complex of homotopy invariant Nisnevich sheaves with transfers on the category $\Sm/k$ of smooth $k$-varieties; see \cite[Lemma~2.1.2]{AEH} building on work of Spie\ss-Szamuely and Orgogozo. 
 Hence it defines
an object  $\uL\to\uG$ in Voevodsky's category of motivic complexes with $\uL$ in degree $0$. In fact, it is even
an object of the full subcategory $\DMgm(k,\Q)$; see \cite[Proposition~5.2.1]{AEH}. We denote it $M_\gm(\uL\to \uG)$.

\begin{cor}\label{cor:one_vs_geom}
\index{geometric motives!vs $1$-motives}
Let $M=[L\to G]$ be a $1$-motive over $k$. Then
\[ V(M)^\vee\isom H^*(M_\gm(\uL\to \uG))=H^0(M_\gm(\uL\to \uG)).\]
\end{cor}
\begin{proof} 
The equality on the right is \cite[Proposition~7.2.3]{AEH}. 
In order to compare $V(M)^\vee$ and $H^0(M_\gm(\uL\to \uG))$ it suffices to give a natural isomorphism
\[ H^0(M_\gm(\uL\to\uG))\to H^1(G,Z)\]
with $Z$ as in Proposition~\ref{prop:compare}. The main result of
\cite{AEH} is to describe $\uG$ as a direct summand of $M_\gm(G)$. Its cohomology
in any contravariant Weil cohomology theory (e.g. de Rham or singular cohomology) agrees
with $H^1(G)$. In particular,
\[ H^*(M_\gm(\uG))=H^1(G).\]
The projection map
$M(G)\to \uG$ is given by the summation map 
\[ M(G)(S)=\Cor(S,G)\to G(S).\]
On the other hand, $\uL=M(Z')$ with $Z'=\{P_1,\dots,P_r\}$. There is 
a natural projection $M(Z)\to M(Z')$.

Our comparison isomorphism is induced by the morphism
of motivic complexes
\[ [M(Z)\to M(G)]\to [M(Z')\to \uG].\]
We apply the long exact cohomology sequence for the stupid filtration:
\[\begin{xy}\xymatrix{
H^0(G)\ar[r]&H^0(Z)\ar[r]& H^1(G,Z)\ar[r] &H^1(G)\ar[r]&0\\
0\ar[r]&H^0(Z')\ar[u]\ar[r]&H^0(M_\gm(\uL\to \uG))\ar[u]\ar[r]&H^1(G)\ar[u]^\id \ar[r]&0
}\end{xy}\]
The vertical map on the left induces an isomorphism 
\[ \Q^r\isom H^0(Z')\to H^0(Z)/H^0(G)=\Q^{r+1}/\Delta(\Q).\]
Together with the identity on the right this induces an isomorphism in the middle. 
\end{proof}

We now start on the proof of Proposition~\ref{prop:compare}.

\begin{lemma}\label{lem:lem1}\index{de Rham realisation!of a $1$-motive}
Let $M=[L\to G]\in\onemot_k$ and $Z\to G$ as in Proposition~\ref{prop:compare}.
Then
\[ V_\dR(M)^\vee\isom H^1_\dR(Z\to G).\]
\end{lemma}
\begin{proof}
We refer the reader to Chapter~\ref{sec:one-mot} for the construction of the algebraic group $M^\natural\to G$ and the realisations.
Note that $\RGamma_\dR(Z)= H^0(Z,\Oh_Z)\isom k^{r+1}$ because $Z$ is of dimension zero. Hence
\[ \RGamma_\dR(G,Z)^i=\begin{cases}R\Gamma_\dR(G)^i&i\neq 1\\
              \RGamma_\dR(G)^1\oplus H^0(Z,\Oh_Z) &i=1
  \end{cases}\]
We claim that there is a natural map in the derived category\[ V_\dR(M)^\vee[-1]=\coLie(M^\natural)[-1]\to \RGamma_\dR(G,Z).\]
It is induced by the composition
\[ \coLie(M^\natural)\to \Omega^1(M^\natural)\to \RGamma_\dR(M^\natural)^1
\leftarrow \RGamma_\dR(G)^1,\]
where the maps are
\begin{itemize}
\item extension of an element of the cotangent space to a unique equivariant differential form (it is closed because the co-Lie bracket is trivial);
\item association of a section of the de Rham complex to a section of its Godement resolution;
\item functoriality of $\RGamma_\dR$ (a quasi-isomorphism by homotopy invariance)
\end{itemize}
together with the map
\[ \coLie(M^\natural)\xrightarrow{0}H^0(Z,\Oh_Z).\]
We compare the long exact sequence of the cone with the exact sequence for
$V_\dR(M)$: 
{\small
\[
\begin{xy}\xymatrix{
0\ar[r]& H^0_\dR(G)\ar[r]& H^0_\dR(Z)\ar[r]& H^1_\dR(G,Z)\ar[r]& H^1_\dR(G)\ar[r]& 0\\
 &0\ar[r]&\Hom(L,\Ga)\ar[r]\ar@{.>}[u]&\coLie(M^\natural)\ar[r]\ar[u]&\coLie(G^\natural)\ar[r]\ar[u]&0
}\end{xy}\]
}
The square on the right commutes by naturality. The dotted arrow does not exist as a natural map, but we get an induced map
\[ \Hom(L,\Ga)\to H^0_\dR(Z)/H^0_\dR(G).\]
We make it
explicit. Pick $\alpha:L\to\Ga$. It defines a differential
form $\omega(\alpha)$ on the algebraic group 
$V=\Ga(\Hom(L,\Ga)^\vee)$. In coordinates: let $t_i$ be the coordinate on $V$
corresponding to the basis vector $e_i$ of $L$. Then 
\[ \omega(\alpha)=\sum_{i=1}^r \alpha(e_i)dt_i.\]
We have an exact sequence
\[ 0\to G^\natural\to M^\natural\to V\to 0,\]
hence we can pull $\omega(\alpha)$ back to $M^\natural$. This defines
a class in $H^1_\dR(G,Z)$. Its image in $H^1_\dR(G)$ vanishes by construction.
This implies that the class is exact. Indeed,
\[ \omega(\alpha)=d\left(\sum_{i=1}^r\alpha(e_i)t_i+c\right)\]
for any $c\in k$. We lift $Z\to G$ to $Z\to M^\natural$ by mapping
$P_i$ to the image of $e_i$ in $M^\natural$. We get
a class in $H^0_\dR(Z)=H^0(Z, \Oh_Z)$  by restricting our
function $\sum_i\alpha(e_i)t_i+c$ to $Z$:
\[ P_i\mapsto \alpha(e_i)+c.\]
(Note that $\alpha(e_0)=\alpha(0)=0$.) The equivalence class in
$H^0_\dR(Z)/H^0_\dR(G)$ is independent of $c$ and hence well-defined. 

Now that we have an explicit formula, it is obvious that the map is bijective.
Hence it remains to show that $\coLie(G^\natural)\to H^1_\dR(G)$ is an isomorphism. Both $\coLie$ and $H^1_\dR$ are exact functors on the category of semi-abelian varieties, hence it suffices to consider the cases $G=\Gm$ and $G=A$ abelian.
In the first case, $\Gm^\natural=\G_m$ and the invariant differential 
$dt/t$ is known to generate $H^1_\dR(\Gm)$.

Let $G=A$ be an abelian variety. In this case $\Omega^1(A)=\Omega^1(A)^A\isom\coLie(A)$. By Hodge theory, we have the short exact sequence
\[ 0\to \Omega^1(A)\to H^1_\dR(A)\to H^1(A,\Oh_A)\to 0.\]
The last group also identifies with $\Ext^1(A,\Ga)$. The exact sequence is
compatible with the sequence
\[ 0\to\coLie(A)\to\coLie(A^\natural)\to \Ext^1(A,\Ga)\to 0.\]
Hence $\coLie(A^\natural)\to H^1_\dR(A)$ is an isomorphism as well.
\end{proof}

\begin{lemma}\label{lem:lem2}
\index{singular realisation!of a $1$-motive}
Let $M=[L\to G]\in\onemot_k$ and $Z\to G$ as in Proposition~\ref{prop:compare}.
Then
\[ V_\sing(M)^\vee\isom H^1_\sing(Z\to G).\]
\end{lemma}
\begin{proof}
It is more natural to give the argument in terms of homology. We use the description via $C^\infty$-chains; see Remark~\ref{rem:smooth_chains} or \cite[Definition~2.2.2]{period-buch}. We work with integral coefficients throughout and omit them from the notation 

Recall that, by construction, $f:L\to G$ has an injective lift $f^\natural:L\to M^\natural$. By abuse of notation, the map $P_i\mapsto f^\natural(e_i)$ is also denoted $\tilde{f}:Z\to M^\natural$.
 By homotopy invariance, $H_1^\sing(Z\to G)\isom H_1^\sing(Z\to M^\natural)$. From now on we work with the latter.  

Given an algebraic variety $X/k$ let
$S^\infty_*(X)$ be the chain complex of $C^\infty$-chains on $X^\an$ with integral coefficients. For a morphism $f:Y\to X$, we then define
$H_i^\sing(X,Y)$ as $H_i$ of the complex
\[ S^\infty_*(X,Y):=\cone( S^\infty_*(Y)\to S^\infty_*(X)).\] 
As $Z$ is a disjoint union of points,
$S^\infty_n(Z)=\Z[Z]$ (linear combinations of points) for all $n\geq 0$ and
the map of complexes
\[ S^\infty_*(Z)\to \Z[Z][0]\] 
is a quasi-isomorphism, where as usual $\Z[Z][0]$ denotes the complex concentrated in degree $0$. We define a map of complexes
\[ S^\infty_*(M^\natural)\to [\Lie(M^{\natural,\an})\to M^{\natural,\an}]\]
as follows: 
\begin{itemize}
\item in degree $1$ a path $\gamma:[0,1]\to M^{\natural,\an}$ is mapped to 
$I(\gamma)\in\Lie(M^{\natural,\an})$ (see Section~\ref{ssec:paths});
\item in degree $0$ a formal sum in $\Z[M^{\natural,\an}]$ is mapped to its
sum in $M^{\natural,\an}$.
\end{itemize}
This is compatible with the differential because $\exp(I(\gamma))=\gamma(1)-\gamma(0)$. 
It is easy to see that the diagram of complexes
\[\begin{xy}\xymatrix{
S^\infty_*(Z)\ar[r]\ar[d]&S^\infty_*(M^\natural)\ar[d]\\
L[0]\ar[r]&[\Lie(M^{\natural,\an})\to M^{\natural,\an}]
}\end{xy}\]
commutes. Hence we get morphisms of complexes
\begin{align*}
 S^\infty_*(M^\natural,Z)&\to \cone\left(L[0]\to [\Lie(M^{\natural,\an})\to M^{\natural,\an}]\right)\\
&=[L\oplus \Lie(M^{\natural,\an})\to M^{\natural,\an}]\\
&\to [\Lie(M^{\natural,\an})\to M^{\natural,\an}/f^\natural(L)].
\end{align*}
By definition
\[ T_\sing(M)=\ker(\Lie(M^{\natural,\an})\to M^{\natural,\an}/f^{\natural}(L))\]hence we have defined a natural morphism in the derived category
\[ S^\infty_*(M^\natural,Z)\to T_\sing(M)[-1].\]

We compare the long exact sequence of the cone with the exact sequence for
$V_\sing(M)$:
\[
\begin{xy}\xymatrix{
0\ar[r]& H_1^\sing(M^\natural)\ar[r]\ar@{.>}[d]& H_1^\sing(Z\to M^{\natural})\ar[r]\ar[d]& H_0^\sing(Z)\ar@{->>}[r]\ar@{.>}[d]&H_0^\sing(M^\natural) \\
 0\ar[r]&T_\sing(G)\ar[r]&T_\sing(M)\ar[r]&L\ar[r]&0
}\end{xy}\]
We make the dotted maps explicit. The elements of $H_1^\sing(M^\natural)$ are represented
by closed loops in $M^{\natural,\an}$. Let $\gamma$ be such a closed loop.
Then the end points of $\tilde{\gamma}$ have the same image in $M^{\natural,\an}$, hence $\tilde{\gamma}(1)-\tilde{\gamma}(0)\in \ker(\Lie(M^{\natural,\an})\to M^{\natural,\an})=T_\sing(G)$. The dotted map on the left is an isomorphism.

The kernel of $H_0^\sing(Z)\to H_0^\sing(G)$ is generated by formal
differences $P_i-P_0$. Choose a path $\gamma_i$ from $\tilde{f}(P_0)=0$ to 
$\tilde{f}(P_i)$ in
$M^{\natural,\an}$. Then 
\[ (-\gamma_i,P_i-P_0)\in S^\infty_1(M^\natural,Z)\]
is
in the kernel of boundary map. Hence its cohomology class is the preimage of
$P_i-P_0$. Its image in $\Lie(M^{\natural,\an})$ is
given by $\tilde{\gamma}_i(0)-\tilde{\gamma}_i(1)$ for a lift
$\tilde{\gamma}_i$ of $\gamma_i$.  We may choose $\tilde{\gamma}_i(0)=0$, then
$-\tilde{\gamma}_i(1)$ is in the preimage of $\tilde{f}(P_i)$. (Which preimage
depends on the choice of $\gamma_i$. It is only unique
up to $2$-chains.) 
Its equivalence class modulo $T_\sing(G)$ is nothing but
the image in $M^{\natural,\an}$, hence $\tilde{f}(P_i)=f^\natural(e_i)$. 
The map $\ker(H_0(Z)\to H_0(M^\natural))\to L$ is also bijective.

This finishes the proof.
\end{proof}

\begin{lemma}\label{lem:lem3}
\index{period isomorphism!for $1$-motives}
The comparison maps for de Rham and singular cohomology are compatible with
the period isomorphism.
\end{lemma}
\begin{proof}
In keeping with the proof of the last lemma, we prefer to check 
compatibility with the period pairing:
\[\begin{xy}\xymatrix{
H^1_\dR(G,Z)\times H_1^\sing(G,Z)\ar[r]&\C\\
V_\dR(M)^\vee\times T_\sing(M)\ar[u]\ar[ru]
}\end{xy}\]
A priori, the pairing on the $1$-motive level is simply evaluation of an
element of $\coLie(M^\natural)$ on $\sigma\in \Lie(M)_\C$. However, we have
already made the translation to the integration of equivariant 
differential forms on $M^{\natural,\an}$ along paths.

We now view the same differential form as a class in de Rham cohomology and the path as a class in singular homology. The 
 cohomological version of the period pairing is also given by integration in these special cases. 
\end{proof}

\begin{proof}[Proof of Proposition~\ref{prop:compare}]
Combine Lemmas \ref{lem:lem1}, \ref{lem:lem2} and \ref{lem:lem3}
\end{proof}

\end{appendix}

\chapter*{List of Notations}\label{sec:not}

\subsection*{General}\ \\

\begin{tabular}{p{3cm}l}
 $k$&  algebraically closed field with a fixed embedding $k\to\C$\\
$V^\vee$ &dual vector space\\
$\langle \cdot,\cdot\rangle$&natural pairing between a vector space and its dual\\
$X^\an$&analytic space attached to an algebraic variety over $k$\\
$G^\an$&complex Lie group attached to an algebraic group over $k$\\
$\mathrm{Var}_k$&varieties over $k$, i.e. reduced $k$-schemes of finite type\\
$\Sm_k$&smooth varieties over $k$\\
$D^b(\Ah)$&bounded derived category of an abelian category\\
\end{tabular}

\subsection*{Chapter 2}\ \\

\begin{tabular}{p{3cm}l}
$\Ah\tensor\Q$&isogeny category of an additive category $\Ah$\\
$\Z[\Ch]$&additive hull of a category\\
\end{tabular}

\subsection*{Chapter 3}\ \\

\begin{tabular}{p{3cm}l}
$\Delta_n$&standard $n$-dimensional simplex\\
$S_n(X)$& space of singular $n$-chains\\
$H_n^\sing(X,\Q)$&singular homology\\
$H_n^\sing(X,Y;\Q)$&relative singular homology\\
$H^n_\dR(X)$&algebraic de Rham cohomology\\
$H^n_\dR(X,Y)$&relative algebraic de Rham cohomology\\
$\RGammatilde_\dR(X,\Uf)$&explicit complex computing de Rham cohomology\\
\end{tabular}

\subsection*{Chapter 4}\ \\

\begin{tabular}{p{3cm}l}
$\grp$&category of commutative connected algebraic groups over $k$\\
$\Ga$&additive group\\
$\Gm$&multiplicative group\\
$A^\vee$&dual abelian variety\\
$X(T)$&character group of a torus\\
$\Gm(\Xi)$& dual torus\\
$\Ext^1(A,B)$&Yoneda-extension group\\
$G^\natural$&universal vector extension of a semi-abelian variety\\
$J(Y)$&generalised Jacobian of a smooth curve\\
\end{tabular}

\subsection*{Chapter 5}\ \\

\begin{tabular}{p{3cm}l}
$\Lie(G)$, $\gg$&Lie algebra of an algebraic group or a complex Lie group\\
$\exp_G$&exponential map of commutative complex Lie group\\
$I(\gamma)$&path integral as inverse of the exponential map\\
$\log_G$ &multi-valued inverse of $\exp_G$\\
\end{tabular}

\subsection*{Chapter 6}\ \\

\begin{tabular}{p{3cm}l}
$\Ann(u)$&annhilator of an element\\
\end{tabular}

\subsection*{Chapter 7}\ \\

\begin{tabular}{p{3cm}l}
$\VV$&category of pairs of vector spaces with a period isomorphism\\
$\Per(V)$&set of periods of an object $V\in \VV$\\
$\Per\langle V\rangle$&space of periods of an object $V\in\VV$\\
$\Per(\Ch)$&set of periods of a catgory\\
$\Perform(\Ch)$&space of formal periods of an additive category\\
$\ev$&evaluation map for formal periods\\
$V^\vee$ &external dual of $V\in\VV$\\
$\End(T)$&endormorphism algebra of a functor $T$ on $\Ch$\\
$\Ah(\Ch,T)$&coalgebra of a functor $T$ on $\Ch$\\
\end{tabular}

\subsection*{Chapter 8}\ \\

\begin{tabular}{p{3cm}l}
$[L\to G]$& $1$-motive with lattice part $L$ and semi-abelian part $G$\\
$\onemot_k$&the category of   iso-$1$-motives over $k$\\
$T_\sing(M)$&integral singular realisation of a $1$-motive\\
$V_\sing(M)$&rational singular realisation of a $1$-motives\\
$M^\natural$&universal vector extension of a $1$-motives\\
$\onemotgen_k$&category of generalised $1$-motives\\
$V_\dR(M)$& de Rham realisation of a $1$-motive\\
$\phi_M$&period isomorphism for a $1$-motive\\
$\MHS_k$&the category of   mixed $\Q$-Hodge structures
over $k$\\
\end{tabular}

\subsection*{Chapter 9}\ \\

\begin{tabular}{p{3cm}l}
$\Per(M)$&set of periods of a $1$-motive\\
$\Per\langle M\rangle$&space of periods of a $1$-motive\\
$\int_\sigma\omega$&period pairing for a $1$-motive\\
$\coLie(G)$&dual of the Lie algebra\\
$\VdR{M}$&dual of the de Rham realisation of $M$\\
$E(M)$&endomorphism algebra for the period realisation of a $1$-motive\\
$V$&functor $\onemot_\Qbar\to\VVneu$\\
\end{tabular}

\subsection*{Chapter 10}\ \\

\begin{tabular}{p{3cm}l}
(GS)&Gelfond-Schneider formulation of Hilbert's 7th problem\\
(B) &Baker formulation of Hilbert's 7th problem\\
\end{tabular}

\subsection*{Chapter 11}\ \\

\begin{tabular}{p{3cm}l}
$\delta(M)$&dimension of the space of periods of $M$\\
\end{tabular}

\subsection*{Chapter 12}\ \\

\begin{tabular}{p{3cm}l}
$H^i(X,Y)$&object in $\VV$ attached to $Y\subset X$\\
$\partial$&(co)boundary map for relative (co)homology\\
$\Per(X,Y,i)$&set of periods for $H^i(X,Y)$\\
$\Per^i$&set of $i$-periods for all $X,Y$\\
$\Z[D]^0$&divisors of degree $0$ supported on $D$\\
$J(C)$& generalised Jacobian of a smooth curve\\
$d_1\DMgm(k,\Q)$&geometric motives of dimension at most $1$\\
$d_1\MMNeff(k,\Q)$&Nori motives of degree at most $1$\\
\end{tabular}

\subsection*{Chapter 13}\ \\

\begin{tabular}{p{3cm}l}
$\Omega^1(C)$&regular algebraic differential forms on a smooth projective curve $C$\\
\end{tabular}

\subsection*{Chapter 14}\ \\

\begin{tabular}{p{3cm}l}
$l(\sigma)$&element of $\Lie(J(C)^\an)$ attached to the chain $\sigma$\\
$l^\circ(\sigma)$&element of $\Lie(J(C^\circ)^\an)$ attached to the chain $\sigma$\\
$[\omega]$& class of a differential form in de Rham cohomology \\
$[\sigma]$&class of a chain in singular homology\\
$\Z_\sigma$&subgroup of $\Z[D]^0$ generated by $\partial\sigma$\\
\end{tabular}

\subsection*{Chapter 15}\ \\

\begin{tabular}{p{3cm}l}
$\delta(M)$&dimension of the period space of a $1$-motive\\
$g(B)$&dimension of the abelian variety $B$\\
$e(B)$&dimension of $\End_\Q(B)$ for an abelian variety $B$\\
$\rk_B(L,B)$&$L$-rank of $M$ with respect to the simple abelian variety $B$\\
$\rk_B(T,M)$&$T$-rank of $M$ with respect to $B$\\
$\rk_{\Gm}(L,M)$& $L$-rank of a Baker motive in $\Gm$\\
$X_\red$&image of the lattice $X(T)$ in $A^\vee(\Qbar)_\Q$\\
$X_\sat$&saturation of $X_\red$\\
$G_\red$&semi-abelian variety with torus lattice $X_\red$\\
$G_\sat$&semi-abelian variety with torus lattice $X_\sat$\\
$L_\red$&image of the lattice $L$ in $A(\Qbar)_\Q$\\
$L_\sat$&saturation of $L_\red$\\
$M_\red$&reduced motive constructed from $M$\\
$M_\sat$&saturated motive constructed from $M$\\
\end{tabular}

\subsection*{Chapter 16}\ \\

\begin{tabular}{p{3cm}l}
$\Per_\tate(M)$&Tate periods for $M$\\
$\Per_2(M)$&periods of the 2nd kind wrt closed paths\\
$\Per_\alg(M)$&algebraic periods\\
$\Per_3(M)$&periods of the 3rd kind wrt closed paths\\
$\Per_\inc(M)$&periods of the 2nd kind wrt non-closed paths\\
$\Per_\mix(M)$&periods of the 3rd kind wrt non-closed paths\\
$\Per_\baker(M)$&periods of the 3rd kind wrt non-closed paths in the Baker case\\
$\delta_?(M)$& dimension of $\Per_?(M)$ \\
\end{tabular}

\subsection*{Chapter 17}\ \\

\begin{tabular}{p{3cm}l}
$\Phi$&the map $L_\Q\tensor X(T)_\Q\to \Per_\mix(M)$\\
$R_1(M)$&primitive relations on $L_\Q\tensor X(T)_\Q$\\
$R_n(M)$&$n$-fold relations on $L_\Q\tensor X(T)_\Q$\\
$R_\inc(M)$&full relation space\\
\end{tabular}

\subsection*{Chapter 18}\ \\

\begin{tabular}{p{3cm}l}
$\wp(z;\Lambda)$&Weierstraß $\wp$-function for the lattice $\Lambda$\\
$g_2$, $g_3$&lattice sums\\
$\zeta(z;\Lambda)$&Weierstraß $\zeta$-function for the lattice $\Lambda$\\
$\sigma(z;\Lambda)$&Weierstraß $\sigma$-function for the lattice $\lambda$\\
$\omega$&standard regular invariant differential on $E$\\
$\eta$&standard differential of the second kind on $E$\\
$\xi_P$&standard differential with simple pole in $P$\\
$F(w,u)$&exponential of $\xi_P$\\
\end{tabular}

\subsection*{Chapter 19}\ \\

\begin{tabular}{p{3cm}l}
$\xi(\lambda)$&elliptic differential form on $\C$\\
$\gamma_{p,q}$&path from $p$ to $q$\\
$I_{p,q}(\lambda)$&Euler integral from $p$ to $q$\\
$C_\lambda$&elliptic curve in Legendre normal form\\
$\omega(\lambda)$&standard invariant differential form on $C_\lambda$\\
$\tilde{\gamma}_{p,q}$&lift of $\gamma_{p,q}$\\
$c_{p,q}$&closed path around $p$ and $p$\\
$\tilde{c}_{p,q}$&lift of $c_{p,q}$\\
$F(1/2,1/2,1;\lambda)$&hypergeometric function\\
$\B(p,q)$&Euler's Beta-function\\
$C$&Legendre family of elliptic curves\\
$\eta(\lambda)$& standard differential form of the second kind on $C_\lambda$\\
$F(a,b,c;\lambda)$&hypergeometric function in general\\
$\omega(a,b,c;\lambda)$&differential form in the Euler integral\\
$\Omega(a,b,c;\lambda)$&Euler integral\\
$\omega(b,c-b)$&differential form in the Beta integral\\
$\B(b,c-b)$&Euler's Beta-function\\
$C_N(\lambda)$& curve with affine equation $y^N=x^r(1-x)^s(1-\lambda x)^t$\\
$C_p$& degeneration $C_p(0)$\\
$X_p(\lambda)$&normalisation of $C_p(\lambda)$\\
$X_p$&normalisation of $C_p$\\
$J_p(\lambda)$&Jacobian of $X_p(\lambda)$\\
$J_p$ &Jacobian of $X_p$\\
$\omega_n^{u,v,w}$&differential form on $X_p(\lambda)$\\
$\omega_n^{u,v}$&differential form on $X_p$\\
$\omega_n$&$\omega^{u,v,w}_n$ or $\omega^{u,v}_n$ for distinguished choice of $u,v,w$\\
$\langle x\rangle$&fractional part of a rational number $x$\\

\end{tabular}

\subsection*{Appendix A}\ \\

\begin{tabular}{p{3cm}l}
$\Q\Vect$&category of finite dimensional $\Q$-vector spaces\\
$E\Mod$&finitely generated left $E$-modules\\
$D$&diagram, i.e. oriented graph\\
$T$&representation of a diagram\\
$\pairseff$&Nori's diagram of (effective) pairs\\
$\MMNeff(k,\Q)$&category of effective Nori motives\\
$H_\sing$&singular realisation of Nori motives\\
$H_\Nori$&representation of $\pairseff$ in Nori motives\\
$H_\hodge$&Hode realisation of Nori motives\\
$H$&period realisation of Nori motives\\
$H_\dR$&de Rham realisation of Nori motives\\
$d_n\MMN(k,\Q)$&degree filtration on Nori motives\\
$\Q(-1)$&Lefschetz motive\\
$\MMN(k,\Q)$&category of all Nori motives\\
$\Per^\eff$&set of all periods of effective Nori motives\\
$E(M)$&endomorphism algebra of $H_\sing|_{\langle M\rangle}$\\
\end{tabular}

\subsection*{Appendix B}\ \\

\begin{tabular}{p{3cm}l}
$\DMgmeff(k,\Q)$&Voevodsky's triangulated category of effective geometric motives\\
$M$&standard functor $\mathrm{Var}_k\to \DMgmeff(k,\Q)$, also representation of $\pairseff_k$\\
$H_\sing$&singular realisation of geometric motives\\
$d_n\DMgmeff(k,\Q)$&filtration by dimension\\

\end{tabular}

\subsection*{Appendix C}\ \\

\begin{tabular}{p{3cm}l}
$V$&functor $\onemot_k\to\VVneu$\\
$\RGamma_\dR(X,Y)$&functorial complexes computing relative de Rham cohomology\\
$\RGamma_\sing(X,Y)$&functorial complexes computing singular cohomology\\
$\RGammatilde_\dR(X,Y,\Uf)$& \v{C}ech complexes computing relative de Rham cohomology\\
$\uG$&Nisnevich sheaf with transfers defined by $G$\\
$\uL$ &Nisnevich sheaf with transfers defined by $L$\\
$M_\gm(\cdot)$&geometric motive defined by a complex of Nisnevich sheaves with transfers\\
$V(M)^\vee$& external dual of $V(M)$\\
\end{tabular}
\bibliographystyle{alpha} 
\bibliography{periods}

\printindex
\end{document}